\documentclass[11pt,letterpaper]{amsart}

\usepackage[top=2.6cm,
bottom=2.6cm,
left=2.3cm,
right=2.3cm]{geometry}

\usepackage[english]{babel}
\usepackage{bm}
\usepackage{setspace}

\usepackage{stmaryrd}
\usepackage{upgreek}
\usepackage{setspace}

\usepackage{color}

\usepackage[utf8]{inputenc}
\usepackage{hyperref}
\usepackage{relsize}
\usepackage{bm}
\usepackage{amssymb,amsmath}
\usepackage{dsfont}
\usepackage[all]{xy}
\usepackage{mathtools}
\usepackage{mathrsfs}
\usepackage{comment}
\usepackage{braket}

\usepackage{combelow}

\renewcommand{\theta}{\uptheta}
\renewcommand{\iota}{\upiota}
\renewcommand{\alpha}{\upalpha}
\renewcommand{\beta}{\upbeta}
\renewcommand{\gamma}{\upgamma}
\renewcommand{\delta}{\updelta}
\renewcommand{\zeta}{\upzeta}
\renewcommand{\pi}{\uppi\hspace{0.05em}}
\renewcommand{\xi}{\upxi}
\renewcommand{\chi}{\upchi}
\renewcommand{\sigma}{\upsigma}
\renewcommand{\Lambda}{\Uplambda}
\renewcommand{\Gamma}{\Upgamma}
\renewcommand{\phi}{\upphi}
\renewcommand{\nu}{\upnu}
\renewcommand{\tau}{\uptau}
\renewcommand{\mu}{\upmu}
\renewcommand{\eta}{\upeta}

\newtheorem{theorem}{Theorem}[section]
\newtheorem{thmx}{Theorem}

\newtheorem{proposition}[theorem]{Proposition}
\newtheorem{lemma}[theorem]{Lemma}
\newtheorem{conjecture}[theorem]{Conjecture}
\newtheorem{corollary}[theorem]{Corollary}

\theoremstyle{definition}
\newtheorem{definition}[theorem]{Definition}
\newtheorem{assumption}[theorem]{Assumption}

\theoremstyle{remark}
\newtheorem{example}[theorem]{Example}
\newtheorem{remark}[theorem]{Remark}

\renewcommand{\AA}{\BoA}

\newcommand{\CC}{\BoC}

\DeclareMathOperator{\B}{B\!}

\DeclareMathOperator{\Yang}{\mathbf{Y}}

\newcommand{\kac}{\mathtt{a}}
\newcommand{\dd}{\mathbf{d}}
\newcommand{\ttt}{\mathbf{t}}
\DeclareMathOperator{\reldim}{\mathrm{reldim}}
\newcommand{\ee}{\mathbf{e}}

\DeclareMathOperator{\vect}{vect}
\newcommand{\JH}{\mathtt{JH}}

\DeclareMathOperator{\IH}{IH}

\let \ol=\overline
\let \ul=\underline

\DeclareMathOperator{\triv}{\scriptscriptstyle{triv}}

\newcommand{\ff}{\mathbf{f}}

\DeclareMathOperator{\shuff}{\mathtt{sh}}

\newcommand{\lazy}{\bm{e}}

\DeclareMathOperator{\opp}{op}
\DeclareMathOperator{\MO}{\mathtt{MO}}

\DeclareMathOperator{\Tang}{T^*\!\!}

\newcommand{\NN}{\BoN}
\newcommand{\PP}{\mathbb{P}}

\newcommand{\QQ}{\mathbb{Q}}

\newcommand{\ZZ}{\BoZ}

\newcommand{\Msp}{\mathcal{M}}

\newcommand{\ICS}{\mathcal{IC}}

\newcommand{\RR}{\BoR}

\DeclareMathOperator{\Tens}{T}

\DeclareMathOperator{\sst}{-ss}

\DeclareMathOperator{\cyc}{cyc}
\DeclareMathOperator{\Free}{\mathbf{Free}}

\DeclareMathOperator{\Hom}{Hom}

\DeclareMathOperator{\End}{End}

\DeclareMathOperator{\Ext}{Ext}

\DeclareMathOperator{\Gr}{\mathbf{Gr}}

\DeclareMathOperator{\T}{T}

\DeclareMathOperator{\Vect}{Vect}

\DeclareMathOperator{\Eu}{Eu}

\DeclareMathOperator{\BM}{BM}

\newcommand{\fg}{\mathfrak{g}}
\newcommand{\fn}{\mathfrak{n}}

\DeclareMathOperator{\Nak}{\mathbf{N}}
\DeclareMathOperator{\NakMod}{\mathbb{N}}

\DeclareMathOperator{\supp}{supp}

\DeclareMathOperator{\Grp}{Grp}
\DeclareMathOperator{\Image}{Im}

\DeclareMathOperator{\codim}{codim}
\DeclareMathOperator{\Perv}{\mathbf{Perv}}

\DeclareMathOperator{\IC}{IH}

\DeclareMathOperator{\Sym}{\mathbf{Sym}}

\DeclareMathOperator{\op}{op}
\DeclareMathOperator{\Spec}{Spec}
\DeclareMathOperator{\Gl}{GL}

\DeclareMathOperator{\UEA}{\mathbf{U}}

\DeclareMathOperator{\id}{id}
\DeclareMathOperator{\Jac}{Jac}

\DeclareMathOperator{\Tr}{Tr}

\DeclareMathOperator{\pt}{pt}

\DeclareMathOperator{\Char}{char}

\DeclareMathOperator{\rk}{rk}

\DeclareMathOperator{\vir}{vir}
\DeclareMathOperator{\Ob}{Ob}
\DeclareMathOperator{\Ho}{\mathcal{H}}
\DeclareMathOperator{\HO}{\mathbf{H}}

\DeclareMathOperator{\K}{K}
\DeclareMathOperator{\sw}{\mathbf{sw}}

\DeclareMathOperator{\Coha}{\mathcal{A}}
\DeclareMathOperator{\HCoha}{\HO\!\mathcal{A}}

\DeclareMathOperator{\NS}{NS}

\DeclareMathOperator{\TS}{\mathtt{TS}}

\newcommand{\vDelta}{\bm{\Delta}}
\newcommand{\vmult}{\bm{m}}

\newcommand{\vact}{\bm{a}}
\newcommand{\vtau}{\bm{\tau}}
\DeclareMathOperator{\BoMo}{BM}
\DeclareMathOperator{\dimred}{\mathtt{dr}}
\DeclareMathOperator{\Lie}{\mathbf{Lie}}
\DeclareMathOperator{\Res}{\mathrm{Res}}

\DeclareMathOperator{\Diag}{\Upxi}

\DeclareMathOperator{\BPS}{\mathtt{BPS}}

\newcommand{\cdotsh}{\!\cdot\!}

\newcommand\ldotsh{\makebox[1em][c]{.\hfil.\hfil.}}
\newcommand{\BPSh}{\mathcal{BPS}}

\newcommand{\Dub}{\mathcal{D}}

\newcommand{\BoN}{\mathbf{N}}
\newcommand{\BoD}{\mathbf{D}}
\newcommand{\BoQ}{\mathbf{Q}}
\newcommand{\BoC}{\mathbf{C}}

\newcommand{\BoZ}{\mathbf{Z}}
\newcommand{\BoA}{\mathbf{A}}
\newcommand{\BoR}{\mathbf{R}}
\newcommand{\BoF}{\mathbf{F}}

\newcommand{\CF}{\mathcal{F}}
\newcommand{\CH}{\mathcal{H}}

\newcommand{\CG}{\mathcal{G}}
\newcommand{\CL}{\mathcal{L}}
\newcommand{\CM}{\mathcal{M}}

\newcommand{\CV}{\mathcal{V}}

\newcommand{\Fn}{\mathfrak{n}}
\newcommand{\Fm}{\mathfrak{m}}
\newcommand{\Ft}{\mathfrak{t}}

\newcommand{\FA}{\mathfrak{A}}
\newcommand{\FB}{\mathfrak{B}}
\newcommand{\FC}{\mathfrak{C}}

\newcommand{\FM}{\mathfrak{M}}
\newcommand{\FP}{\mathfrak{P}}
\newcommand{\FS}{\mathfrak{S}}

\newcommand{\FZ}{\mathfrak{Z}}
\newcommand{\Fp}{\mathfrak{p}}
\newcommand{\Fg}{\mathfrak{g}}

\DeclareMathOperator{\wt}{\mathbf{w}}

\newcommand{\SA}{\mathscr{A}}
\newcommand{\SC}{\mathscr{C}}

\newcommand{\SG}{\mathscr{G}}
\newcommand{\SL}{\mathscr{L}}
\newcommand{\SR}{\mathscr{R}}

\newcommand{\fh}{\mathfrak{h}}
\newcommand{\fp}{\mathfrak{p}}

\newcommand{\phip}[1]{{}^{\mathfrak{p}}\!\phi_{#1}}

\DeclareMathOperator{\EU}{\mathfrak{EU}}

\DeclareMathOperator{\CoHA}{\mathcal{A}}

\DeclareMathOperator{\Abel}{Ab}
\DeclareMathOperator{\Stab}{Stab}
\DeclareMathOperator{\CritStab}{CStab}
\DeclareMathOperator{\Ch}{\mathfrak{C}}
\DeclareMathOperator{\fr}{\mathtt{fr}}
\DeclareMathOperator{\Tt}{\mathbb{T}}
\DeclareMathOperator{\At}{\mathbb{A}}
\DeclareMathOperator{\rmat}{\mathbf{r}}
\DeclareMathOperator{\FNak}{\mathfrak{N}}
\DeclareMathOperator{\iso}{iso}
\DeclareMathOperator{\Frac}{Frac}
\DeclareMathOperator{\cochar}{cochar}
\DeclareMathOperator{\E}{\mathsf{E}}
\newcommand{\fz}{\mathfrak{z}}
\newcommand{\Sh}[1]{\mathcal{#1}}
\DeclareMathOperator{\dAtt}{Att}
\newcommand{\Att}[1]{\dAtt_{\mathfrak{#1}}}
\newcommand{\StabC}[1]{{\Stab}_{\mathfrak{#1}}}
\DeclareMathOperator{\loc}{\mathtt{loc}}
\DeclareMathOperator{\gr}{gr}

\usepackage{tikz-cd}
\tikzset{
    labl/.style={anchor=south, rotate=90, inner sep=.5mm}
}
\tikzset{
    labl2/.style={anchor=south, rotate=-90, inner sep=.5mm}
}
\title[Okounkov's conjecture via BPS Lie algebras]{Okounkov's conjecture via BPS Lie algebras}

\author{Tommaso Maria Botta and Ben Davison}

\date{21st December 2023}

\address{T. M. Botta: Department of Mathematics, Columbia University, MC 4418, 2990 Broadway, New York, USA}
\email{tommaso.botta@columbia.eduemail}

\address{B. Davison: School of Mathematics, University of Edinburgh, Edinburgh EH9 3FD, United Kingdom}
\email{ben.davison@ed.ac.uk}

\begin{document}
\onehalfspacing

\maketitle
\begin{abstract}
Let $Q$ be an arbitrary finite quiver. We use nonabelian stable envelopes to relate representations of the Maulik--Okounkov Lie algebra $\mathfrak{g}^{\MO}_Q$ to representations of the BPS Lie algebra associated to the tripled quiver $\tilde Q$ with its canonical potential. We use this comparison to provide an isomorphism between the Maulik--Okounkov Lie algebra and the BPS Lie algebra.  Via this isomorphism we prove Okounkov's conjecture, equating the graded dimensions of the Lie algebra $\mathfrak{g}^{\MO}_Q$ with the coefficients of Kac polynomials.  Via general results regarding cohomological Hall algebras in dimensions two and three we furthermore give a complete description of $\mathfrak{g}^{\MO}_Q$ as a generalised Kac--Moody Lie algebra with Cartan datum given by intersection cohomology of singular Nakajima quiver varieties, and prove a conjecture of Maulik and Okounkov, stating that their Lie algebra is obtained from a Lie algebra defined over the rationals, by extension of scalars. Finally, we explain how our results suggest the correct definition of critical stable envelopes in vanishing cycle cohomology.
\end{abstract}

\setcounter{tocdepth}{1} 
\tableofcontents
\vspace{-40pt}
\section{Introduction}

\subsection{Background and main results}
Stable envelopes were introduced by Maulik and Okounkov in \cite{MO19}, initially to define and study a geometric R-matrix formalism in the context of cohomology of Nakajima quiver varieties.  Since then they have been extensively developed, across different cohomology theories, and have found beautiful applications wherever the study of geometric representation theory meets symplectic geometry \cite{smirnov2021quantum, danilenko2022stable, danilenko2022quantum} or quantum field theory \cite{SMIRNOVhypertoric, rimanyi20193d, Rimanyi_2019full, kononov2020pursuing, kononov2022pursuing, botta2023bow}.
In addition, stable envelopes are a formidable means to produce new mathematics. Indeed, whenever studied in physically meaningful settings, they have proven a fundamental tool to effectively translate physical predictions to mathematically precise statements, opening the way to outstanding developments in diverse areas such as enumerative geometry \cite{okounkov2017lectures, Okounkov2016QuantumDE}, knot theory \cite{aganagic2023knot}, and the Langlands program \cite{aganagic2018quantum}.

The result of the development of stable envelopes in the book \cite{MO19} is the purely geometric construction of a Yangian-style quantum group associated to any quiver $Q$, defined by applying the R-matrix formalism that arises by considering stable envelopes across variation of torus actions on quiver varieties associated to $Q$.  After the fundamental work \cite{Nak94, Groj96,Nak97,Nak98, Var00} on constructing representations of quantum groups out of the singular cohomology of Nakajima quiver varieties, Maulik and Okounkov showed that one can construct new quantum groups themselves directly out of this singular cohomology and their stable envelopes, using their R-matrices and the formalism of Faddeev, Reshetikhin and Takhtadzhyan \cite{FRT}.

In the event that the underlying quiver Q is of type ADE, the resulting algebra $\Yang^{\MO}_Q$ contains the usual Yangian associated to the underlying graph of $Q$ \cite{mcbreen2013}, while the case of the one-loop quiver recovers quantum affine $\hat{\mathfrak{gl}_1}$, as explained in depth in the second half of \cite{MO19}.  In these cases the theory recovers and sheds new light on the realisation of various classical Lie algebras and their associated Yangians as subalgebras of the endomorphism algebra of cohomology of Nakajima quiver varieties.  For a general quiver $Q$, Maulik and Okounkov demonstrated that there is a Lie subalgebra $\fg^{\MO,T}_Q\subset \Yang^{\MO}_Q$ for which their Yangian satisfies a PBW theorem: there is an isomorphism of underlying graded vector spaces $\Sym(\fg^{\MO,T}_Q[u])\cong \Yang^{\MO}_Q$, just as in the case of ordinary Yangians.  The superscript $T$ refers to a torus: the definition of the Yangian, and the Lie algebra $\fg_Q^{\MO,T}$, involves a choice of torus $T$ which is required to admit a character by which $T$ scales the natural symplectic forms on Nakajima quiver varieties.  In particular, $T$ is required to be nontrivial.

In contrast with the standard source of Yangians, defined in terms of generators and relations, actually accessing the Lie algebra $\fg_Q^{\MO,T}$, finding non-trivial generators, or even calculating the dimensions of its graded components outside of (extended) ADE type, has been an open problem from the beginning of the subject.  Study in this direction has been heavily inspired by the conjecture below, which in turn presupposes a conjecture of Maulik and Okounkov \cite[\S 5.3.2]{MO19} stating that there exists a cohomologically graded, $\BoZ^{Q_0}$-graded Lie algebra $\fg^{\MO}_{Q}$, from which $\fg^{\MO,T}_Q$ is obtained by $\HO_T$-linear extension: $\fg^{\MO,T}_Q\cong \fg^{\MO}_Q\otimes\HO_T$ with $\HO_T\coloneqq \HO^*(\B T,\BoQ)$.  

The conjecture relates the graded dimensions of $\fg^{\MO,T}_Q$ to \textit{Kac polynomials}.  Given a quiver $Q$, dimension vector $\dd\in\BoN^{Q_0}$ and prime power $q=p^n$, Kac defined $\kac_{Q,\dd}(q)$ to be the number of isomorphism classes of $Q$-representations $\rho$ over a field of cardinality $q$ with dimension vector $\dd$ for which $\rho\otimes_{\BoF_q} \overline{\BoF}_q$ is indecomposable as a $\overline{\BoF}_qQ$-module.  He showed that this is in fact a polynomial function $\kac_{Q,\dd}(q)\in\BoZ[q]$ -- we refer the reader to \cite{schiffmann2018kac} for a comprehensive survey of these polynomials, including more recent developments.  Okounkov's conjecture relates the graded dimensions of $\fg^{\MO}_Q$ to the coefficients of these polynomials:
\begin{conjecture}[Okounkov \cite{OConj}]
\label{O_conj}
Let $Q$ be an arbitrary finite quiver, and fix a dimension vector $\dd\in\BoN^{Q_0}$.  There is an equality of generating series
\begin{align*}
\chi_{q^{1/2}}(\fg^{\MO}_{Q,\dd})\coloneqq &\sum_{k \in \BoZ}\dim\left((\fg^{\MO}_{Q,\dd})^k\right)q^{k/2}=\kac_{Q,\dd}(q^{-1}).
\end{align*}
In the above equation, $(\fg^{\MO}_{Q,\dd})^k$ is the $k$-th cohomologically graded piece of the $\dd$-th graded piece of $\fg^{\MO}_{Q}$.
\end{conjecture}
A consequence of the main result of the present article is the following.
\begin{thmx}
\label{OC_thm}
Conjecture \ref{O_conj} is true: for every quiver $Q$ the Maulik--Okounkov Lie algebra is defined over $\BoQ$, and for any dimension vector $\dd\in\BoN^{Q_0}$, there is an equality $\chi_{q^{1/2}}(\fg^{\MO}_{Q,\dd})=\kac_{Q,\dd}(q^{-1})$.
\end{thmx}

The Kac polynomials are central objects in the combinatorics of quiver representations.  They were introduced by Victor Kac in \cite{Kac83}, where it was conjectured that their constant coefficients record the root space multiplicities of associated Kac--Moody Lie algebras, and it was asked what kind of Lie theoretic interpretation the higher coefficients might have.  Kac's constant term conjecture was proved by Tamas Hausel in \cite{Hau10}, while Kac's wider question can be answered via the theory of \textit{BPS Lie algebras}, which we come to below.  

The equality in the theorem is a corollary of an isomorphism between $\fn^{\MO,T,+}_Q$, the positive half of the Maulik--Okounkov Lie algebra, and the BPS Lie algebra $\fg_{\tilde{Q},\tilde{W}}^{T}$.  This Lie algebra, first defined in \cite{QEAs}, has somewhat different origins; it comes from the enumerative geometry of 3-Calabi--Yau categories, and more precisely cohomological Donaldson--Thomas theory \cite{KS2,Da13,Sz14,Kin3}.  Before stating a stronger result (Theorem \ref{main_thm}) regarding these Lie algebras, we provide some background on BPS Lie algebras and their links with enumerative geometry.

Given a 3-dimensional Calabi--Yau category $\mathscr{C}$, a Bridgeland stability condition $\zeta$, and a class $\gamma$ in the numerical Grothendieck group of $\mathscr{C}$, the BPS invariant $\omega^{\zeta}_{\gamma}$ is a numerical invariant ``counting'' \textit{primitive} $\zeta$-semistable objects in $\mathscr{C}$ of class $\gamma$. As such, the BPS invariants provide the building blocks of the DT invariants, which count \emph{all} $\zeta$-semistable objects in class $\gamma$.  The archetypal example involves setting $\mathscr{C}$ to be the category of coherent sheaves on a Calabi--Yau 3-fold \cite{Thomas1}.  In mathematical applications these invariants can often be refined, meaning that we can define a natural polynomial invariant $\Omega^{\zeta}_{\gamma}(q)\in\BoZ[q^{\pm 1}]$ satisfying $\Omega^{\zeta}_{\gamma}(1)=\omega^{\zeta}_{\gamma}$.  

Inspired by work of Harvey and Moore \cite{HM98}, Kontsevich and Soibelman defined in \cite {KS2} the critical cohomological Hall algebra $\HO\!\CoHA^{\zeta}_{\mathscr{C}}$ associated to certain 3CY categories $\mathscr{C}$, namely categories of modules over Jacobi algebras, in order to give a mathematical definition of algebras of BPS states.  This critical cohomological Hall algebra, for which the underlying vector space is the vanishing cycle cohomology of the potential for which the moduli space is the critical locus, reveals a richer structure to the enumerative geometry of $\mathscr{C}$. The refined BPS invariants are encoded in the (virtual) Poincar\'e series of $\HO\!\CoHA^{\zeta}_{\mathscr{C}}$, which are obtained by forgetting the algebra structure and counting dimensions of pieces in the Hodge filtration on the underlying vector space.  It was proved by Mozgovoy \cite{Moz11} that, starting with a quiver $Q$, there is a ``tripled'' quiver $\tilde{Q}$ with potential $\tilde{W}$ such that the refined BPS invariants of the associated Jacobi algebra $\Jac(\tilde{Q},\tilde{W})$ precisely recover the Kac polynomials of $Q$, (see \S \ref{quiver_sec} for the definitions) providing the clue that the right hand side of Conjecture \ref{O_conj} can also be studied via Donaldson--Thomas theory.

To get from critical cohomological Hall algebras/refined BPS state counts to Maulik--Okounkov Yangians, we need to pick the right kinds of 3-Calabi--Yau category $\mathscr{C}$.  Such a choice is provided by the 3-Calabi-Yau completion (in the sense of Keller \cite{Kel11}) of the category of perfect modules for the preprojective algebra $\Pi_Q$ for the quiver $Q$.  This 3-Calabi--Yau completion turns out to be precisely the category of perfect modules over Mozgovoy's Jacobi algebra $\Jac(\tilde{Q},\tilde{W})$ (see also \cite[\S 4.3]{ginz}, and, again, \S \ref{quiver_sec} for definitions).  The relation between the cohomological Donaldson--Thomas theory of the resulting 3-Calabi--Yau category and the Borel--Moore homology of the stack of $\Pi_Q$-modules passes via the method of dimensional reduction, which is a powerful method for deriving results regarding stacks of objects in 2-Calabi--Yau categories from the cohomological Donaldson--Thomas theory of the associated 3-Calabi--Yau category (see \cite{BBS} \cite{Da16}, \cite{Kin22}, \cite{KK21} for applications and further developments).  Here, it yields an isomorphism between $\HO\!\CoHA^T_{\Pi_Q}$, for which the underlying graded vector space is $\HO^{\BoMo}(\FM^T(\Pi_Q),\BoQ)$, and the critical cohomological Hall algebra associated to the pair $(\tilde{Q},\tilde{W})$.

The cohomological Hall algebra $\HO\!\CoHA^T_{\Pi_Q}$ was defined and studied independently by Schiffmann and Vasserot in their solution to the AGT conjecture \cite{ScVa13} (see also \cite{YZ18}). Connections between these Hall algebras and generalised Kac--Moody Lie algebras were developed in \cite{Bo16}; in the presence of edge loops, it was observed in [loc. cit.] that the spherical subalgebra of $\HO\!\CoHA^T_{\Pi_Q}$ already takes us from Kac--Moody algebras into the world of \emph{generalised} Kac--Moody algebras (also known as Borcherds algebras \cite{Bor88}).  The deep relations between these Hall algebras and Kac polynomials are further explored in \cite{BSV17}.  There is now an extensive literature devoted to comparing the CoHAs $\HO\!\CoHA^T_{\Pi_Q}$, or the critical cohomological Hall algebras obtained from associated 3-Calabi--Yau categories, to Yangians and vertex algebras \cite{SchVa17,LY20,GaYa20,GLY21,RSYZ20,RSYZ20b,Da22,BuRa23,MNSV23,jindal2024coha}.  

The existing state of the art for general comparisons between preprojective CoHAs and the Maulik--Okounkov Yangian is the following result of Schiffmann and Vasserot \cite{ScVa20}; after applying the functor $(-)_K=-\otimes_{\HO_T} K$, where $K$ is the function field of $\HO_T$, they showed that there is an embedding $\jmath\colon( \HO\!\CoHA_{\Pi_Q}^T)_K\hookrightarrow (\Yang^{\MO,+}_{Q})_K$, using generation of the domain by specific elements of $(\HO\!\CoHA_{\Pi_Q,n\delta_i}^T)_K$. Here, $n\delta_i=(0,\ldots,0,n,0,\ldots,0)$ is an arbitrary dimension vector supported on the $i$-th vertex of the quiver.  The generation statement was subsequently strengthened by Negu\cb{t}, who showed that it is sufficient to choose $n=1$ \cite{neguct2023generators} to generate the algebra $(\HO\!\CoHA_{\Pi_Q}^T)_K$. The target $\Yang^{\MO,+}_{Q}$ is the positive half of the Maulik--Okounkov Yangian, defined with respect to the same torus $T$ appearing on the left hand side; it is generated by the positive half of $\fg^{\MO,T}_Q$ and tautological classes. It was further conjectured in \cite{ScVa20} that the inclusion $\jmath$ is in fact an isomorphism; this conjecture follows from Corollary \ref{Y_cor} below.  We remark here that without localising, both the cohomological Hall algebra, and the Maulik--Okounkov Yangian are almost never spherically generated (see e.g. \cite[Prop.5.7]{Da22} and \cite[Cor.11.5.2]{jindal2024coha}) and identifying the non-spherical generators of the two algebras is an essentially new challenge.  See \cite{schiffmann2018kac} for an overview of the many connections between preprojective CoHAs, Kac polynomials and Yangians, and in particular, background to \cite[Conj.4.10]{schiffmann2018kac}, which we prove below as Theorem \ref{main_thm}.  

Both $\fn^{\MO,T,+}_Q$ and $\fn_{\Pi_Q}^{T,+}\coloneqq \fg_{\tilde{Q},\tilde{W}}^T$ are positive halves in weight space decompositions of bigger Lie algebras $\fg^{\MO,T}_Q$ and $\fg_{\Pi_Q}^T$ respectively, and our second main theorem, from which Theorem \ref{OC_thm} follows, is a comparison between these bigger Lie algebras:

\begin{thmx}
\label{main_thm}
There is an isomorphism of $\BoZ^{Q_0}$-graded, cohomologically graded Lie algebras $\fg^{T}_{\Pi_Q}\cong \fg^{\MO,T}_Q$ intertwining the natural actions on cohomology of Nakajima quiver varieties.
\end{thmx}
We note that the cohomology of quiver varieties is isomorphic to the critical cohomology of certain moduli spaces of $\Jac(\tilde{Q},\tilde{W})$-representations via dimensional reduction, so that our results give a somewhat surprising action of $\fg^T_{\MO}$ in \textit{critical} cohomology; an intriguing question is the extent to which the (doubled) Lie algebra $\fg^T_{\MO}$ acts on the critical cohomology of \textit{other} moduli spaces of $\Jac(\tilde{Q},\tilde{W})$-modules.

We deduce the result of Theorem \ref{main_thm} regarding the full Lie algebras from the analogous result regarding the positive halves, which, in turn, is obtained by combining two main ingredients. The first one is the set of results on the representation theory of (full) BPS Lie algebras from \cite{DHSM23}.  In particular we use from [loc. cit.] the description of the cohomology of Nakajima quiver varieties in terms of direct sums of irreducible lowest weight representations for the full BPS Lie algebra. The second main ingredient is the theory of nonabelian stable envelopes, which we come to below. They form the object of our final main theorem (Theorem \ref{NaSE_vs_BPS}).

Thanks to the papers \cite{preproj3,DHSM22,DHSM23} the structure of $\fg_{\Pi_Q}^T$ is well-understood: it is a generalised Kac--Moody Lie algebra with infinite Cartan datum.  Theorem \ref{main_thm} enables us to transfer the results of these papers into the study of Maulik--Okounkov Yangians, and we summarise the outcome in the following corollaries.
\begin{corollary}
\label{OQ_cor}
The Lie algebra $\fg_Q^{\MO,T}$ is obtained via $\HO_T$-linear extension of scalars from a cohomologically graded, $\BoZ^{Q_0}$-graded Lie algebra $\fg_Q^{\MO}$ with finite-dimensional graded pieces.  Moreover, the Lie algebra $\fg_Q^{\MO}$ is a generalised Kac--Moody Lie algebra with spaces of (positive) Chevalley generators canonically identified with intersection cohomologies of singular quiver varieties.
\end{corollary}
See \S \ref{GKM_thm_sec} for a more thorough description of the spaces of Chevalley generators for the Maulik--Okounkov Lie algebras.
\begin{corollary}
\label{Y_cor}
There is an isomorphism of $\BoN^{Q_0}$-graded, cohomologically graded algebras $\Yang_Q^{\MO,+}\cong \HO\!\CoHA_{\Pi_Q}^T$, intertwining the respective actions on cohomology of Nakajima quiver varieties.
\end{corollary}
Shortly after this paper appeared, an independent proof of Corollary \ref{Y_cor}, which implies Conjecture \ref{O_conj} using the calculation of the graded dimensions of $\HO\!\CoHA_{\Pi_Q}^T$ in \cite{preproj}, appeared in \cite{SV23O}.

Conversely, results from \cite{MO19} regarding the representation theory of $\Yang^{\MO}_Q$ allow us to settle open problems regarding the CoHA $\HO\!\CoHA_{\Pi_Q}^T$, for example with the following corollary.  We refer the reader to \S \ref{sec: Nakajima quiver varieties} for the definition of Nakajima quiver varieties and their stack-theoretic quotients, and \S \ref{section: Stable envelopes for quiver varieties} for a discussion of the extended torus $T_{\ff}$.
\begin{corollary}
\label{faithful_cor}
Let $\NakMod^{T_{\bullet}}_Q=\bigoplus_{\ff\in\BoN^{Q_0}}\NakMod^{T_{\ff}}_{Q,\ff}=\bigoplus_{\ff,\dd\in\BoN^{Q_0}}\HO^*(\Nak^{T_{\ff}}_Q(\ff,\dd),\BoQ)$ denote the $T_{\ff}$-equivariant cohomology of all Nakajima quiver varieties $\Nak_Q(\ff,\dd)$  for $Q$, across all choices of framings $\ff\in\BoN^{Q_0}$.  The natural $\HO\!\CoHA_{\Pi_Q}^T$-action on $\NakMod^{T_{\bullet}}_Q$ is faithful.
\end{corollary}

Our proof of Theorem \ref{main_thm} makes essential use of \textit{nonabelian stable envelopes \cite{aganagic2017quasimap, Okounkov_2021}}. This notion traces back to Okounkov's approach to K-theoretic curve counts \cite{okounkov2017lectures} and is closely related to both the representation theoretic and the physical interpretation of the latter. From the mathematical perspective, these counts enumerate \emph{quasimaps} $\mathcal{C} \dashrightarrow X$ from a curve $\mathcal{C}$ to a quiver variety $X$ \cite{CIOCANFONTANINE201417} and provide a possible systematic approach to quantum K-theory that parallels Givental's theory \cite{PUSHKAR219, Koroteev_2021}. In cohomology, the nonabelian stable envelope can be viewed as a choice of section to the Kirwan map from the ring of tautological classes, and thus provides a constructive, explicit approach to Kirwan surjectivity (see e.g. \cite{Kirwan1,MR3773791} for background on Kirwan surjectivity).

In quasimap-theory, marked points on the curve $\mathcal{C}$ correspond to \emph{insertions} in the count and come in three different flavors—descendant, relative, and nonsingular—corresponding to different geometric constraints imposed on the moduli space of quasimaps $\mathsf{MQ}(\mathcal{C},X)$  at each marked point. 
The connection with representation theory follows from the fact that, for $\mathcal{C}=\PP^1$, counting quasimaps to a type $Q$ quiver variety with a relative insertion at $0$ and a nonsingular insertion at $\infty$ gives a fundamental solution of the quantum difference equation and the qKZ equations for the quantum affine algebra $\UEA_{\hbar} (\hat \fg^{\MO}_Q)$ \cite{okounkov2017lectures, Okounkov2016QuantumDE}.

Relative and non-singular insertions on the one side and descendant insertions on the other are drastically different: the former live in the K-theory $\K_T(X)$ of the quiver variety $X$, while the latter live in the K-theory of its ambient stack $\mathfrak{X}\supset X$ of all representations, which is much larger. Counts with descendant insertions directly link physics and quiver varieties, now seen as Higgs branches of $3d$ gauge theories with $\mathcal{N}=4$ supersymmetry. Indeed, it is shown in \cite{aganagic2018quantum} that quasimap counts with descendant insertions can be written as Mellin--Barnes type integrals and identified with the supersymmetric partition function of the underlying QFT \cite{moore_integrating, Nekrasov_counting}.

Despite living in different spaces, relative and descendant counts can be related: the nonabelian stable envelope provides the bridge.
In its original formulation, the nonabelian stable envelope is a map $ \Psi: \K_T(X)\to \K_T(\mathfrak{X})$ providing an extension of a class on $X\subset \mathfrak{X}$ to the ambient stack $\mathfrak{X}$ that \emph{preserves} the counts. In other words, the relative count with insertion $\alpha \in \K_T(X)$ equals the descendant count with insertion $\Psi(\alpha) \in \K_T(\mathfrak{X})$ \cite{aganagic2017quasimap}. As a remarkable consequence, the nonabelian stable envelope turns the Mellin--Barnes type integrals into solutions of difference equations, hence recovering and vastly generalising the integral presentations of qKZ and quantum difference equations known in the representation theory literature.

In this article, we recast the nonabelian stable envelope in the cohomological setting and show that it is the natural answer to another comparison problem. Namely, we show that it identifies the cohomology of the quiver variety $\Nak^T_Q(\ff,\dd)$ with the appropriate graded piece of the BPS Lie algebra for $\Pi_{Q_{\ff}}$, the preprojective algebra of the framed quiver $Q_{\ff}$. 
This graded piece of the BPS Lie algebra is identified with a subspace of the Borel--Moore homology of the stack $\mathfrak{X}=\FM^T_{(\dd,1)}(\Pi_{Q_{\ff}})$ of $\Pi_{Q_{\ff}}$-modules.  More precisely, we prove the following theorem.  As in Corollary \ref{faithful_cor}, we define $T_{\ff}$ to be the fiber product of $T$ with the torus rescaling the framing arrows, while we define $T_{\ff}^0=T_{\ff}/\BoC^*$ to be the quotient by the diagonal torus scaling all of the framing arrows with the same weight, and thus acting trivially on all Nakajima quiver varieties.  We refer the reader to \S \ref{section: Stable envelopes for quiver varieties} for a full discussion of these tori and their actions.
\begin{thmx}[Theorem \ref{NakBPS_prop}]
\label{NaSE_vs_BPS}
Let $Q$ be a quiver, and let $\dd,\ff\in\BoN^{Q_0}$ be dimension vectors with $\ff$ nonzero.  Then the nonabelian stable envelope
\[
\Psi_{\ff,\dd}\colon\HO^*\!\left(\Nak^{T_{\ff}}_Q(\ff,\dd),\BoQ\right)\rightarrow \HO\!\CoHA_{\Pi_{Q_{\ff}},(\dd,1)}^{T_{\ff}^0}
\]
is injective, and the image of its restriction to $\HO^*\!\left(\Nak^{T^0_{\ff}}_Q(\ff,\dd),\BoQ\right)$ is the degree $(\dd,1)$ piece of the BPS Lie algebra for the category of $\Pi_{Q_{\ff}}$-modules: 
\[
\Psi_{\ff,\dd}\!\left(\HO^*\!\left(\Nak^{T_{\ff}^0}_Q(\ff,\dd),\BoQ\right)\right)=\Fn^{{T_{\ff}^0},+}_{\Pi_{Q_{\ff}},(\dd,1)}\subset \HO\!\CoHA_{\Pi_{Q_{\ff}},(\dd,1)}^{T_{\ff}^0}.
\]
Specialising to $T$-equivariant cohomology, we deduce that the nonabelian stable envelope
\[
\Psi'_{\ff,\dd}\colon\HO^*\!\left(\Nak^{T}_Q(\ff,\dd),\BoQ\right)\rightarrow \HO\!\CoHA_{\Pi_{Q_{\ff}},(\dd,1)}^{T} 
\]
is an isomorphism onto its image, which is precisely the BPS cohomology $\Fn^{{T},+}_{\Pi_{Q_{\ff}},(\dd,1)}\subset \HO\!\CoHA_{\Pi_{Q_{\ff}},(\dd,1)}^{T}$.
\end{thmx}

Since the stability condition that forms part of the definition of the Nakajima quiver varieties is generic, there is a canonical identification between $\HO^*\!\left(\Nak^{T_{\ff}}_Q(\ff,\dd),\BoQ\right)$ and the BPS cohomology (see \S \ref{BPS_coh_sec} for details).  With Corollary \ref{ThmC_cor} we strengthen Theorem \ref{NaSE_vs_BPS} by showing that under this identification, the nonabelian stable envelope $\Psi_{\ff,\dd}$ is equal to the inclusion morphism $\Fn^{{T_{\ff}^0},+}_{\Pi_{Q_{\ff}},(\dd,1)}\hookrightarrow \HO\!\CoHA_{\Pi_{Q_{\ff}},(\dd,1)}^{T_{\ff}^0}$.

Restricting attention to framing dimension vectors $\ff$ of the form $\delta_i=(0,\ldots,0,1,0,\ldots,0)$ for $i\in Q_0$ there is no difference between $T$ and $T^0_{\ff}$.  By construction, the $T$-equivariant cohomology of Nakajima quiver varieties for these dimension vectors form modules for the Lie algebra $\fg^{\MO,T}_{Q}$, while the image $\bigoplus_{\dd\in\BoN^{Q_0}}\Fn^{T,+}_{\Pi_{Q_{\ff}},(\dd,1)}$ forms a natural module for the BPS Lie algebra $\bigoplus_{\dd\in\BoN^{Q_0}}\Fg^{T}_{\Pi_{Q},\dd}\cong \bigoplus_{\dd\in\BoN^{Q_0}}\Fg^{T}_{\Pi_{Q_{\ff}},(\dd,0)}$.  Theorem \ref{NaSE_vs_BPS} thus states that nonabelian stable envelopes provide the means to intertwine the representation theory of the Maulik--Okounkov Lie algebra and the BPS Lie algebra, finally allowing us to directly compare these Lie algebras.

\subsection{The structure of the proofs of Theorems \ref{OC_thm}, \ref{main_thm} and \ref{NaSE_vs_BPS}}
\label{subsec: strategy for thm C}
In the paper, we first prove Theorem \ref{NaSE_vs_BPS}, then use it to prove Theorem \ref{main_thm}.  Theorem \ref{OC_thm} then follows from known facts about the BPS Lie algebra $\fg_{\tilde{Q},\tilde{W}}$ (Propositions \ref{BPS_char_function} and \ref{Triv_BPS_ext} along with Definition \ref{BPSLAdef}).

So we start with Theorem \ref{NaSE_vs_BPS}.  Since it is known that the graded dimensions of the source and target are the same, it is sufficient to prove injectivity and surjectivity of $\Psi_{\ff,\bullet}$.  Injectivity follows from our recasting (in \S \ref{more_SE_sec}) of the nonabelian stable envelope as a correspondence given by a limit of diagonal cycles --- see Corollary \ref{corollary normalization nonabelian stable envelope} for the precise statement.  

The first idea of the proof of surjectivity is to lift the stable envelope map to a morphism of complexes of perverse sheaves on the stack-theoretic $T_{\ff}$-quotient of the affinization of $\Nak_Q(\ff,\dd)$, meaning that we find a morphism of such complexes that recovers the nonabelian stable envelope after passing to derived global sections.  Our goal then is to show that the image of this lifted map is exactly the perverse sheaf lifting $\Fn^{{T},+}_{\Pi_{Q_{\ff}},(\dd,1)}$.  

This is proved in two steps.  Observe that $\Fn^{{T},+}_{\Pi_{Q_{\ff}},(\bullet,1)}\coloneqq \bigoplus_{\dd\in\BoN^{Q_0}} \Fn^{{T},+}_{\Pi_{Q_{\ff}},(\dd,1)}$ may be considered as a module over $\Fn^{{T},+}_{\Pi_Q}=\bigoplus_{\dd\in\BoN^{Q_0}} \Fn^{{T},+}_{\Pi_{Q},\dd}$.
\begin{itemize}
\item{Step 1 (generators):} We show that the image of $\Psi_{\ff,\bullet}$ contains the lowest weight generators of $\Fn^{{T},+}_{\Pi_{Q_{\ff}},(\bullet,1)}$.  This follows from their identification in terms of intersection complexes (see \S \ref{GKM_thm_sec}), and becomes a straightforward consequence of injectivity of $\Psi_{\ff,\bullet}$ (Proposition \ref{ChevG_prop}).
\item{Step 2:} We show that the image of $\Psi_{\ff,\bullet}$ is closed under the $\Fn^{{T},+}_{\Pi_Q}$ action.  This uses an inductive argument, along with compatability between the nonabelian stable envelope, and its (localised) inverse with the product, and coproduct (respectively) in an extended cohomological Hall algebra $\HO\!\CoHA_{\Pi_{Q_{\fr}}}^{T_{\bullet}}$.  
\end{itemize}
This extended cohomological Hall algebra is introduced in \S \ref{SE_to_COHA}, where it is shown in Proposition \ref{mult_compat} that the nonabelian stable envelope intertwines the stable envelope with the multiplication $\vmult$ in $\HO\!\CoHA_{\Pi_{Q_{\fr}}}^{T_{\bullet}}$.  The situation with the coproduct is a little more subtle, and it is generally \textit{not} the case that the diagram obtained from \eqref{eq:diagram compatibility product} by reversing the horizontal arrows and replacing $\Stab_+$ by its (localised) inverse, and the product with the coproduct $\Delta$, commutes.  Lemma \ref{cor_pack}, informally, says that the diagram obtained this way \textit{does} commute up to low order in certain variables.  This gives us the tool required to compare Maulik--Okounkov R matrices with the composition $\Delta\circ \vmult$, at least at the level of \textit{classical} r matrix.  The rest of the proof of Step 2 leverages this compatibility.

Since $\Fn^{{T},+}_{\Pi_{Q_{\ff}},(\bullet,1)}$ can be viewed as a module over the (doubled) Lie algebra $\Fn^{{T}}_{\Pi_Q}$, and by definition, the vector space $\NakMod^{T_{\bullet}}_Q$ forms a module for $\fg_{Q}^{\MO,T}$, after identifying these vector spaces using Theorem \ref{NaSE_vs_BPS} we may consider the two Lie algebras as Lie subalgebras of the same Lie algebra of endomorphisms of a fixed vector space.  The proof of Theorem \ref{main_thm} in \S \ref{main_thm_sec} uses the same analysis as the proof of Theorem \ref{NaSE_vs_BPS} to show that they define the \textit{same} Lie subalgebra.

\subsection{Critical stable envelopes}
The set of critical cohomological Hall algebras that we show are isomorphic to positive halves of Maulik--Okounkov Yangians is very special; in fact for \textit{any} quiver $Q$ with potential $W\in \BoC Q/[\BoC Q,\BoC Q]$, preserved by the action of some torus $T$ there is an equivariant critical cohomological Hall algebra $\HO\!\CoHA^T_{Q,W}$ (see \S \ref{gen_CoHA_sec}).  Moreover, as long as $Q$ is symmetric, the CoHA $\HO\!\CoHA^T_{Q,W}$ satisfies a PBW theorem, like that satisfied for the Maulik--Okounkov Yangian.  More generally we consider instead the CoHA $\HO\!\CoHA^{T,\zeta}_{Q,W,\theta}$ associated to some suitably generic stability condition $\zeta$ and slope $\theta$, and then we may drop the condition that $Q$ is symmetric and we still have the same kind of PBW theorem (Theorem \ref{PBW_thm}).  This suggests the question of whether there is a ``critical'' version of the Maulik--Okounkov stable envelope construction; one of the main consequences of Theorem \ref{NaSE_vs_BPS} is that the answer is yes, and so this paper should be seen as motivating and justifying a general construction of ``critical stable envelopes'' whenever we have a critical cohomological Hall algebra at hand.

For quivers with potential, the construction is as follows.  Let $Q$ be a quiver, let  $W\in\BoC Q/[\BoC Q,\BoC Q]$ be a potential, and pick a vertex $v\in Q_0$ and a sufficiently generic stability condition $\zeta\in\BoQ^{Q_0}$.  For a given positive integer $r$, let us assume that we have a function $\wt\colon Q_1\rightarrow \BoZ^r$ such that $\wt(a)=\pm\delta_m=(0,\ldots,0,\pm 1,0\ldots,0)$ if $a$ is an arrow which is not a loop and has source or target equal to $v$, and $\wt(a)=0$ otherwise. Let us also assume that $W$ has weight $0$ with respect to $\wt$.  We define the ``abelianization'' $Q^{\Abel}$ of $Q$ with respect to this data to be the quiver with vertex set $(Q_0\setminus\{v\}) \cup\{\infty_1,\ldots,\infty_r\}$, and with arrows to and from the new vertices determined by the weighting $\wt$: see \S \ref{extended_quiver_sec} for the exact construction, and the recipe for producing a potential $W^{\Abel}$ for $Q^{\Abel}$ out of a potential $W$ for $Q$.  We furthermore assume that we are given a decomposition of $r$, i.e. $r=r_1+\ldots+r_l$ where $1\leq r_m$ for every $m$, and a decomposition, of the same length, of some fixed dimension vector $\dd\in\BoN^{Q_0\setminus \{v\}}$, i.e. an expression $\dd=\dd^{(1)}+\ldots+\dd^{(l)}$ with $\dd^{(m)}\in\BoN^{Q_0\setminus \{v\}}$.  From this starting data, we define the quiver $Q^{(m)}$ to be the quiver obtained from $Q$ by deleting all of the arrows to and from $v$ that are not loops and do not satisfy $\wt(a)=\pm\delta_m$.  Representations of this quiver are acted on by the torus $A^{(m)}=(\BoC^*)^{r_m}$ acting on the arrows with weights determined by $\wt$; this is the restriction to $Q^{(m)}$ of the $A=(\BoC^*)^r$ action on $Q$, acting with weights also determined by $\wt$.  We define the dimension vector $\dd^{[m]}$ for $Q^{\Abel}$ by extending $\dd^{(m)}$ by setting the value at $\infty_{m'}$ to be $1$ if $r^{(1)}+\ldots +r^{(m-1)}<m'\leq r^{(1)}+\ldots +r^{(m)}$ and $0$ otherwise.  We denote by $(\dd,\underline{1})$ the dimension vector for $Q^{\Abel}$ obtained by extending $\dd$ by setting the value at $\infty_m$ to be $1$ for all $m$.  Since we choose $\zeta$ to be generic, for each $m$ we may identify $\HO^*(\CM^{A^{(m)},\zeta\sst}_{(\dd^{(m)},1)}(Q^{(m)}),\phip{\Tr(W)}\BoQ^{\vir})$ with BPS cohomology for the quiver $Q^{\Abel}$ with potential $W^{\Abel}$ and dimension vector $\dd^{[m]}$.  We have denoted by $\phip{\Tr(W)}\BoQ^{\vir}$ the perverse sheaf of vanishing cycles for the function $\Tr(W)$.  For each $m$ there is an inclusion
\[
\jmath\colon \HO^*(\CM^{A^{(m)},\zeta\sst}_{(\dd^{(m)},1)}(Q^{(m)}),\phip{\Tr(W)}\BoQ^{\vir})\hookrightarrow \HO^*(\FM_{\dd^{[m]}}(Q^{\Abel}),\phip{\Tr(W)}\BoQ^{\vir})
\]
which is defined via the inclusion of BPS cohomology, along with the action of tautological classes: see \S \ref{BPS_coh_sec} for details.
Then we define the \textit{critical stable envelope} associated to this data as the morphism $\CritStab=\mathrm{res}\circ\vmult\circ\jmath^{\otimes l}$ making the following diagram commute:
\[
    \begin{tikzcd}
    \bigotimes_{m=1}^l\HO^*(\CM^{A^{(m)},\zeta\sst}_{(\dd^{(m)},1)}(Q^{(m)}),\phip{\Tr(W)}\BoQ^{\vir})\arrow[d,"\jmath^{\otimes l}"]\arrow[r,"\CritStab"]&\HO^*(\CM^{A,\zeta\sst}_{(\dd,1)}(Q),\phip{\Tr(W)}\BoQ^{\vir})\\
    \bigotimes_{m=1}^l\HO^*(\FM_{\dd^{[m]}}(Q^{\Abel}),\phip{\Tr(W^{\Abel})}\BoQ^{\vir})\arrow[r,"\vmult"]&\HO^*(\FM_{(\dd,\underline{1})}(Q^{\Abel}),\phip{\Tr(W^{\Abel})}\BoQ^{\vir}).\arrow[u,"\mathrm{res}"]
    \end{tikzcd}
\]
The morphism $\vmult$ is the CoHA multiplication recalled in \S \ref{gen_CoHA_sec}, while $\mathrm{res}$ is the restriction map in vanishing cycle cohomology.

Choosing $\tilde{Q}_{\ff}$ to be the tripled quiver obtained from some framed quiver, and $\tilde{W}_{\ff}$ to be the canonical cubic potential on this quiver (see \S \ref{quiver_sec} for exact definitions), we recover the Maulik--Okounkov stable envelope from this construction via dimensional reduction; this is the bridge that we build and then use in this paper in order to compare the Maulik--Okounkov Yangian with certain critical CoHAsin this paper.  More generally, though, this recipe gives a definition of the critical stable envelope for which we can take much more general choices of $Q$ and $W$, since it is built solely from the inclusion of BPS cohomology and the CoHA multiplication in the Kontsevich--Soibelman critical cohomological Hall algebra.
\subsection{Conventions}
\label{conventions_sec}
Unless otherwise specified, all schemes (and, more generally, algebraic stacks) are defined over the complex numbers $\BoC$. Accordingly, $\dim(X)$, $\codim_Y(X)$ etc. are the complex dimensions.

For a fixed torus $T$, let $X$ be a $T$-equivariant variety, or let $\FM$ be a stack presented as a global quotient $Y/(G\times T)$ for some algebraic group $G$, and some variety $Y$.  Then we define the shifted perverse t-structure on $\Dub_c(X/T)$ and $\Dub_c(\FM)$ by setting ${}^{\fp'}\!\vtau^{\leq i}\coloneqq {}^{\fp}\vtau^{\leq i-\dim(T)}$ and ${}^{\fp'}\!\vtau^{\geq i}\coloneqq {}^{\fp}\vtau^{\geq i-\dim(T)}$.  Objects in the heart of this t-structure are precisely those complexes $\CF$ such that $(X\rightarrow X/T)^*\CF$ is perverse, in the variety case, or $(Y/G\rightarrow Y/(G\times T))^*\CF$ is perverse in the stack case.  We denote the heart of this t-structure $\Perv'(X)$ in the variety case, $\Perv'(\FM)$ in the stack case.  We then take the shift of the Verdier duality functor $\BoD$ preserving this heart.

Given a space $X$, we denote by $\Dub^+(\Perv(X))$ the full subcategory of the unbounded derived category of constructible complexes on $X$ containing those objects $\CF$ such that for each connected component $X'\subset X$ the cohomology sheaves ${}^{\Fp}\CH^i(\CF)$ vanish for $i\ll 0$.

Given a vector bundle $V$ on a space $X$ we denote by $\Eu(V)\in\HO^{2\rk(V)}(X,\BoQ)$ the Euler class of $V$.

If the definition of some object $P^T$ depends on a choice of a complex torus $T$, for which it makes sense to pick $T=\{1\}$ the trivial torus, and we write $P$ (i.e. the torus is omitted from the notation) this means that we make the trivial choice $T=\{1\}$.

If the definition of some object $P_{\dd}$ depends on the choice of a dimension vector $\dd\in \BoN^{Q_0}$ for a quiver $Q$, and the dimension vector $\dd$ is omitted, or replaced with $\bullet$, then the direct sum over all $\dd\in\BoN^{Q_0}$ is intended.  If the definition of some object $P^{\zeta}$ depends on a stability condition $\zeta\in\BoQ^{Q_0}$ for a quiver $Q$, and the stability condition $\zeta$ is omitted, the degenerate stability condition $\zeta_{\mathrm{deg}}=(0,\ldots,0)$, for which all $\BoC Q$-modules are semistable, is intended.

Given an algebraic group $G$ we denote by $\B G$ the stack theoretic quotient $\pt/G$.  We write $\HO_G\coloneqq \HO(\B G,\BoQ)$.

If $\CL$ is a local system on a locally closed smooth irreducible substack $Z$ of a stack $X$, possibly defined as a quotient of our fixed torus $T$, we write $\ICS_{\overline{Z}}(\CL)$ for the intermediate extension of $\CL$ to $\overline{Z}$, considered as a complex on $X$ via the direct image along the closed embedding $\overline{Z}\hookrightarrow X$.  To be precise, regarding degree shifts, $\ICS_{\ol{Z}}(\CL)[\mathrm{reldim}_{\B T}(Z)]$ is the intermediate extension of the (equivariant) perverse sheaf $\CL[\mathrm{reldim}_{\B T}(Z)]$.  If $X$ is integral we denote by $\ICS(X)=\ICS_{X}(\BoQ_{X^{\mathrm{sm}}})[\mathrm{reldim}_{\B T}(X)]$ the intersection complex on $X$.  This is a simple perverse sheaf.

Given an algebra $A$ and an $A$-module $M$ we define
\[
\Tens_A(M)\coloneqq\bigoplus_{n\geq 0}M^{\otimes_A n}; \quad \quad\quad M^{\otimes_A n}\coloneqq \underbrace{M\otimes_A\cdots\otimes_A M}_{n\textrm{ times}}
\]
which we consider as a functor from $A$-modules to $A$-algebras in the standard way.

We set $\BoN\coloneqq \BoZ_{\geq 0}$.
\subsection{Acknowledgements}
BD's research is funded by a Royal Society University Research Fellowship. TMB is supported as a part of NCCR SwissMAP, a National Centre of Competence in Research, funded by the Swiss National Science Foundation (grant number 205607) and grant 200021\_196892 of the Swiss National Science Foundation. Special thanks are due to Davesh Maulik and Andrei Okounkov for years of inspiration and encouragement, and for producing the monumental \cite{MO19} in the first place.  Thanks are also due to Giovanni Felder for countless discussions about Yangians and their geometric representation theory.

\section{Quivers, algebras, representations}
\subsection{Quivers and dimension vectors}
A \textit{quiver} $Q=(Q_1,Q_0,s,t)$ is given by a set of arrows $Q_1$, a set of vertices $Q_0$, and two maps $s,t\colon Q_1\rightarrow Q_0$ taking an arrow to its source and target, respectively.  Given a quiver $Q$ we define $Q^{\opp}$ to be the quiver satisfying $Q^{\opp}_0=Q_0$, and for which there is an arrow $a^*$ for each $a\in Q_1$, with $s(a^*)=t(a)$ and $t(a^*)=s(a)$.  We say that a quiver $Q$ is \textit{symmetric} if for all $i,j\in Q_0$ the number of arrows $a\in Q_1$ with $s(a)=i$ and $t(a)=j$ is equal to the number of arrows with $s(a)=j$ and $t(a)=i$.  This is equivalent to the existence of an isomorphism $Q\cong Q^{\opp}$ acting as the identity on $Q_0$.

We define the set of \emph{dimension vectors} to be the monoid $\BoN^{Q_0}$. We give this monoid the natural partial order: $\dd>\dd'$ if $\dd\neq \dd'$ and for all $i\in Q_0$ there is an inequality $\dd_i\geq \dd'_i$.
The \emph{support} $\supp(\dd)$ of a dimension vector $\dd\in\BoN^{Q_0}$ is the set of $i\in Q_0$ such that $\dd_i\neq 0$.  Given a number $n\in\BoN$ we denote by $\underline{n}\in\BoN^{Q_0}$ the constant dimension vector satisfying $\underline{n}_i=n$ for all $i\in Q_0$.  

We define the Euler form
\[
\chi_Q\colon\BoN^{Q_0}\times\BoN^{Q_0}\rightarrow \BoZ;\quad\quad
(\dd,\ee)\mapsto\sum_{i\in Q_0} \dd_i\ee_i-\sum_{a\in Q_1}\dd_{s(a)}\ee_{t(a)}
\]
and we define $(-,-)_Q$ to be its symmetrisation: $(\dd',\dd'')_Q\coloneqq\chi_Q(\dd',\dd'')+\chi_Q(\dd'',\dd')$.  We denote by $\dd'\cdotsh\dd''\coloneqq\sum_{i\in Q_0}\dd'_i\dd''_i$ the usual dot product of dimension vectors.

An edge loop is an arrow $a$ with $s(a)=t(a)$.  We define the set of \emph{simple real roots} to be the set of dimension vectors $\delta_i\in\BoN^{Q_0}$ for $i$ a vertex not supporting any edge loops.  Here and elsewhere we denote by $\delta_i$ the dimension vector satisfying $(\delta_i)_j=1$ if $i=j$ and $0$ otherwise.  

We define the set of \textit{primitive positive roots}
\[
\Sigma_{Q}\coloneqq \left\{\dd\in \BoN^{Q_0}\lvert 2-(\dd,\dd)_Q>\sum_{r=1}^t (2-(\dd^{(r)},\dd^{(r)})_Q) \textrm{ for all }(\dd^{(1)},\ldots,\dd^{(t)})\in S_{\dd}\right\}
\]
where $S_{\dd}$ is the set of all decompositions $\dd=\dd^{(1)}+\ldots +\dd^{(t)}$ with $t>1$ and every $\dd^{(r)}\in \BoN^{Q_0}\setminus \{0\}$.  Equivalently (by \cite[Thm.1.2]{CB01} and \cite[Lem.5.1]{DHSM23}), $\Sigma_Q$ is the set of positive dimension vectors such that there exists a simple $\Pi_Q$-module, with $\Pi_Q$ defined as in \eqref{PA_def}.  We define the set of \textit{simple positive roots}
\[
\Phi_Q^+\coloneqq\Sigma_Q\cup \{l\dd\lvert l\in\BoZ_{\geq 1},\dd\in \Sigma_{Q}\textrm{ and }(\dd,\dd)_Q=0\}.
\]
We decompose $\Phi_Q^+=\Phi_Q^{+,\mathrm{real}}\cup \Phi_Q^{+,\mathrm{im}}$ into real and imaginary roots, where 
\[
\Phi_Q^{+,\mathrm{real}}\coloneqq \{\dd\in\Phi_Q^+\lvert (\dd,\dd)_Q=2\};\quad\quad\Phi_Q^{+,\mathrm{im}}\coloneqq \{\dd\in\Phi_Q^+\lvert (\dd,\dd)_Q<2\}.
\]
We furthermore decompose $\Phi_Q^{+,\mathrm{im}}=\Phi_Q^{+,\mathrm{iso}}\cup \Phi_Q^{+,\mathrm{hyp}}$ into isotropic and hyperbolic roots
\[
\Phi_Q^{+,\mathrm{iso}}\coloneqq\{\dd\in\Phi_Q^+\lvert (\dd,\dd)_Q=0\};\quad\quad\Phi_Q^{+,\mathrm{hyp}}\coloneqq\{\dd\in\Phi_Q^+\lvert (\dd,\dd)_Q<0\}.
\]
Given $\dd\in \Phi_Q^{+,\mathrm{iso}}$ there is a unique $\dd'\in \Sigma_Q$ such that $\dd=n\dd'$ with $n\geq 1$.

\subsection{Quiver algebras}
\label{quiver_sec}
If $K$ is a field, we denote by $KQ$ the free path algebra of $Q$ over $K$.  This algebra has a $K$-basis given by paths in $Q$, including ``lazy'' paths $\lazy_i$ of length zero, for each $i\in Q_0$, which start and end at $i$.  If $p,q$ are two paths in $Q$ we define
\[
p\cdot q\coloneqq \begin{cases} pq&\textrm{if }q\textrm{ ends where }p\textrm{ begins}\\
0&\textrm{otherwise.}
\end{cases}
\]
Here, $pq$ denotes the concatenation of $p$ and $q$, and we read paths from right to left (the same way we compose morphisms in a category).

Given a quiver $Q$ we first define the double $\overline{Q}$ by setting $\overline{Q}_0=Q_0$ and setting $\overline{Q}_1=Q_1\coprod Q^{\opp}_1$.  We define the preprojective algebra
\begin{equation}
\label{PA_def}
\Pi_Q\coloneqq\BoC\overline{Q}/\langle \sum_{a\in Q_1}[a,a^*]\rangle.
\end{equation}
A \textit{potential} for a quiver $Q$ is an element $W\in (\BoC Q)_{\cyc}\coloneqq \BoC Q/[\BoC Q,\BoC Q]_{\mathrm{vect}}$ where the subscript means that we take the quotient by the vector space spanned by commutators.  A potential is thus a $\BoC$-linear combination of cyclic paths in $Q$.  

For an arbitrary quiver $Q$, we form the \emph{tripled quiver} $\tilde{Q}$ by setting
\[
\tilde{Q}_0=Q_0;\quad\quad \tilde{Q}_1=\overline{Q}_1\coprod \{\omega_i\;\lvert \;i\in Q_0\}
\]
with $s(\omega_i)=t(\omega_i)=i$.  The tripled quiver $\tilde{Q}$ carries the canonical cubic potential which we denote
\begin{equation}
\label{CCP}
\tilde{W}\coloneqq \left(\sum_{a\in Q_1} [a,a^*]\right)\left(\sum_{i\in Q_0} \omega_i\right).
\end{equation}
\subsection{Extended algebras}
\label{extended_quiver_sec}

Given a quiver $Q$ and dimension vector $\ff$ we define the \emph{framed quiver} $Q_{\ff}$ by setting
\[
(Q_{\ff})_0=Q_0\coprod \{\infty\};\quad\quad (Q_{\ff})_1=Q_1\coprod \{r_{i,m}\;\lvert \; i\in Q_0,\;1\leq m\leq \ff_i\}.
\]
We set $s(r_{i,n})=\infty$ and $t(r_{i,n})=i$ for all $i\in Q_0$ and $n\in \BoN$.  Given $\dd\in\BoN^{Q_0}$ and $n\in \BoN$ we denote by $(\dd,n)\in\BoN^{Q_{\ff}}$ the dimension vector for $Q_{\ff}$ satisfying $(\dd,n)_{\infty}=n$ and $(\dd,n)_i=\dd_i$ for $i\in Q_0$.

Similarly, if $\ff',\ff''\in\BoN^{Q_0}$ are a pair of dimension vectors, we define 
\[
(Q_{\ff',\ff''})_0=Q_0\coprod \{\infty',\infty''\};\quad\quad (Q_{\ff',\ff''})_1=Q_1\coprod \{r'_{i,m},r''_{i,n}\;\lvert \; i\in Q_0,\;1\leq m\leq \ff'_i,\;1\leq n\leq \ff''_i\}.
\]
We set $s(r'_{i,n})=\infty'$, $s(r''_{i,n})=\infty''$, and $t(r'_{i,n})=i=t(r''_{i,n})$.  Given $\dd\in\BoN^{Q_0}$, and $n',n''\in \BoN$, we define $(\dd,n',n'')\in\BoN^{Q_{\ff',\ff''}}$ to be the dimension vector satisfying $(\dd,n',n'')_{\infty'}=n'$, $(\dd,n',n'')_{\infty''}=n''$ and $(\dd,n',n'')_i=\dd_i$ for $i\in Q_0$.  There are natural inclusions $Q_{\ff'}\hookrightarrow Q_{\ff',\ff''}$ and $Q_{\ff''}\hookrightarrow Q_{\ff',\ff''}$ sending the vertex labelled $\infty$ in the domain to the vertex labelled $\infty'$ or $\infty''$, respectively, in the target.

\begin{remark}
We write $\tilde{Q}_{\ff}$ and $\tilde{Q}_{\ff',\ff''}$ to denote the quivers obtained by framing and \textit{then} tripling, as opposed to the more logical, but also more cumbersome $\widetilde{Q_{\ff}}$ etc.
\end{remark}

We next introduce a kind of partial abelianization of quivers with potential.  Say we have fixed a quiver $Q$ with potential $W$, a vertex $v\in Q_0$ and a non-negative integer $r\geq 1$. Let $L=\{\omega_1,\ldots,\omega_p\}$ be the set of loops based at $v$.  We assume that we are given a weighting function $\wt\colon Q_1\rightarrow \BoZ^r$ such that $\wt(a)=0$ if $a\in L$ or $a$ does not either start or end at $v$, and otherwise $\wt(a)=\pm\delta_m=(0,\ldots,0,\pm 1,0\ldots,0)$ for some $1\leq m\leq r$.  Then we define $Q^{\Abel}$ by setting $Q^{\Abel}_0=(Q_0\setminus \{v\})\cup\{\infty_1,\ldots,\infty_r\}$ and $Q^{\Abel}_1=(Q_1\setminus L )\cup \{\omega_{p',m}\;\lvert\; 1\leq p'\leq p,\; 1\leq m\leq r\}$. 
We define source and target functions $s',t'\colon Q^{\Abel}_1\rightarrow Q^{\Abel}_0$ by
\[
s'(a)=\begin{cases} s(a) &\textrm{if }a\in Q_1\setminus L\textrm{ and }s(a)\neq v\\
\infty_m &\textrm{if }s(a)=v \textrm{ and }\wt(a)=\pm \delta_m\\
\infty_m&\textrm{ if }a=\omega_{p',m}\end{cases};\quad\quad  t'(a)=\begin{cases} t(a) &\textrm{if }a\in Q_1\setminus L\textrm{ and }t(a)\neq v\\
\infty_m &\textrm{if }t(a)=v \textrm{ and }\wt(a)=\pm \delta_m\\
\infty_m&\textrm{ if }a=\omega_{p',m}.\end{cases}
\]
We define the potential $W^{\Abel}$ by replacing each instance of $\omega_{p'}$ with $\sum_{m=1}^r \omega_{p',m}$.  The new $W$ is a linear combination of cyclic paths in $Q^{\Abel}$, which we denote $W^{\Abel}$.

\begin{remark}
\label{fr_from_ab}
Fix a decomposition $\ff=\ff'+\ff''$. Then $Q_{\ff',\ff''}$ is obtained as a partial abelianization of $Q_{\ff}$ in the obvious way: we consider the weighting satisfying $\wt(r_{i,m})=(1,0)$ if $m\leq \ff'_i$ and $\wt(r_{i,m})=(0,1)$ if $m>\ff'_i$.
\end{remark}

\subsection{Quiver representations and Kac polynomials}
Let $M$ be a $KQ$-module.  The \emph{dimension vector} of $M$ is defined to be the $Q_0$-tuple $\dim_Q(M)\coloneqq (\dim_K(\lazy_i\cdot M))_{i\in Q_0}\in\BoN^{Q_0}$.  A left $KQ$-module can be identified with a $Q$-representation over the field $K$ in the natural way: we assign the vector spaces $\lazy_i\cdotsh M$ to the vertices of $Q$ and the linear maps
\[
M(a)\colon \lazy_{s(a)}\cdotsh M\rightarrow \lazy_{t(a)}\cdotsh M;\quad\quad
v\mapsto a\cdot v
\]
to the arrows.

For $q=p^n$ a prime power, let $\BoF_q$ denote the field of order $q$, and let $\ol{\BoF}_q$ be its algebraic closure.  If $M$ is a $\BoF_q Q$-module, we say that it is \emph{absolutely indecomposable} if $M\otimes_{\BoF_q}\overline{\BoF}_q$ is indecomposable as a $\overline{\BoF}_qQ$-module.  Given a quiver $Q$, a dimension vector $\dd\in\BoN^{Q_0}$, and a prime power $q$ Kac \cite{Kac83} defined the numbers
\[
\kac_{Q,\dd}(q)\coloneqq\#\left\{\begin{array}{c}
\textrm{isomorphism classes of absolutely indecomposable }\\
\dd\textrm{-dimensional }\BoF_q Q\textrm{-modules}
\end{array}\right\}.
\]
Furthermore, Kac proved that the assignment $q\mapsto\kac_{Q,\dd}(q)$ is a polynomial function with integer coefficients, and conjectured that in fact $\kac_{Q,\dd}(q)\in\BoN[q]$.  This conjecture was proved in \cite{HLRV13} by identifying Kac polynomials with Poincar\'e polynomials of sign-isotrivial summands (with respect to a natural Weyl group action) of the cohomology of certain Nakajima quiver varieties.

Given a quiver $Q$, a \emph{stability condition} is a tuple $\zeta\in\BoQ^{Q_0}$.  The \emph{slope} of a nonzero $K Q$-module $M$ is defined to be the ratio 
\[
\mu(M)\coloneqq \mu(\dim_Q(M))\coloneqq \dim_Q(M)\cdot \zeta/\dim_K(M).
\]
Given a stability condition $\zeta\in \BoQ^{Q_0}$ and a slope $\theta\in\BoQ$ we define
\begin{align*}
\Lambda_{\theta}^{\zeta,+}\coloneqq &\{0\}\cup \{\dd\in\BoN^{Q_0}\setminus \{0\}\;|\: \mu(\dd)=\theta\}\\
\Lambda_{\theta}^{\zeta}\coloneqq &\{\dd-\ee\;\lvert\; \dd,\ee\in\Lambda_{\theta}^{\zeta,+} \}.
\end{align*}
Given $\theta\in\BoQ$, we say that a stability condition $\zeta\in\BoQ^{Q_0}$ is $\theta$\emph{-generic} if for all $\dd',\dd''\in \Lambda_{\theta}^{\zeta}$ there is equality $\chi_Q(\dd',\dd'')=\chi_Q(\dd'',\dd')$. We say that $\zeta$ is generic for the dimension vector $\dd\in\BoN^{Q_0}$ if for all $\dd'<\dd$ we have $\mu(\dd')\neq \mu(\dd)$.

A $K Q$-module $M$ is $\zeta$-\emph{semistable} if, for all proper nonzero submodules $M'\subset M$, the weak inequality $\mu(M')\leq \mu(M)$ is satisfied, and $M$ is $\zeta$-\textit{stable} if the strong inequality $\mu(M')< \mu(M)$ is satisfied.  The full subcategory of $\zeta$-semistable $K Q$-modules with dimension vector in $\Lambda_{\theta}^{\zeta}$ is a Serre subcategory of the category of $K Q$-modules, with simple objects given by $\zeta$-stable modules of slope $\theta$.

\subsection{Sign twists}
\label{tau_twist_sec}
Let $Q$ be a quiver, let $\theta\in\BoQ$ be a slope, and fix a $\theta$-generic stability condition $\zeta\in\BoQ^{Q_0}$.  We define the $\BoZ/2\cdotsh \BoZ$-bilinear form
\begin{align*}
\tau\colon &\Lambda^{\zeta}_{\theta}\times\Lambda^{\zeta}_{\theta}\rightarrow \BoZ/2\cdotsh\BoZ\\
&(\dd,\dd')\mapsto \chi_{Q}(\dd,\dd)\chi_Q(\dd',\dd')+\chi_Q(\dd,\dd')+ 2\cdotsh\BoZ.
\end{align*}
Let $\mathscr{C}$ be a symmetric monoidal $\BoQ$-linear category, with symmetrising natural transformation $\sw\colon \CF\otimes\CG\rightarrow \CG\otimes\CF$.  We assume that $\mathscr{C}$ satisfies the property that for all objects $\CF\in\Ob(\mathscr{C})$ there is a canonical decomposition $\CF\cong \bigoplus_{\dd\in \Lambda_{\theta}^{\zeta}}\CF_{\dd}$, which is respected by the symmetric monoidal structure, in the sense that $(\CF'\otimes\CF'')_{\dd}=\bigoplus_{\dd',\dd''\in\Lambda_{\theta}^{\zeta}\;\lvert\; \dd'+\dd''=\dd}\CF_{\dd'}\otimes \CF_{\dd''}$, and $\sw\lvert_{\CF_{\dd'}\otimes\CF_{\dd''}}$ factors through a sum of isomorphisms $\sw_{\dd',\dd''}\colon \CF_{\dd'}\otimes\CF_{\dd''}\rightarrow \CF_{\dd''}\otimes\CF_{\dd'}$.  We denote by $\mathscr{C}_{\tau}$ the same symmetric monoidal category, but with the modified symmetrising natural transformation 
\[
\sw_{\tau,\dd',\dd''}=(-1)^{\tau(\dd',\dd'')}\sw_{\dd',\dd''}.
\]
By genericity of $\zeta$, we have $\sw_{\tau,\dd'',\dd'}\circ\sw_{\tau,\dd',\dd''}=\sw_{\dd'',\dd'}\circ\sw_{\dd',\dd''}=\id$.
\begin{example}
\label{gvs_examp}
Let $\Vect_{\Lambda^{\zeta,+}\oplus \BoZ}$ be the category of $\Lambda^{\zeta,+}_{\theta}\oplus \BoZ$-graded vector spaces.  For $V\in\Vect_{\Lambda_{\theta}^{\zeta,+}\oplus \BoZ}$ we write $V_{\dd}^k$ for the $(\dd,k)$-graded piece.  Then we have a natural decomposition $V=\bigoplus_{\dd\in\Lambda^{\zeta}_{\theta}}V_{\dd}$ where $V_{\dd}=\bigoplus_{k\in\BoZ}V_{\dd}^k$ for $\dd\in \Lambda^{\zeta,+}_{\theta}$ and $V_{\dd}=0$ otherwise.  We consider the extra $\BoZ$-grading as a cohomological grading, so that the Koszul sign rule applies in the symmetrising morphism for this category: if $\alpha,\alpha'$ are elements of $V,V'$ with bidegree $(\dd,k)$ and $(\dd',k')$ respectively, then $\sw(\alpha\otimes\alpha')=(-1)^{kk'}(\alpha'\otimes\alpha)$.  Let us furthermore assume that for each $\dd\in\BoN^{Q_0}$ the graded piece $V_{\dd}$ has vanishing odd or even cohomology, depending on whether $\chi_Q(\dd,\dd)$ is even or odd, respectively; this condition holds, for example, for the shuffle algebras $\HO\!\CoHA_Q^T$ considered in \S \ref{shuffle sec}.  In the $\tau$-twisted symmetric monoidal category $(\Vect_{\Lambda^{\zeta,+}\oplus \BoZ})_{\tau}$ we have 
\begin{equation}
\label{twisted_sym}
\sw_{\tau}(\alpha\otimes\alpha')=(-1)^{\chi_Q(\dd,\dd')}(\alpha'\otimes\alpha).
\end{equation}
\end{example}
These $\tau$-twists first appeared in \cite[\S 1]{KS2} and \cite{Efimov}.  With the swap morphism $\tau$-twisted this way, the shuffle algebra $\HO\!\CoHA^T_{Q}$ is a commutative algebra in $(\Vect_{\Lambda^{\zeta,+}\oplus \BoZ})_{\tau}$, as long as $Q$ is graded-symmetric; see \S \ref{shuffle sec} for definitions and details.

\subsection{Weightings and torus actions}
\label{wtngs_ta_sec}

Fix a lattice $N=\BoZ^{r}$, which we call a lattice of weights.  If a lattice of weights has been fixed, we set $T=\Hom_{\Grp}(N,\BoC^*)$.  If $\Ft$ is the Lie algebra of $T$, then there is a natural identification of complex vector spaces $N_{\BoC}\cong \Ft^*$.  Given $\mathbf{n}\in N $ we denote by $\ttt(\mathbf{n})\in\HO^2(\B T,\BoQ)\cong \BoQ[\Ft^*]$ the corresponding cohomology class.

An $N$-graded quiver is a finite quiver $Q$, along with a weighting of the edges, defined by a map of sets $\wt\colon Q_1\rightarrow N$.  Given a path $p=a_1\ldots a_m$ we define $\wt(p)=\sum_{n=1}^m\wt(a_n)$.    We say that a linear combination $\sum_{s\in S} \lambda_s p_s$ of paths, with each $\lambda_s\in \BoC$, is homogeneous if $\wt(p_s)$ does not depend on $s$.  Given a $N$-weighting $\wt$ of $Q$ we define a $N$-weighting $\wt^{\opp}$ of $Q^{\opp}$ by setting $\wt^{\opp}(a^*)=-\wt(a)$ for each $a\in Q_1$.  We say that $(Q,\wt)$ is \textit{graded-symmetric} if there is an isomorphism of graded quivers $(Q,\wt)\cong (Q^{\opp},\wt^{\opp})$ preserving the vertices.  For example this will hold if $Q$ is symmetric in the usual sense and the weighting function $\wt$ is trivial.  
\begin{example}
For a less trivial example, pick any weighting $\wt\colon Q_1\rightarrow N$ of a quiver $Q$ (not necessarily symmetric).  We extend this to a weighting of $\overline{Q}$ by setting $\wt(a^*)=\wt^{\opp}(a^*)=-\wt(a)$.  The quiver $\overline{Q}$ is graded-symmetric for this weighting.
\end{example}

If $t\in T$, and $M$ is a $\BoC Q$-module, we obtain a new $\BoC Q$-module $t\cdot M$ which has the same underlying vector space as $M$, and for which for $a\in Q_1$ we define $(t\cdot M)(a)=t(\wt(a))M(a)$.  If $W\in (\BoC Q)_{\mathrm{cyc}}$ is a potential, we say that $\wt$ is $W$\emph{-preserving} if $\wt(\overline{W})=0$ for $\overline{W}\in\BoC Q$ some homogeneous lift of $W$.  Equivalently we say that $W$ is $T$\emph{-invariant}.

For the purposes of embedding cohomological Hall algebras inside shuffle algebras, we will frequently make the following assumption.
\begin{assumption}
\label{weighting_assumption}
Let $Q$ be a quiver, and let $\tilde{Q}$ be the associated tripled quiver.  Let $\wt\colon\tilde{Q}_1\rightarrow \BoZ^r$ be a $\tilde{W}$-preserving weighting.  We assume that there is a morphism of lattices $\BoZ^r\rightarrow \BoZ^2$ such that, denoting by $\wt'\colon \tilde{Q}_1\rightarrow \BoZ^2$ the composite weighting, $\wt'$ is given by $\wt'(a)=(1,0)$ and $\wt'(a^*)=(0,1)$ for every $a\in Q_1$, and $\wt'(\omega_i)=(-1,-1)$ for every $i\in Q_0$.
\end{assumption}

\subsection{Generalised Kac--Moody algebras}
\label{GKM_sec}
Let $Q$ be a quiver, which as ever we assume to be finite.  Let $A$ be a cohomologically graded commutative algebra, and let $V$ be a free $\BoN^{Q_0}$-graded, cohomologically graded $A$-module, meaning that there is a decomposition
\[
V=\bigoplus_{\dd\in\BoN^{Q_0}}V_{\dd}
\]
where each $V_{\dd}$ is a cohomologically graded vector space, carrying an $A$-action that respects the cohomological degrees, and each $V_{\dd}$ is a free $A$-module. We define
\[
V^{\vee}\coloneqq \bigoplus_{\dd\in-\BoN^{Q_0}}(V^{\vee})_{\dd};\quad\quad(V^{\vee})_{\dd}\coloneqq\Hom_{A}(V_{-\dd},A).
\]
Since $\tilde{Q}_0=Q_0$ we may consider $V$ as an element in the category of $\BoN^{\tilde{Q}_0}$-graded modules, which we give the $\tau$-twisted symmetric monoidal structure defined in \S \ref{tau_twist_sec}, with respect to the quiver $\tilde{Q}$.  We assume that $V_{\dd}=0$ for $\dd\notin \Phi^+_Q$, each $V_{\dd}$ has finite rank as an $A$-module, and $V_{\delta_i}\cong A$ as graded $A$-modules if $\delta_i$ is a simple real root.  We next define a generalised Kac--Moody algebra, and generalised Kac--Moody Lie algebra, associated to this data.

Set $\fh_Q=\BoQ^{Q_0}$. Where there is no risk of ambiguity, we will simply write $\fh=\fh_Q$.   We embed $\fh_Q$ in $\fh_{Q,A}\coloneqq \fh_Q\otimes A$ via $h\mapsto h\otimes 1$.  We define $\fg_{A,V}$ to be the free $A$-linear Lie algebra generated by $\fh_{Q,A}\oplus V\oplus V^{\vee}$, subject to the following relations
\begin{enumerate}
\item
For all $h,h'\in \fh$ we set $[h,h']=0$;
\item
For $g\in (V\oplus V^{\vee})_{\dd}$ and $h\in \fh$ we set $[h,g]=(h,\dd)_Q \cdot g$;
\item
For $v\in V_{\dd}$ and $v'\in (V^{\vee})$ we set $[v,v']=v'(v)\cdot \dd\in\fh$;
\item
For $g\in V_{\dd}$ and $g'\in V_{\dd'}$, or $g\in (V^{\vee})_{-\dd}$ and $g'\in (V^{\vee})_{-\dd'}$ (where $\dd$ and $\dd'$ are not equal to the same real simple root) we set $[g,-]^{1-(\dd,\dd')_Q}(g')=0$ if either $(\dd,\dd')_Q=0$ or $\dd$ is a real simple root.
\end{enumerate}
If $A=\BoQ$ we omit it from the notation.  Since the relations are defined over $\BoQ$, if $V$ is a $\BoN^{Q_0}$-graded, cohomologically graded vector space, we find that $\fg_{A,V\otimes A}=\fg_{V}\otimes A$.

The relations (4) in the list above are called the \textit{Serre relations}.  They play a prominent role due to the following fact \cite{Bor88}: there is a decomposition as $A$-modules $\fg_{A,V}\cong \fn_{A,V}^+\oplus\fh_A\oplus \fn_{A,V}^-$ where $\fn_{A,V}^+$ is the free $A$-linear Lie algebra generated by $V$, modulo the Serre relations for $V$, and $\fn_{A,V}^-$ is the free Lie algebra generated by $V^{\vee}$, modulo the Serre relations for $V^{\vee}$.

\begin{definition}
\label{Chev_def}
We refer to homogeneous elements $e\in V\subset \fn_{A,V}^+$ and $f\in V^{\vee}\subset \fn_{A,V}^-$ as the positive and negative \textit{Chevalley generators} of $\fg_{A,V\otimes A}$, respectively. We further refer to either $\fn_{A,V}^+$ or $\fn_{A,V}^-$ as a half GKM algebra with coefficients in $A$.
\end{definition}		

We denote by $\UEA_A(\Fg_{A,V})$ the quotient of the free unital associative algebra generated over $A$ by $\fh_A\oplus V\oplus V^{\vee}$, subject to the same relations.  This is naturally identified with the universal enveloping algebra of $\Fg_{A,V}$ in the category of $A$-modules, justifying the notation.

If $T$ is a torus and $A=\HO_T$, we define $\fg^T_V\coloneqq \fg_{\HO_T,V}$ and $\fn^{T,\pm}_{V}=\fn_{\HO_T,V}^{\pm}$.

\subsection{Half GKM Lie algebras in perverse sheaves}
\label{GKM_perv_sec}
We recall the construction of half GKM Lie algebras in categories of perverse sheaves defined and considered in \cite{DHSM23}. In contrast with Definition \ref{Chev_def}, we do not actually consider these ``half'' algebras as halves of some bigger algebraic object.

We fix as above an algebraic torus $T$.  We may consider $\BoN^{Q_0}\times \B T$ as a stack, consisting of a copy of the stack $\B T$ for every $\dd\in \BoN^{Q_0}$.  This is a monoid in stacks over $\B T$, with the monoid structure $(\BoN^{Q_0}\times \BoN^{Q_0})\times \B T\rightarrow \BoN^{Q_0}\times \B T$ provided by addition of dimension vectors.

Let $X$ be a commutative monoid in stacks over $\BoN^{Q_0}\times \B T$.  Equivalently, $X$ admits a decomposition $X=\coprod_{\dd\in\BoN^{Q_0}} X_{\dd}$, and is equipped with a monoid map $\mu\colon X\times_{\B T} X\rightarrow X$ such that $\mu\lvert_{X_{\dd'}\times_{\B T} X_{\dd''}}$ factors through the inclusion of $X_{\dd'+\dd''}$, and there is an identification $X_0=\B T$.  Commutativity is the condition that $\mu\circ s=\mu$, where $s$ is the automorphism of $X\times_{\B T} X$ swapping the factors.  We assume furthermore that $X_{\dd}=\B T$ if $\dd$ is a simple real root.  We define a symmetric tensor structure on $\Dub^+(\Perv(X))$ by setting
\[
\CF\boxdot\CG\coloneqq \mu_*(\pi_{1}^*\CF\otimes\pi_2^*\CG)
\]
where $\pi_1,\pi_2\colon X\times_{\B T} X\rightarrow X$ are the two projections.  This induces a symmetric tensor structure on $\Perv'(X)$ if $\mu$ is finite, see e.g. \cite[\S 3]{DHSM22}.  In this section we will assume that the monoid map $\mu\colon X \times_{\B T} X\rightarrow X$ is finite.  

An object $\CF$ of $\Perv'(X)$ admits a canonical direct sum decomposition
\[
\CF=\bigoplus_{\dd\in\BoN^{Q_0}}\CF_{\dd};\quad\quad\CF_{\dd}\coloneqq \CF\lvert_{X_{\dd}},
\]
so that $\Perv'(X)$ satisfies the assumptions of \S \ref{tau_twist_sec}.  We give this category the $\tau$-twist of the above tensor structure, as defined in \S \ref{tau_twist_sec}, with $\tau$ defined with respect to the symmetric quiver $\tilde{Q}$.

Let $\SG\in \Perv'(X)$.  We assume that $\SG_{\dd}=0$ if $\dd\notin \Phi^+_{Q}$, and $\SG_{\dd}=\BoQ_{X_{\dd}}$ if $\dd$ is a simple real root.  We first consider the free algebra object $\Tens_{\boxdot}(\SG)$.  This is the infinite direct sum
\[
\Tens_{\boxdot}(\SG)\coloneqq \bigoplus_{i\geq 0} \SG^{\boxdot i}
\]
where $\SG^{\boxdot i}$ denotes the $i$-fold tensor product of $\SG$ with itself with respect to the tensor structure $\boxdot$.  This infinite direct sum is well defined since for all $\dd\in\BoN^{Q_0}$ there are only finitely many decompositions $\dd=\dd^{(1)}+\ldots +\dd^{(m)}$ with each $\dd^{(j)}\in\Phi^+_Q$.  The perverse sheaf $\Tens_{\boxdot}(\SG)$ carries a natural associative product, and thus a Lie bracket given by the commutator bracket.  We denote by $\Free_{\Lie}(\SG)$ the free Lie subalgebra object generated by $\SG$.  If $\SL$ is a Lie algebra object, we denote by $\UEA_{\boxdot}(\SL)$ the universal enveloping algebra of $\SL$.  It is defined up to isomorphism by the universal property: if $f\colon \SL\rightarrow \SA$ is a morphism of Lie algebra objects, where $\SA$ is an algebra, considered as a Lie algebra via the commutator bracket, then $f$ extends uniquely to an algebra morphism $\UEA_{\boxdot}(\SL)\rightarrow \SA$.  See \cite[\S 3.2]{DHSM23} for a more leisurely account of this construction, and the following one.

If $\SA$ is an algebra object, or Lie algebra object in $\Perv'(X)$, and $\dd',\dd''\in\BoN^{Q_0}$, we denote by $L_{\SA_{\dd'}}(\SA_{\dd''})$ the image of the morphism
\[
[-,-]\lvert_{\SA_{\dd'}\boxdot\SA_{\dd''}}\colon \SA_{\dd'}\boxdot\SA_{\dd''}\rightarrow \SA_{\dd'+\dd''}.
\]
In the algebra case we define the Lie bracket to be the commutator $(m- m\circ \sw_{\tau})$ as usual.  Inductively, for $m\geq 2$, we define $L^m_{\SA_{\dd'}}(\SA_{\dd''})$ to be the image of the morphism $[-,-]\lvert_{\SA_{\dd'}\boxdot\SA_{(m-1)\dd'+\dd''}}$ restricted to $\SA_{\dd'}\boxdot L^{m-1}_{\SA_{\dd'}}(\SA_{\dd''})$.
For $\dd',\dd''\in \Phi^+_Q$ we define
\[
\SR_{\dd',\dd''}(\SA)\coloneqq 
\begin{cases}
L_{\SA_{\dd'}}(\SA_{\dd''})& \textrm{if }(\dd',\dd'')_Q=0\\
L_{\SA_{\dd'}}^{1-(\dd',\dd'')_Q}(\SA_{\dd''})&\textrm{if }\dd'\in\Phi_Q^{+, \mathrm{real}}\\
0&\textrm{otherwise.}
\end{cases}
\]
For $\SG$ as above, we define $\Fn^+_{\SG}$ to be the quotient of $\SA=\Free_{\Lie}(\SG)$ by the Lie ideal generated by $\SR_{\dd',\dd''}(\SA)$ for all $\dd',\dd''\in\Phi^+_Q$.  Equivalently, let $\SA_{\SG}$ be the quotient of the free associative algebra $\Tens_{\boxdot}(\SG)$ by the two-sided ideal generated by $\SR_{\dd',\dd''}(\Tens_{\boxdot}(\SG))$ for all $\dd',\dd''\in\Phi^+_Q$, and define $\Fn^+_{\SG}$ to be the Lie subalgebra of $\SA_{\SG}$ generated by the subobject $\SG\subset \SA_{\SG}$.  As an algebra object in $\Perv'(X)$, we have a natural isomorphism $\SA_{\SG}\cong \UEA_{\boxdot}(\Fn^+_{\SG})$.
\section{Cohomological Hall algebras}
\subsection{Moduli stacks}
\label{moduli_stacks_sec}
Let $Q$ be a quiver which is graded via a weighting $\wt\colon Q_1\rightarrow N=\BoZ^{r}$ for some $r\in\BoN$.  We continue to set $T=\Hom_{\Grp}(N,\BoC^*)$.   Fix a dimension vector $\dd\in\BoN^{Q_0}$.  We denote by $\FM^T_{\dd}(Q)$ the quotient of the stack of $\dd$-dimensional $\BoC Q$-modules by the $T$-action.  Explicitly, we define
\[
\BoA_{\dd}(Q)\coloneqq \prod_{a\in Q_1}\Hom(\BoC^{\dd_{s(a)}},\BoC^{\dd_{t(a)}}).
\]
This variety is acted on by the gauge group $\Gl_{\dd}\coloneqq \prod_{i\in Q_0}\Gl_{\dd_i}(\BoC)$ by change of basis.  The $\Gl_{\dd}$-action commutes with the natural $T$-action defined in \S \ref{wtngs_ta_sec}, so that we obtain a $\Gl_{\dd}\times T$-action on $\BoA_{\dd}(Q)$.  To save space we sometimes write $\Gl_{\dd}^{\wt}=\Gl_{\dd}\times T$.  There is an equivalence of stacks between $\FM^T_{\dd}(Q)$ and the stack-theoretic quotient $\BoA_{\dd}(Q)/\Gl_{\dd}^{\wt}$.  

Given a pair $\dd',\dd''\in \BoN^{Q_0}$ we denote by $\FM^T_{\dd',\dd''}(Q)$ the $T$-quotient of the stack of short exact sequences of $\BoC Q$-modules $0\rightarrow M'\rightarrow M\rightarrow M''\rightarrow 0$ with $M'$ of dimension $\dd'$ and $M''$ of dimension $\dd''$.  Let $\dd=\dd'+\dd''$.  We denote by $\BoA_{\dd',\dd''}(Q)\subset \BoA_{\dd}(Q)$ the subspace of homomorphisms preserving each of a set of fixed flags $0\subset \BoC^{\dd'_i}\subset \BoC^{\dd}$ for $i\in Q_0$.  We define the subgroup $\Gl_{\dd',\dd''}\subset \Gl_{\dd}$ to be the subgroup preserving these same flags, and set $\Gl_{\dd',\dd''}^{\wt}=\Gl_{\dd',\dd''}\times T$.  There is an equivalence of stacks $\FM^T_{\dd',\dd''}(Q)\simeq \BoA_{\dd',\dd''}(Q)/\Gl^{\wt}_{\dd',\dd''}$.

We define the coarse moduli space $\CM_{\dd}(Q)=\mathrm{Spec}(\Gamma(\mathcal{O}_{\BoA_{\dd}(Q)})^{\Gl_{\dd}})$.  Closed points of $\CM_{\dd}(Q)$ are in bijection with isomorphism classes of $\dd$-dimensional semisimple $\BoC Q$-modules.  We denote by $\JH_{\dd}\colon \FM_{\dd}(Q)\rightarrow \CM_{\dd}(Q)$ the affinization morphism.  At the level of points this morphism sends modules to their semisimplifications, i.e. the direct sum of the subquotients appearing in some (equivalently, any) Jordan--H\"older filtration.  The torus $T$ acts on $\CM_{\dd}(Q)$, and we denote by $\CM^T_{\dd}(Q)=\Msp_{\dd}(Q)/T$ the stack-theoretic quotient.   We continue to denote by the same symbols $\JH_{\dd}$ the induced morphism
\[
\JH_{\dd}\colon \FM^T_{\dd}(Q)\rightarrow \CM_{\dd}^T(Q).
\]

We will also consider versions of the above stacks incorporating stability conditions.  Let $\zeta\in\BoQ^{Q_0}$ be a stability condition, then we denote by $\BoA^{\zeta\sst}_{\dd}(Q)\subset \BoA_{\dd}(Q)$ the open subvariety, points of which correspond to $\zeta$-semistable $\dd$-dimensional $\BoC Q$-modules.  We set $\FM^{T,\zeta\sst}_{\dd}(Q)\coloneqq \BoA^{\zeta\sst}_{\dd}(Q)/\Gl_{\dd}^{\wt}$.  The natural morphism $\FM^{T,\zeta\sst}_{\dd}(Q)\rightarrow \FM^{T}_{\dd}(Q)$ is an open embedding.  We construct the coarse moduli space $\CM^{\zeta\sst}_{\dd}(Q)$ of $\zeta$-semistable $\dd$-dimensional $\BoC Q$-modules as in \cite{King}.  Closed points of this quasi-projective variety are in bijection with isomorphism classes of $\zeta$-polystable $\dd$-dimensional $\BoC Q$-modules, and the natural morphism $\pi\colon\CM^{\zeta\sst}_{\dd}(Q)\rightarrow \CM_{\dd}(Q)$ is a GIT quotient map, and hence projective (see \cite{King}).  By construction, the torus $T$ acts on $\CM^{\zeta\sst}_{\dd}(Q)$, and the morphism $\pi$ is $T$-equivariant.  We define the stack $\CM^{T,\zeta\sst}_{\dd}(Q)\coloneqq \CM^{\zeta\sst}_{\dd}(Q)/T$.  We denote by 
\[
\JH^{\zeta}_{\dd}\colon \FM^{T,\zeta\sst}_{\dd}(Q)\rightarrow \CM^{T,\zeta\sst}_{\dd}(Q)
\]
the natural morphism.  

Let $Q$ be a quiver, and assume that $\wt\colon\overline{Q}_1\rightarrow \BoZ^r$ is chosen so that the element $\rho=\sum_{a\in Q_1}[a,a^*]$ is homogeneous, so that we may define $\wt(\rho)=\wt(a)+\wt(a^*)$ independently of the choice of $a\in Q_1$.  We thus obtain a morphism $\lambda'\colon T\rightarrow \BoC^*$, sending $\Hom(N,\BoC^*)\ni t\mapsto t(\wt(\rho))$, and an induced morphism $\lambda\colon \B T\rightarrow \B \BoC^*$.  We denote by $u$ the tautological generator of $\HO^2(\B \BoC^*,\BoQ)=\BoQ[u]$.  Throughout the rest of the paper we define $\hbar\coloneqq\lambda^*(u)$.  This will define the character through which $T$ scales the natural symplectic forms on Nakajima quiver varieties, as in \cite{MO19}.

We may identify $\BoA_{\dd}(\overline{Q})=\Tang\BoA_{\dd}(Q)$.  We denote by $\mathfrak{gl}_{\dd}$ the Lie algebra of $\Gl_{\dd}$.  Then $\BoA_{\dd}(\ol{Q})$ carries a natural symplectic form and admits the $T$-equivariant moment map
\begin{align*}
\mu_{\dd}\colon &\BoA_{\dd}(\overline{Q})\rightarrow \mathfrak{gl}_{\dd};
\quad \quad (M(a),M(a^*))_{a\in Q_1}\mapsto \sum_{a\in Q_1}[M(a),M(a^*)]
\end{align*}
where the target is acted on by scaling via the character $\hbar$. The subvariety $(\mu_{\dd})^{-1}(0)$ thus carries a $T$-action.  We define 
\[
\FM^T_{\dd}(\Pi_Q)=(\mu_{\dd})^{-1}(0)/\Gl_{\dd}^{\wt}\quad\quad\textrm{and} \quad\quad\CM^T_{\dd}(\Pi_Q)=\mathrm{Spec}(\Gamma(\mathcal{O}_{(\mu_{\dd})^{-1}(0)})^{\Gl_{\dd}})/T.
\]

If $Q$ is a quiver and $\ff\in\BoN^{Q_0}$ is a dimension vector we denote by 
\begin{equation}
\label{ldef}
l\colon \CM^T(\Pi_Q)\hookrightarrow \CM^T(\Pi_{Q_{\ff}})
\end{equation}
the extension by zero map.  At the level of points, it takes the semisimple $\Pi_Q$-module $M$ to the semisimple $\Pi_{Q_{\ff}}$-module for which the dimension vector is $(\dim_Q(M),0)$, and the induced $\Pi_Q$-module defined by the inclusion $\Pi_Q\hookrightarrow \Pi_{Q_{\ff}}$ is $M$.  For every $\dd\in\BoN^{Q_0}$ this morphism restricts to an isomorphism $l_{\dd}\colon \CM^T_{\dd}(\Pi_Q)\xrightarrow{\cong} \CM_{(\dd,0)}^T(\Pi_{Q_{\ff}})$.

\subsection{Nakajima quiver varieties and stacks}
\label{sec: Nakajima quiver varieties}

We refer to \cite{Nak94,Nak98,ginz2012} for background on the material in this subsection.  Nakajima quiver varieties play a fundamental role in this paper; we prove Theorems \ref{OC_thm} and \ref{main_thm} by realising both the Lie algebras in Theorem \ref{main_thm} as the same Lie subalgebra of the endomorphism Lie algebra of the cohomology of Nakajima quiver varieties.

Let $Q$ be a quiver and fix a dimension vector $\ff\in \BoN^{Q_0}$.  We denote by $\mu_{\ff,\dd}\colon \BoA_{(\dd,1)}(\overline{Q_{\ff}})\rightarrow \mathfrak{gl}_{\dd}$ the composition of $\mu_{(\dd,1)}\colon \BoA_{(\dd,1)}(\overline{Q_{\ff}})\rightarrow \mathfrak{gl}_{(\dd,1)}$ with the natural projection.  
\begin{lemma}
There is an equality $(\mu_{(\dd,1)})^{-1}(0)=(\mu_{\ff,\dd})^{-1}(0)$ of subvarieties of $\BoA_{(\dd,1)}(\overline{Q_{\ff}})$.
\end{lemma}
\begin{proof}
Obviously the left hand side is a subvariety of the right hand side: it is the subvariety cut out by the equation 
\[
\sum_{i\in Q_0,1\leq m\leq \ff_i}\rho_{r^*_{i,m}}\rho_{r_{i,m}}=0
\]
for elements $(\rho_a)_{a\in \overline{Q_{\ff}}}$ of $(\mu_{\ff,\dd})^{-1}(0)\subset \BoA_{(\dd,1)}(\overline{Q_{\ff}})$.  Since $(\dd,1)_{\infty}=1$, this equation is equivalent to
\begin{equation}
\label{redundancy_eq}
\sum_{i\in Q_0,1\leq m\leq \ff_i}\Tr(\rho_{r^*_{i,m}}\rho_{r_{i,m}})=0.
\end{equation}
By cyclic invariance of the trace, and the defining equations of $(\mu_{\ff,\dd})^{-1}(0)$ respectively, we have
\begin{align*}
\sum_{a\in Q_1}\Tr([\rho_{a},\rho_{a^*}])+\sum_{i\in Q_0,1\leq m\leq \ff_i}\Tr([\rho_{r_{i,m}},\rho_{r^*_{i,m}}])&=0\\
\sum_{a\in Q_1}\Tr([\rho_{a},\rho_{a^*}])+\sum_{i\in Q_0,1\leq m\leq \ff_i}\Tr(\rho_{r_{i,m}}\rho_{r^*_{i,m}})&=0,
\end{align*}
implying \eqref{redundancy_eq}.
\end{proof}

Let $\zeta\in\BoQ^{(Q_{\ff})_0}\cong \BoQ^{Q_0}\oplus \BoQ$ be generic for the dimension vector $(\dd,1)$
Since $Q_{\ff}$ and $\overline{Q_{\ff}}$ have the same vertex sets, this also defines a stability condition for $\overline{Q_{\ff}}$. This ensures that $\zeta$-semistability for $(\dd,1)$-dimensional $\Pi_{Q_{\ff}}$-modules is equivalent to $\zeta$-stability.
We define
\[
\mu^{\zeta}_{\ff,\dd}\colon \BoA_{(\dd,1)}^{\zeta\sst}(\overline{Q_{\ff}})\rightarrow \mathfrak{gl}_{\dd}
\]
by restricting $\mu_{\ff,\dd}$ to the stable locus.  We fix a weighting $\wt\colon (\overline{Q}_{\ff})_1\rightarrow \BoZ^{r}$ for which $\sum_{a\in (Q_{\ff})_1}[a,a^*]$ is homogeneous.  It follows that the $T$-action on $\BoA^{\zeta\sst}_{(\dd,1)}(\overline{Q_{\ff}})$ preserves $(\mu^{\zeta}_{\ff,\dd})^{-1}(0)$.  We define the Nakajima quiver stack
\begin{align*}
\Nak^{\zeta, T}_Q(\ff,\dd)\coloneqq &(\mu^{\zeta}_{\ff,\dd})^{-1}(0)/\Gl^{\wt}_{\dd}.
\end{align*}

In the case in which $T$ is trivial $\Nak^{\zeta, T}_Q(\ff,\dd)$ is a \textit{smooth} variety, and it follows that in general $\Nak^{\zeta, T}_Q(\ff,\dd)$ is a smooth stack. Given two generic stability conditions $\zeta$ and $\zeta'$, the associated quiver varieties $\Nak^{\zeta}_Q(\ff,\dd)$ and $\Nak^{\zeta'}_Q(\ff,\dd)$ are in general diffeomorphic as real manifolds but not isomorphic as schemes \cite[Cor. 4.2]{Nak94}.

The Nakajima variety $\Nak^{\zeta}_Q(\ff,\dd)$ may be realised as a GIT quotient (as in \cite{King}, see \cite[\S 3]{Nak98}).  As such, the affinization morphism 
\[
\pi\colon \Nak^{\zeta}_Q(\ff,\dd)\rightarrow \CM_{(\dd,1)}(\Pi_{Q_{\ff}})
\]
is projective.  It follows that, even in the case in which $T$ is nontrivial, the morphism $ \Nak^{ \zeta,T}_Q(\ff,\dd)\rightarrow \CM^T_{(\dd,1)}(\Pi_{Q_{\ff}})$,
obtained by passing to stack-theoretic quotients is representable and projective.  We denote this morphism also by $\pi$ when it is unlikely to cause confusion.

For a general stability condition $\zeta$, generic for $(\dd,1)$, we define
\[
\NakMod^{\zeta, T}_{Q,\ff,\dd}\coloneqq \HO(\Nak^{\zeta, T}_Q(\ff,\dd),\BoQ^{\vir});\quad\quad\NakMod^{\zeta, T}_{Q,\ff}\coloneqq\bigoplus_{\dd\in\BoN^{Q_0}}\NakMod^{\zeta, T}_{Q,\ff,\dd}
\]
where $\BoQ^{\vir}\coloneqq \BoQ[\dim(\Nak^{\zeta}_Q(\ff,\dd))]$ is identified with the constant equivariant perverse sheaf on the $T$-equivariant variety $\Nak^{\zeta}_Q(\ff,\dd)$.

Set $\zeta_{\mathrm{deg}}^+\coloneqq (0,\ldots,0,1)\in \BoQ^{Q_0}\oplus \BoQ\cong \BoQ^{(Q_{\ff})_0}$.  This stability condition is generic for the dimension vector $(\dd,1)$, and will be our default choice for Nakajima varieties throughout the whole article. When we revert to such default choice, we drop $\zeta_{\mathrm{deg}}^+$ from the notation. For example, we will systematically write $\Nak^T(\ff,\dd)$ and $\NakMod^T_{\ff,\dd}$ instead of $\Nak^{\zeta_{\mathrm{deg}}^+,T}(\ff,\dd)$ and $\NakMod^{\zeta_{\mathrm{deg}}^+, T}_{\ff, \dd}$. The objects $\NakMod^T_{Q,\ff}$ will form the underlying $\HO_T$-modules of a collection of modules for both the (doubled) BPS Lie algebra $\fg_{\Pi_Q}^T$ (defined in \S \ref{GKM_thm_sec}) and the Maulik--Okounkov Lie algebra $\Fg^{\MO,T}_Q$ (defined in \S \ref{MO_LA_sec}).

\subsection{Cohomological Hall algebras}
\label{gen_CoHA_sec}
We recall the construction of the critical cohomological Hall algebra associated to a graded quiver with potential, as defined in \cite[\S 7]{KS2} (see also \cite{Da13,preproj}), with required modifications to take account of weightings.   

Given a $N$-graded symmetric quiver $Q$ and a $T$-invariant potential $W$, we define the function $\Tr(W)$ on $\FM^T(Q)$, which takes a point representing a module $M$ to the complex number $\Tr(M(\overline{W}))$, where $\overline{W}$ is any lift of $W$ to an element of $\BoC Q$.  We denote by $\BoQ^{\vir}$ the constant $T$-equivariant perverse sheaf on $\FM^T(Q)$.  Precisely, we shift the constant sheaf via the formula $\BoQ^{\vir}\lvert_{\FM_{\dd}^T(Q)}=\BoQ[-\chi_Q(\dd,\dd)]$.
We define 
\[
\HO\!\CoHA^T_{Q,W,\dd}\coloneqq \HO^*(\FM^T_{\dd}(Q),\phip{\Tr(W)}\BoQ^{\vir});\quad \quad \HO\!\CoHA^T_{Q,W}\coloneqq \bigoplus_{\dd\in\BoN^{Q_0}}\HO\!\CoHA^T_{Q,W,\dd}.
\]
Here, $\phip{\Tr(W)}=\phi_{\Tr(W)}[-1]$ is the perverse-exact shift of the vanishing cycles functor, see \cite[Chapter 8]{KSsheaves} for details.  If $\zeta\in\BoQ^{Q_0}$ is a stability condition, and $\theta\in\BoQ$ is a slope, we likewise define
\[
\HO\!\CoHA^{T,\zeta}_{Q,W,\theta}\coloneqq \bigoplus_{\dd\in\Lambda^{\zeta,+}_{\theta}}\HO\!\CoHA^{T,\zeta}_{Q,W,\dd};\quad \quad\HO\!\CoHA^{T,\zeta}_{Q,W,\dd}\coloneqq \HO^*(\FM^{T,\zeta\sst}_{\dd}(Q),\phip{\Tr(W)}\BoQ^{\vir}).
\]
It can be beneficial to think of $\HO\!\CoHA^T_{Q,W}$ as a special case of $\HO\!\CoHA^{T,\zeta}_{Q,W,\theta}$, obtained by setting $\zeta=(0,\ldots,0)$ and $\theta=0$, so we will spell out everything only in the presence of a stability condition.

For $\dd',\dd''\in\Lambda_{\theta}^{\zeta,+}$ there is a commutative diagram 
\[
\xymatrix{
\FM^{T,\zeta\sst}_{\dd'}(Q)\times_{\B T} \FM^{T,\zeta\sst}_{\dd''}(Q)\ar[d]^{\JH^{\zeta}_{\dd'}\times_{\B T} \JH^{\zeta}_{\dd''}}& \ar[l]_-{q_1\times q_3} \FM^{T,\zeta\sst}_{\dd',\dd''}(Q)\ar[r]^-{q_2}& \FM^{T,\zeta\sst}_{\dd}(Q)\ar[d]^{\JH^{\zeta}_{\dd}}
\\
\CM^{T,\zeta\sst}_{\dd'}(Q)\times_{\B T} \CM^{T,\zeta\sst}_{\dd''}(Q)\ar[rr]^-{\oplus}&&\CM^{T,\zeta\sst}_{\dd}(Q)
}
\]
where $q_1,q_2,q_3$ are the forgetful maps taking a short exact sequence $0\rightarrow M'\rightarrow M\rightarrow M''\rightarrow 0$ to $M',M,M''$ respectively.  Then as in \cite[\S 7]{KS2} and \cite[\S 3]{Da13} we consider the natural adjunction morphisms 
\begin{align*}
\alpha\colon &\JH^{\zeta}_{\dd,*}\left(q_{2,*}\phip{\Tr(W)}\BoQ^{\vir}_{\FM^{T,\zeta\sst}_{\dd',\dd''}(Q)}[(\dd',\dd'')_Q]\rightarrow \phip{\Tr(W)}\BoQ^{\vir}_{\FM^{T,\zeta\sst}_{\dd}(Q)}\right)\\
\beta\colon &\oplus_*(\JH^{\zeta}_{\dd'}\times_{\B T} \JH^{\zeta}_{\dd''})_*\left(\phip{\Tr(W)}\BoQ^{\vir}_{\FM^{T,\zeta\sst}_{\dd'}(Q)\times_{\B T} \FM^{T,\zeta\sst}_{\dd''}(Q)}\xrightarrow{\cong} \phip{\Tr(W)}(q_1\times q_3)_*\BoQ^{\vir}_{\FM^{T,\zeta\sst}_{\dd',\dd''}(Q)}[(\dd',\dd'')_Q]\right)
\end{align*}
Via the structure morphism $\B T\rightarrow \pt$ we obtain the morphism $p\colon \FM^{T,\zeta\sst}_{\dd'}(Q)\times_{\B T} \FM^{T,\zeta\sst}_{\dd''}(Q)\rightarrow \FM^{T,\zeta\sst}_{\dd'}(Q)\times \FM^{T,\zeta\sst}_{\dd''}(Q)$ and the adjunction morphism
\[
\HO(\delta)\colon\HO^*(\FM^{T,\zeta\sst}_{\dd'}(Q)\times \FM^{T,\zeta\sst}_{\dd''}(Q),\phip{\Tr(W)}\BoQ^{\vir})\rightarrow \HO^*(\FM^{T,\zeta\sst}_{\dd'}(Q)\times_{\B T} \FM^{T,\zeta\sst}_{\dd''}(Q),\phip{\Tr(W)}\BoQ^{\vir})
\]
which we compose with the inverse of the Thom--Sebastiani isomorphism \cite{Ma01}
\[
\TS\colon\HO^*(\FM^{T,\zeta\sst}_{\dd'}(Q)\times \FM^{T,\zeta\sst}_{\dd''}(Q),\phip{\Tr(W)}\BoQ^{\vir})\cong \HO^*(\FM^{T,\zeta\sst}_{\dd'}(Q),\phip{\Tr(W)}\BoQ^{\vir})\otimes \HO^*(\FM^{T,\zeta\sst}_{\dd''}(Q),\phip{\Tr(W)}\BoQ^{\vir}).
\]
We set
\begin{align*}
\vmult_{\dd',\dd''}=\HO(\alpha)\HO(\beta)\HO(\delta) \TS^{-1}\colon \HO\!\CoHA^{T,\zeta}_{Q,W,\dd'}\otimes\HO\!\CoHA^{T,\zeta}_{Q,W,\dd''}\rightarrow \HO\!\CoHA^{T,\zeta}_{Q,W,\dd}
\end{align*}
and define the CoHA multiplication $\vmult\colon \HO\!\CoHA^{T,\zeta}_{Q,W,\theta}\otimes\HO\!\CoHA^{T,\zeta}_{Q,W,\theta}\rightarrow \HO\!\CoHA^{T,\zeta}_{Q,W,\theta}$ by summing over all pairs $\dd',\dd''\in\Lambda_{\theta}^{\zeta,+}$.  The proof that this provides an associative multiplication is in \cite[\S 7]{KS2}.  Note that under favourable circumstances (for instance if the mixed Hodge structures on all of these cohomology groups are pure, which will be the case for all examples in our paper) the canonical morphism
\begin{align*}
\gamma\colon &\HO^*(\FM^{T,\zeta\sst}_{\dd'}(Q),\phip{\Tr(W)}\BoQ^{\vir})\otimes_{\HO_T} \HO^*(\FM^{T,\zeta\sst}_{\dd''}(Q),\phip{\Tr(W)}\BoQ^{\vir})\rightarrow \\&\HO^*(\FM^{T,\zeta\sst}_{\dd'}(Q)\times_{\B T} \FM^{T,\zeta\sst}_{\dd''}(Q),\phip{\Tr(W)}\BoQ^{\vir})
\end{align*}
through which $\HO(\delta) \TS^{-1}$ factors is an isomorphism (see \cite[\S 9]{preproj}).

The morphism $\alpha\beta$ itself provides an algebra structure on the object 
\[
\CoHA^{T,\zeta}_{Q,W,\theta}\coloneqq \bigoplus_{\dd\in\Lambda_{\theta}^{\zeta}}\JH^{\zeta}_*\phip{\Tr(W)}\BoQ^{\vir}_{\FM_{\dd}^{T,\zeta\sst}(Q)}
\]
in the tensor category $\Dub_{\mathrm{c}}^+(\coprod_{\dd\in\Lambda^{\zeta}_{\theta}}\CM_{\dd}^{T,\zeta\sst}(Q))$ of locally bounded below constructible complexes, which is referred to as the \textit{relative Hall algebra}.  

The structure morphism $\BoA_\dd^{\zeta\sst}(Q)\rightarrow \pt$ induces the morphism $\omega\colon \FM^{T,\zeta\sst}_{\dd}(Q)\rightarrow \B \Gl_{\dd}^{\wt}$.  The function $\Tr(W)$ on $\FM^{T,\zeta\sst}_{\dd}(Q)$ factors through $\omega\times \id \colon \FM^{T,\zeta\sst}_{\dd}(Q)\rightarrow  \B \Gl_{\dd}^{\wt}\times \FM^{T,\zeta\sst}_{\dd}(Q) $, inducing the morphism $\HO_{\Gl_{\dd}^{\wt}}\otimes\HO\!\CoHA^{T,\zeta}_{Q,W,\dd}\rightarrow \HO\!\CoHA^{T,\zeta}_{Q,W,\dd}$, defining the action of tautological classes on $\HO\!\CoHA^{T,\zeta}_{Q,W,\dd}$.

See \cite[\S 5.1]{QEAs} for details of the relative Hall algebra, and \cite{KS2,Da13, QEAs} for a fuller account of cohomological Hall algebras and vanishing cycle cohomology.
\subsection{BPS cohomology}
\label{BPS_coh_sec}
Let $Q$ be a symmetric quiver, with a weighting $\wt\colon Q_1\rightarrow \BoZ^r$, and a $\wt$-invariant potential $W$.  Fix a dimension vector $\dd\in \BoN^{Q_0}$ and a stability condition $\zeta\in\BoQ^{Q_0}$. Recall the shifted perverse t-structure $\fp'$ introduced in \S\ref{conventions_sec}.  By \cite[Thm.A]{QEAs} the perverse cohomology sheaves ${}^{\fp'}\!\CH^i\!\left(\JH^{\zeta}_*\phip{\Tr(W)}\BoQ_{\FM^{T,\zeta\sst}_{\dd}(Q)}^{\vir}\right)$ vanish for $i\leq 0$.  We define the \textit{BPS sheaf}
\[
\BPSh^{T,\zeta}_{Q,W,\dd}\coloneqq {}^{\fp'}\!\CH^1\!\left(\JH^{\zeta}_*\:\phip{\Tr(W)}\BoQ_{\FM^{T,\zeta\sst}_{\dd}(Q)}^{\vir}\right).
\]
By definition, it is an element of $\Perv'(\CM^{T,\zeta\sst}_{\dd}(Q))$.  We define
\begin{align*}
\Fg^{T,\zeta}_{Q,W,\dd}\coloneqq &\HO^*(\CM^{T,\zeta\sst}_{\dd}(Q),{}^{\fp'}\!\vtau^{\leq 1}\JH^{\zeta}_*\phip{\Tr(W)}\BoQ_{\FM^{T,\zeta\sst}_{\dd}(Q)}^{\vir})\\
=&\HO^*(\CM^{T,\zeta\sst}_{\dd}(Q),\BPSh^{T,\zeta}_{Q,W,\dd}[-1]).
\end{align*}
Via the natural transformation ${}^{\fp'}\!\vtau^{\leq 1}\rightarrow \id$ we obtain the morphism $\iota\colon \Fg^{T,\zeta}_{Q,W,\dd}\rightarrow \HO\!\CoHA^{T,\zeta}_{Q,W,\dd}$.  By \cite[Thm.C]{QEAs} $\iota$ is an injection, and the image is closed under the commutator Lie bracket.  Throughout the rest of the paper we use this fact to consider $\Fg^{T,\zeta}_{Q,W,\theta}\coloneqq \bigoplus_{\dd\in\Lambda^{\zeta}_{\theta}}\Fg^{T,\zeta}_{Q,W,\dd}$ as a Lie subalgebra of $\HO\!\CoHA^{T,\zeta}_{Q,W,\theta}$. For the degenerate stability condition $\zeta_{\mathrm{deg}}=(0,\ldots,0)$ and $\theta_{\mathrm{deg}}=0$ we abbreviate $\Fg_{Q,W}^T\coloneqq \Fg_{Q,W,\theta_{\mathrm{deg}}}^{T,\zeta_{\mathrm{deg}}}$.

The stack $\FM^{T,\zeta\sst}_{\dd}(Q)$ carries a canonical family of $\BoC Q$-modules $\CF$, and so a tautological line bundle, obtained by taking the determinant line bundle of the underlying vector bundle of $\CF$.  As such we obtain the map $\eta\times \id\colon \FM^{T,\zeta\sst}_{\dd}(Q)\rightarrow \B \BoC^*\times \FM^{T,\zeta\sst}_{\dd}(Q) $ where $\eta$ is the classifying morphism of the determinant line bundle. In terms of the morphism $\omega$ from the previous subsection, we can write $\eta=\det \circ \omega$ where $\det\colon \B \Gl_{\dd}^{\wt}\rightarrow \B \BoC^*$ is induced by composing the projection $\Gl_{\dd}^{\wt}\rightarrow \Gl_{\dd}$ with the determinant morphism $\Gl_{\dd}\rightarrow \BoC^*$. Passing to vanishing cycle cohomology, we obtain an action
\[
\vact\colon \HO_{\BoC^*}\otimes \HO\!\CoHA^{T,\zeta}_{Q,W,\dd} \rightarrow \HO\!\CoHA^{T,\zeta}_{Q,W,\dd}.
\] 

We define the morphism 
\[
\Theta\colon \Sym_{\HO_T}\!\left(\HO_{\BoC^*}\otimes \Fg^{T,\zeta}_{Q,W,\theta} \right)\rightarrow \HO\!\CoHA^{T,\zeta}_{Q,W,\theta}
\]
as follows.  Firstly, the domain is defined to be the subspace
\[
\bigoplus_{n\geq 0}\!\left(\left(\HO_{\BoC^*}\otimes \Fg^{T,\zeta}_{Q,W,\theta} \right)^{\otimes_{\HO_T}n}\right)^{\mathfrak{S}_n}\subset\Tens_{\HO_T}\!\left(\HO_{\BoC^*}\otimes\Fg^{T,\zeta}_{Q,W,\theta} \right)
\]
where the symmetric group acts via the symmetric tensor structure on cohomologically graded, $\Lambda_{\theta}^{\zeta,+}$-graded vector spaces $\tau$-twisted as in \S \ref{tau_twist_sec}, and the right hand side is as defined in \S \ref{conventions_sec}.  Secondly, given any element of $(\HO\!\CoHA^{T,\zeta}_{Q,W,\theta})^{\otimes _{\HO_T}n}$ (in particular, any element that is a $\tau$-twisted symetric tensor) we use the iterated Hall algebra multiplication $\vmult$ to evaluate it to an element of $\HO\!\CoHA^{T,\zeta}_{Q,W,\theta}$.  Thirdly, given an element of $\HO_{\BoC^*}\otimes \Fg^{T,\zeta}_{Q,W,\theta} $ we use the inclusion $\Fg^{T,\zeta}_{Q,W,\theta}\subset \HO\!\CoHA^{T,\zeta}_{Q,W,\theta}$ along with the action morphism $\vact$ to obtain an element of $\HO\!\CoHA^{T,\zeta}_{Q,W,\theta}$.  The morphism $\Theta$ is obtained by combining these two operations.
We may now recall the main theorem of \cite{QEAs}.
\begin{theorem} \cite[Thm.C]{QEAs}
\label{PBW_thm}
The PBW morphism $\Theta$ is an isomorphism.
\end{theorem}
As an immediate corollary, we record the following.
\begin{corollary}
The algebra $\HO\!\CoHA^{T,\zeta}_{Q,W,\theta}$ is generated by $\Fg^{T,\zeta}_{Q,W,\theta}$ and the multiplication by tautological classes.
\end{corollary}
The following result turns out to be an indispensable tool in calculating cohomology of BPS sheaves.  In particular, in this paper we will use it to identify cohomology of Nakajima quiver varieties with cohomology of certain BPS sheaves (see \S \ref{BPS_to_Nak_sec} for more details).
\begin{proposition}\cite[Lem.4.7]{To17}
\label{Toda_prop}
Let $Q$ be a symmetric quiver, let $W\in \BoC Q/[\BoC Q,\BoC Q]_{\vect}$ be a potential, let $\wt\colon Q_1\rightarrow \BoZ^r$ be a $W$-preserving weighting, and let $\zeta\in\BoQ^{Q_0}$ be a stability condition.  Fix $\dd\in\BoN^{Q_0}$ and consider the commutative diagram
\[
\begin{tikzcd}
\FM^{T,\zeta\sst}_{\dd}(Q) \arrow[d,"\JH^{\zeta}"]\arrow[r,hook,"j"]&\FM^{T}_{\dd}(Q)\arrow[d,"\JH"]
\\
\CM^{T,\zeta\sst}_{\dd}(Q)\arrow[r,"\pi"]&\CM^T_{\dd}(Q).
\end{tikzcd}
\]
There is a canonical isomorphism
\begin{equation}
\label{Toda_iso}
\pi_*\BPSh_{Q,W,\dd}^{T,\zeta}\cong \BPSh^T_{Q,W,\dd}.
\end{equation}

\end{proposition}

\begin{remark}
\label{toda_rem}
Let $\zeta$ be generic, so that $\JH^{\zeta}$ is a $\BoC^*$-gerbe, and the morphism $\Theta$, restricted to degree $\dd\in\BoN^{Q_0}$ becomes the isomorphism
\[
\Theta_{\dd}\colon \HO_{\BoC^*}\otimes \Fg^{T,\zeta}_{Q,W,\theta}\rightarrow  \HO\!\CoHA^{T,\zeta}_{Q,W,\dd}.
\]
Passing to derived global sections, and acting by tautological classes, the isomorphism \eqref{Toda_iso} yields the inclusion
\begin{equation}
\label{jdef}
j\colon \HO_{\BoC^*}\otimes \Fg^{T,\zeta}_{Q,W,\theta}\hookrightarrow \HO\!\CoHA^{T}_{Q,W,\dd}.
\end{equation}
By construction, the morphism 
\[
\jmath=j \circ \Theta_{\dd}^{-1}\colon \HO\!\CoHA^{T,\zeta}_{Q,W,\dd}\rightarrow \HO\!\CoHA^{T}_{Q,W,\dd}
\]
is a section to the restriction morphism $r\colon \HO\!\CoHA^{T}_{Q,W,\dd}\rightarrow  \HO\!\CoHA^{T,\zeta}_{Q,W,\dd}$.
\end{remark}

\subsection{Preprojective CoHAs}
\label{subsection:Preprojective CoHA}
By restriction, a weighting $\wt\colon \tilde{Q}_1\rightarrow \BoZ^r$ induces a weighting $\wt'=\wt\lvert_{\overline{Q}_1}\colon \overline{Q}_1\rightarrow \BoZ^r$.  We assume that $\tilde{W}$ has weight zero with respect to the weighting $\wt$.  If we assume that $Q$ is connected, the preprojective relation $\sum_{a\in Q_1}[a,a^*]$ is homogeneous for the $\wt'$-weighting. We call such a weighting of $\overline{Q}$ \textit{semi-invariant}. Recall that we set $\hbar=\ttt(\wt(a))+\ttt(\wt(a^*))\in\HO^2(\B T,\BoQ)$.  We define the stack $\FM^{T}_{\dd}(\Pi_Q)$ as in \S \ref{moduli_stacks_sec}.  We set
\[
\BoD\BoQ_{\FM^{T}_{\dd}(\Pi_Q)}^{\vir}\coloneqq (\BoD\BoQ_{\FM^{T}_{\dd}(\Pi_Q)})[(\dd,\dd)_Q],
\]
which we may consider as a constructible complex on $\FM^T_{\dd}(\overline{Q})$ via direct image along the inclusion $\FM^T_{\dd}(\Pi_Q)\hookrightarrow\FM^T_{\dd}(\overline{Q})$.  We define the $\BoN^{Q_0}$-graded, cohomologically graded $\HO_T$-module $\HO\!\CoHA_{\Pi_Q}^T$ by setting 
\begin{align*}
\HCoha_{\Pi_Q,\dd}^{T}=&\HO^{\BoMo}(\FM^{T}_{\dd}(\Pi_Q),\BoQ^{\vir})\\
\coloneqq&\HO^*(\FM^{T}_{\dd}(\Pi_Q),\BoD\BoQ_{\FM^{T}_{\dd}(\Pi_Q)}^{\vir}).
\end{align*}
The $\HO_T$-module $\HCoha_{\Pi_Q,\dd}^{T}$ is free by \cite[Cor.9.7]{preproj}.  Let 
\begin{equation}
\label{dim_proj}
p\colon \FM_{\dd}^T(\tilde{Q})\rightarrow \FM_{\dd}^T(\overline{Q})
\end{equation}
be the natural forgetful morphism.  By \cite[Cor.A.9]{Da13} the natural morphism
\begin{equation}
\label{dim_red}
\BoD\BoQ_{\FM^{T}_{\dd}(\Pi_Q)}^{\vir}\rightarrow \pi_*\phip{\Tr(\tilde{W})}\BoQ^{\vir}_{\FM_{\dd}(\tilde{Q})}
\end{equation}
is an isomorphism.  Passing to derived global sections, we obtain the isomorphism of cohomologically graded, $\BoN^{Q_0}$-graded $\HO_T$-modules
\begin{equation}
\label{gdim_red}
\dimred\colon \HCoha^T_{\Pi_Q}\cong \HCoha^T_{\tilde{Q},\tilde{W}}.
\end{equation}
We endow $\HCoha^T_{\Pi_Q}$ with an algebra structure by pulling back the algebra structure on $\HCoha^T_{\tilde{Q},\tilde{W}}$, defined as a special case of the construction in \S \ref{gen_CoHA_sec}, via the isomorphism $\dimred$.  By \cite[Thm.4.4]{KR17}, \cite{YZ16} this\footnote{To be precise, we have that $\dimred^*\vmult_{\dd',\dd''}$ differs from the Schiffmann--Vasserot multiplication by the sign $(-1)^{\dd'\cdotsh\dd''}$.  See \cite[Lem.4.1]{KR17} for the provenance of this sign, which appears again in the proof of Proposition \ref{mult_compat}.} is the same algebra structure as was introduced by Schiffmann and Vasserot in \cite{ScVa13} (and then generalised from the case of the Jordan quiver by Yang and Zhao in \cite{YZ18}).

\begin{definition}
\label{BPSLAdef}
We define the \textit{BPS Lie algebra} for the category of $\Pi_Q$-modules
\[
\Fn_{\Pi_Q}^{T,+}\coloneqq\dimred^{-1}(\Fg_{\tilde{Q},\tilde{W}}^{T}).
\]
\end{definition}
Since we consider $\Fg_{\tilde{Q},\tilde{W}}^{T}$ as a subobject of $\HCoha^T_{\tilde{Q},\tilde{W}}$, we may consider $\Fn_{\Pi_Q}^{T,+}$ as a subobject of $\HCoha^T_{\Pi_Q}$, via $\dimred$ and a Lie subalgebra for the commutator Lie bracket.  As usual, since the stability condition $\zeta$ is omitted, we mean the case $\zeta_{\mathrm{deg}}=(0,\ldots,0)$.

The BPS Lie algebra carries a cohomological grading, so that we can take its characteristic function, yielding the connection with Kac polynomials:
\begin{proposition}\cite[\S 8.1]{preproj}
\label{BPS_char_function}
There is an equality of generating series 
\begin{align*}
\chi_{q^{1/2}}(\Fn^+_{\Pi_Q,\dd})\coloneqq& \sum_{k\in \BoZ}\dim((\Fn^+_{\Pi_Q,\dd})^k)q^{k/2}\\
=&\kac_{Q,\dd}(q^{-1}).
\end{align*}
\end{proposition}
The next proposition says that, considering $\Fn^{T,+}_{\Pi_Q}$ as a family of Lie algebras over $\Ft^*$, it is in fact the trivial deformation of the algebra $\Fn^{+}_{\Pi_Q}$ defined by setting all of the equivariant parameters to be zero.
\begin{proposition}\cite[Cor.11.9]{DHSM23}
\label{Triv_BPS_ext}
Let $\wt\colon\tilde{Q}_1\rightarrow \BoZ^r$ be a weighting, with associated torus $T=\Hom_{\Grp}(\BoZ^r,\BoC^*)$, and assume that $\tilde{W}$ is $\wt$-invariant.  There is an isomorphism of Lie algebras
\[
\Fn^{T,+}_{\Pi_Q}\cong \Fn^+_{\Pi_Q}\otimes \HO_T
\]
where the Lie bracket on the right hand side is the $\HO_T$-linear extension of the Lie bracket on $\Fn^+_{\Pi_Q}$: if $\alpha,\beta\in \fn^+_{\Pi_Q}$ and $p,q\in \HO_T$ then $[\alpha\otimes p,\beta\otimes q]\coloneqq [\alpha,\beta]\otimes pq$.
\end{proposition}

\subsection{Relative and perverse associated graded CoHA}
\label{perverse_CoHA_sec}
The morphism $\oplus\colon \CM(\Pi_Q)\times\CM(\Pi_Q)\rightarrow \CM(\Pi_Q)$ taking a pair of modules to their direct sum is $T$-equivariant and finite by \cite[Lem.2.1]{Meinhardt14}.  We denote by the same symbol the induced finite morphism of stacks
\[
\oplus\colon\CM^T(\Pi_Q)\times_{\B T}\CM^T(\Pi_Q)\rightarrow \CM^T(\Pi_Q).
\]
In this section we denote by $\JH_{\dd}\colon \FM^{T}_{\dd}(\Pi_Q)\rightarrow \CM^T_{\dd}(\Pi_Q)$ the morphism to the stack-theoretic quotient of the coarse moduli space. By \cite[Thm.C]{Da21a} the decomposition theorem holds for the morphism $\JH_{\dd}$, for the complex $\BoD\BoQ_{\FM^T_{\dd}(\Pi_Q)}^{\vir}$.  More precisely, we can write
\begin{align}
\JH_{\dd,*}\BoD\BoQ_{\FM^T_{\dd}(\Pi_Q)}^{\vir}\cong&\bigoplus_{i\in2\cdot \BoZ_{\geq 0}}{}^{\mathfrak{p}'}\!\!\Ho^i(\JH_{\dd,*}\BoD\BoQ_{\FM^T_{\dd}(\Pi_Q)}^{\vir})\label{dec_thm_1}\\
\cong&\bigoplus_{i\in 2\cdot\BoZ_{\geq 0}}\bigoplus_{j\in S_i} \ICS_{\overline{Z_j/T}}(\CL_j)[\dim(Z_j)-i]\nonumber
\end{align}
where $Z_j\subset X_{\dd}(\Pi_Q)$ are $T$-invariant locally closed irreducible smooth subvarieties of $\CM_{\dd}(\Pi_Q)$, and $\CL_j$ are simple local systems on $Z_j/T$.  We define 
\[
\Coha_{\Pi_Q,\dd}^T\coloneqq \JH_{\dd,*}\BoD\BoQ_{\FM^T_{\dd}(\Pi_Q)}^{\vir};\quad\quad\quad\CoHA^T_{\Pi_Q}\coloneqq\bigoplus_{\dd\in\BoN^{Q_0}}\CoHA^T_{\Pi_Q,\dd}.
\]

Consider the commutative diagram
\begin{equation}
\label{forg_diag}
\xymatrix{
\FM^T(\tilde{Q})\ar[d]^{\tilde{\JH}}\ar[r]^{p}&\FM^T(\overline{Q})\ar[d]^{\overline{\JH}}\\
\CM^T(\tilde{Q})\ar[r]^{\varpi}&\CM^T(\overline{Q}),
}
\end{equation}
where the two vertical maps are the usual morphisms, obtained by taking the $T$-quotient of the respective affinization morphisms. Since $\varpi$ is a morphism of monoids, the direct image $\varpi_*\Coha_{\tilde{Q},\tilde{W}}^T$ of the algebra object $\Coha_{\tilde{Q},\tilde{W}}^T$ inherits an algebra structure from the algebra structure $\vmult$ on $\Coha_{\tilde{Q},\tilde{W}}^T$.  Furthermore, via \eqref{dim_red} we have
\begin{align*}
\varpi_*\Coha_{\tilde{Q},\tilde{W}}^T\cong& \ol{\JH}_*p_*\phip{\Tr(\tilde{W})}\BoQ^{\vir}_{\FM(\tilde{Q})}\\
\cong &\Coha_{\Pi_Q}^T.
\end{align*}
So $\Coha_{\Pi_Q}^T$ inherits an algebra structure, and the algebra structure on $\HCoha_{\Pi_Q}^T$ induced by applying the derived global sections functor to $\Coha_{\Pi_Q}^T$ is the same as the one induced via the isomorphism \eqref{gdim_red}.

We set 
\begin{equation}
\label{ZPSA}
\Coha_{\Pi_Q}^{T,0}\coloneqq {}^{\Fp'}\!\!\vtau^{\leq 0}\Coha_{\Pi_Q}^{T}.  
\end{equation}
This carries the multiplication induced by composing the morphisms
\[
\Coha_{\Pi_Q}^{T,0}\boxdot\Coha_{\Pi_Q}^{T,0}\cong {}^{\Fp'}\!\!\vtau^{\leq 0}(\Coha_{\Pi_Q}^T\boxdot \Coha_{\Pi_Q}^T)\xrightarrow{{}^{\Fp'}\!\!\vtau^{\leq 0}\vmult}{}^{\Fp'}\!\!\vtau^{\leq 0}\Coha_{\Pi_Q}^T
\]
where $\vmult$ is the multiplication morphism for $\Coha_{\Pi_Q}^T$.  We set 
\[
\FP^0\!\HCoha_{\Pi_Q}^{T}\coloneqq \HO^*(\CM^T(\Pi_Q),\Coha_{\Pi_Q}^{T,0}),
\]
which is an associative algebra, as it is obtained by taking derived global sections of the algebra object $\Coha_{\Pi_Q}^{T,0}$ in $\Perv'(\CM^T(\Pi_Q))$.  The natural transformation ${}^{\Fp'}\!\!\vtau^{\leq 0}\rightarrow \id$ induces the morphism of algebra objects $\FP^0\!\HCoha_{\Pi_Q}^{T}\hookrightarrow \HCoha_{\Pi_Q}^T$ which is an inclusion by the decomposition \eqref{dec_thm_1}; it is the inclusion of the zeroth piece of the perverse filtration induced by this decomposition.  The inclusion $\Fn_{\Pi_Q,\dd}^{T,+}\hookrightarrow \HCoha_{\Pi_Q}^T$ factors through the inclusion of $\FP^0\!\HCoha_{\Pi_Q}^{T}$.

More generally we define 
\[
\FP^n\HO\!\CoHA_{\Pi_Q,\dd}^T=\HO^*(\CM^T(\Pi_Q),{}^{\Fp'}\!\vtau^{\leq n}\!\CoHA_{\Pi_Q,\dd}^T),
\]
defining an ascending filtration on $\HO\!\CoHA_{\Pi_Q}^T$ (again using \eqref{dec_thm_1}), which is respected by the algebra structure on $\HO\!\CoHA_{\Pi_Q}^T$.  This is the perversely filtered algebra introduced and studied in \cite{Da21a,preproj3}.  We denote by $\Gr_{\FP}\HO\!\CoHA_{\Pi_Q}^T$ the resulting associated graded algebra.

\begin{remark}
\label{less_perv_rem}
There is a second natural perverse filtration on $\HO\!\CoHA_{\Pi_Q}^T$, induced by the isomorphism $\dimred\colon \HO\!\CoHA_{\Pi_Q}^T\cong \HO\!\CoHA_{\tilde{Q},\tilde{W}}^T$.  By \cite{QEAs}, an arbitrary critical cohomological Hall algebra carries a perverse filtration, defined in this instance via the perverse t-structure on the category of constructible complexes on $\CM^T(\tilde{Q})$, and we may pull back the resulting filtration on $\HO\!\CoHA_{\tilde{Q},\tilde{W}}^T$ along $\dimred$.  This is \textit{not} the same as the perverse filtration that we consider in this paper, which is introduced as the ``less perverse'' filtration in \cite{preproj3}.
\end{remark}

\subsection{Realisations of Kac--Moody algebras}
\label{GKM_thm_sec}
We summarise the main results of \cite{DHSM23}.  Firstly, by \cite[Thm.B]{preproj3}, for each primitive positive root $\dd\in\Sigma_Q$ there is a canonical inclusion $\Upsilon_{\dd}\colon \ICS(\CM^T_{\dd}(\Pi_Q))\hookrightarrow \CoHA_{\Pi_Q,\dd}^{T,0}$.  We set $\CG(Q)_{\dd}\coloneqq\ICS(\CM^T_{\dd}(\Pi_Q))$ for $\dd\in\Sigma_Q$. For $\dd\in \Phi_Q^+$ not a primitive positive root, let $\dd'\in \Sigma_{Q}$ be the unique imaginary simple root such that $\dd= n \dd'$ for some $n\in\BoZ_{\geq 2}$.  Let $\Diag_n\colon \CM^T_{\dd'}(\Pi_Q)\rightarrow \CM^T_{\dd}(\Pi_Q)$ be the embedding taking a semisimple module $M$ to its $n$-fold direct sum.  Then by \cite[Thm.7.37]{Da21a} there is a canonical inclusion $\Upsilon_{\dd}\colon\Diag_{n,*}\ICS(\CM^T_{\dd'}(\Pi_Q))\hookrightarrow  \CoHA_{\dd}^{T,0}(\Pi_Q)$.  We set $\CG(Q)_{\dd}\coloneqq \Diag_{n,*}\ICS(\CM^T_{\dd'}(\Pi_Q))$ for such roots.  Then we set
\[
\CG(Q)\coloneqq\bigoplus_{\dd\in\Phi^+_Q}\CG(Q)_{\dd}.
\]
\begin{theorem}\cite[Thm.1.1, Thm.11.2, Thm.1.16]{DHSM23}
\label{KMA_thm}
The embedding $\Upsilon\colon\CG(Q)\hookrightarrow \CoHA^{T,0}_{\Pi_Q}$ extends uniquely to an isomorphism of algebra objects 
\[
\UEA_{\boxdot}\!\left(\Fn^{T,+}_{\CG(Q)}\right)\cong \CoHA_{\Pi_Q}^{T,0}.
\]
Taking derived global sections, we deduce that there is an isomorphism
\[
\UEA_{\HO_T}\left(\Fn^{T,+}_{\HO^*(\CM^T(\Pi_Q),\CG(Q))}\right)\cong \FP^0\!\HO\!\CoHA^T_{\Pi_Q}.
\]
Composing this isomorphism with the natural inclusion of $\Fn^{T,+}_{\HO^*(\CM^T(\Pi_Q),\CG(Q))}$ into the left hand side, we obtain an isomorphism of Lie algebras
\[
h\colon\Fn^{T,+}_{\HO^*(\CM^T(\Pi_Q),\CG(Q))}\cong \fn^{T,+}_{\Pi_Q}.
\]
\end{theorem}
In words, the zeroth perverse piece of the preprojective CoHA is isomorphic to the positive half of the universal enveloping algebra of the generalised Kac--Moody Lie algebra generated by derived global sections of $\CG(Q)$, as defined in \S \ref{GKM_sec}.  The morphism $h^{-1}$ provides an embedding of Lie algebras $\fn^{T,+}_{\Pi_Q}\hookrightarrow \fg'^{T,+}_{\Pi_Q}$, where we define
\begin{equation}
\label{gprime_def}
\fg'^{T,+}_{\Pi_Q}\coloneqq \Fg^{T}_{\HO^*(\CM^T(\Pi_Q),\CG(Q))}.
\end{equation}

We refer to homogeneous elements $e'\in \HO^*(\CM^T(\Pi_Q),\CG(Q))$ and $f'\in \HO^*(\CM^T(\Pi_Q),\CG(Q))^{\vee}$ as Chevalley generators, as in Definition \ref{Chev_def}.

In defining generalised Kac--Moody Lie algebras, there is some flexibility over how we define the Cartan subalgebra.  Above, we have followed the convention of \cite{DHSM23}.  In order to match with \cite[\S 1.2.9]{MO19} it is convenient to work with a larger Cartan subalgebra, which we define as follows.  Let $Q$ be a quiver, and let $\overrightarrow{Q}$ be the \textit{principally framed} quiver obtained by adding a new vertex $i'$ and an arrow $i'\rightarrow i$ for every $i\in Q_0$, as in \cite[\S 2.1.1]{MO19}.  We write dimension vectors (with potentially negative entries) for this larger quiver as pairs $(\dd,\ff)\in (\BoZ^{Q_0})^{\oplus 2}$, where $\ff$ is the assignment of numbers to the new vertices $i'$.  Then $\dd\mapsto ((\dd,\ul{0}),-)_{\overrightarrow{Q}}$ defines an injection $\fh^T_Q\hookrightarrow (\fh^T_{\overrightarrow{Q}})^{\vee}$.  We use this to define an embedding of $\HO_T$-modules, extending the Cartan algebra
\[
\fg'^{T}_{\Pi_Q}\hookrightarrow \fg^{T}_{\Pi_Q} \coloneqq \fn^{T,+}_{\Pi_Q}\oplus  (\fh^T_{\overrightarrow{Q}})^{\vee}\oplus\fn^{T,-}_{\Pi_Q}.
\]
We extend the Lie bracket in the obvious way: $(\fh^T_{\overrightarrow{Q}})^{\vee}$ is Abelian, and $[h,g]=h(\dd)g$ for $g\in \fn^{T,\pm}_{\Pi_Q,\dd}$ and $h\in(\fh^T_{\overrightarrow{Q}})^{\vee}$.
\subsection{BPS cohomology from Nakajima quiver varieties}
\label{BPS_to_Nak_sec}
Fix a nonzero framing dimension vector $\ff\in\BoN^{Q_0}$.  We fix the stability condition $\zeta^+=\delta_{\infty}\in\BoQ^{Q_0}$ satisfying $(\zeta^+)_i=0$ for all $i\in Q_0$ and $(\zeta^+)_{\infty}=1$.  Recall that we have defined $\Nak_Q^T(\ff,\dd)=\CM^{T,\zeta^+\sst}_{(\dd,1)}(\Pi_{Q_{\ff}})$.  We extend the analogue of diagram \eqref{forg_diag} in the framed setup, forming the following commutative diagram:
\begin{equation}
\label{bBPS_diag}
\begin{tikzcd}
&\FM_{(\dd,1)}^T(\tilde{Q}_{\ff})\arrow[d,"\tilde{\JH}"]\ar[r,"p"]&\FM_{(\dd,1)}^T(\overline{Q_{\ff}})\ar[d,"\overline{\JH}"]\\
&\CM_{(\dd,1)}^T(\tilde{Q}_{\ff})\ar[r,"\varpi"]&\CM_{(\dd,1)}^T(\overline{Q_{\ff}})&\CM^T_{(\dd,1)}(\Pi_{Q_{\ff}})\arrow[l,hook',swap,"\xi"]\\
\Nak_Q^T(\ff,\dd)\times\BoA^1\ar[hook,r,"\tilde{\xi}"]&\CM_{(\dd,1)}^{T,\zeta^+\sst}(\tilde{Q}_{\ff})\ar[u,swap,"\tilde{\pi}"]\ar[r,"\varpi^+"]&\CM_{(\dd,1)}^{T,\zeta^+\sst}(\overline{Q_{\ff}})\ar[u,swap,"\overline{\pi}"]&\Nak^T_{Q}(\ff,\dd).\arrow[l,hook',swap,"\overline{\xi}"]\arrow[u,swap,"\pi"]
\end{tikzcd}
\end{equation}

We denote by $\overline{\xi}\colon \Nak_Q^T(\ff,\dd)\hookrightarrow \CM^{\zeta^+\sst}_{(\dd,1)}(\overline{Q_{\ff}})$ the obvious embedding, induced by the surjection $\BoC \overline{Q_{\ff}}\rightarrow \Pi_{Q_{\ff}}$.  We denote by $\tilde{\xi}\colon \Nak_Q^T(\ff,\dd)\times\BoA^1\hookrightarrow \CM^{T,\zeta^+\sst}_{(\dd,1)}(\tilde{Q}_{\ff})$ the closed embedding taking a pair consisting of a $\zeta^+$-stable $\Pi_{Q_{\ff}}$-module $\rho$ and a scalar $z\in\BoC$ to the $\zeta^+$-stable $\BoC \tilde{Q}_{\ff}$-module for which the underlying $\overline{Q_{\ff}}$-action is the given $\Pi_{Q_{\ff}}$-action, and the extra loops $\omega_i$ all act by multiplication by $z$.  We denote by 
\[
\tilde{W_{\ff}}=\left(\sum_{a\in Q_{\ff}}[a,a^*]\right)\left(\sum_{i\in (Q_{\ff})_0}\omega_i\right)
\]
the canonical cubic potential on $\tilde{Q}_{\ff}$.
\begin{proposition}\cite[Prop.6.3]{preproj3}
\label{dr_lift}
There is an isomorphism of perverse sheaves
\[
\BPSh^{T,\zeta^+}_{\tilde{Q}_{\ff},\tilde{W}_{\ff},(\dd,1)}\cong \tilde{\xi}_*\ICS(\Nak_Q^T(\ff,\dd)\times\BoA^1).
\]
\end{proposition}
\begin{corollary}
\label{dr_lift_cor}
There is an isomorphism of semisimple perverse sheaves
\[
c\colon \varpi^+_*\BPSh^{T,\zeta^+}_{\tilde{Q}_{\ff},\tilde{W}_{\ff},(\dd,1)}[-1]\rightarrow \overline{\xi}_*\ICS(\Nak^T_Q(\ff,\dd))=\overline{\xi}_*\BoQ_{\Nak^T_Q(\ff,\dd)}^{\vir}.
\]
\end{corollary}
The equality follows from smoothness of $\Nak^T_Q(\ff,\dd)$.
\begin{corollary}
\label{Nak_BPS_comp}
There is an isomorphism of cohomologically graded $\HO_T$-modules 
\[
\Fg^T_{\Pi_{Q_{\ff}},(\dd,1)}\cong \HO^*(\Nak^T_Q(\ff,\dd),\BoQ^{\vir}).
\]
\end{corollary}
\begin{proof}
By Definition \ref{BPSLAdef} and Proposition \ref{Toda_prop} we have isomorphisms
\begin{align*}
\Fg^T_{\Pi_{Q_{\ff}},(\dd,1)}=&\HO\!^*\left(\CM_{(\dd,1)}^T(\tilde{Q}_{\ff}),\BPSh^T_{\tilde{Q}_{\ff},\tilde{W}_{\ff},(\dd,1)}[-1]\right)\\
\cong &\HO\!^*\left(\CM_{(\dd,1)}^{T,\zeta^+\sst}(\tilde{Q}_{\ff}),\BPSh^{T,\zeta^+}_{\tilde{Q}_{\ff},\tilde{W}_{\ff},(\dd,1)}[-1]\right)
\end{align*}
and so the result follows from Corollary \ref{dr_lift_cor}.
\end{proof}
We will also use the following corollary.
\begin{corollary}
\label{KMLA_Nak_comp}
There is an isomorphism of semisimple perverse sheaves
\[
(\xi_*\fn^{T,+}_{\CG(Q_{\ff})})_{(\dd,1)}\cong \overline{\pi}_{*}\overline{\xi}_*\ICS(\Nak_Q^T(\ff,\dd)).
\]
\end{corollary}
\begin{proof}
By \cite[Thm.1.16]{DHSM23} there is an isomorphism $a\colon \omega_*\BPSh^T_{\tilde{Q}_{\ff},\tilde{W}_{\ff},(\dd,1)}[-1]\rightarrow \xi_*(\fn^{T,+}_{\CG(Q_{\ff})})_{(\dd,1)}$.  By Lemma \ref{Toda_prop} there is an isomorphism $b\colon \tilde{\pi}_*\BPSh^{T,\zeta^+}_{\tilde{Q}_{\ff},\tilde{W}_{\ff},\dd}[-1]\rightarrow \BPSh^T_{\tilde{Q}_{\ff},\tilde{W}_{\ff},\dd}[-1]$.  Now we compose isomorphisms, with $c$ as in Corollary \ref{dr_lift_cor}:
\begin{align*}
\xi_*(\fn^{T,+}_{\CG(Q_{\ff})})_{(\dd,1)}&\xrightarrow{a^{-1}}\omega_*\BPSh^T_{\tilde{Q}_{\ff},\tilde{W}_{\ff},(\dd,1)}[-1]\\
&\xrightarrow{\omega_*b^{-1}}\omega_*\tilde{\pi}_*\BPSh^{T,\zeta^+}_{\tilde{Q}_{\ff},\tilde{W}_{\ff},(\dd,1)}[-1]\\
&\xrightarrow{\cong} \overline{\pi}_{*}\omega^+_*\BPSh^{T,\zeta^+}_{\tilde{Q}_{\ff},\tilde{W}_{\ff},(\dd,1)}[-1]\\
&\xrightarrow{\overline{\pi}_{*}c}\overline{\pi}_{*}\overline{\xi}_*\ICS(\Nak_Q^T(\ff,\dd)).
\end{align*}
\end{proof}

The next proposition gives a complete description of the $\HO_T$-modules $\NakMod_{Q,\ff}^T$ as $\fg^T_{\Pi_Q}$-modules, which will be very useful when comparing with the action of $\fg^{\MO,T}_Q$ in the proof of Theorem \ref{main_thm}, given in \S \ref{main_thm_sec}.
\begin{proposition}\cite[\S 9]{DHSM23}
\label{lowest_weight_prop}
The $\BoN^{Q_0}$-graded, cohomologically graded $\HO_T$-module $\NakMod_{Q,\ff}^T$ is a direct sum of irreducible lowest weight representations for the generalised Kac--Moody Lie algebra $\fg^T_{\Pi_Q}$, where the action is provided by the isomorphisms $l_*\colon \fg^T_{\Pi_Q,\dd}\rightarrow \fg^T_{\Pi_Q,(\dd,0)}$ and Corollary \ref{Nak_BPS_comp}.  The inclusion of positive Chevalley generators for the Lie algebra $\fg^T_{\Pi_{Q_{\ff}}}$
\[
\bigoplus_{(\dd,1)\in\Phi^+_{Q_{\ff}}}\IC^*(\CM^T_{(\dd,1)}(\Pi_{Q_{\ff}}),\BoQ^{\vir})\hookrightarrow \NakMod^T_{Q,\ff}
\]
is the inclusion of the lowest weight vectors for the $\fg_{\Pi_Q}^T$-action.
\end{proposition}

\section{Shuffle bi-algebra}
\label{shuffle sec}
In this section we recall the definition of the shuffle algebra associated to a graded quiver, as well as its (localised) coproduct.  As is clear from the name, in order to define the coproduct, some localisation must first be taken care of, which we spell out in the following subsection.
\subsection{Localised tensor products}
\label{localisation_notation}
Let $Q$ be a symmetric quiver, with weighting $\wt\colon Q_1\rightarrow \BoZ^r$.  Fix $T=\Hom_{\Grp}(\BoZ^r,\BoC^*)$.    We work in the symmetric tensor category $\mathscr{C}=(\Vect_{\BoN^{Q_0}\oplus \BoZ})_{\tau}$ as in Example \ref{gvs_examp}.  Recall that this category has as objects vector spaces equipped with a $\BoN^{Q_0}\oplus \BoZ$-grading.  The $\BoZ$-factor is considered as keeping track of cohomological degrees, so that in the symmetrising isomorphism $V\otimes V'\rightarrow V'\otimes V$ the Koszul sign rule applies with respect to it.  The subscript $\tau$ indicates that we also sign-twist the usual symmetriser by the bilinear form in \S \ref{tau_twist_sec}.  See \eqref{twisted_sym} for the exact formula.

For $\dd\in\BoN^{Q_0}$ we define $\FS_{\dd}=\prod_{i\in Q_0} \FS_{\dd_i}$, the corresponding product of symmetric groups.  Given $\dd',\dd''\in\BoN^{Q_0}$  we define the Euler classes
\begin{align*}
\EU_{\dd',\dd''}^T(Q_1)\coloneqq &\prod_{a\in Q_1}\prod_{\substack{1\leq m\leq \dd'_{s(a)}\\1\leq n\leq \dd''_{t(a)}}}(x_{(2),t(a),n}-x_{(1),s(a),m}-\ttt(\wt(a)))\in \HO_{\Gl^{\wt}_{\dd'}}\otimes_{\HO_T}\HO_{\Gl^{\wt}_{\dd''}}\\
\EU_{\dd',\dd''}^T(Q_0)\coloneqq &\prod_{i\in Q_0}\prod_{\substack{1\leq m\leq \dd'_{i}\\1\leq n\leq \dd''_{i}}}(x_{(2),i,n}-x_{(1),i,m})\in \HO_{\Gl^{\wt}_{\dd'}}\otimes_{\HO_T}\HO_{\Gl^{\wt}_{\dd''}}.
\end{align*}
Here and throughout, where a tensor product of several cohomology rings of the form $\HO_{\Gl^{\wt}_{\ldots}}$ is being considered, we use the notation $x_{(m),\ldots}$ to denote the variables in the $m$th tensor factor.  We identify
\[
\HO_{\Gl_{\dd}^{\wt}}=\HO_T[x_{i,l}\lvert \; i\in Q_0,\; 1\leq l\leq \dd_i]^{\FS_{\dd}}.
\]
Let $\CF,\CG$ be objects of $\Vect_{\BoN^{Q_0}\oplus \BoZ}$, and assume that for each $\dd\in \BoN^{Q_0}$, each of $\CF_{\dd}$ and $\CG_{\dd}$ is given the structure of a $\HO_{\Gl^{\wt}_{\dd}}$-module.  
We write $\Gl_{\dd'\times\dd''}^{\wt}=\Gl_{\dd'}\times\Gl_{\dd''}\times T$.  We define the \emph{localised tensor product}
\begin{align*}
\CF\tilde{\otimes}\CG&\coloneqq\bigoplus_{\dd\in\BoN^{Q_0}}\left(\CF\tilde{\otimes}\CG\right)_{\dd}\\
\left(\CF\tilde{\otimes}\CG\right)_{\dd}&\coloneqq \bigoplus_{\dd'+\dd''=\dd}\left(\CF_{\dd'}\otimes_{\HO_T}\CG_{\dd''}\right)\otimes_{\HO_{\Gl^{\wt}_{\dd'\times \dd''}}}\HO_{\Gl^{\wt}_{\dd'\times\dd''}}[z_{\dd',\dd''}^{-1}]\\
z_{\dd',\dd''}&=\EU_{\dd',\dd''}^T(Q_1)\EU_{\dd',\dd''}^T(Q^{\opp}_1)\in \HO_{\Gl_{\dd'\times\dd''}^{\wt}}.
\end{align*}
In the definition, we give the quiver $Q^{\opp}$ the weighting $\wt^{\opp}$ from \S \ref{wtngs_ta_sec}.  We endow this tensor product with a localised symmetriser: for $\dd',\dd''\in \BoN^{Q_0}$ we define 
\begin{equation}
\label{swap_factor}
\tilde{\sw}_{\dd',\dd''}=\sw_{\dd',\dd''}\circ \left((-1)^{\tau(\dd', \dd'')}\EU^T_{\dd',\dd''}(Q^{\op}_1)^{-1}\EU^T_{\dd',\dd''}(Q_1)\cdot\right)
\end{equation}
where $\sw_{\dd',\dd''}$ is the usual swap $\CF_{\dd'}\otimes \CG_{\dd''}\rightarrow \CG_{\dd''}\otimes \CF_{\dd'}$ (incorporating the usual Koszul sign rule with respect to the cohomological grading).

Let $V^{(1)},\ldots, V^{(n)}$ be objects of $\SC$, and let us assume that for each $\dd\in\BoN^{Q_0}$ and each $1\leq m \leq n$ the cohomologically graded vector space $V_{\dd}^{(m)}$ carries an action of 
\[
\HO_{\Gl_{\dd}^{\wt}}=\HO_T[x_{i,l}\lvert \; i\in Q_0,\; 1\leq l\leq \dd_i]^{\FS_{\dd}},
\]
and $V^{(m)}_{\dd}$ is a free $\HO_T$-module for the induced $\HO_T$-module structure.  For $\dd^{(1)},\ldots,\dd^{(n)}\in\BoN^{Q_0}$ a $n$-tuple of dimension vectors, the cohomologically graded $\HO_T$-module $V^{(1)}_{\dd^{(1)}}\otimes_{\HO_T}\cdots\otimes_{\HO_T} V^{(n)}_{\dd^{(n)}}$ carries an action of 
\begin{equation}
\label{coh_ring_not}
\HO_{\Gl^{\wt}_{\dd^{(1)}}}\otimes_{\HO_T}\cdots\otimes_{\HO_T} \HO_{\Gl^{\wt}_{\dd^{(n)}}}=\HO_T[x_{(m),i,k}\lvert \;1\leq m\leq n,\; i\in Q_0,\; 1\leq k\leq \dd^{(m)}_i]^{\FS_{\dd^{(1)}}\times\cdots\times \FS_{\dd^{(n)}}}.
\end{equation}

For $1\leq m',m''\leq n$ with $m'\neq m''$, we set
\[
z_{(m'm''),\dd',\dd''}\coloneqq z_{\dd',\dd''}\lvert_{\substack{x_{(1),\ldots}\mapsto x_{(m'),\ldots}\\ x_{(2),\ldots}\mapsto x_{(m''),\ldots}}}.
\]
Generalising the construction above, we define $(V^{(1)}\otimes_{\HO_T}\cdots\otimes_{\HO_T} V^{(n)})_{(m'm'')}$ by replacing the summand $(V^{(1)}_{\dd^{(1)}}\otimes_{\HO_T}\cdots\otimes_{\HO_T} V^{(n)}_{\dd^{(n)}})\subset V^{(1)}\otimes_{\HO_T}\cdots\otimes_{\HO_T} V^{(n)}$ with the localisation
\[
(V^{(1)}_{\dd^{(1)}}\otimes_{\HO_T}\cdots\otimes_{\HO_T} V^{(n)}_{\dd^{(n)}})\otimes_B B[(z_{(m'm''),\dd^{(m')},\dd^{(m'')}})^{-1}]
\]
where $B=\HO_{\Gl^{\wt}_{\dd^{(1)}}}\otimes_{\HO_T}\cdots\otimes_{\HO_T} \HO_{\Gl^{\wt}_{\dd^{(n)}}}$.
\begin{remark}
\label{spec_def_rem}
Let $Q$ be graded-symmetric in the sense introduced in \S \ref{wtngs_ta_sec}.    Then $\EU^T_{\dd',\dd''}(Q^{\op}_1)=\EU^T_{\dd',\dd''}(Q_1)$.  Let $\alpha\in V$ and $\alpha'\in V'$ be homogeneous of $\BoN^{Q_0}$-degrees $\dd,\dd'$ respectively, and of cohomological degrees with parity equal to $\chi_Q(\dd,\dd)$ and $\chi_Q(\dd',\dd')$ respectively.
Then
\[
\tilde{\sw}(\alpha \otimes\alpha')=(-1)^{\chi_Q(\dd,\dd')}(\alpha'\otimes \alpha)=\sw_{\tau}(\alpha\otimes\alpha')
\]
cf. Example \ref{gvs_examp}.
This holds, in particular, if $Q$ is symmetric in the usual sense and $T$ is trivial.  We can think of non-symmetric weightings on a symmetric quiver $Q$ as providing spectral deformations of the $\tau$-twisted symmetric monoidal structure on $\SC$, considered as a braided monoidal category. 
\end{remark}

\subsection{The shuffle algebra as a CoHA}
\label{shuffle_sec}
Let $Q$ be a symmetric quiver.  Considering the CoHA associated to the \textit{zero} potential on $\BoC Q$, we recover the shuffle algebra \cite[\S 1]{KS2}, as we now recall.

For a function $f$ on a stack $\FM$, there is a canonical natural transformation $\phip{f}\rightarrow z_*z^*$, where $z\colon f^{-1}(0)\rightarrow \FM$ is the inclusion (see \cite[Chapter 8]{KSsheaves}).  In the special case $f=0$ this induces the natural isomorphism $\phip{f}\BoQ_{\FM}\rightarrow \BoQ_{\FM}$.  Setting $\FM=\FM_{\dd}^T(Q)$ we obtain the isomorphisms 
\[
\HO\!\CoHA_{Q,0,\dd}^T\cong \HO^*(\FM^T_{\dd}(Q),\BoQ^{\vir})\cong \HO_T[x_{i,m}\;\lvert\;i\in Q_0,1\leq m\leq \dd_i]^{\FS_{\dd}}[e]
\]
for the final isomorphism, the cohomological degree of each $x_{i,m}$ is two, and the overall cohomological shift is given by setting $e=-\chi_{Q}(\dd,\dd)$.
\begin{definition}
We define $\HO\!\CoHA_{Q}^T$ to be the CoHA $\HO\!\CoHA_{Q,W}^T$ after setting $W=0$.
\end{definition}
We consider $\HO\!\CoHA_{Q}^T$ as an algebra object of $\Vect_{\BoN^{Q_0}\oplus \BoZ}$, as in the previous subsection.  Note that $\HO\!\CoHA_{Q}^T$ satisfies the cohomological parity condition in Remark \ref{spec_def_rem}.

The algebra $\HO\!\CoHA_{Q}^T$ admits an elegant description as a shuffle algebra, we refer to \cite[\S 1]{KS2} for proofs\footnote{The addition of the torus $T$ associated to the weighting $\wt\colon Q_1\rightarrow \BoZ^r$ generalises the setup in \cite{KS2} very slightly, but does not necessitate new proofs.}.  Given dimension vectors $\dd',\dd''\in\BoN^{Q_0}$ we let $\shuff_{\dd',\dd''}\subset \FS_{\dd'+\dd''}$ be the set of permutations $(\sigma_i)_{i\in Q_0}$ such that for $i\in Q_0$ and $1\leq m<n\leq \dd'_i$ or $\dd'_i+1\leq m<n\leq \dd'_i+\dd''_i$ we have $\sigma_i(m)<\sigma_i(n)$.  Such permutations are commonly referred to as shuffles.  Set $\dd=\dd'+\dd''$.  We define an isomorphism of $\HO_T$-algebras 
\begin{align}
\label{jmath_def}
\jmath\colon&\HO_T[x_{(1),i,m}\lvert i\in Q_0, 1\leq m\leq\dd'_i]\otimes_{\HO_T} \HO_T[x_{(2),i,m}\lvert i\in Q_0, 1\leq m\leq\dd''_i]\rightarrow \\&\HO_T[x_{i,m}\lvert i\in Q_0, 1\leq m\leq\dd_i]\nonumber
\end{align}
sending $x_{(1),i,m}\otimes 1\mapsto x_{i,m}$ and $1\otimes x_{(2),i,m}\mapsto x_{i,\dd'_i+m}$.  Then the multiplication is given by 
\begin{align}
\vmult_{\dd',\dd''}\colon &\HO\!\CoHA_{Q,\dd'}^T\otimes_{\HO_T}\HO\!\CoHA_{Q,\dd''}^T\rightarrow \HO\!\CoHA_{Q,\dd}^T\nonumber \\ \label{shuff_prod}
&f\otimes g\mapsto \sum_{\sigma\in \shuff_{\dd',\dd''}}\sigma\left(\jmath(f\otimes g)\EU^T_{\dd',\dd''}(Q_1)\EU^T_{\dd',\dd''}(Q_0)^{-1}\right).
\end{align}
As explained in \cite[\S 1]{KS2}, despite the rational functions appearing in \eqref{shuff_prod}, the sum is a genuine polynomial.
\subsection{Shuffle algebra coproduct}
\label{Sh_alg_cop}
There is a natural localised coproduct on $\HO\!\CoHA_{Q}^T$, which we now recall. Let $\dd',\dd''\in \BoN^{Q_0}$ and set $\dd=\dd'+\dd''$.  We define the morphism $\imath\colon \HO\!\CoHA_{Q,\dd}^T\rightarrow \HO\!\CoHA_{Q,\dd'}^T\otimes_{\HO_T}\HO\!\CoHA_{Q,\dd''}^T$ by restricting $\jmath^{-1}$ from \eqref{jmath_def} to $\FS_{\dd}$-invariant elements.  Then we define
\begin{align*}
\vDelta_{\dd',\dd''}\colon&\HO\!\CoHA_{Q,\dd}^T\rightarrow \HO\!\CoHA_{Q,\dd'}^T\tilde{\otimes}\HO\!\CoHA_{Q,\dd''}^T\\
&f\mapsto \EU^T_{\dd',\dd''}(Q_0)\EU^T_{\dd',\dd''}(Q_1)^{-1}\cdot\imath(f),
\end{align*}
and define $\vDelta$ by taking the direct sum over all $\dd',\dd''\in\BoN^{Q_0}$.  Then it is an elementary exercise to show that the following diagram commutes
\[
\begin{tikzcd}
\HO\!\CoHA_Q^T\otimes_{\HO_T}\HO\!\CoHA_Q^T\arrow[r,"\vmult"]\arrow[d,"u \circ(\vDelta\otimes_{\HO_T}\vDelta)"]&\HO\!\CoHA_Q^T\arrow[ddd,"\vDelta"]
\\
(\HO\!\CoHA_Q^T\otimes_{\HO_T}\HO\!\CoHA_Q^T\otimes_{\HO_T}\HO\!\CoHA_Q^T\otimes_{\HO_T}\HO\!\CoHA_Q^T)_{(12),(14),(23),(34)}\arrow[d,"(\id\otimes_{\HO_T}\tilde{\sw}\otimes_{\HO_T}\id)_{(12),(14),(23),(34)}"]
\\
(\HO\!\CoHA_Q^T\otimes_{\HO_T}\HO\!\CoHA_Q^T\otimes_{\HO_T}\HO\!\CoHA_Q^T\otimes_{\HO_T}\HO\!\CoHA_Q^T)_{(13),(14),(23),(24)}\ar[d,"="]
\\
(\HO\!\CoHA_Q^T\otimes_{\HO_T}\HO\!\CoHA_Q^T)\tilde{\otimes}(\HO\!\CoHA_Q^T\otimes_{\HO_T}\HO\!\CoHA_Q^T)\ar[r,"\vmult\tilde{\otimes}\vmult"]&\HO\!\CoHA_Q^T\tilde{\otimes}\HO\!\CoHA_Q^T
\end{tikzcd}
\]
where we denote by $u$ the localisation morphism
\[
u\colon (\HO\!\CoHA_Q^T\otimes_{\HO_T}\HO\!\CoHA_Q^T\otimes_{\HO_T}\HO\!\CoHA_Q^T\otimes_{\HO_T}\HO\!\CoHA_Q^T)_{(12),(14)}\rightarrow (\HO\!\CoHA_Q^T\otimes_{\HO_T}\HO\!\CoHA_Q^T\otimes_{\HO_T}\HO\!\CoHA_Q^T\otimes_{\HO_T}\HO\!\CoHA_Q^T)_{(12),(14),(23),(34)}.
\]
We will often abbreviate
\[
\alpha\ast\beta=\vmult(\alpha\otimes \beta);\quad\quad (\alpha\otimes \alpha')\ast (\beta\otimes\beta')=\vmult\tilde{\otimes}\vmult\circ (\id\otimes\tilde{\sw}\otimes \id)(\alpha\otimes \alpha'\otimes \beta\otimes\beta').
\]
The coproduct has a geometric definition, which we recall from \cite{Da13}.  The paper [loc. cit.] deals with the case of \textit{general} potentials, and so the presentation here can be a great deal simpler.  Let ${}^-\FM_{\dd',\dd''}^T(Q)$ be the $T$-quotient of the stack of short exact sequences $0\leftarrow M'\leftarrow M\leftarrow M''\leftarrow 0$ of $\BoC Q$-modules for which $M',M''$ have dimension vectors $\dd'$ and $\dd''$ respectively.  There is a natural isomorphism ${}^-\FM_{\dd',\dd''}^T(Q)\cong \FM_{\dd'',\dd'}^T(Q)$, swapping the roles of the projections $q_1,q_3$ to $\FM^T(Q)$.  We then consider the diagram
\[
\xymatrix{
\FM^{T}_{\dd'}(Q)\times_{\B T} \FM^{T}_{\dd''}(Q)&& \ar[ll]_-{{}^-(q_1\times q_3)} {}^-\FM^{T}_{\dd',\dd''}(Q)\ar[r]^-{{}^-q_2}& \FM^{T}_{\dd}(Q).
}
\]
Taking cohomology yields the diagram
\[
\xymatrix{
C^{\FS_{\dd'}\times\FS_{\dd''}}\ar[rr]^{{}^-(q_1\times q_3)^*}_{=} && C^{\FS_{\dd'}\times\FS_{\dd''}}& \ar@{_{(}->}[l]_-{{}^-q^*_2} C^{\FS_{\dd}}
}
\]
where $C=\HO_T[x_{i,m}\;\lvert\;i\in Q_0,1\leq m\leq \dd_i]$.  Then $\vDelta$ is obtained by composing $({}^-(q_1\times q_3)^*)^{-1}\circ{}^-q^*_2$ with multiplication by $\EU^T_{\dd',\dd''}(Q_1)^{-1}\EU^T_{\dd',\dd''}(Q_0)$, which we remark is the (virtual) Euler characteristic of ${}^-(q_1\times q_3)$.

We finish with an easy but crucial observation:
\begin{proposition}
\label{mult_prop}
Let $Q$ be a symmetric quiver.  Let $\dd',\dd''\in\BoN^{Q_0}$ be dimension vectors with disjoint support, and set $\dd=\dd'+\dd''$.  The source and the target of the multiplication
\[
\vmult_{\dd',\dd''}\colon \HO\!\CoHA^T_{Q,\dd'}\otimes_{\HO_T} \HO\!\CoHA^T_{Q,\dd''}\rightarrow \HO\!\CoHA^T_{Q,\dd}
\]
carry a $B=\HO_{\Gl_{\dd'}}\otimes \HO_{\Gl^{\wt}_{\dd''}}$-action.  Let $\Fm\subset \HO_{\Gl^{\wt}_{\dd''}}$ be the maximal graded ideal.  For a $B$-module $M$ we denote by $M_I$ the localisation at the prime ideal $I=\HO_{\Gl_{\dd'}}\otimes \Fm$.  We consider the morphism induced by the product
\[
\vmult'_{\dd',\dd''}\colon \left(\HO\!\CoHA^T_{Q,\dd'}\otimes_{\HO_T} \HO\!\CoHA^T_{Q,\dd''}\right)_I\rightarrow (\HO\!\CoHA^T_{Q,\dd})_I
\]
as well as the localised coproduct
\[
\vDelta'_{\dd',\dd''}\colon (\HO\!\CoHA^T_{Q,\dd})_I\rightarrow \left(\HO\!\CoHA^T_{Q,\dd'}\otimes_{\HO_T} \HO\!\CoHA^T_{Q,\dd''}\right)_I.
\]
The composition of these two morphisms (in either order) is given by the identity.  In particular, they are both isomorphisms.
\end{proposition}
\begin{proof}
Both the source and the target of the multiplication morphism can be identified with the algebra of symmetric functions $\HO_T[x_{i,m}\lvert i\in Q_0, 1\leq m\leq \dd_i]^{\FS_{\dd}}$.  Under this isomorphism, the formula \eqref{shuff_prod} is given by
\[
\vmult_{\dd',\dd''}(f\otimes g)=\EU^T_{\dd',\dd''}(Q_1)\cdot (f\otimes g).
\]
The key point is that since $\supp(\dd')\cap \supp(\dd'')=\emptyset$ the set $\shuff_{\dd',\dd''}$ contains only the trivial permutation.  Likewise, the localised comultiplication $\vDelta_{\dd',\dd''}$ is given explicitly by multiplication by\\ $\EU^T_{\dd',\dd''}(Q_1)^{-1}$, completing the proof.
\end{proof}

\subsection{From the critical CoHA to the shuffle algebra}
\label{sec: From critical CoHA to shuffle algebra}
Let $Q$ be a symmetric quiver, with weighting $\wt\colon Q_1\rightarrow \BoZ^{r}$, and $\wt$-invariant potential $W$.  Let $\wt'\colon Q_1\rightarrow \BoZ$ be a second weighting, which factors through the inclusion $\BoN\hookrightarrow \BoZ$, and for which $W$ is homogeneous, with strictly positive weight $\wt'(W)>0$.  Let $Q'\subset Q$ be the full subquiver containing exactly those edges $a$ such that $\wt'(a)=0$.  Then $\Tr(W)_{\dd}^{-1}(0)\subset \BoA_{\dd}(Q)$ is $\Gl_{\dd}^{\wt}$-equivariantly contractible onto $\BoA_{\dd}(Q')$.  It follows that the restriction morphism in cohomology
\[
\HO^*(\FM^T_{\dd}(Q),\BoQ)\rightarrow \HO^*(\Tr(W)_{\dd}^{-1}(0)/\Gl_{\dd}^{\wt},\BoQ)
\]
is an isomorphism.  

Set $\FM=\FM^T_{\dd}(Q)$.  We abuse notation by also denoting by $\Tr(W)$ the function on $\FM$ induced by $\Tr(W)$.  Let $\imath\colon \Tr(W)^{-1}(0)\hookrightarrow\FM $ be the inclusion.  From the distinguished triangle relating the vanishing cycle functor and the nearby cycle functor, as well as the adjunction morphism for $\imath_*$, we obtain morphisms
\[
\phip{\Tr(W)}\BoQ_{\FM}^{\vir}\rightarrow \imath_*\imath^*\BoQ^{\vir}_{\FM}\leftarrow \BoQ^{\vir}_{\FM}.
\]
Passing to derived global sections we obtain morphisms
\[
\HO\!\CoHA^T_{Q,W,\dd}\xrightarrow{s} \HO^*(\Tr(W)^{-1}(0),\imath^*\BoQ^{\vir}_{\FM})\xleftarrow{r} \HO\!\CoHA^T_{Q,\dd}.
\]
We have seen that $r$ is an isomorphism of $\HO_T$-modules for every $\dd\in \BoN^{Q_0}$.  We define $\xi_{\dd}=r^{-1}s\colon \HO\!\CoHA^T_{Q,W,\dd}\rightarrow \HO\!\CoHA^T_{Q,\dd}$.
\begin{proposition}
\label{map_to_shuff}
The morphism $\xi\colon \HO\!\CoHA^T_{Q,W}\rightarrow \HO\!\CoHA^T_{Q}$ obtained by taking the sum of the $\xi_{\dd}$ defined above across all $\dd\in\BoN^{Q_0}$ is an algebra morphism.
\end{proposition}
\begin{proof}
To remove clutter, we define stacks $\FA,\FB,\FC$ along with morphisms between them via the following diagram
\[
\xymatrix{
\FA\ar[d]^{\coloneqq}&\ar[l]_{q}\FB\ar[d]^{\coloneqq}\ar[r]^{q_2}&\FC\ar[d]^{\coloneqq}\\
\FM^{T}_{\dd'}(Q)\times_{\B T} \FM^{T}_{\dd''}(Q)& \ar[l]_-{
q_1\times q_3} \FM^{T}_{\dd',\dd''}(Q)\ar[r]^-{q_2}& \FM^{T}_{\dd}(Q).
}
\]
The stack $\FA$ carries two functions $f_{\dd'}$ and $f_{\dd''}$ obtained by pulling back the function $\Tr(W)$ along the projections to $\FM^{T}_{\dd'}(Q)$ and $\FM^{T}_{\dd''}(Q)$ respectively.  We set $f=f_{\dd'}+f_{\dd''}$.  Using contractibility of all relevant stacks, all of the arrows in the commutative diagram
\[
\xymatrix{
&\HO^*(\FA,\BoQ)\ar[dl]_{r^+}\ar[d]^{r_\cap}\\
\HO^*(f^{-1}(0),\BoQ)\ar[r]^-{r'}&\HO^*(f_{\dd'}^{-1}(0)\cap f_{\dd''}^{-1}(0),\BoQ)
}
\]
obtained by taking restriction morphisms in cohomology are isomorphisms.  Since all vertical arrows in the leftmost diagram
\[
\xymatrix{
\BoQ_{\FA}\ar[d]\ar[r]&q_*\BoQ_{\FB}\ar[d]&\HO^*(\FA,\BoQ)\ar[d]^{r^+}\ar[r]&\HO^*(\FB,\BoQ)\ar[d]\\
\BoQ_{f^{-1}(0)}\ar[r]&q_*\BoQ_{(fq)^{-1}(0)}&\HO^*(f^{-1}(0),\BoQ)\ar[r]&\HO^*((fq)^{-1}(0),\BoQ)\\
\phip{f}\BoQ_{\FA}\ar[u]\ar[r]&\phip{f}q_*\BoQ\ar[u]&\HO^*(\FA,\phip{f}\BoQ_{\FA})\ar[u]_{a}\ar[r]&\HO^*(\FB,\phip{fq}\BoQ_{\FB})\ar[u]
}
\]
are obtained from natural transformations, the diagram commutes.  Since the diagram to its right is obtained by taking derived global sections, it also commutes.  Now (ignoring cohomological shifts) we have $\xi_{\dd'}\otimes_{\HO_T}\xi_{\dd''}=(r_{\cap})^{-1}r'a$, so from the commutativity of the above diagrams we deduce that the diagram
\[
\xymatrix{
\HO^*(\FA,\BoQ)\ar[r]&\HO^*(\FB,\BoQ)\ar[d]_b\\
&\HO^*((fq)^{-1}(0),\BoQ)\\
\HO^*(\FA,\phip{f}\BoQ_{\FA})\ar[uu]^{\xi_{\dd'}\otimes_{\HO_T}\xi_{\dd''}}\ar[r]&\HO^*(\FB,\phip{fq}\BoQ_{\FB})\ar[u]^c
}
\]
commutes.  By the same argument, setting $e=-\reldim(q_2)$, the diagrams 
\[
\xymatrix{
(q_2)_*\BoQ_{\FB}\ar[r]\ar[d]&\BoQ_{\FC}[e]\ar[d]&\HO^*(\FB,\BoQ)\ar[d]_b\ar[r]&\HO^*(\FC,\BoQ)[e]\ar[d]^r\\
(q_2)_*\BoQ_{(\Tr(W)(q_2))^{-1}(0)}\ar[r]&\BoQ_{\Tr(W)^{-1}(0)}[e]&\HO^*((fq)^{-1}(0),\BoQ)\ar[r]&\HO^*(\Tr(W)^{-1}(0),\BoQ)[e]\\
(q_2)_*\phip{\Tr(W)}\BoQ_{\FB}\ar[u]\ar[r]&\phip{\Tr(W)}\BoQ_{\FC}\ar[u]&\HO^*(\FB,\phip{fq}\BoQ_{\FB})\ar[u]^c\ar[r]&\HO^*(\FC,\phip{\Tr(W)}\BoQ_{\FC})[e]\ar[u]_s
}
\]
commute, and the proposition follows from the definition $\xi_{\dd}=r^{-1}s$.
\end{proof}

Now we fix our quiver to be the tripled quiver $\tilde{Q}$ associated to some quiver $Q$, and consider its canonical cubic potential $\tilde{W}$ defined as in \eqref{CCP}.  If we let $\wt'\colon \tilde{Q}_1\rightarrow \BoZ$ be the weighting sending all arrows to $1$, it is clear that the conditions at the start of this subsection are met.  Let $p$ be the morphism \eqref{dim_proj}.  Let $\FZ=p^{-1}(\FM^T_{\dd}(\Pi_Q))$.  Since $\FZ\rightarrow \FM^T_{\dd}(\Pi_Q)$ is an affine fibration, it induces an isomorphism $\HO^{\BoMo}(\FM^T_{\dd}(\Pi_Q),\BoQ)\rightarrow \HO^{\BoMo}(\FZ,\BoQ)$; it does not respect cohomological degrees, which we temporarily ignore.  Direct image along the closed embedding $\FZ\hookrightarrow \FM^T(\tilde{Q})$ induces the morphism $\HO^{\BoMo}(\FZ,\BoQ)\rightarrow \HO^*(\FM^T(\tilde{Q}),\BoQ)$, and composing these morphisms, we define the morphism
\[
\ol{\xi}\colon \HO\!\CoHA^T_{\Pi_Q}\rightarrow \HO\!\CoHA^T_{\tilde{Q}}.
\]
\begin{proposition}
\label{prop: embedding preprojective CoHA in shuffle algebra}
The diagram of $\HO_T$-algebras
\[
\xymatrix{
\HO\!\CoHA^T_{\Pi_Q}\ar[dr]_{\overline{\xi}}\ar[r]^-{\dimred}_-{\cong}&\HO\!\CoHA^T_{\tilde{Q},\tilde{W}}\ar[d]^{\xi}\\
&\HO\!\CoHA^T_{\tilde{Q}}
}
\]
commutes.
\end{proposition}
\begin{proof}
Set $e=-(\dd,\dd)_{\tilde{Q}}$, and let $\iota\colon \Tr(\tilde{W})^{-1}(0)\hookrightarrow \FM^T(\tilde{Q})$ be the inclusion.  Applying the natural transformations $\phip{\Tr(\tilde{W})}\rightarrow \iota_*\iota^*$ and $\id\rightarrow \iota_*\iota^*$ to the morphism $\BoD\BoQ_{\FZ}[e]\rightarrow \BoD\BoQ^{\vir}_{\FM^{T}_{\dd}(\tilde{Q})}\cong \BoQ^{\vir}_{\FM^{T}_{\dd}(\tilde{Q})}$ and passing to global sections we get the top and bottom squares, respectively, of the diagram
\[
\xymatrix{
\HO\!\CoHA^T_{\Pi_Q}\ar[d]^{\cong}\ar[r]^{\dimred}&\HO\!\CoHA^T_{\tilde{Q},\tilde{W}}\ar[d]_s\\
\HO\!\CoHA^T_{\Pi_Q}\ar[r]&\HO^*(\Tr(\tilde{W})^{-1}(0),\iota^*\BoQ^{\vir})\\
\HO\!\CoHA^T_{\Pi_Q}\ar[r]^{\ol{\xi}}\ar[u]_{\cong}&\HO\!\CoHA^T_{\tilde{Q}}.\ar[u]^r
}
\]
It follows that these squares commute, and the proposition follows from the equality $\xi=r^{-1}s$.
\end{proof}
We recall the following result, see \cite[Thm.A(c)]{ScVa20}, though see also \cite[Thm.10.2]{preproj} for the precise version we need here:
\begin{proposition}
\label{shuff_embed}
Let $\wt\colon \tilde{Q}_1\rightarrow \BoZ^r$ be a $\tilde{W}$-preserving weighting of the edges of $\tilde{Q}$ satisfying Assumption \ref{weighting_assumption}.  Then the morphism $\overline{\xi}$ is an embedding of algebras.
\end{proposition}

\section{The coproduct on $\HO\!\CoHA^T_{\Pi_Q}$}
\subsection{The induced coproduct}
This section is devoted to the coproduct on $\HO\!\CoHA^T_{\Pi_Q}$, which is an essential component of the comparison with algebra structures built out of stable envelopes.  As we have an isomorphism of algebras $\dimred\colon \HCoha^T_{\Pi_Q}\cong \HCoha^T_{\tilde{Q},\tilde{W}}$, where on the right hand side we have an example of a critical cohomological Hall algebra for the quiver $\tilde{Q}$ with potential $\tilde{W}$, we can appeal to the general construction of localised coproducts for such cohomological Hall algebras in \cite{Da13}.  For a slightly more self-contained exposition, in this subsection we obtain this coproduct instead by pulling back along the morphism $\ol{\xi}$ from \S \ref{sec: From critical CoHA to shuffle algebra}.
\begin{proposition}
\label{res_cop_def}
Let $\vDelta_{\dd',\dd''}\colon \HO\!\CoHA^T_{\tilde{Q},\dd}\rightarrow  \HO\!\CoHA^T_{\tilde{Q},\dd'}\tilde{\otimes}\HO\!\CoHA^T_{\tilde{Q},\dd''}$ denote the degree $(\dd',\dd'')$ summand of the localised coproduct on $\HO\!\CoHA^T_{\tilde{Q}}$.  Then $\vDelta_{\dd',\dd''}\lvert_{\Image(\overline{\xi}_{\dd})}$ factors through the inclusion
\[
\overline{\xi}_{\dd'}\tilde{\otimes}\overline{\xi}_{\dd''}\colon \HO\!\CoHA^T_{\Pi_Q,\dd'}\tilde{\otimes}\HO\!\CoHA^T_{\Pi_Q,\dd''}\hookrightarrow \HO\!\CoHA^T_{\tilde{Q},\dd'}\tilde{\otimes}\HO\!\CoHA^T_{\tilde{Q},\dd''}.
\]
\end{proposition}
\begin{proof}
The proof is the same as for Proposition \ref{map_to_shuff}, this time using that the diagrams induced by restriction
\[
\xymatrix{
(q_2)_*\BoQ_{\FB}\ar[d]&\ar[l]\BoQ_{\FC}\ar[d]&\HO^*(\FB,\BoQ)\ar[d]_b&\ar[l]\HO^*(\FC,\BoQ)\ar[d]^r\\
(q_2)_*\BoQ_{(\Tr(\tilde{W})\circ q_2)^{-1}(0)}&\ar[l]\BoQ_{\Tr(\tilde{W})^{-1}(0)}&\HO^*((fq)^{-1}(0),\BoQ)&\ar[l]\HO^*(\Tr(\tilde{W})^{-1}(0),\BoQ)\\
(q_2)_*\phip{\Tr(\tilde{W})}\BoQ_{\FB}\ar[u]&\ar[l]\phip{\Tr(\tilde{W})}\BoQ_{\FC}\ar[u]&\HO^*(\FB,\phip{fq}\BoQ_{\FB})\ar[u]^c&\ar[l]\HO^*(\FC,\phip{\Tr(\tilde{W})}\BoQ_{\FC})\ar[u]_s
}
\]
commute.
\end{proof}
Assume that the weighting $\wt\colon\tilde{Q}_1\rightarrow \BoZ^r$ satisfies Assumption \ref{weighting_assumption}, so that $\overline{\xi}$ is an embedding by Proposition \ref{shuff_embed}.  We abuse notation by denoting also by $\vDelta_{\dd',\dd''}\colon \HO\!\CoHA_{\Pi_Q,\dd}^T\rightarrow \HO\!\CoHA_{\Pi_Q,\dd'}^T\tilde{\otimes}\HO\!\CoHA_{\Pi_Q,\dd''}^T$ the coproduct morphism induced by restricting the localised coproduct on $\HO\!\CoHA_{\tilde{Q}}^T$ along $\ol{\xi}$, which makes sense, due to Proposition \ref{res_cop_def}.

\begin{remark}
It follows formally that this operation induces a localised bialgebra structure on $\HO\!\CoHA_{\Pi_Q}^T$, in the sense of \cite{Da13}.  
\end{remark}
We record here an easy, but essential corollary of Proposition \ref{mult_prop}.  
\begin{corollary}
\label{mult_cor}
Let $Q$ be a quiver, and let $\wt\colon\tilde{Q}_1\rightarrow \BoZ^r$ be a weighting satisfying Assumption \ref{weighting_assumption}.  Let $\dd',\dd''\in\BoN^{Q_0}$ be dimension vectors with disjoint support, and set $\dd=\dd'+\dd''$.  We define localisation at the ideal $I$ as in Proposition \ref{mult_prop}.  We consider the morphism induced by the product
\[
\vmult'_{\dd',\dd''}\colon \left(\HO\!\CoHA^T_{\Pi_Q,\dd'}\otimes_{\HO_T} \HO\!\CoHA^T_{\Pi_Q,\dd''}\right)_I\rightarrow (\HO\!\CoHA^T_{\Pi_Q,\dd})_I
\]
as well as the localised coproduct
\[
\vDelta'_{\dd',\dd''}\colon (\HO\!\CoHA^T_{\Pi_Q,\dd})_I\rightarrow \left(\HO\!\CoHA^T_{\Pi_Q,\dd'}\otimes_{\HO_T} \HO\!\CoHA^T_{\Pi_Q,\dd''}\right)_I.
\]
These two morphisms are inverses of each other.
\end{corollary}
\begin{proof}
Both the multiplication and the localised comultiplication are obtained from restriction along $\overline{\xi}$, and so this follows immediately from Proposition \ref{mult_prop}.\end{proof}
\begin{remark}
Under the condition $\supp(\dd')\cap \supp(\dd'')=\emptyset$ required in Corollary \ref{mult_cor} there is an equality $\EU^T_{\dd',\dd''}(\tilde{Q}_1)=\EU^T_{\dd',\dd''}(\overline{Q}_1)$.
\end{remark}

\subsection{The raw coproduct}
\label{raw_cop_sec}
In this subsection we recall the definition of the localised coproduct on the cohomological Hall algebra $\HO\!\CoHA_{Q,W}^T$ associated to an arbitrary quiver with potential from \cite{Da13}, and show that the above coproduct on $\HO\!\CoHA^T_{\Pi_Q}$ is a special case, pulled back along the isomorphism $\dimred$.  Since this result is not required for the rest of the paper, we will be brief.  

Let $N=\BoZ^r$ be a lattice, and set $T=\Hom_{\Grp}(N,\BoC^*)$.  Let $Q$ be a $N$-graded quiver, and let $W$ be a $T$-invariant potential on it.  For simplicity, we assume that $Q$ is symmetric, although not necessarily graded-symmetric.  We consider the adjunction morphism
\[
\breve{\alpha}\colon \JH_{\dd,*}\left(\phip{\Tr(W)}\BoQ^{\vir}_{\FM^{T}_{\dd}(Q)}\rightarrow q_{2,*}\phip{\Tr(W)}\BoQ^{\vir}_{\FM^{T}_{\dd',\dd''}(Q)}[-(\dd',\dd'')_Q]\right).
\]
Fix $\dd',\dd''\in\BoN^{Q_0}$.  We will use the morphism defined in terms of tautological classes
\begin{align*}
E\colon &\HO\!\CoHA^{T}_{Q,W,\dd'}\otimes_{\HO_T}\HO\!\CoHA^{T}_{Q,W,\dd''}\rightarrow \HO\!\CoHA^{T}_{Q,W,\dd'}\tilde{\otimes}\HO\!\CoHA^{T}_{Q,W,\dd''}\\
&f\mapsto  \EU^T_{\dd',\dd''}(Q_1)^{-1}\EU^T_{\dd',\dd''}(Q_0)\cdot f.
\end{align*} 
Let $\dd=\dd'+\dd''$.  Then we define the coproduct morphism
\begin{align*}
\vDelta_{\dd',\dd''}\colon &\HO\!\CoHA^{T}_{Q,W,\dd}\rightarrow \HO\!\CoHA^{T}_{Q,W,\dd'}\tilde{\otimes}\HO\!\CoHA^{T}_{Q,W,\dd''}\\
&f\mapsto (E\circ \TS\circ \HO(\beta)^{-1}\circ\HO(\breve{\alpha}))(f)
\end{align*}
with $\beta$ as in \S \ref{gen_CoHA_sec}.  We define $\vDelta\colon \HO\!\CoHA^{T}_{Q,W}\rightarrow \HO\!\CoHA^{T}_{Q,W}\tilde{\otimes}\HO\!\CoHA^{T}_{Q,W}$ by taking the sum of the $\vDelta_{\dd',\dd''}$ over all pairs of dimension vectors $\dd',\dd''\in\BoN^{Q_0}$.

Since the target of $\vDelta$ is localised, the sense in which $\vDelta$ is compatible with the product takes some explaining; we use the localised tensor category structure introduced in \S \ref{localisation_notation}.  In the notation of that section, compatibility between the product and localised coproduct is expressed by the following proposition.
\begin{proposition}\cite[Thm.5.13]{Da13}
The following diagram commutes (with notation as in \S \ref{Sh_alg_cop})
\[
\xymatrix{
\HO\!\CoHA^{T}_{Q,W}\otimes_{\HO_T} \HO\!\CoHA^{T}_{Q,W}\ar[d]^{u \circ(\vDelta\otimes_{\HO_T}\vDelta)}\ar[r]^{\vmult}&\HO\!\CoHA^{T}_{Q,W}\ar[dd]^{\vDelta}
\\
(\HO\!\CoHA^{T}_{Q,W}\otimes_{\HO_T}\HO\!\CoHA^{T}_{Q,W}\otimes_{\HO_T} \HO\!\CoHA^{T}_{Q,W}\otimes_{\HO_T}\HO\!\CoHA^{T}_{Q,W})_{(12),(14),(23),(34)}\ar[d]^{\id\otimes_{\HO_T}\tilde{\sw}\otimes_{\HO_T}\id}&
\\
(\HO\!\CoHA^{T}_{Q,W}\otimes_{\HO_T}\HO\!\CoHA^{T}_{Q,W}\otimes_{\HO_T} \HO\!\CoHA^{T}_{Q,W}\otimes_{\HO_T}\HO\!\CoHA^{T}_{Q,W})_{(13),(14),(23),(24)}\ar[r]^-{\vmult}&\HO\!\CoHA^{T}_{Q,W}\tilde{\otimes}\HO\!\CoHA^{T}_{Q,W}.
}
\]
\end{proposition}
\begin{proposition}
For $T$ induced by an arbitrary $\tilde{W}$-preserving weighting $\wt\colon \tilde{Q}_1\rightarrow \BoZ^r$, the morphism $\overline{\xi}\colon \HO\!\CoHA_{\Pi_Q}^T\rightarrow \HO\!\CoHA_{\tilde{Q}}^T$ is a morphism of localised bialgebras.  In particular, in case $T$ satisfies Assumption \ref{weighting_assumption}, the localised bialgebra structure on $\HO\!\CoHA_{\Pi_Q}^T$ induced by pulling back along the embedding $\overline{\xi}$ is the same as the one defined as in \cite{Da13}.
\end{proposition}
\begin{proof}
The analogous result for the morphism $\xi\colon \HO\!\CoHA_{\tilde{Q},\tilde{W}}^T\rightarrow \HO\!\CoHA_{\tilde{Q}}^T$ follows from the same arguments as in Proposition \ref{map_to_shuff}.  Then the result follows from Proposition \ref{prop: embedding preprojective CoHA in shuffle algebra}, since we obtain all algebraic structures on $\HO\!\CoHA_{\Pi_Q}^T$ by pulling back along the dimension reduction isomorphism $\dimred$.
\end{proof}

\subsection{Coproduct on the perverse associated graded algebra}
\label{PAGC}
In this subsection, we describe how the localised coproduct on $\HO\!\CoHA_{\Pi_Q}^T$ interacts with the perverse filtration introduced in \S \ref{perverse_CoHA_sec}.  
\begin{proposition}
\label{coarse_perv}
Let $\dd^{(1)},\ldots,\dd^{(r)}\in \BoN^{Q_0}$, and let $\CV$ be a vector bundle on $\FM^T_{\dd^{(1)}}(\overline{Q})\times_{\B T}\cdots\times_{\B T} \FM^T_{\dd^{(r)}}(\overline{Q})$.  Set $B=\HCoha_{\Pi_Q,\dd^{(1)}}^T\otimes_{\HO_T}\cdots \otimes_{\HO_T}\HCoha_{\Pi_Q,\dd^{(r)}}^T$.  Let $\alpha\in \End(B)$ be the action of multiplication by $\Eu(\CV)$.  Then $\alpha(\FP^iB)\subset \FP^{i+2\rk(\CV)}B$.
\end{proposition}
\begin{proof}
This is similar to proof of \cite[Lem.5.8]{QEAs}, to which we refer for extra explanation.  Let $\tilde{\CV}=p^*\CV$, where $p$ is as below: 
\begin{align*}
p\colon &\FM^T_{\dd^{(1)}}(\tilde{Q})\times_{\B T}\cdots\times_{\B T} \FM^T_{\dd^{(r)}}(\tilde{Q})\rightarrow \FM^T_{\dd^{(1)}}(\overline{Q})\times_{\B T}\cdots\times_{\B T} \FM^T_{\dd^{(r)}}(\overline{Q})\\
\JH\colon &\FM^T_{\dd^{(1)}}(\overline{Q})\times_{\B T}\cdots\times_{\B T} \FM^T_{\dd^{(r)}}(\overline{Q})\rightarrow \CM^T_{\dd^{(1)}}(\overline{Q})\times_{\B T}\cdots\times_{\B T} \CM^T_{\dd^{(r)}}(\overline{Q}).
\end{align*}
Denote by $\tilde{\FM}$ the domain of $p$.  Set $d=-\sum_{i=1}^r\chi_{\tilde{Q}}(\dd^{(i)},\dd^{(i)})$.  Then $\alpha$ is induced by applying $\JH_*p_*\phip{\Tr(\tilde{W})}$ to the morphism of complexes $l\colon \BoQ_{\FM}[d]\rightarrow \BoQ_{\FM}[d+2\rk(\CV)]$ determined by the Euler class of $\tilde{\CV}$.  In particular, $\alpha$ is induced by applying the global sections functor to a morphism
\[
\CoHA^T_{\Pi_Q,\dd^{(1)}}\boxdot\cdots\boxdot\CoHA^T_{\Pi_Q,\dd^{(r)}}\xrightarrow{l'}\CoHA^T_{\Pi_Q,\dd^{(1)}}\boxdot\cdots\boxdot\CoHA^T_{\Pi_Q,\dd^{(r)}}[2\rk(\CV)]
\]
with $l'=\JH_*\pi_*\phip{\Tr(\tilde{W})}l$.  The result follows from the commutativity of the following diagram:
\[
\xymatrix{
\CoHA^T_{\Pi_Q,\dd^{(1)}}\boxdot\cdots\boxdot\CoHA^T_{\Pi_Q,\dd^{(r)}}\ar[r]^-{l'}&\CoHA^T_{\Pi_Q,\dd^{(1)}}\boxdot\cdots\boxdot\CoHA^T_{\Pi_Q,\dd^{(r)}}[2\rk(\CV)]
\\
\ar[u]{}^{\fp'}\!\vtau^{\leq i}\left(\CoHA^T_{\Pi_Q,\dd^{(1)}}\boxdot\cdots\boxdot\CoHA^T_{\Pi_Q,\dd^{(r)}}\right)\ar[r]^-{{}^{\fp'}\!\vtau^{\leq i}l'}&\left({}^{\fp'}\!\vtau^{\leq i+2\rk(\CV)}\left(\CoHA^T_{\Pi_Q,\dd^{(1)}}\boxdot\cdots\boxdot\CoHA^T_{\Pi_Q,\dd^{(r)}}\right)\right)[2\rk(\CV)]\ar[u]
}
\]
\end{proof}
The next lemma follows by the same proof as \cite[Lem.6.3]{QEAs}, with only minor modifications to take account of the fact that here we consider the less perverse filtration (see Remark \ref{less_perv_rem}).
\begin{lemma}
\label{AB_lemma}
Let $\CV$ be a $\Gl_{\dd^{(1)}}\times\cdots \times \Gl_{\dd^{(r)}}\times T$-equivariant vector bundle on $\BoA_{\dd^{(1)}}(\ol{Q})\times \cdots\times\BoA_{\dd^{(r)}}(\ol{Q})$ satisfying the property that there is a coweight $\BoC^*\hookrightarrow \Gl_{\dd^{(1)}}\times\cdots \times \Gl_{\dd^{(r)}}\times T$ for which the induced $\BoC^*$-action on the fibre $\CV_0$ over the trivial quiver representation $0\in \BoA_{\dd^{(1)}}(\ol{Q})\times \cdots\times\BoA_{\dd^{(r)}}(\ol{Q})$ has no fixed subspace.  Set $B^i=\Gr_{\FP}^{i}(\HCoha_{\Pi_Q,\dd^{(1)}}^T\otimes_{\HO_T}\cdots \otimes_{\HO_T}\HCoha_{\Pi_Q,\dd^{(r)}}^T)$.  Then the morphism provided by Proposition \ref{coarse_perv}
\[
\Eu(\CV)\cdot \colon B^i\rightarrow B^{i+2\rk(\CV)}
\]
is injective.
\end{lemma}
For $\dd',\dd''\in\BoN^{Q_0}$ consider the degree $(\dd',\dd'')$ piece of the target of the localised coproduct: $\HO\!\CoHA_{\Pi_Q,\dd'}^T\tilde{\otimes}\HO\!\CoHA^T_{\Pi_Q,\dd''}$.  We endow this with a perverse filtration by setting
\[
\FP^n\left(\HO\!\CoHA^T_{\Pi_Q,\dd'}\tilde{\otimes}\HO\!\CoHA^T_{\Pi_Q,\dd''}\right)\coloneqq \sum_{l\geq 0}(\EU^T_{\dd',\dd''}(\tilde{Q}^{\op}_1)\EU^T_{\dd',\dd''}(\tilde{Q}_1))^{-l}\cdot \FP^{n+ 2ld}\left(\HO\!\CoHA^T_{\Pi_Q,\dd'}\otimes_{\HO_T}\HO\!\CoHA^T_{\Pi_Q,\dd''}\right)
\]
where $d=\deg(\EU^T_{\dd'',\dd'}(\tilde{Q}_1))$, and we embed 
\[
\left(\HO\!\CoHA_{\Pi_Q,\dd'}^T\otimes_{\HO_T}\HO\!\CoHA^T_{\Pi_Q,\dd''}\right)\hookrightarrow \left(\HO\!\CoHA^T_{\Pi_Q,\dd'}\tilde{\otimes}\HO\!\CoHA^T_{\Pi_Q,\dd''}\right)
\]
via Lemma \ref{AB_lemma}.  The following proposition follows from the construction of the coproduct:
\begin{proposition}
The (localised) coproduct respects the perverse filtrations.
\end{proposition}
\begin{proof}
The morphism $\HO(\beta)^{-1}\circ\HO(\breve{\alpha})$ is induced by applying the derived global sections functor to a morphism $\CoHA^T_{\Pi_Q,\dd}\rightarrow \CoHA^T_{\Pi_Q,\dd'}\boxdot\CoHA^T_{\Pi_Q,\dd''}[2\chi_{\tilde{Q}}(\dd',\dd'')]$, so that
\[
(\HO(\beta)^{-1}\circ\HO(\breve{\alpha}))(\FP^i\!\HO\!\CoHA^T_{\Pi_Q,\dd})\subset \left(\FP^{i-2\chi_{\tilde{Q}}(\dd',\dd'')}\!\left( \HO\!\CoHA^T_{\Pi_Q,\dd'}\otimes_{\HO_T}\HO\!\CoHA^T_{\Pi_Q,\dd''}\right)\right)[2\chi_{\tilde{Q}}(\dd',\dd'')].
\]
Then by Lemma \ref{AB_lemma} and the definition of the perverse filtration on the target, the multiplication map
\[
e= \EU^T_{\dd',\dd''}(\tilde{Q}_0)\EU^T_{\dd',\dd''}(\tilde{Q}_1)^{-1}\cdot \colon \left(\HO\!\CoHA_{\Pi_Q,\dd'}^T\otimes_{\HO_T}\HO\!\CoHA^T_{\Pi_Q,\dd''}\right)\rightarrow \left(\HO\!\CoHA_{\Pi_Q,\dd'}^T\tilde{\otimes}\HO\!\CoHA^T_{\Pi_Q,\dd''}\right)
\]
satisfies
\[
e\left(\left( \FP^{i-2\chi_{\tilde{Q}}(\dd',\dd'')}\!\left( \HO\!\CoHA^T_{\Pi_Q,\dd'}\otimes_{\HO_T}\HO\!\CoHA^T_{\Pi_Q,\dd''}\right)\right)[2\chi_{\tilde{Q}}(\dd',\dd'')]\right)\subset\FP^i(\HO\!\CoHA_{\Pi_Q,\dd'}\tilde{\otimes}\HO\!\CoHA_{\Pi_Q,\dd''}).
\]
\end{proof}
We denote by $\vDelta^{\FP}\colon \Gr_{\FP}\HO\!\CoHA_{\Pi_Q}^T\rightarrow \Gr_{\FP}\left(\HO\!\CoHA^T_{\Pi_Q}\tilde{\otimes}\HO\!\CoHA^T_{\Pi_Q}\right)$ the morphism induced by $\vDelta$.

For $X$ a stack, we define an endofunctor for the derived category of constructible sheaves:
\begin{align*}
{}^{\fp}\! \CH^{\bullet}\colon &\Dub_c(X)\rightarrow \Dub_c(X)\\
&\CF\mapsto \bigoplus_{i\in \BoZ}({}^{\fp}\!\CH^{i}\CF)[-i].
\end{align*}
Since the decomposition theorem holds for all the complexes we consider in this paper, the effect of this endofunctor on objects is rather mild; it takes them to (non-canonically) isomorphic objects.  The effect on \textit{morphisms} is more dramatic; it will often send them to zero.  
\begin{example}
Let $\CF=\ICS_Z$ for $Z$ some closed irreducible subvariety of $X$.  We consider $\CF$ as a perverse sheaf on $X$ via direct image along $Z\hookrightarrow X$.  Let $\CG=\bigoplus_{i\in \BoZ} \CG_i[-i]$ be a complex, with each $\CG_i$ a semisimple perverse sheaf, such that $\ICS_Z$ does not feature amongst any of the summands of $\CG_0$.  Then although 
\begin{align*}
\Hom_{\Dub_c(X)}(\CF,\CG)\cong&\Hom_{\Dub_c(X)}({}^{\fp}\! \CH^{\bullet}\CF,{}^{\fp}\! \CH^{\bullet}\CG)\\
\cong &\bigoplus_{i\geq 0}(\Ext^i(\CF,\CG_{-i}))
\end{align*}
may well be nontrivial, if $\beta$ is any homomorphism in $\Hom_{\Dub_c(X)}(\CF,\CG)$, we have ${}^{\fp}\! \CH^{\bullet}\beta=0$.  
\end{example}
The vanishing described in the previous example is key to the following proposition.
\begin{proposition}
\label{CG_prim_prop}
Let $e\in \Fg^{T}_{\Pi_Q}\subset \Gr_{\FP}^0\!\HO\!\CoHA^T_{\Pi_Q}$ be a Chevalley generator. Then $e$ is primitive for $\vDelta^{\FP}$, meaning that $\vDelta^{\FP}(e)=e\otimes 1+1\otimes e$. Equivalently, if $e\in \Fg^{T}_{\Pi_Q,\dd}$ is a Chevalley generator, and $\dd'+\dd''=\dd$ with $\dd'\neq 0\neq \dd''$, then $\vDelta_{\dd',\dd''}(e)\in \FP^{-2}\left(\HO\!\CoHA^T_{\Pi_Q,\dd'}\tilde{\otimes}\HO\!\CoHA^T_{\Pi_Q,\dd''}\right)$.
\end{proposition}
\begin{proof}
Let $e\in \Fg^{T}_{\Pi_Q,\dd}$, and let $\dd',\dd''\in\BoN^{Q_0}$ satisfy $\dd'+\dd''=\dd$, with $\dd'\neq 0\neq \dd''$.  In particular we assume that $\dd$ is not a real simple root (for otherwise the proposition is trivially true).  By definition, the morphism $\vDelta_{\dd',\dd''}^{\FP}$ factors through 
\begin{align*}
\Gr_{\FP}\HO(\varpi_*\breve{\alpha})\colon \Gr_{\FP}\!\HO\!\CoHA^T_{\Pi_Q,\dd}\rightarrow&\HO^*(\CM^T_{\Pi_Q,\dd},{}^{\Fp}\!\Ho^{\bullet}\!\varpi_*\tilde{\JH}_*\phip{\Tr(\tilde{W})}(q_2)_*\BoQ_{\FM_{\dd',\dd''}^{T}(\tilde{Q})})\\
\cong & \HO^*(\CM^T_{\Pi_Q,\dd},{}^{\Fp}\!\Ho^{\bullet}\!\oplus_*(\varpi_*\times\varpi_*)(\tilde{\JH}\times\tilde{\JH})_*\phip{\Tr(\tilde{W})}\BoQ_{\FM^T_{\dd'}(\tilde{Q})\times_{\B T}\FM^T_{\dd''}(\tilde{Q})})\\
\cong &\HO^*(\CM^T_{\Pi_Q,\dd}, {}^{\Fp}\!\Ho^{\bullet}\!\!\CoHA_{\Pi_Q,\dd'}^T\boxdot {}^{\Fp}\!\Ho^{\bullet}\!\CoHA_{\Pi_Q,\dd''}^T)
\end{align*}
where we have omitted some overall cohomological shifts, and $\varpi$ is as in \eqref{forg_diag}.  By definition of the perverse associated graded morphism, the morphism $\Gr_{\FP}\HO(\varpi_*\breve{\alpha})$ is induced by a morphism of graded perverse sheaves
\begin{equation}
\label{coprod_supp_arg}
{}^{\fp}\!\Ho^{\bullet}\!\CoHA^T_{\Pi_Q,\dd}\rightarrow {}^{\Fp}\!\Ho^{\bullet}\!\CoHA_{\Pi_Q,\dd'}^T\boxdot {}^{\Fp}\!\Ho^{\bullet}\!\CoHA_{\Pi_Q,\dd''}^T[n]
\end{equation}
where $n=2\sum_{a\in \overline{Q}}\dd'_{s(a)}\dd''_{t(a)}$ (this accounts for the omitted cohomological shift).  By assumption, $e$ is obtained, after passing to derived global sections, from the inclusion of a simple summand $\CG(Q)_{\dd}\subset {}^{\fp}\!\Ho^0\!\CoHA^T_{\Pi_Q,\dd}$, cf. \S\ref{GKM_thm_sec}.  On the other hand, by Theorem \ref{KMA_thm}, or by direct inspection of possible supports, this simple summand is not contained in the target of \eqref{coprod_supp_arg}, and we deduce that $\vDelta^{\FP}_{\dd',\dd''}(e)=0$.

This shows that $\vDelta_{\dd',\dd''}(e)\in \FP^{-1}\left(\HO\!\CoHA^T_{\Pi_Q,\dd'}\tilde{\otimes}\HO\!\CoHA^T_{\Pi_Q,\dd''}\right)$.  On the other hand, by \eqref{dec_thm_1} we deduce that $\FP^{2n+1}\left(\HO\!\CoHA^T_{\Pi_Q,\dd'}\otimes_{\HO_T}\HO\!\CoHA^T_{\Pi_Q,\dd''}\right)=\FP^{2n}\left(\HO\!\CoHA^T_{\Pi_Q,\dd'}\otimes_{\HO_T}\HO\!\CoHA^T_{\Pi_Q,\dd''}\right)$ for all $n\in\BoN$.  It follows that $\FP^{2n+1}\left(\HO\!\CoHA^T_{\Pi_Q,\dd'}\tilde{\otimes}\HO\!\CoHA^T_{\Pi_Q,\dd''}\right)=\FP^{2n}\left(\HO\!\CoHA^T_{\Pi_Q,\dd'}\tilde{\otimes}\HO\!\CoHA^T_{\Pi_Q,\dd''}\right)$ for all $n\in\BoZ$, in particular for $n=-1$.
\end{proof}

Although we will not use the following corollary, it is helpful to point out regardless.  It illustrates the sense in which taking a \textit{non}-symmetric weighting of $\tilde{Q}$ constitutes a deformation away from the category of cocommutative Hopf algebras, i.e. it provides the passage to quantum groups.  The corollary follows from Proposition \ref{CG_prim_prop}, the fact that under the graded symmetry assumption $\tilde{\sw}$ does not involve inverting any Euler classes, and the fact that $\Gr_{\FP}^0\!\HO\!\CoHA_{\Pi_Q}^T$ is generated by the Chevalley generators.
\begin{corollary}
Let $\wt\colon \tilde{Q}_1\rightarrow \BoZ^{r}$ be a weighting for which $\tilde{Q}$ is graded-symmetric.  Then the morphism 
\[
\vDelta^{\FP,0}\colon \Gr_{\FP}^0\!\HO\!\CoHA_{\Pi_Q}^T\rightarrow \Gr_{\FP}^0\!\left(\HO\!\CoHA_{\Pi_Q}^T\tilde{\otimes}\HO\!\CoHA_{\Pi_Q}^T\right)
\]
factors through the inclusion of $\Gr_{\FP}^0\!\HO\!\CoHA_{\Pi_Q}^T\otimes_{\HO_T}\Gr_{\FP}^0\!\HO\!\CoHA_{\Pi_Q}^T$.  In particular, the localised coproduct is the coproduct induced by the standard coproduct on $\UEA_{\HO_T}(\Fg^{T}_{\Pi_Q})\cong \Gr_{\FP}^0\!\HO\!\CoHA_{\Pi_Q}^T$.
\end{corollary}

\section{From coproducts to lowering operators}
\label{coprod_to_lowering}
\subsection{Perverse vs order filtration}
Let $Q$ be a quiver, and let $\wt\colon\tilde{Q}_1\rightarrow \BoZ^r$ be a weighting satisfying Assumption \ref{weighting_assumption}, so the natural algebra homomorphism $\overline{\xi}\colon \HO\!\CoHA^T_{\Pi_Q}\rightarrow \HO\!\CoHA^T_{\tilde{Q}}$ is an embedding by Proposition \ref{shuff_embed}.  Let $i\in Q_0$ be a vertex, and let $\dd\in\BoN^{Q_0}$ be a dimension vector.  Via $\overline{\xi}$ we consider elements of $\HO\!\CoHA^T_{\Pi_Q,\delta_i}\tilde{\otimes}\HO\!\CoHA^T_{\Pi_Q,\dd}$ as a subspace of rational functions in the function field of $\BoQ[x_{(1),i,1}]\otimes \HO_T[x_{(2),j,m}\lvert j\in Q_0, 1\leq m\leq\dd_j]$.  For $\alpha\in \HO\!\CoHA^T_{\Pi_Q,\delta_i}\tilde{\otimes}\HO\!\CoHA^T_{\Pi_Q,\dd}$ we denote the $n$th coefficient of the resulting formal expansion in the variable $x_{(1),i,1}^{-1}$ by $\alpha_{(n)}$.  In a little more detail, we write
\[
\overline{\xi}(\alpha)=\sum_{n\in \BoZ}x_{(1),i,1}^{-n} \otimes \ol{\xi}(\alpha)_{(n)}
\]
with $\ol{\xi}(\alpha)_{(n)}\in \HO_T[x_{(2),j,m}\lvert j\in Q_0, 1\leq m\leq\dd_j]$.  Since $\HO\!\CoHA^T_{\Pi_Q,\dd}$ is closed under the action of tautological classes in $\HO_{\Gl_{\dd}^{\wt}}$, each $\ol{\xi}(\alpha)_{(n)}$ is in the image of $\overline{\xi}\colon \HO\!\CoHA^T_{\Pi_Q,\dd}\rightarrow \HO\!\CoHA^T_{\tilde {Q},\dd}$, and so there exist unique $\alpha_{(n)}$ such that we can write
\[
\overline{\xi}(\alpha)=\sum_{n\in \BoZ}x_{(1),i,1}^{-n} \otimes \overline{\xi}(\alpha_{(n)}).
\]
\begin{proposition}
\label{ptrack}
Let $\alpha\in \FP^{-2n}\left(\HO\!\CoHA^T_{\Pi_Q,\delta_i}\tilde{\otimes}\HO\!\CoHA^T_{\Pi_Q,\dd}\right)$.  Then $\alpha_{(m)}=0$ for $m<n$.
\end{proposition}
\begin{proof}
We have $\FP^{2r}\left(\HO\!\CoHA^T_{\Pi_Q,\delta_i}\otimes_{\HO_T}\HO\!\CoHA^T_{\Pi_Q,\dd}\right)=\sum_{s\geq 0}\!\left(x_{(1),i,1}^s\otimes_{\HO_T} \FP^{2r-2s}\HO\!\CoHA^T_{\Pi_Q,\dd}\right)$ and so
\[
\FP^{2r}\left(\HO\!\CoHA^T_{\Pi_Q,\delta_i}\tilde{\otimes}\HO\!\CoHA^T_{\Pi_Q,\dd}\right)=\sum_{\substack{s,l\geq 0\\2r+4el-2s\geq 0}}(x_{(1),i,1}^{s}\otimes\FP^{2r+4el-2s}\HO\!\CoHA^T_{\Pi_Q,\dd})\cdot(\EU^T_{\delta_i,\dd}(\tilde{Q}_1)\EU^T_{\delta_i,\dd}(\tilde{Q}^{\op}_1))^{-l}
\]
where $e$ is the homogeneous degree of $\EU^T_{\delta_i,\dd}(\tilde{Q}_1)$, which is also the homogeneous degree of $\EU^T_{\delta_i,\dd}(\tilde{Q}^{\op}_1)$, and is also the order of vanishing of $x_{(1),i,1}^{-1}$ in either of these expressions, considered as formal power series expansion in $x_{(1),i,1}^{-1}$.  In the above sum, after expanding in powers of $x_{(1),i,1}^{-1}$ the summands that contribute the lowest powers are given by setting $2r+4el-2s= 0$, and their expansions all have order $-r$ in the variable $x_{(1),i,1}^{-1}$.
\end{proof}

Let $\dd\in\BoN^{Q_0}$ and set $\dd'=\dd-\delta_i$, with $i\in Q_0$. Assume that $\dd'\in \BoN^{Q_0}$.  Expanding in powers of $x_{(2),i,1}^{-1}$ we find
\begin{align}
\label{constexp}
&\EU^T_{\dd',\delta_i}(\tilde{Q}_1^{\opp})^{-1}\EU^T_{\dd',\delta_i}(\tilde{Q}_1)=\\ \nonumber
&\prod_{\substack{a\in Q_1\lvert t(a)=i\\ 1\leq n\leq \dd_{t(a)}}}(x_{(2),i,1}-x_{(1),s(a),n}-\ttt(\wt(a)))(x_{(2),i,1}-x_{(1),t(a^*),n}+\ttt(\wt(a^*)))^{-1}\cdot\\ 
\nonumber &\prod_{\substack{a\in Q_1\lvert t(a^*)=i\\1\leq n\leq \dd_{s(a)}}}(x_{(2),i,1}-x_{(1),s(a^*),1}-\ttt(\wt(a^*)))(x_{(2),i,1}-x_{(1),t(a),n}+\ttt(\wt(a)))^{-1}\cdot\\
\nonumber&\prod_{1\leq n\leq \dd_i}(x_{(2),i,1}-x_{(1),i,n}-\ttt(\wt(\omega_i)))(x_{(2),i,1}-x_{(1),i,n}+\ttt(\wt(\omega_i)))^{-1}\\
=& 1-(\delta_i,\dd')_Q \hbar x_{(2),i,1}^{-1}+O\left(x_{(2),i,1}^{-2}\right). \nonumber
\end{align}

\begin{proposition}
\label{res_prop}
Let $\dd,\ff\in\BoN^{Q_0}$ be nonzero dimension vectors for the quiver $Q$.  Let $\wt\colon \tilde{Q}_1\rightarrow \BoZ^{ r}$ be a weighting satisfying Assumption \ref{weighting_assumption}, with associated torus $T$.  Let $\alpha\in \Fg^{T}_{\Pi_{Q_{\ff}},(\dd,1)}\subset \HO\!\CoHA^T_{\Pi_{Q_{\ff}}}$.  Then $(\vDelta_{(\ul{0},1),(\dd,0)}(\alpha))_{(m)}=0$ for $m\leq 0$.
\end{proposition}
\begin{proof}
Firstly, let us assume that $\alpha$ is a Chevalley generator. By Proposition \ref{CG_prim_prop}  $\vDelta_{(\ul{0},1),(\dd,0)}(\alpha)\in \FP^{-2}\left(\HO\!\CoHA^T_{_{\Pi_{Q_{\ff}}},(\ul{0},1)}\tilde{\otimes}\HO\!\CoHA^T_{_{\Pi_{Q_{\ff}}},(\dd,0)}\right)$.  The result then follows from Proposition \ref{ptrack}.  

Next, assume that there exists $\alpha'\in \fg^T_{\Pi_{Q_{\ff}},(\dd,0)}$ such that 
\[
\alpha=[e_{\infty},\alpha']\coloneqq e_{\infty}\ast\alpha'-(-1)^{\chi_{\tilde{Q}_{\ff}}((\ul{0},1),(\dd,0))}\alpha'\ast e_{\infty}
\]
where $e_{\infty}=x_{\infty,1}^0\in \Fg^{T}_{\Pi_{Q_{\ff}},(\ul{0},1)}=\HO_T[x_{\infty,1}]$ is a Chevalley generator.  Note that $\chi_{\tilde{Q}_{\ff}}((\ul{0},1),(\dd,0))=-\ff\cdotsh\dd$.  Then $\overline{\xi}(\vDelta(e_{\infty}))=x_{(1),\infty,1}^0\otimes 1 +1\otimes x_{(2),\infty,1}^0$ and we calculate
\begin{align*}
\overline{\xi}(\vDelta_{(\ul{0},1),(\dd,0)}(\alpha))=&(x_{(1),\infty,1}^0\otimes 1)\ast(1\otimes \overline{\xi}(\alpha'))-(-1)^{\ff\cdot\dd}(1\otimes \overline{\xi}(\alpha')) \ast (x_{(1),\infty,1}^0\otimes 1)\\
=&\left(1-(\EU^T_{(\dd,0),\delta_{\infty}}((\tilde{Q}^{\opp}_{\ff})_1)^{-1}\EU^T_{(\dd,0),\delta_{\infty}}((\tilde{Q}_{\ff})_1))\right)\cdot(x_{(1),\infty,1}^0\otimes \overline{\xi}(\alpha'))\\
=&-(\ff\cdotsh\dd) \hbar x^{-1}_{(1),\infty,1}\cdot(x_{(1),\infty,1}^0\otimes \overline{\xi}(\alpha'))+O(x^{-2}_{(1),\infty,1})
\end{align*}
as required.
A general element of $\Fg^{T}_{\Pi_{Q_{\ff}},(\dd,1)}$ can be written as a linear combination of elements $\gamma=[\ldotsh[\alpha,\beta_l],\beta_{\l-1}],\ldotsh,\beta_1]$ where $\beta_s\in \Fg^{T}_{\Pi_{Q_{\ff}},(\bullet,0)}$ are Chevalley generators and either $\alpha\in\Fg^{T}_{\Pi_{Q_{\ff}},(\dd',1)}$ is a Chevalley generator, with $\dd'\neq \ul 0$, or $\alpha=[e_{\infty},\alpha']$ as above.  Then we find, using the cases already considered,
\begin{align*}
\vDelta_{(\ul{0},1),(\dd,0)}(\gamma)=&\sum_{m\geq 1} [\ldotsh[x_{(1),\infty,1}^{-m}\otimes\vDelta_{(\ul{0},1),(\dd',0)}(\alpha)_{(m)},1\otimes\beta_l],1\otimes\beta_{l-1}],\ldotsh],1\otimes\beta_1]\\
=&x_{(1),\infty,1}^{-1}\otimes [\ldotsh[\vDelta_{(\ul{0},1),(\dd',0)}(\alpha)_{(1)},\beta_l],\beta_{l-1}],\ldotsh],\beta_1 ]+O(x^{-2}_{(1),\infty,1})
\end{align*}
and the result holds for general $\gamma$.
\end{proof}
The following strengthening of Proposition \ref{res_prop} is proved the same way, by induction on length of Lie words.
\begin{corollary}
\label{perv_prim}
Let $\alpha\in\Fg^T_{\Pi_Q}\subset \HO\!\CoHA_{\Pi_Q,\dd}^T$, and let $\dd',\dd''\in\BoN^{Q_0}$ be nonzero dimension vectors satisfying $\dd'+\dd''=\dd$.  Then $\vDelta_{\dd',\dd''}(\alpha)\in\FP^{-2}\left(\HO\!\CoHA_{\Pi_Q}^T\tilde{\otimes}\HO\!\CoHA_{\Pi_Q}^T\right)$.
\end{corollary}
For $\alpha\in \HO\!\CoHA^T_{\Pi_Q,\delta_i}\tilde{\otimes}\HO\!\CoHA^T_{\Pi_Q,\dd}$ we write $\Res_{x_{(1),i,1}}\alpha\coloneqq \alpha_{(1)}$.  We record the following important corollaries of the proof of Proposition \ref{res_prop}.
\begin{corollary}
\label{old_lowering_argument}
Let $e_{\infty}\in \Fg^T_{\Pi_{Q_{\ff}},(\ul{0},1)}$ be a Chevalley generator, and let $\alpha\in \Fg^T_{\Pi_{Q_{\ff}}, (\dd,0)}$, then
\[
\Res_{x_{(1),\infty,1}}\circ\vDelta_{(\ul{0},1),(\dd,0)}([e_{\infty},\alpha])=-(\ff\cdotsh \dd)\hbar\alpha. 
\]
In particular, if $\ff\cdot \dd\neq 0$ then $\hbar\cdot \Fg_{\Pi_Q,\dd}^T\subset \Image\!\left(\Res_{x_{(1),\infty,1}}\circ\vDelta_{(\ul{0},1),(\dd,0)}\colon \Fg^T_{\Pi_{Q_{\ff}},(\dd,1)}\rightarrow \HO\!\CoHA^T_{\Pi_Q,\dd}\right)$.
\end{corollary}
\begin{corollary}
\label{Res_cor_2}
The morphism $\Res_{x_{(1),\infty,1}}\circ \vDelta_{(\ul{0},1),(\bullet,0)}\colon \Fg^T_{\Pi_{Q_{\ff}},(\bullet,1)}\rightarrow \HO\!\CoHA_{\Pi_Q}^T$ is a morphism of right $\Fn^{T,+}_{\Pi_Q}$-modules, where $\Fn^{T,+}_{\Pi_Q}$ acts on the target via the commutator Lie bracket, and on the domain via the embedding $l_*\colon \Fg^T_{\Pi_Q}\hookrightarrow \Fg^T_{\Pi_{Q_{\ff}}}$ and the Lie bracket on $\Fg^T_{\Pi_{Q_{\ff}}}$.
\end{corollary}
\begin{proof}
Let $\gamma$ be as in the proof of Proposition \ref{res_prop} and let $\beta_0\in \fg^T_{\Pi_{Q_{\ff}},(\bullet,0)}$.  Then calculating as at the end of the proof of Proposition \ref{res_prop} we find
\[
\vDelta_{(\ul{0},1),(\dd,0)}([\gamma,\beta_0])=x_{(1),\infty,1}^{-1}\otimes [\ldotsh[\vDelta_{(\ul{0},1),(\dd',0)}(\alpha)_{(1)},\beta_l],\beta_{l-1}],\ldotsh],\beta_0 ]+O(x^{-2}_{(1),\infty,1})
\]
and so $(\Res_{x_{(1),\infty,1}}\circ \vDelta_{(\ul{0},1),(\bullet,0)})([\gamma,\beta_0])=[(\Res_{x_{(1),\infty,1}}\circ \vDelta_{(\ul{0},1),(\bullet,0)})(\gamma),\beta_0]$.
\end{proof}
Also we note the following proposition.
\begin{proposition}
\label{alpha_one_perv}
The morphism 
\[
\Res_{x_{(1),\infty,1}}\circ\vDelta_{(\ul{0},1),(\dd,0)}\colon \Fg_{\Pi_{Q_{\ff}}, (\dd,1)}^T\rightarrow \HO\!\CoHA_{\Pi_Q}^T
\]
factors through the inclusion $\FP^0\!\HO\!\CoHA_{\Pi_Q}^T\hookrightarrow \HO\!\CoHA_{\Pi_Q}^T$.
\end{proposition}
\begin{proof}
Let $\alpha\in \Fg_{\Pi_{Q_{\ff}}(\dd,1)}^T$.  Set $\beta=\vDelta_{(\ul{0},1),(\dd,0)}(\alpha)$.  By Corollary \ref{perv_prim} and Proposition \ref{coarse_perv} it follows that 
\[
\EU^T_{\delta_{\infty},(\dd,0)}(\tilde{Q}_1)\cdot\beta \in \FP^{2e-2}\left(\HO\!\CoHA^T_{\Pi_{Q_{\ff}},(\ul{0},1)}\otimes_{\HO_T}\HO\!\CoHA^T_{\Pi_{Q_{\ff}},(\dd,0)}\right)
\]
where $e$ is the degree of $\EU^T_{\delta_{\infty},(\dd,0)}(\tilde{Q}_1)$.  Examining highest order terms in $x_{(1),\infty,1}$ we find
\[
\EU^T_{\delta_{\infty},(\dd,0)}(\tilde{Q}_1)\cdot\beta=x_{(1),\infty,1}^{e-1}\otimes\beta_{(1)}+\textrm{lower order terms}
\]
implying that $\beta_{(1)}\in \FP^0\HO\!\CoHA^T_{\Pi_{Q_{\ff}},(\dd,0)}=\FP^0\HO\!\CoHA^T_{\Pi_Q,\dd}$.
\end{proof}

\section{Stable envelopes}

\subsection{Axiomatic definition}
\label{sec: axiomatic definition stab}

Let $X$ be a smooth quasi-projective symplectic variety with an action of a torus $\Tt$ rescaling the symplectic form with weight $\hbar$, and let $\At\subset \ker({\hbar})$. As technical assumptions, we also assume the existence of a $\Tt$-equivariant proper map $X\to X_0$ with $X_0$ affine, and that $X$ is $\Tt$-equivariantly formal. All these assumptions are satisfied in the case when $X=\Nak_Q(\ff,\dd)$ is a quiver variety, which is the main focus of this paper.

By definition, a chamber $\mathfrak{C}$ is a connected component of $(\cochar(\At)\otimes_{\ZZ} \RR) \setminus \Delta$, where $\Delta$ is the hyperplane arrangement determined by the $\At$-weights of the normal bundle of $X^{\At}$. 
One says that a weight $\chi\in \Char(\At)$ is attracting (resp. repelling) with respect to $\mathfrak{C}$ if $ \chi\circ \sigma (t)=t^k$ with $k>0$ (resp. $k<0$) for any $\sigma\in \mathfrak{C}$. The trivial character is called the $\At$-fixed character. 
Accordingly, the restriction of the tangent bundle $TX\in \K_{\Tt}(X)$ to the $\At$-fixed locus decomposes as
\[
TX|_{X^{\At}}=N_{X^{\At}}^+ +TX^{\At}+N_{X^{\At}}^-,
\]
where $N_{X^{\At}}^+=TX|_{X^{\At}}^+$ and $N_{X^{\At}}^-=TX|_{X^{\At}}^-$ are the attracting and repelling parts of $N_{X^{\At}}$, the normal bundle of $X^{\At}$ in $X$.

Let $F\subset X^{\At}$ be a connected component. Since $X$ is symplectic and its symplectic form is rescaled by $\hbar$, it follows that $N_{F}^+=\hbar \otimes (N_{F}^-)^\vee$ in $\K_{\Tt}(F)$. Consequently, the restriction of $N_F$ to any fixed point is $(-1)^{\codim_X(F)/2}$ times a square in $\K_{\At}(\pt)=\Char(\At)$. By definition, a \emph{polarisation} $\varepsilon_{X,\At}$ is a choice of square root for each connected component. Following Maulik and Okounkov, we write $\pm \Eu(N^{-}_F)$ for that multiple of the Euler class  $\Eu(N^{-}_F)$ that coincides with the polarisation $\varepsilon_{X,\At}$ when restricted to any point in $F$. This means that $\pm\Eu(N^{-}_F)$ equals $\Eu(N^{-}_F)$ up to a sign. 

Given a subvariety $Y\subseteq X^{\At}$, we define its attracting set $\Att{C}(Y)$ as:
\[
\Att{C}(Y)=\{x\in X \; | \; \lim_{t\to0} \sigma(t)\cdot x\in Y, \text{ for any $\sigma \in \mathfrak{C}$} \}.
\]
Notice that there is a canonical map of sets $\Att{C}(Y)\to Y$ obtained by taking the limit of the $\BoC^*$-action induced by $\sigma$. Moreover, if $Y=F\subset X^{\At}$ is a connected component, then $\Att{C}(F)\to F$ is an affine bundle of rank $\codim_X(F)/2$.  Similarly, we define $\text{Att}^f_{\mathfrak{C}}(F)$ as the minimal closed subset of $X$ containing $F$ and closed under taking $\Att{C}(-)$. 

Since $X$ is quasiprojective, the choice of a chamber determines a partial ordering on the set of fixed components of $X^{\At}$ \cite[\S 3.1.2]{MO19}, defined by taking the transitive closure of the relation
\[
\ol{\Att{C}(F)}\cap F'\neq \emptyset \quad \Rightarrow \quad F'\leq_{\mathfrak{C}} F.
\]
Equivalently, $F\geq_{\mathfrak{C}} F'$ iff $\Att{C}^f(F)\cap F'\neq \emptyset$.

We now recall Maulik and Okounkov's 
definition of stable envelopes.

\begin{theorem}\cite[ Prop.3.5.1]{MO19}
\label{thm: definition stab}
Let $\mathfrak{C}$ be a chamber for the action of $\At$ on $X$ and let $\varepsilon_X$ be a polarisation. There exists a unique $\Tt$-invariant Lagrangian cycle 
\[
\Sh{L}_{\mathfrak{C}}\in \HO^{\BoMo}_{\Tt} (X^{\At}\times X,\BoQ),
\]
proper over $X$, such that

\begin{enumerate}
    \item[(i)]For every fixed component $ F\subseteq  X^{\At}$, the restriction of $\Sh{L}_{\mathfrak{C}}$ to $F\times F$ equals $\pm\Eu(N^-_F)\cap [\Delta]$, according to polarisation, where $[\Delta]\subset \HO^{\BoMo}_{\Tt}(F\times F,\BoQ)$ is the fundamental class of the diagonal;
    \item[(ii)] For any $F\subset X^{\At}$, the cycle $\Sh{L}_{\mathfrak{C}}\lvert_{F\times X}$ is supported on $F\times \Att{C}^f(F)$, i.e. its restriction to $(F\times X)\setminus (F\times \Att{\mathfrak{C}}^f(F))$ is zero;
    \item[(iii)] For every $F'<_{\mathfrak{C}} F$, the restriction of $\Sh{L}_{\mathfrak{C}}$ to $F\times F'$ has $\At$-degree less than $\codim(F')/2$.
\end{enumerate}
\end{theorem}
In (iii), the $\At$-degree of a class in $\HO^*_{\Tt}(X^{\At})$ is defined as the cohomological degree in the equivariant variables associated with the torus $\At$. This can be explicitly computed via a splitting
\[
\HO^*_{\Tt}(X^{\At},\BoQ)\cong \HO^*_{\Tt/\At}(X^{\At},\BoQ)\otimes \HO^*_{\At}(\pt,\BoQ)
\]
by determining the cohomological degree in the second term of the tensor product.
Notice that, although the isomorphism above is not canonical, different isomorphisms produce the same filtration by cohomological degree in the $\At$-equivariant parameters. Expressed in terms comparable to \S \ref{perverse_CoHA_sec}: there is a natural morphism of stacks $p\colon X^{\At}/\Tt\rightarrow X^{\At}/(\Tt/\At)$, and the filtration is defined by setting the $i$th piece to be $\HO^*(X^{\At}/(\Tt/\At),\vtau^{\leq 2i}p_*\BoQ_{X^{\At}/\Tt})\subset \HO^*(X^{\At}/\Tt,\BoQ)$.  Here we consider the unshifted locally constant sheaf, and truncation functors for the standard (non-perverse) t-structure. 

For a given $\Tt$-equivariant smooth variety $X$, we denote by $\BoQ^{\vir}$ the constant $\Tt$-equivariant perverse sheaf on $X$. Explicitly, we have that for each $X'\subset X$ a connected component, $\BoQ^{\vir}\lvert_{X'}=\BoQ[\dim(X')]$.
Following Maulik and Okounkov, we define the stable envelope as the $\HO_{\Tt}$-linear map
\[
\StabC{C}: \HO^*_{\Tt}(X^{\At}, \BoQ^{\vir})\to \HO^*_{\Tt}(X,\BoQ^{\vir}), \qquad \gamma\mapsto (p_2)_* (\Sh{L}_{\mathfrak{C}}\cap (p_1)^*\gamma)
\]
defined by the cycle $\Sh{L}_{\mathfrak{C}}$ via convolution. For details about convolution in Borel--Moore homology, see \cite[\S 2.7]{chriss2009representation}.
Notice that, because of the cohomological shift encoded in $\BoQ^{\vir}$, the map $\StabC{C}$ preserves the cohomological degree.

The next lemma is a key result for the application of stable envelopes in the geometric representation theory of Yangians. 

\begin{lemma}\cite[Lem.3.6.1]{MO19}
\label{lma:triangle lemma}
Let $\mathfrak{C}$ be a chamber for the action of $\At$ on $X$ and let $\mathfrak{C}'\subset \mathfrak{C}$ be a face, of some lower dimension. Let $\At'\subset \At$ be the torus whose Lie algebra is the span of $\mathfrak{C}'$ in the Lie algebra of $\At$ and let $\mathfrak{C}/\mathfrak{C}'$ be the projection of $\mathfrak{C}$ on the Lie algebra of $\At/\At'$. We also assume that the polarisations satisfy $\varepsilon_{X,\At}= \varepsilon_{X,\At'} \varepsilon_{X^{\At'},\At/\At'}$.
Then the following diagram
\begin{equation*}
    \begin{tikzcd}
    \HO^*_{\Tt}(X^{\At},  \BoQ^{\vir})\arrow[dr, swap,  "\Stab_{\mathfrak{C}/\mathfrak{C}'}"]\arrow[rr, "\Stab_{\mathfrak{C}}"] & & \HO^*_{\Tt}(X,  \BoQ^{\vir})\\
    & \HO^*_{\Tt}(X^{\At'},  \BoQ^{\vir})\arrow[ur, swap, "\Stab_{\mathfrak{C}'}"] &
    \end{tikzcd}
\end{equation*}
is commutative.
\end{lemma}

\subsection{Stable envelopes for quiver varieties}
\label{section: Stable envelopes for quiver varieties}
Now let $X=\Nak_Q(\ff,\dd)$ be a quiver variety and set $\Tt= T_{\ff}$, where $T_{\ff}=T\times A_{\ff}$ is the product of the torus $T$ associated with a semi-invariant weighting $\wt: (\overline{Q_{\ff}})_1\to \ZZ^r$ in the sense of \S \ref{subsection:Preprojective CoHA} and the usual rank $|\ff|=\sum_{i\in Q_0} \ff_i$ torus $A_{\ff}$ rescaling the framing arrows $r_{i,m}$ with weight $1$ and the arrows $r_{i,m}^*$ with weight $-1$. The latter can be either thought of as the torus arising from an additional weighting of the framing directions or as the maximal torus of the group $\prod_{i\in Q_0} \Gl_{\ff_i}(\BoC)$ reparametrising the framing  of $\overline{Q_{\ff}}$.
In particular, we have $\hbar=\ttt(\wt(a)+\wt(a^*))$ for any $a\in (Q_{\ff})_1$ and hence $A_{\ff}\subset \ker(\hbar)$.  The torus $A_{\triv}\coloneqq \BoC^*\subset A_{\ff}$ scaling all of the framing edges with weight one may be identified with the factor in the gauge group $\Gl_{(\dd,1)}$ corresponding to the vertex $\infty$, and acts trivially on all Nakajima quiver varieties.  We define $A^0_{\ff}\coloneqq A_{\ff}/A_{\triv}$ and 
\begin{equation}
\label{reduced_T_def}
T_{\ff}^0\coloneqq T\times A_{\ff}^0.
\end{equation}
Throughout this section we assume $\hbar\neq 0$; in particular, the torus $T$ cannot be trivial.

Fix $k\in\BoN$, let $e_1, \dots, e_k$ be the standard basis of $\ZZ^k$, and consider a decomposition $\ff=\sum_{j=1}^k \ff^{(j)}$ with $\ff^{(j)}\in\BoN^{Q_0} \setminus \{0\}$ for all $j=1,\dots, k$. Let $\wt: (\overline{Q_{\ff}})_1\to \ZZ^{k}$ be the weighting defined on the framing arrows $\lbrace r_{i,m}, r^*_{i,m}\rbrace\in (\overline{Q_{\ff}})_1$ by
\[
\wt(r_{i,m})= e_l \qquad \wt(r^*_{i,m})= -e_l  \qquad \forall \quad \sum_{j=1}^{l-1} \ff^{(j)}_i <m \leq\sum_{j=1}^{l+1} \ff^{(j)}_i
\]
and zero on the other arrows of the quiver. Let $A\subset A_{\ff}$ be the rank $k$ subtorus
induced by $\wt$. 
Its action preserves the symplectic form of $\Nak_Q(\ff,\dd)$ and admits $k!$ chambers $\mathfrak{C}_{\sigma}$, corresponding to the elements of the symmetric group $\sigma\in \FS_k$. In the notation of Maulik and Okounkov, $\mathfrak{C}_{\sigma}$ is the chamber $\lbrace a_{\sigma(1)}>\dots>a_{\sigma(k)} \rbrace$.

By \cite[Lem.3.2]{Nakajima_tensorI}, the $A$-fixed locus of $\Nak_Q(\ff,\dd)$ has the form 
\[
\bigsqcup_{\dd^{(1)}+\dots+\dd^{(k)}=\dd}\Nak_Q(\ff^{(1)},\dd^{(1)})\times\dots\times \Nak_{Q}(\ff^{(k)},\dd^{(k)})\hookrightarrow \Nak_Q(\ff,\dd),
\]
where the inclusion is induced by taking direct sums of representations. Passing to cohomology, we deduce that for every choice of chamber $\mathfrak{C}_{\sigma}$ the stable envelope is a $\HO_{T_{\ff}}$-linear map
\begin{equation}
\label{stable envelope nakajima varieties}
    \StabC{C_\sigma}: \bigoplus_{\dd^{(1)}+\dots+\dd^{(k)}=\dd}\NakMod^{T_{\ff^{(1)}}}_{Q,\ff^{(1)},\dd^{(1)}}\otimes_{\HO_T}\dots\otimes_{\HO_T} \NakMod^{T_{\ff^{(k)}}}_{Q, \ff^{(k)},\dd^{(k)}}\to \NakMod^{T_{\ff}}_{Q, \ff,\dd}.
\end{equation}

By Theorem \ref{thm: definition stab}, a choice of polarisation is required, in order to uniquely determine the stable envelopes. One can tune the polarisation to get an arbitrary sign in the diagonal restriction of the stable envelopes only for one possible choice of chamber. The signs for the other chambers are then determined in such a way that the leading order of the R-matrix expansion \eqref{eq:R-matrix expansion} is the identity. In our setting, we choose the following conventions.
For a given quiver variety $X= \Nak_Q(\ff,\dd)$ and any $A$-fixed component $F=\Nak_Q(\ff^{(1)},\dd^{(1)})\times\dots\times \Nak_{Q}(\ff^{(k)},\dd^{(k)})\hookrightarrow \Nak_Q(\ff,\dd)$ as above, we choose the polarisation $\varepsilon_{X,A}$ in such a way that the restriction of $\Sh{L}_{\mathfrak{C}_{\id}}$ on $F\times F$ equals $(-1)^{\delta} \Eu(N^-_F)\cap [\Delta]$ with $\delta=\prod_{i<j} \dd^{(i)}\cdotsh\dd^{(j)}$. We fix this choice of polarisation for the rest of the article.

\begin{remark}
\label{rem: polarisation opposite chamber}
Let $\tau\in \FS_{k}$ be the permutation of maximal length. Then the chamber $(\mathfrak{C}_{\id})^{\text{opp}}$ opposite to $\mathfrak{C}_{\id}$ in the chamber arrangement is equal to $\mathfrak{C}_{\tau}$. Because of the choice of polarisation $\varepsilon_{X,A}$ above, the restriction of $\Sh{L}_{\mathfrak{C}_{\tau}}$ on $F\times F$ equals $(-1)^{\delta+\frac{1}{2}\codim_{X}(F)} \Eu(N^+_F)\cap [\Delta]$.
\end{remark}

\begin{remark}
    If $k=2$, which is the relevant case for us, there are only two chambers $\mathfrak{C_{\pm}}$, corresponding to the two elements of the symmetric group $\FS_2$. In this case, any $A$-fixed component in $(\Nak_Q(\ff,\dd))^{A}$ is of the form
\[
F_{\eta}=\Nak_Q(\ff^{(1)},\eta)\times \Nak_Q(\ff^{(2)},\dd-\eta)\qquad \eta\in \BoQ^{Q_0}
\]
and the partial order determined by $\mathfrak{C_{\pm}}$ on the set of fixed components in $\Nak_Q(\ff,\dd)$ is detected by the dimension vector $\eta\in \NN^{Q_0}$. Indeed, it is shown in \cite[\S 3.2.4]{MO19} that the relation $F_{\eta}>_{\mathfrak{C_{\pm}}}F_{\eta'}$ implies $\pm(\eta-\eta')>0$. However, we stress that the opposite implication is generally false.
\end{remark}

\subsection{Deformation of quiver varieties}
\label{section: Deformation of quiver varieties}

Fix a pair of dimension vectors $\ff,\dd\in \NN^{Q_0}$ such that $\Nak_Q(\ff,\dd)$ is non-empty.
Nakajima quiver varieties admit a deformation $\tilde \mu: \widetilde \Nak_Q(\ff,\dd)\to \AA^{Q_0}$ that is nicely compatible with stable envelopes. In this section we recall its main features. The center $\fz_{\dd}$ of the Lie algebra $\mathfrak{gl}_{\dd}=\Lie(\Gl_{\dd})$ can be naturally identified with an affine subspace of $\BoA^{Q_0}$ of dimension equal to the number of vertices $i\in Q_0$ such that $\dd_i\neq 0$. The moment map $\mu^{\zeta}_{\ff,\dd}\colon \BoA_{(\dd,1)}^{\zeta\sst}(\overline{Q_{\ff}})\rightarrow \mathfrak{gl}_{\dd}$ introduced in \S \ref{sec: Nakajima quiver varieties} is $\Gl_{\dd}$-equivariant, hence it descends to a map
\[
\tilde \mu': (\mu^{\zeta}_{\ff,\dd})^{-1} (\fz_{\dd})/\Gl_{\dd}\to \fz_{\dd}.
\]
This morphism is the universal deformation of the quiver variety $\Nak_Q(\ff,\dd)=(\mu^{\zeta}_{\ff,\dd})^{-1}(0)/\Gl_{\dd}$, see \cite{Kal09} for a comprehensive discussion.  However, for us it will be convenient to consider a deformation whose base space is independent of the dimension vector $\dd$. Hence, we define
\[
\tilde \mu: \widetilde \Nak_Q(\ff,\dd)\to \AA^{Q_0}
\]
as the pullback of $\tilde \mu'$ along the natural projection $p: \AA^{Q_0}\to \fz_{\dd}$. By construction, the map $\tilde \mu$ factors through the affine quotient $\widetilde\CM_{(\dd,1)}(\Pi_{Q_{\ff}}):=\Spec(\Gamma(\Sh{O}_{(\mu_{\ff,\dd})^{-1} (\fz_{\dd})})^{\Gl_{(\dd,1)}})\times_{\fz_{\dd}}\AA^{Q_0}$, i.e. we have a commutative diagram
\[
\begin{tikzcd}
    \widetilde \Nak_Q(\ff,\dd) \arrow[r, "\tilde \pi"]\arrow[d, "\tilde \mu"] & \widetilde\CM_{(\dd,1)}(\Pi_{Q_{\ff}})\arrow[dl]\\
    \AA^{Q_0}&
\end{tikzcd}
\]
The affinization morphism $\pi:\Nak_Q(\ff,\dd)\to \CM_{(\dd,1)}(\Pi_{Q_{\ff}})$ introduced in \S \ref{sec: Nakajima quiver varieties} is the fiber over the origin $0\in \AA^{Q_0}$. More generally, for every $t\in \AA^{Q_0}$ we can consider the fibers
\[
\begin{tikzcd}
	\Nak_{Q,t}(\ff,\dd)\arrow[r, hookrightarrow] \arrow[d]&\widetilde \Nak_Q(\ff,\dd) \arrow[d, "\tilde \mu"]\\
	\lbrace t \rbrace \arrow[r, hookrightarrow] & \AA^{Q_0}.
\end{tikzcd}
\]

It is well known that $\AA^{Q_0}$ has a natural wall-chamber arrangement such that, for $t\in \AA^{Q_0}$ away from the walls, all the points are $\zeta$-stable and hence the map $\pi_t: \Nak_{Q,t}(\ff,\dd)\to \CM_{(\dd,1),t}(\Pi_{Q_{\ff}})$ induced by $\tilde \pi$ is an isomorphism. We denote the complement of the walls in $\AA^{Q_0}$ by $U$. Note that $U$ is dense in $\AA^{Q_0}$ and $\Nak_{Q,t}(\ff,\dd)$ is affine for all $t\in U$. We call the fibers over such $t\in U$ affine fibers. 

As for quiver varieties, we can also consider analogous deformations of the stacks of representations of the preprojective algebra of a quiver $Q$. Namely, we define $\widetilde \FNak_Q(\ff,\dd)\to \AA^{Q_0}$ as the pullback of
\[
(\mu_{\ff,\dd})^{-1}(\fz_{\dd})/\Gl_{\dd}\to \fz_{\dd}
\]
along $p: \AA^{Q_0}\to \fz_{\dd}$.  Given a quiver $Q$, let $\lazy_i\in \BoC Q$ 
be the idempotent associated to some vertex $i\in Q_0$. Notice that the fiber
\[
\begin{tikzcd}
	\FNak_{Q,t}(\ff,\dd)\arrow[r, hookrightarrow] \arrow[d]& \widetilde \FNak_Q(\ff,\dd)\arrow[d]\\
	\lbrace t \rbrace \arrow[r, hookrightarrow] &\AA^{Q_0}
\end{tikzcd}
\]
over $t=(t_i)_{i\in Q_0}$ is the stack quotient by $\Gl_{\dd}$ of the space of $(\dd,1)$-dimensional representations of the deformed preprojective algebra $\Pi^t_{Q_{\ff}}=\BoC\overline{Q_{\ff}}/\langle\rho_t\rangle $ defined by the two-sided ideal generated by
\[
\rho_t=\sum_{a\in (Q_\ff)_1} [a,a^*]-\sum_{i\in Q_0} t_i \lazy_i.
\]
Here, the elements $\lazy_i\in \BoC\overline{Q_{\ff}}$ are the idempotents of the subalgebra $\BoC \ol{Q}\subset \BoC\overline{Q_{\ff}}$. By construction, there is an open inclusion of $\AA^{Q_0}$-stacks 
\begin{equation}
\label{eq: deformation of }
    \widetilde \Nak_Q(\ff,\dd)\hookrightarrow \widetilde \FNak_Q(\ff,\dd).
\end{equation}
In particular, the quiver variety $\Nak_Q(\ff,\dd)$ is an open subspace of the central fiber $\FNak_{Q}(\ff,\dd):= \FNak_{Q,0}(\ff,\dd)$, i.e. the fiber over $t=0$.

As discussed above, over $U\subset \AA^{Q_0}$ all representations of $\overline{Q_{\ff}}$ are stable and the group $\Gl_{\dd}$ acts freely on them. Hence, we deduce the following proposition.
\begin{proposition}
\label{prop: usual resolution vs stacky resol}
    The inclusion \eqref{eq: deformation of } restricts to an isomorphism over $U\subset\AA^{Q_0}$. In particular, for all $t\in U$, the fiber $\Nak_{Q,t}(\ff,\dd)$ is smooth, affine, and canonically isomorphic to the quotient stack $\FNak_{Q,t}(\ff,\dd)$.
\end{proposition}

\begin{remark}
    So far in the current section we have kept the stability condition fixed; however, the deformation construction for quiver varieties can be exploited to relate their cohomologies for different stability conditions $\zeta\in \BoQ^{(Q_{\ff})_0}$. Let $\pi^{\zeta}\colon \Nak^{\zeta}_Q(\ff,\dd)\rightarrow \CM_{(\dd,1)}(\Pi_{Q_{\ff}})$ be the affinization morphism. The previous proposition remains true for arbitrary generic stability conditions, and so given two generic stability conditions $\zeta,\zeta'\in \BoQ^{Q_{\ff}}$, the associated families $\widetilde \Nak^{\zeta}_Q(\ff,\dd)$ and $\widetilde \Nak^{\zeta'}_Q(\ff,\dd)$ over $\AA^{Q_0}$ are flops of each other, and hence give a canonical morphism 
    \begin{equation}
    \label{eq: canonical iso cohomology quiver varieties}
        (\pi^\zeta)_{\ast}\BoQ^{\vir}_{\Nak^{\zeta}(\ff,\dd)}\to (\pi^{\zeta'}_{\ast})\BoQ^{\vir}_{\Nak^{\zeta'}(\ff,\dd)},
    \end{equation}
    which by the topological triviality of both families \cite[\S4.2]{Nak97} is actually an isomorphism. From a slightly different (but equivalent) perspective, this morphism can also be interpreted as a Steinberg correspondence on the product $ \Nak^{\zeta}_Q(\ff,\dd)\times  \Nak^{\zeta'}_Q(\ff,\dd)$, cf. \cite[\S4.10]{MO19}.
\end{remark}

\subsection{Stable envelopes via specialisation}
\label{sec: Stable envelopes via specialisation}
As shown in \cite[\S 3.7]{MO19} the class $\Sh{L}_{\mathfrak{C}}$ defining the stable envelope of an equivariant symplectic resolution can be constructed via specialisation in Borel--Moore homology. We now recall this construction in the setting of quiver varieties, whose universal deformation was described in the previous subsection.

Let the torus $T_{\ff}$ act on $\AA^{Q_0}$ via scaling by the character $\hbar$ defined in \S \ref{section: Stable envelopes for quiver varieties}. For this choice of torus action on the base, the deformation map $\tilde \mu: \widetilde \Nak_Q(\ff,\dd)\to \AA^{Q_0}$ becomes $T_{\ff}$-equivariant. Since $T_{\ff}$ uniformly rescales the base $\AA^{Q_0}$ with weight $\hbar$, the complement $U$ of the hyperplane arrangement in $\AA^{Q_0}$ is preserved by this action. Hence, the preimage $\widetilde \Nak_{Q,U}(\ff,\dd)=\tilde\mu^{-1}( U)$ carries an action of $T_{\ff}$. 

Fix a subtorus $A\subseteq A_{\ff}$ as in \S \ref{section: Stable envelopes for quiver varieties}.
Since all the fibers of $ \widetilde\Nak_{Q,U}(\ff,\dd):=\tilde\mu^{-1}( U)$ are affine, the attracting set $\Att{C_\sigma}(\widetilde \Nak_{Q,U}(\ff,\dd)^{A})$ is closed in $ \widetilde\Nak_{Q,U}(\ff,\dd)$ and hence defines a fundamental class 
\[
[\Att{C_\sigma}(\widetilde\Nak_{Q,U}(\ff,\dd)^{A})]\in \HO^{\BoMo}_{T_{\ff}}(\widetilde\Nak_{Q,U}(\ff,\dd)^{A}\times_{\AA^{Q_0}}\widetilde\Nak_{Q,U}(\ff,\dd), \BoQ).
\]
The cycle $\Sh{L}_{\mathfrak{C}_\sigma}$ defining the stable envelope $\StabC{C_\sigma}$ is then recovered as the specialisation of this class to the central fiber of the deformation space. More precisely, we have the following proposition.
\begin{proposition}\cite[Thm.3.7.4]{MO19}
\label{prop: stab via specialzation}
    The class $\Sh{L}_{\mathfrak{C}_\sigma}$ from Theorem \ref{thm: definition stab} coincides, up to a sign due to the choice of polarisation, with the specialisation
    \[
    \lim_{t\to 0}\;[\Att{C_\sigma}(\widetilde\Nak_{Q,U}(\ff,\dd)^{A})]\in \HO^{\BoMo}_{T_{\ff}}(\Nak_{Q}(\ff,\dd)^{A}\times_{\AA^{Q_0}} \Nak_{Q}(\ff,\dd), \BoQ).
    \]
\end{proposition}
For details about the specialisation map, see \cite[\S 2.6]{chriss2009representation}. 

\begin{remark}
Although specialisation in Borel--Moore homology is highly dependent on the geometry of the deformation space $\tilde \mu: \widetilde \Nak_Q(\ff,\dd)\to \AA^{Q_0}$ in an open neighborhood of $0\in \AA^{Q_0}$, it is common to write the statement of Proposition \ref{prop: stab via specialzation} as
\[
\Sh{L}_{\mathfrak{C}_\sigma}=\lim_{t\to 0}\;[\Att{C_\sigma}(\Nak_{Q,t}(\ff,\dd)^{A})],
\]
i.e. by thinking of $\Sh{L}_{\mathfrak{C}_\sigma}$ as the limit of a family of fundamental classes $[\Att{C_\sigma}(\Nak_{Q,t}(\ff,\dd)^{A})]_{t\in U}$ living in the affine fibers of $\widetilde\Nak_Q(\ff,\dd)\times_{\AA^{Q_0}}\widetilde\Nak_Q(\ff,\dd)^A$. In the continuation of the article, we adopt this notation. 
\end{remark}

\section{Maulik-Okounkov Yangians}

\subsection{R-matrices}
\label{sect: R-matrices}
Consider a framing decomposition $\ff=\ff'+\ff''$. Following the analysis from \S \ref{section: Stable envelopes for quiver varieties}, we associate to such a decomposition a two-dimensional torus $A$ acting on $\Nak_Q(\ff,\dd)$ and two chambers $\mathfrak{C}_{+},\mathfrak{C_-}\subset \cochar(A)\otimes_{\ZZ}\RR$, corresponding to the identity element and the nontrivial permutation in the symmetric group $\FS_2$, respectively. Taking direct sums, the stable envelope \eqref{stable envelope nakajima varieties} gives $\HO_{T_{\ff}}$-linear maps 
\begin{equation}
    \label{eq: direct sum of stabs}
    \StabC{\pm}: \NakMod^{T_{\ff'}}_{Q,\ff'}\otimes_{\HO_T} \NakMod^{T_{\ff''}}_{Q, \ff''}\to \NakMod^{T_{\ff}}_{Q, \ff}.
\end{equation}
These maps are injective morphisms of finitely generated free $\HO_{T_{\ff}}$-modules of the same rank, but surjectivity fails for degree reasons. However, they become isomorphisms after inverting
sufficiently many equivariant parameters, so we can define their composition
\[
R_{\ff',\ff''}:=\StabC{+}^{-1}\circ \StabC{-}\in \End_{\HO_{T_{\ff}}}(\NakMod^{T_{\ff'}}_{Q,\ff'}\otimes_{\HO_{T}} \NakMod^{T_{\ff''}}_{Q,\ff''})\otimes_{\HO_{T_{\ff}}} \Frac(\HO_{T_{\ff}})
\]
after applying the functor $-\otimes_{\HO_{T_{\ff}}}\Frac(\HO_{T_{\ff}})$. 
\begin{remark}
\label{loc_rem}
Consider the natural isomorphism $\HO_{T_{\ff}}\cong\HO_{A_{\ff'}}\otimes \HO_{A_{\ff''}}\otimes \HO_T$, and the prime ideal $\mathfrak{p}=\HO_{A_{\ff'}}\otimes\mathfrak{m}$, where $\mathfrak{m}\subset \HO_{A_{\ff''}}\otimes \HO_T$ is the unique maximal graded ideal.  If $M$ is a $\HO_{T_{\ff}}$-module, we will often write $M_{\loc}$ for the localisation at $\mathfrak{p}$.  Then $(\StabC{+})_{\loc}$ is already an isomorphism.  This less drastic localisation (than tensoring with the function field) will prove useful when comparing with the coproduct.
\end{remark}

Whenever the subscripts in $R_{\ff',\ff''}$ are clear from the context we will simply write $R$.  This operator is called the Maulik-Okounkov R-matrix, and its name is justified  by the following pivotal result:
\begin{proposition}\cite[\S 4.1.9]{MO19}
    The Maulik-Okounkov R-matrix is a solution of the Yang-Baxter equation with spectral parameters. Namely, for every decomposition $\ff=\ff^{(1)}+\ff^{(2)}+\ff^{(3)}$ we have\footnote{Here the superscripts indicate the factors on which the operators act: for example, $R^{(12)}=R_{\ff^{(1)},\ff^{(2)}}\otimes \id$. }.
    \begin{equation}
        \label{eq: YBE}
        R^{(12)}R^{(13)}R^{(23)}=R^{(23)}R^{(13)}R^{(12)}
    \end{equation}
    in
    \[
    \End_{\HO_{T_{\ff}}}(\NakMod^{T_{\ff^{(1)}}}_{Q,\ff^{(1)}}\otimes_{\HO_T} \NakMod^{T_{\ff^{(2)}}}_{Q,\ff^{(2)}}\otimes_{\HO_T} \NakMod^{T_{\ff^{(3)}}}_{Q,\ff^{(3)}})\otimes_{\HO_{T_{\ff}}} \Frac(\HO_{T_{\ff}}).
    \]
\end{proposition}

Let $\alpha\in \Char(A)$ be the weight whose zero locus is the wall separating the two chambers $\mathfrak{C}_{\pm}\subset \cochar(A)\otimes_{\ZZ}\RR$. It is the unique nontrivial tangent weight in the normal bundle of $\Nak_Q(\ff,\dd)^A\subset \Nak_Q(\ff,\dd)$, and as in \S \ref{sec: axiomatic definition stab} we filter $\NakMod^{T_{\ff'}}_{Q,\ff'}\otimes_{\HO_T} \NakMod^{T_{\ff''}}_{Q,\ff''}$ by degree in $\alpha$.  By \cite[\S 4.1.7]{MO19} the R-matrix can be expanded as a formal power series
\begin{equation}
    \label{eq:R-matrix expansion}
    R(\alpha)= 1+ \hbar \sum_{k\geq 0} {R_k}\alpha^{-k}
\end{equation}
with coefficients
\[
R_k \in  \End_{\HO_{T_{\ff}}}(\NakMod^{T_{\ff'}}_{Q,\ff'}\otimes_{\HO_T} \NakMod^{T_{\ff''}}_{Q,\ff''})
\]
of $\alpha$-degree zero.  Notice in particular that all the $R_k$ live in the non-localised endomorphism ring, and $R_1$ is uniquely determined by the expansion \eqref{eq:R-matrix expansion}.

The operator $\rmat:=R_1$ is called the classical r-matrix. It plays a crucial role in the theory of quantum groups and also in this paper. Plugging the expansion above into the Yang-Baxter equation \eqref{eq: YBE},
it follows that
\[
[\rmat^{(12)}, \rmat^{(13)}+ \rmat^{(23)}]=[\rmat^{(23)}, \rmat^{(12)}+ \rmat^{(13)}]=0.
\]
Setting $\bar\rmat^{(ij)}=\rmat^{(ij)}/{(a_i-a_j)}$, the previous equation can be equivalently rewritten as follows:
\[
[\bar\rmat^{(12)}, \bar\rmat^{(13)}]+[\bar\rmat^{(12)}, \bar\rmat^{(23)}]+[\bar\rmat^{(13)}, \bar\rmat^{(23)}]=0.
\]
This is the well-known classical Yang-Baxter equation with spectral parameters.


\subsection{FRT construction of the Maulik-Okounkov Yangian}

The FRT construction, named after Faddeev, Reshetikhin and Takhtadzhyan, is a general procedure that allows the reconstruction of a quasi-triangular Hopf algebra from the braidings of its category of representations or, equivalently, from its R-matrices. This procedure, which we now review, was exploited by Maulik and Okounkov to produce a Yangian $\Yang_Q^{\MO, T}$ for an arbitrary quiver $Q$ starting from the R-matrices defined in the previous section.  If a torus $T$ is fixed we will employ\footnote{This is not in violation of the convention regarding $T$ superscripts in \S \ref{conventions_sec}, since a priori it does \textit{not} make sense to choose $T=\{1\}$ in the construction of $\Yang_Q^{\MO,T}$.} the abbreviation $\Yang_Q^{\MO}\coloneqq \Yang_Q^{\MO,T}$.

Fix a quiver $Q$ and let $\delta_i\in \NN^{Q_0}$ be a framing vector such that $(\delta_i)_i=1$ and $(\delta_i)_j=0$ for $j\neq i$.  Since the torus $A_{\delta_i}\subset T_{\delta_i}$ acts trivially on $\Nak^{T}_{Q}(\delta_i,\dd)$, we have
\begin{equation}
\label{eq: noncanonical splitting}
    \NakMod^{T_{\delta_i}}_{Q,\delta_i} \cong \NakMod^{T}_{Q,\delta_i}[a].
\end{equation}
In other words, $\NakMod^{T_{\delta_i}}_{Q,\delta_i}$ is isomorphic to the polynomial ring in the variable $a$ with coefficients in $\NakMod^{T}_{Q,\delta_i}$. Here, $a$ is the generator of $\HO_{A_{\delta_i}}=\QQ[a]$.
Set
\[
\NakMod^{T,\vee}_{Q,\delta_i}=\bigoplus_{\dd\in-\BoN^{Q_0}}\Hom_{\HO_T}(\NakMod^T_{Q,\delta_i,-\dd}, \HO_T)
\]
to be the $\BoZ^{Q_0}$-graded dual $\HO_T$-module.  Since the cohomology $\NakMod^{T}_{Q,\delta_i}$ is a free $\HO_T$-module, its dual $\NakMod^{T,\vee}_{Q,\delta_i}$ is also free. Thus, the trace gives a perfect pairing of $\BoZ^{Q_0}$-graded $\HO_T$-modules
\[
\Tr: \NakMod^{T}_{Q,\delta_i}\otimes_{\HO_T} \NakMod^{T,\vee}_{Q,\delta_i}\to \HO_T.
\]
For any $i\in Q_0$, let $m_{i}(a)$ denote a polynomial in $a$ with values in $
\NakMod^{T}_{Q,\delta_i}\otimes_{\HO_T} \NakMod^{T,\vee}_{Q,\delta_i}$.

\begin{definition}\cite[\S 5.2.6]{MO19}
\label{def: MO Yangian}
The Maulik-Okounkov Yangian $\Yang_Q^{\MO}$ is the subalgebra
\[
\Yang_Q^{\MO}\subset \prod_{\ff\in\NN^{Q_0}}\End_{\HO_{T_\ff}}\left(\NakMod^{T_\ff}_{Q,\ff}\right)
\]
generated by the operators
\begin{equation}
\label{E_def}
\E(m_i(a)):=\frac{1}{\hbar} \Res_{a}\left(\Tr_{\delta_i}\left( (m_i(a)\otimes 1) \circ R_{\delta_i,\ff}(a) \right)\right)
\end{equation}
for all $i\in Q_0$ and $m_i(a)\in \NakMod^{T}_{Q,\delta_i}\otimes_{\HO_T} \NakMod^{T,\vee}_{Q,\delta_i}[a]$. 
\end{definition}
Here, $R_{\delta_i,\ff}(a)$ is the expansion in $a$ of the R-matrix defined in \S \ref{sect: R-matrices} and $\Tr_{\delta_i}$ denotes the partial trace over the auxiliary space $\NakMod^{T}_{Q,\delta_i}$. Notice that equation \eqref{eq:R-matrix expansion} implies that $\Res_{a}\left(\Tr\left( (m_i(a))\otimes 1) \circ R_{\delta_i,\ff}(a) \right)\right)$ is divisible by $\hbar$ and hence $\E(m_i(a))\lvert_{\NakMod^{T_\ff}_{Q,\ff}}$ is a well defined element of $\End_{\HO_{T_\ff}}\left(\NakMod^{T_\ff}_{Q,\ff}\right)$.

An important set of relations for the Yangian is directly inherited from the Yang-Baxter equation. Indeed, since the operator $\E$ is linear, we get a map
\[
\E: \text{Tensor Algebra}\left(\bigoplus_{i\in Q_0}\left( \NakMod^{T}_{Q,\delta_i}\otimes_{\HO_T} \NakMod^{T,\vee}_{Q,\delta_i}[a] \right)\right)\twoheadrightarrow \Yang_Q^{\MO} 
\]
factoring through the ideal generated by so-called  RTT=TTR relations
\[
\left(m_i(a_i)\otimes m_j(a_j)\right) \circ R_{\delta_i,\delta_j}= R_{\delta_i,\delta_j}\circ \left(m_j(a_j)\otimes m_i(a_i) \right).
\]
This follows at once by the definition and the Yang-Baxter equation \eqref{eq: YBE}.

\begin{remark}
\label{rem: larger framing elements in the yangian}
    The domain of the map $m(a)\mapsto \E(m(a))$ can be naturally extended by allowing endomorphims 
    \[
    m(a)\in \NakMod^{T}_{Q,\ff}\otimes_{\HO_T} \NakMod^{T,\vee}_{Q,\ff}[a] 
    \]
    for arbitrary framing dimensions $\ff\in \NN^{Q_0}$. The resulting operators $\E(m(a))$ are already contained in $\Yang_Q^{\MO}$ \cite[\S5.2.14]{MO19}, so adding them in the definition of $\Yang_Q^{\MO}$ would simply generate redundance. Nonetheless, considering these more general framings can be useful, cf. the definition of the operators $\psi(\beta)$ and $\psi(\lambda)$ after Proposition \ref{gen_by_taut}, and Remark \ref{redundancy_remark}.
\end{remark}

\begin{remark}
\label{rem: Yangians for different stability conditions}
    Different choices of generic stability condition $\zeta\in\BoN^{Q_0}\oplus \BoN\cong \BoN^{(Q_{\ff})_0}$ give rise to different Yangians. However, these algebras are canonically isomorphic. Indeed, consider two generic stability conditions $\zeta$ and $\zeta'$ and the associated Yangians $\Yang_Q^{\MO, \zeta}$ and $\Yang_Q^{\MO, \zeta'}$. Let $\Xi_{\ff,\dd}: \NakMod^{\zeta,  T_\ff}_{Q, \ff, \dd}\to \NakMod^{\zeta',  T_\ff}_{Q, \ff, \dd}$ be the canonical isomorphism obtained from \eqref{eq: canonical iso cohomology quiver varieties} by taking derived global sections. Set also $\Xi_{\ff}=\bigoplus_{\dd\in \NN^{Q^0}}\Xi_{\ff,\dd}$ and $\Xi=\prod_{\ff\in\NN^{Q_0}}\Xi_{\ff}$.  
    Then $\Xi$ induces the canonical isomorphism 
    \[
    \prod_{\ff\in\NN^{Q_0}}\End_{\HO_{T_\ff}}\left(\NakMod^{\zeta, T_\ff}_{Q,\ff}\right)\to \prod_{\ff\in\NN^{Q_0}}\End_{\HO_{T_\ff}}\left(\NakMod^{\zeta', T_\ff}_{Q,\ff}\right)\qquad \lambda \mapsto \Xi \circ \lambda  \circ \Xi^{-1},
    \] 
    which maps $\Yang_Q^{\MO, \zeta}$ isomorphically onto $\Yang_Q^{\MO, \zeta'}$. The last claim follows at once from Definition \ref{def: MO Yangian} and the compatibility of the stable envelopes with $\Xi_{\ff,\dd}$ \cite[\S4.10.5]{MO19}.
\end{remark}


\subsection{The Maulik-Okounkov Lie algebra}
\label{MO_LA_sec}

The Yangian $\Yang_Q^{\MO}$ admits a natural filtration by the $a$-degree, that is, by setting 
\begin{equation}
    \label{eq:filtration degree yangian}
    \deg(\E(m(a)))=\deg_a(m(a))
\end{equation}
on the generators. The Maulik-Okounkov Lie algebra $\fg^{\MO,T}_{Q}$ is defined as the degree zero piece of $\Yang_Q^{\MO} $, with Lie bracket given by the commutator. In other words, $\fg^{\MO,T}_{Q}\subset Y^{\MO}_Q$ is spanned by operators 
\[
\E(m_i(a))\in \prod_{\ff\in\NN^{Q_0}}\End_{\HO_{T_{\ff}}}\left(\NakMod^{T_{\ff}}_{Q,\ff}\right)
\]
such that $m_i(a)=m_i\in\NakMod^{T}_{Q,\delta_i}\otimes_{\HO_T} \NakMod^{T,\vee}_{Q,\delta_i}[a]$ is a constant polynomial. By the definition of $\E$, this is equivalent to considering the $\HO_{T}$-module spanned by all operators of the form  
\[
\Tr_{\delta_i}\left( (m_i\otimes 1) \circ \rmat_{\delta_i,\ff} \right)\in \End_{\HO_{T_{\ff}}}\left(\NakMod^{T_{\ff}}_{Q,\ff}\right),
\]
where $\rmat_{\delta_i,\ff}\in \End_{\HO_T}(\NakMod^{T}_{Q,\delta_i}\otimes_{\HO_T} \NakMod^{T_{\ff}}_{Q,\ff})$ is the classical r-matrix. Notice in particular that $\fg^{\MO,T}_{Q}$ is defined as a $\HO_{T}$-linear Lie algebra. In the next theorem we record the main properties of the Maulik-Okounkov Lie algebra.
\begin{theorem}\cite[\S 5.3]{MO19}
The Lie algebra $\fg^{\MO,T}_{Q}$ admits a triangular decomposition
\[
\fg^{\MO,T}_{Q}= \fn^{\MO,T,-}_{Q}\oplus \bar\fh^{\MO,T}_{Q}\oplus\fn^{\MO,T,+}_{Q}
\]
where $\bar\fh^{\MO,T}_{Q}$ is a maximal commutative subalgebra and
\[
\fn^{\MO,T,\pm}_{Q}=\bigoplus_{\dd\in \NN^{Q_0}\setminus\{\ul{0}\}}\fg^{\MO,T}_{Q,\pm \dd}.
\]
Each root space $\fg^{\MO,T}_{Q,\pm \dd}$ for $\dd\neq \ul 0$ is a projective $\HO_{T}$-module of finite rank spanned by elements $\xi\in \fg^{\MO,T}_{Q, \pm\dd}$ such that 
\[
\xi: \NakMod^{T_{\ff}}_{\ff, \bullet}\to \NakMod^{T_{\ff}}_{\ff, \bullet\pm \dd}.
\]
The root spaces satisfy $[\fg^{\MO,T}_{Q,\dd'}, \fg^{\MO,T}_{Q,\dd''}]\subseteq \fg^{\MO,T}_{Q,\dd'+\dd''}$. Moreover, there exists a bilinear invariant form $(\cdot,\cdot): \fg^{\MO,T}_{Q}\times \fg^{\MO,T}_{Q}\to \HO_{T}$ whose invariant tensor is the classical r-matrix $\rmat$.
\end{theorem}

\begin{remark}
    As a $\HO_T$-module, the Lie algebra $\bar\fh^{\MO,T}_{Q}$ is isomorphic to $(\HO_T)^{\oplus\overrightarrow{Q}_0}$, where $\overrightarrow{Q}$ is the principally framed quiver associated to $Q$ introduced in \S \ref{GKM_thm_sec}. As a consequence, its rank is equal to $2|Q_0|$, which is twice as big as one might guess from the theory of generalised Kac-Moody Lie algebras in \S \ref{GKM_sec}. These extra dimensions are spanned by a rank $|Q_0|$ 
    subalgebra $\mathfrak{z}\subset \bar\fh^{\MO}_{Q}$ whose elements act on the modules $\NakMod^{T_{\ff}}_{Q,\ff}$ by multiplication by a linear function of the framing vector $\ff$. As a consequence, $\mathfrak{z}$ is central in $\fg^{\MO,T}_{Q}$. Theorem \ref{main_thm} will identify $\bar\fh^{\MO}_{Q}$ with the Abelian Lie algebra $(\fh^T_{\overrightarrow{Q}})^{\vee}$.
\end{remark}

The Maulik-Okounkov Lie algebra is the fundamental representation theoretic object underlying the Yangian $\Yang_Q^{\MO}$. Indeed, Maulik and Okounkov proved the following result.
\begin{proposition}\cite[Thm.5.5.1]{MO19}
\label{gen_by_taut}
    The Yangian $\Yang^{\MO}_Q$ is generated by the Lie algebra $\fg^{\MO,T}_{Q}$ and the operators of classical multiplication, i.e. the operators of multiplication by the Chern classes of the tautological bundles on $\Nak^{T}_Q(\ff,\dd)$. Moreover, we have a PBW-type isomorphism
    \[
    \gr\left(\Yang^{\MO}_Q\right)\cong  \Sym_{\HO_T}(\fg^{\MO,T}_Q\otimes \HO^*(\B \BoC^*,\BoQ))
    \]
    induced by the filtration \eqref{eq:filtration degree yangian}.
\end{proposition}
Fix a dimension vector $\ff'\in\BoN^{Q_0}$.  There is a canonical isomorphism $\NakMod^T_{Q,\ff',\ul{0}}\cong \HO_T$ which we use to identify these two $\HO_T$-modules.  An element $\beta\in \NakMod^T_{Q,\ff',\dd}$ determines a unique $\HO_T$-linear homomorphism $\NakMod^T_{Q,\ff',\ul{0}}\rightarrow \NakMod^T_{Q,\ff',\dd}$ taking $1\mapsto \beta$.  Composing with the inclusion $\NakMod^T_{Q,\ff',\dd}\hookrightarrow \NakMod^T_{Q,\ff'}$ and the projection $\NakMod^T_{Q,\ff'}\rightarrow \NakMod^T_{Q,\ff',\ul{0}}$ we obtain an endomorphism $g_{\beta}\in  \End_{\HO_T}(\NakMod_{Q,\ff'})$.

Now fix a decomposition $\ff=\ff'+\ff''$.  We consider $T_{\ff''}^+\coloneqq A_{\triv}\times T_{\ff''}$, with $A_{\triv}$ as at the beginning of \S \ref{section: Stable envelopes for quiver varieties}, as a subtorus of $T_{\ff}$ in the natural way, via the natural embedding $T\times A_{\triv}\times A_{\ff''}\hookrightarrow T\times A_{\ff'}\times A_{\ff''}$ induced by the inclusion $A_{\triv}\subset A_{\ff'}$.  Working in $T^{+}_{\ff''}$ equivariant cohomology of $\Nak_Q(\ff,\dd)$ we define as in \S \ref{sect: R-matrices}
\[
R_{\ff',\ff''}:=\StabC{+}^{-1}\circ \StabC{-}\in \End_{\HO_T}(\NakMod^{T^+}_{Q,\ff'}\otimes_{\HO_T} \NakMod^{T_{\ff''}}_{Q,\ff''})\otimes_{\HO_{T_{\ff''}^+}} \Frac(\HO_{T_{\ff''}^+})
\]
We define 
\begin{equation}
\label{little_psi_def}
\psi(\beta):=\E(g_{\beta})
\end{equation}
as in \eqref{E_def}.  Note that if $\ff'=\delta_i$ for some $i\in Q_0$, and $\beta\in \NakMod^{T}_{Q,\ff',\dd}$, then by definition, $\psi(\beta)\in\Fg^{\MO,T}_{Q, \dd}$.
Given $\lambda\in (\NakMod^T_{Q,\ff,\dd})^{\vee}$ we define $\psi(\lambda)$ the same way.

By Remark \ref{rem: larger framing elements in the yangian}, the operators $\psi(\beta)$ and $\psi(\lambda)$ belong to $\Yang^{\MO}_Q$. They can be written more evocatively as follows. Let $\ket{\ff'}\in \NakMod^T_{Q,\ff',\ul{0}}$ be the vacuum, i.e. the class $1$ under the isomorphism $\NakMod^T_{Q,\ff',\ul{0}}\cong \HO_T$. Let also $\bra{\ff'}\in (\NakMod^T_{Q,\ff'})^\vee$ be its dual. Then for a given $\beta\in \NakMod^T_{\ff',\dd'}$, the class $\psi(\beta)$ is the raising operator:
\begin{equation}
    \label{eq:MO raising operator}
    \psi(\beta): \NakMod^{T_{\ff''}}_{\ff'',\dd''}\to  \NakMod^{T_{\ff''}}_{\ff'',\dd''+\dd'}
    \qquad 
    \alpha\mapsto 
    \left((\bra{\ff'}\otimes \id)\circ \rmat_{\ff',\ff''}\right)(\beta\otimes \alpha).
\end{equation}
Similarly, the operator $\psi(\lambda)$ associated to a class $\lambda\in (\NakMod^T_{Q,\ff',\dd'})^{\vee}$ is the lowering operator
\[
\psi(\lambda): \NakMod^{T_{\ff''}}_{\ff'',\dd''}\to  \NakMod^{T_{\ff''}}_{\ff'',\dd''-\dd'}
\qquad 
\alpha\mapsto 
\left((\lambda\otimes \id)\circ \rmat_{\ff',\ff''}\right)(\ket{\ff'}\otimes \alpha)
\]
and hence is zero if $\dd''<\dd'$. Notice in particular that the vacuum $\ket{\ff''}$ is a lowest weight vector, in the sense that it is killed by the subalgebra $\fn_Q^{\MO,T, -}$ and rescaled by $\bar\fh^{\MO,T}_Q$.  In Proposition \ref{MO_lowest_weights}, we will use the Beilinson--Bernstein--Deligne--Gabber decomposition theorem to construct non-vacuum lowest weight vectors from intersection cohomology.

The following proposition is very useful for beginning to get a handle on the Lie algebra $\fg^{\MO,T}_Q$.
\begin{proposition}\cite[Prop.5.3.4, Prop.5.3.8]{MO19}
\label{MO_span}
Fix $\dd'\in \BoN^{Q_0}$, and pick $\ff'=\delta_i$ with $i\in\supp(\dd')$.  Then the weight space $\fg^{\MO,T}_{Q,\dd'}$ is spanned by operators $\psi(\beta)$ for $\beta\in \NakMod^T_{Q,\ff',\dd'}$, and the weight space $\fg^{\MO,T}_{Q,-\dd'}$ is spanned by operators $\psi(\lambda)$ for $\lambda\in (\NakMod^T_{Q,\ff',\dd'})^{\vee}$.
\end{proposition}

By construction, for every $\ff\in\BoN^{Q_0}$ the $\BoN^{Q_0}$-graded vector space $\NakMod^{T^0_{\ff}}_{Q,\ff}$ is a module for $\fg^{\MO,T}_{Q}$.  In light of Proposition \ref{lowest_weight_prop}, the following proposition will be useful in comparing $\fg^{\MO,T}_Q$ and $\fg^T_{\Pi_Q}$.

\begin{proposition}
\label{MO_lowest_weights}
Pick $\ff''\in\BoN^{Q_0}$ and let $(\dd,1)\in\Phi^+_{Q_{\ff''}}$, so that there is a canonical inclusion
\[
\imath\colon \IC^*(\CM^{T^0_{\ff''}}_{(\dd,1)}(\Pi_{Q_{\ff''}}),\BoQ^{\vir})\hookrightarrow \NakMod^{T^0_{\ff''}}_{Q,\ff''}
\]
arising from the inclusion of the summand with full support in the direct image $(\pi_{\ff''})_*\BoQ^{\vir}_{\Nak^{T^0_{\ff''}}_{Q}(\ff'',\dd)}$.  This is an inclusion of a space of lowest weight vectors for the action of $\fg^{\MO,T}_Q$; elements of $\fn^{\MO,T,-}_Q$ annihilate this subspace.
\end{proposition}

\begin{proof}
Fix $\dd'\in\BoN^{Q_0}$ and pick $\ff'=\delta_i$ with $i\in\supp(\dd')$.  We write $\dd=\dd'+\dd''$.  We embed $Q_{\ff'}$ and $Q_{\ff''}$ inside $Q_{\ff',\ff''}$ as in \S \ref{quiver_sec}.  We consider points of $\Nak^{T}_Q(\ff',\dd')$ as representations of $\overline{Q_{\ff',\ff''}}$ with dimension vector $(\dd',1,0)$, and points of $\Nak^{T^0_{\ff''}}_Q(\dd'',\ff'')$ as representations with dimension vector $(\dd'',0,1)$.  Set $X=\CM_{(\dd,1,1)}(\overline{Q_{\ff',\ff''}})/T_{\ff''}^0$.  Then we define morphisms 
\begin{align*}
\pi_1\colon &Y_1\coloneqq \Nak^{T}_Q(\ff',0)\times_{\B T} \Nak^{T_{\ff''}^0}_Q(\ff'',\dd)\rightarrow X\\
\pi_2\colon &Y_2\coloneqq \Nak^{T}_Q(\ff',\dd')\times_{\B T} \Nak^{T_{\ff''}^0}_Q(\ff'',\dd'')\rightarrow X
\end{align*}
sending a pair of modules to their direct sum.  Now let $\gamma\in \fn^{\MO,T,-}_Q$. By Proposition \ref{MO_span} we may write $\gamma=\psi(\lambda)$ for some $\lambda \in (\NakMod_{Q,\ff',\dd'}^T)^{\vee}$.  By \cite[Prop.4.8.2]{MO19} the operator $\psi(\lambda)$ is induced by a Steinberg correspondence: convolution with a cycle in $\HO^{\BoMo}(Y_1\times_X Y_2,\BoQ)$ gives the required operation
\[
f\colon \HO^{\BoMo}(Y_1,\BoQ^{\vir})\rightarrow \HO^{\BoMo}(Y_2,\BoQ^{\vir}).
\]
Since $\pi_1$ and $\pi_2$ are proper morphisms from smooth varieties, \cite[Lem.8.6.1]{chriss2009representation} implies that $f$ is induced by applying the derived global sections functor to a morphism
\[
\tilde{f}\colon (\pi_{1})_*\BoQ_{Y_1}^{\vir}\rightarrow (\pi_{2})_*\BoQ_{Y_2}^{\vir}.
\]
Let $l^{(2)}\colon \CM_{(\dd'',1)}^{T_{\ff''}^0}(\Pi_{Q_{\ff''}})\hookrightarrow X$ be the extension by zero morphism.  Then the domain of $\tilde{f}$ contains a unique summand with support $\Image(l^{(2)})$, while the target contains no summands with support $\Image(l^{(2)})$.  The inclusion of the cohomology of this summand is precisely the inclusion $\imath$, and we conclude that $f\circ \imath=0$.  Therefore $\gamma$ annihilates the image of $\imath$.  The proposition follows.
\end{proof}

\section{Nonabelian stable envelopes and comparison statements}
\label{NaSE_sec}

\subsection{Nonabelian stable envelopes}
\label{sec: nonabelian stab}
Any quiver variety $\Nak_Q(\ff,\dd)$ can be seen as a locally closed substack of $\BoA_{(\dd,1)}(\overline{Q_{\ff}})/\Gl_{\dd}$. More generally, we can consider the diagram
\begin{equation}
\label{eq: diagram nonabelian stab}
    \begin{tikzcd}
    \Nak_Q(\ff,\dd)\arrow[r,hookrightarrow] & \FNak_Q(\ff,\dd) \arrow[r,hookrightarrow]\arrow[d] & \BoA_{(\dd,1)}(\overline{Q_{\ff}})/\Gl_{\dd} \arrow[d]\\
    & \FM_{(\dd,1)}(\Pi_{Q_{\ff}}) \arrow[r, hookrightarrow] &\FM_{(\dd,1)}(\overline{Q_{\ff}})
\end{tikzcd}
\end{equation}
where the top horizontal arrows are, respectively, open and closed embeddings and the vertical arrows are obtained by taking quotients by $\Gl_{(\ul{0},1)}\cong \BoC^*$. Throughout this section we assume that $\Nak_Q(\ff,\dd)$ is non-empty.  Moreover, we consider the action of $T_{\ff}:=T\times A_{\ff}$ on the spaces in the top row of the diagram above and of $T^0_{\ff}:= T\times (A_{\ff}/\Gl_{(\ul{0},1)})$ on those in the bottom row. To then define the stack $\FM^{T^0_{\ff}}_{(\dd,1)}(\Pi_Q)$ as a global quotient stack as in \S \ref{moduli_stacks_sec} we choose a section of $T_{\ff}\rightarrow T^0_{\ff}$.  The definition of the stack is obviously independent of the choice of section.

The nonabelian stable envelope is a distinguished map
\begin{equation}
    \label{nonabelian stab}
    \overline \Psi_{\ff,\dd}: \HO^{\BoMo}(\Nak^{T_{\ff}}_{Q}(\ff,\dd), \BoQ)\to \HO^{\BoMo}(\FM^{T^0_\ff}_{(\dd,1)}(\overline{Q_{\ff}}), \BoQ)
\end{equation}
preserving certain counts in the context of Okounkov's theory of quasimaps to quiver varieties \cite{okounkov2017lectures}.  As shown by Aganagic and Okounkov in \cite{aganagic2017quasimap}, nonabelian stable envelopes can be reduced to the ``Abelian'' case. i.e. to the standard stable envelopes considered in this article. We now review their construction.  To simplify the exposition, we assume that the stability condition entering in the definition of the quiver variety is $\zeta=\zeta_{\mathrm{deg}}^+$, defined in \S \ref{sec: Nakajima quiver varieties}. For the general case, see \cite[\S 2.1.3]{aganagic2017quasimap}.

If $\dd=0$, then $\Nak^{T_{\ff}}_{Q}(\ff,\dd)=\FM^{T^0_\ff}_{(\dd,1)} (\overline{Q_{\ff}})= \B T_{\ff}$ and we set $\overline \Psi_{\ff,\dd}=\id$. If instead $\dd\neq0$, set $\hat\ff:=\ff+\dd$ and consider the action of the torus $A_{\hat\ff}$ on the quiver variety $\Nak_{Q}(\hat\ff,\dd)$. Following \S \ref{section: Stable envelopes for quiver varieties}, let $A\subset A_{\hat\ff}$ be the two-dimensional\footnote{If $\dd\neq0$ the framing vector $\ff$ cannot be trivial, otherwise the quiver variety $\Nak_Q(\ff,\dd)$ is empty, contradicting our assumption.} torus associated to the splitting $\hat\ff=\ff+\dd$.
The varieties 
\begin{align*}
    \Nak_{Q}(\ff,\dd) & =\Nak_{Q}(\ff,\dd)\times \Nak_{Q}(\dd,\ul{0})\\
    \Nak_{Q}(\dd,\dd) & =\Nak_{Q}(\ff,\ul{0})\times \Nak_{Q}(\dd,\dd)
\end{align*}
can be seen as two fixed components for the $A$-action on $\Nak_{Q}(\ff+\dd,\dd)$. This action admits two chambers; between these, let $\Ch_+$ be the one such that $\Nak_{Q}(\dd,\dd)<_{\Ch_+} \Nak_{Q}(\ff,\dd)$. Let $\Sh{L}_{\Ch_{+}}$ be the Lagrangian cycle associated to this data by Theorem \ref{thm: definition stab}. 
Let also $\Gl'_{\dd}$ be a copy of $\Gl_{\dd}$ and consider the action of the group $A_{\ff}\times \Gl'_{\dd}$ on the Nakajima variety $\Nak_{Q}(\ff+\dd,\dd)$, with $\Gl'_{\dd}$ acting in the natural way on the framing. 
The class $\Sh{L}_{\mathfrak{C}_{+}}$ is invariant under the action of $\Gl'_{\dd}$, so it defines a $\HO_{\Gl'_{\dd}\times T_{\ff}}$-linear map
\[
\Stab_{+}:\HO^*_{T_{\ff}\times \Gl'_{\dd}}(\Nak_{Q}(\ff,\dd), \BoQ^{\vir})\to \HO^*_{T_{\ff}\times \Gl'_{\dd}}(\Nak_{Q}(\ff+\dd,\dd), \BoQ^{\vir}).
\]
Moreover, since $\Gl'_{\dd}$  acts trivially on $\Nak_{Q}(\ff,\dd)$, the map above descends to a $\HO_{T_{\ff}}$-linear map
\begin{equation}
\label{stab part of nonabelian stable envelope}
    \Stab'_{+}:\HO^*_{T_{\ff}}(\Nak_{Q}(\ff,\dd), \BoQ^{\vir})\to \HO^*_{T_{\ff}\times \Gl'_{\dd}}(\Nak_{Q}(\ff+\dd,\dd), \BoQ^{\vir}).
\end{equation}
By definition \cite{aganagic2017quasimap}, the nonabelian stable envelope $\eqref{nonabelian stab}$ is the composition of \eqref{stab part of nonabelian stable envelope} with the pull back with respect to an open inclusion
\[
\BoA_{(\dd,1)}(\overline{Q_{\ff}})/\Gl_{\dd}\hookrightarrow\Nak_{Q}(\ff+\dd,\dd)/\Gl'_{\dd}
\]
that we now define.
Consider the prequotient $(\mu^{\zeta}_{\ff+\dd,\dd})^{-1}(0)$. It can be visualised vertex-wise as follows:
\[
\begin{tikzcd}
\CC^{\dd_i}\arrow[d, shift left=.75ex, "k_{i}"]\arrow[drr, shift left=.75ex, "\chi_i"] & & \\
\CC^{\ff_i}\arrow[u, shift left=.75ex, "j_{i}"] &  & \CC^{\dd_i}.\arrow[ull, shift left=.75ex, "\phi_i"]
\end{tikzcd}
\]
Here, we represent the framing arrows $\lbrace r_{i,m}\rbrace_{m=1,\dots, \ff_i+\dd_i}\in \overline{Q_{\ff}}$ as a map $\CC^{\ff_i}\oplus \CC^{\dd_i}\to \CC^{\dd_i}$ rather than as $\ff_i+\dd_i$ maps from the one dimensional space $\CC$ placed at the vertex $\infty\in (\overline{Q_{\ff}})_0$ to $\CC^{\dd_i}$. Similarly, the dual framing arrows $\lbrace r^*_{i,m}\rbrace_{m=1,\dots, \ff_i+\dd_i}\in \overline{Q_{\ff}}$ are represented as a map $\CC^{\dd_i}\to\CC^{\ff_i}\oplus \CC^{\dd_i}$. The upshot of this notation is that the action of the group $\Gl_{\dd}\times \Gl'_{\dd}\times T_{\ff}$ on the framing arrows is apparent; we let $V'=\CC^{\dd}$ denote the domain of $\phi$ and let $V=\CC^{\dd}$ denote the target, and let $\Gl_{\dd}$ and $\Gl'_{\dd}$ act by change of basis on $V$ and $V'$ respectively.

Let $(\mu^{\iso}_{\ff+\dd,\dd})^{-1}(0)\subset (\mu_{\ff+\dd,\dd})^{-1}(0)$ be the subscheme consisting of representations such that $\phi_i$ is an isomorphism for all $i\in Q_0$. Because of our choice of stability condition $\zeta$, all these representations are stable, hence $(\mu^{\iso}_{\ff+\dd,\dd})^{-1}(0)\subset(\mu^{\zeta}_{\ff+\dd,\dd})^{-1}(0)$.
Now consider the morphism
\[
\BoA_{(\dd,1)}(\overline{Q_{\ff}})\times \Hom(\CC^{\dd},\CC^{\dd})^{\iso}_{\dd}\to (\mu^{\iso}_{\ff+\dd,\dd})^{-1}(0)
\]
sending a representation $(x,x^*, j,k, \phi)$ to $(x,x^*, j,k, \phi, \chi=-\phi^{-1}\cdot \mu_{\ff,\dd}(x,x^*, j,k))$. Note that this map is well defined, because our choice of $\chi$ forces the moment map condition on the target, and it is an isomorphism, with inverse given by the projection $(x,x^*, j,k, \phi, \chi)\mapsto (x,x^*, j,k, \phi)$. Moreover, it is equivariant with respect to $\Gl_{\dd}\times \Gl'_{\dd}\times T_{\ff}$, with $\Gl_{\dd}$ acting by change of basis on both terms of the fiber product and $\Gl'_{\dd}$ acting on the trivial $\Gl'_{\dd}$-torsor $\Hom(V',V)^{\iso}_{\dd}$ from the right.  Hence, taking quotients by $\Gl_{\dd}\times \Gl'_{\dd}$, we get an isomorphism 
\[
\BoA_{(\dd,1)}(\overline{Q_{\ff}})/\Gl_{\dd} \xrightarrow[]{\cong} \Nak^{\iso}_{Q}(\ff+\dd,\dd)/\Gl'_{\dd}
\]
and hence an open immersion 
\begin{equation}
    \label{pull part of nonabelian stable envelope}
    \iota: \BoA_{(\dd,1)}(\overline{Q_{\ff}})/\Gl_{\dd}\hookrightarrow \Nak_{Q}(\ff+\dd,\dd)/\Gl'_{\dd}.
\end{equation}
\begin{definition}\cite[Def.1]{aganagic2017quasimap}
\label{definition: nonabelian stab}
    The nonabelian stable envelope \eqref{nonabelian stab} is the composition $\overline \Psi_{\ff,\dd} := \iota^*\circ \Stab'_{+}$.
\end{definition}
The following result follows from the support property of the stable envelope $\Stab_{+}$.

\begin{proposition}\cite[Prop.1]{aganagic2017quasimap}
\label{proposition support property and normalization of NS}
    Let $\eta: \Nak^{T_{\ff}}_{Q}(\ff,\dd)\hookrightarrow \FM^{T^0_{\ff}}_{(\dd,1)}(\overline{Q_{\ff}})$ be the inclusion induced by diagram \eqref{eq: diagram nonabelian stab}.  The nonabelian stable envelope $\overline \Psi_{\ff,\dd}$ is supported on $\FM^{T^0_{\ff}}_{(\dd,1)}(\Pi_{Q_{\ff}})\subset \FM^{T^0_{\ff}}_{(\dd,1)}(\overline{Q_{\ff}})$. Moreover, we have
    \begin{equation}
    \label{eq: diagonal condition nonabelian stab}
    \eta^*\circ \overline \Psi_{\ff,\dd}= \Eu(\hbar\otimes \Fg_{\dd}) \cap 
    \end{equation}
where $\Fg_{\dd}$ is the tautological bundle on $\Nak_{Q}(\ff,\dd)$ associated to the Lie algebra of $\Gl_{\dd}$.
\end{proposition}

Equation \eqref{eq: diagonal condition nonabelian stab} can be interpreted as the nonabelian analogue of (i) in Theorem \ref{thm: definition stab}.

\subsection{Nonabelian stable envelopes via specialisation}
\label{more_SE_sec}
Motivated by the previous proposition and by the definition of the stable envelope as a correspondence in Borel--Moore homology, we now reinterpret the nonabelian stable envelope $\overline\Psi_{\ff,\dd}$ as the morphism induced by a correspondence 
\[
[\NS]\in \HO^{\BoMo}(\Nak^{T_{\ff}}_{Q}(\ff,\dd)\times_{\B T_{\ff}} \FM^{T^0_{\ff}}_{(\dd,1)}(\Pi_{Q_{\ff}}), \BoQ^{\vir}).
\]
As in the Abelian case reviewed in \S \ref{sec: Stable envelopes via specialisation}, we obtain this correspondence as the specialisation of a simpler class defined on the affine fibers of the universal deformation $\mu: \widetilde\Nak_{Q}(\ff,\dd)\to \AA^{Q_0}$ of the quiver variety $\Nak_{Q}(\ff,\dd)$.
Consider the stacky deformation $\widetilde\FNak_Q(\ff,\dd)\to \AA^{Q_0}$ introduced at the end of \S \ref{section: Deformation of quiver varieties} and the fiber product 
\[
\widetilde\Nak_{Q}(\ff,\dd)\times_{\AA^{Q_0}} \widetilde\FNak_Q(\ff,\dd)\to \AA^{Q_0}.
\]
By Proposition \ref{prop: usual resolution vs stacky resol}, we have a canonical isomorphism $\widetilde\Nak_{Q}(\ff,\dd)|_{U}\cong \widetilde\FNak_Q(\ff,\dd)|_{U}$ and hence, taking quotients by $T_{\ff}$, a flat family of diagonals  
\begin{equation}
    \label{eq:diagonal specializing to NS}
    \Delta_t\subset \Nak^{T_{\ff}}_{Q,t}(\ff,\dd)\times_{\B T_{\ff}}\FM^{\T^0_{\ff}}_{(\dd,1),t}(\Pi_{Q_{\ff}})
\end{equation}
over $t\in U$. 
\begin{definition}
\label{def: NS class}
     The nonabelian stable envelope class $[\NS]$ is the specialisation of $\Delta_t$ to the central fiber $t=0$. In other words, we define $[\NS]=\lim_{t\to 0}\; [\Delta_{t}]$. 
\end{definition}
By definition, $[\NS]$ is a class in 
\[
\HO^{\BoMo}_{T_{\ff}}(\Nak_{Q}(\ff,\dd)\times\FNak_Q(\ff,\dd), \BoQ^{\vir})=\HO^{\BoMo}(\Nak^{T_{\ff}}_{Q}(\ff,\dd)\times_{\B T_{\ff}} \FM^{T^0_{\ff}}_{(\dd,1)}(\Pi_{Q_{\ff}}), \BoQ^{\vir}).
\]
We now state the main result of this subsection.  Consider the diagram
\[
\begin{tikzcd}
    & \Nak^{T_\ff}_{Q}(\ff,\dd)\times_{\B T_{\ff}}\FM^{T^0_\ff}_{(\dd,1)}(\Pi_{Q_{\ff}})  \arrow[d, "\id\times  j"]\arrow[r, " p_2"] \arrow[dl, swap, " p_1"] &  \FM^{T^0_\ff}_{(\dd,1)}(\Pi_{Q_{\ff}}) \arrow[d, "j"]\\
   \Nak^{T_\ff}_{Q}(\ff,\dd) & \Nak^{T_\ff}_{Q}(\ff,\dd)\times_{\B T_{\ff}} \FM^{T^0_\ff}_{(\dd,1)}(\overline{Q_{\ff}}) \arrow[r, "\bar p_2"] \arrow[l, swap, " \bar p_1"] &\FM^{T^0_\ff}_{(\dd,1)}(\overline{Q_{\ff}}).
\end{tikzcd}
\]

\begin{proposition}
\label{prop:nonabelian stable envelope as a correspondence}
The class $[\NS]$ is proper over $\FM^{T^0_\ff}_{(\dd,1)}(\Pi_{Q_{\ff}})\subseteq \FM^{T^0_\ff}_{(\dd,1)}(\overline{Q_{\ff}})$. Moreover, the induced map 
\[
\Psi_{\ff,\dd} : \HO^{\BoMo}(\Nak^{T_\ff}_{Q}(\ff,\dd), \BoQ^{\vir})\to \HO^{\BoMo}(\FM^{T^0_\ff}_{(\dd,1)}(\Pi_{Q_{\ff}}), \BoQ^{\vir})\qquad   \Psi_{\ff,\dd}(\alpha)=(p_2)_{\ast}\left( [\NS]\cap \bar p_1^*(\alpha) \right)
\]
provides the factorization $\overline \Psi_{\ff,\dd}=j_{*}\circ \Psi_{\ff,\dd}$. In other words, the correspondence $[\NS]$ induces the nonabelian stable envelope map of Definition \ref{definition: nonabelian stab}.
\end{proposition}
\begin{remark}
    Notice that the map $\bar p_1$ is smooth, hence the pullback along $\bar p_1$ is well defined. Moreover, the intersection product $[\NS]\cap \bar p_1^*(\alpha)$ takes place in the smooth ambient space $\Nak^{T_\ff}_{Q}(\ff,\dd)\times_{\B T_{\ff}}\FM^{T^0_\ff}_{(\dd,1)}(\overline{Q_{\ff}})$ and, because of the support of $[\NS]$, gives a class in the Borel--Moore homology of
    \[
    \Nak^{T_\ff}_{Q}(\ff,\dd)\times_{\B T_{\ff}} \FM^{T^0_\ff}_{(\dd,1)}(\Pi_{Q_{\ff}})\xhookrightarrow{\id\times j} \Nak^{T_\ff}_{Q}(\ff,\dd)\times_{\B T_{\ff}}\FM^{T^0_\ff}_{(\dd,1)}(\overline{Q_{\ff}}).
    \]
    Hence, the pushforward $(p_2)_*([\NS]\cap \bar p_1^*(\alpha))$ is well defined too.
\end{remark}

\begin{proof}
Recall Definition \ref{definition: nonabelian stab}. The pullback $\iota^*$ induced by the open immersion \eqref{pull part of nonabelian stable envelope} can be seen as the map induced by the correspondence $\Delta^{\iota}$ associated to the diagonal
\[
\Delta\subset \FM^{T^0_{\ff}}_{(\dd,1)}(\overline{Q_{\ff}})\times_{T_{\ff}} \FM^{T^0_{\ff}}_{(\dd,1)}(\overline{Q_{\ff}})\xrightarrow{\iota\times \id} \Nak^{T_{\ff}\times \Gl_{\dd}'}_{Q}(\ff+\dd,\dd)\times_{T_{\ff}} \FM^{T^0_{\ff}}_{(\dd,1)}(\overline{Q_{\ff}}).
\]
Notice that $\Delta^{\iota}$ is proper over $\FM_{(\dd,1)}(\overline{Q_{\ff}})$, so it really gives a well defined map via convolution in Borel--Moore homology. As a result, the nonabelian stable envelope $\overline \Psi_{\ff,\dd}=\iota^*\circ \Stab'_{+}$ is the map associated to the convolution 
\begin{equation}
\label{nonabelian stab as a convolution}
    [\Delta^{\iota}]\ast [\Stab'_{+}]\in  \HO^{\BoMo}(\Nak^{T_{\ff}}_{Q}(\ff,\dd)\times_{\B T_{\ff}} \FM^{T^0_{\ff}}_{(\dd,1)}(\overline{Q_{\ff}}), \BoQ^{\vir}).
\end{equation} 
This convolution product is well defined because the class $[\Delta^{\iota}]$ is proper over $\FM^{T^0_{\ff}}_{(\dd,1)}(\overline{Q_{\ff}})$. Moreover, by Proposition \ref{prop: stab via specialzation} $[\Stab_+]$ is proper over $\Nak^{T_{\ff}}_{Q}(\ff+\dd,\dd)$, hence the cycle \eqref{nonabelian stab as a convolution} is proper over $\FM^{T^0_{\ff}}_{(\dd,1)}(\overline{Q_{\ff}})$.
Consequently, to prove the proposition it suffices to show that the cycle \eqref{nonabelian stab as a convolution} coincides with $[\NS]=\lim_{t\to 0}\; [\Delta_t]$.  Recall that the cycle $[\Stab_{+}]$ is the specialisation to the $t=0$ fiber of a family of fundamental classes $[\Att{+}^t]$, where $t\in U\subset\AA^{Q_0}$ and 
\[
\Att{+}^t\subset \Nak_{Q,t}(\ff,\dd)\times \Nak_{Q,t}(\ff+\dd,\dd)
\]
is the attracting set of the $A\times \Gl_{\dd}'$-fixed component $ \Nak_{Q,t}(\ff,\dd)= \Nak_{Q, t}(\ff,\dd)\times \Nak_{Q,t}(\dd,\ul{0})$ in $ \Nak_{Q,t}(\ff+\dd,\dd)$. Hence, we have 
\[
[\Stab'_{+}]=\lim_{t\to 0}\;[\Att{+}^t/\Gl'_{\dd}]\in \HO^{\BoMo}(\Nak^{T_{\ff}}_{Q}(\ff,\dd)\times_{\B T_{\ff}} \Nak^{T_\ff\times G'_\dd}_Q(\ff+\dd,\dd), \BoQ^{\vir} ).
\]
The class $\Delta^{\iota}$ can be interpreted as a specialisation to the $t=0$ fiber too. Indeed, let $\widetilde\FM_{(\dd,1)}(\overline{Q_{\ff}})$ be the trivial fibration  $\FM_{(\dd,1)}(\overline{Q_{\ff}})\times \AA^{Q_0} \to \AA^{Q_0}$.  Inspecting the definition of $\iota$ from \S \ref{sec: nonabelian stab}, one can check that it naturally extends fiberwise over $\AA^{Q_0}$ to an open immersion
\[
\tilde\iota: \widetilde{\FM}^{T^0_{\ff}}_{(\dd,1)}(\overline{Q_{\ff}})\hookrightarrow \widetilde{\Nak}^{T_{\ff}\times \Gl'_{\dd}}_{Q}(\ff+\dd,\dd).
\]
This immersion gives a flat family of cycles $[\Delta^\iota_t]$ specialising for $t=0$ to $[\Delta^\iota]$. Since the convolution product is compatible with specialisation, see \cite[Prop.7.2.23]{chriss2009representation}, we deduce that $\overline \Psi_{\ff,\dd}$ is induced by the class
\[
\left(\lim_{t\to 0}\;[\Delta^\iota_t]\right) \ast \left(\lim_{t\to 0}\;  [\Att{+}^t/\Gl'_{\dd}]\right) =\lim_{t\to 0}\left( [\Delta^\iota_t]\ast  [\Att{C_+}^t/\Gl'_{\dd}]\right).
\]
However, the class $[\Delta_{\iota,t}]\ast [\Att{+}^t/\Gl'_{\dd}]$ is equal to the restriction of the fundamental class of
\[
\Att{C_+}^t/ \Gl'_{\dd}\subset  \Nak_{Q,t}(\ff,\dd)\times \Nak_{Q,t}(\ff+\dd,\dd)/\Gl'_{\dd}
\]
to the open substack 
\[
\id \times \iota_t: \Nak_{Q,t}(\ff,\dd)\times \BoA_{(\dd,1)}(\overline{Q_{\ff}})/\Gl_{\dd}\hookrightarrow \Nak_{Q,t}(\ff,\dd)\times \Nak_{Q,t}(\ff+\dd,\dd)/\Gl'_{\dd},
\]
where $\iota_t$ is the fiber of $\tilde \iota$ over $t\in U\subset\AA^{Q_0}$. Hence, to prove the proposition, it suffices to show that $(\id\times \iota_t)^{-1}\Att{C_+}^t/ \Gl'_{\dd}= \Delta_t$ for all $t\in U$. First, recall from \S \ref{section: Deformation of quiver varieties} that for all $t\in U$ every representation is semistable, hence $\Nak_{Q,t}(\ff,\dd)=\FNak_{Q,t}(\ff,\dd)\subset\BoA_{(\dd,1)}(\overline{Q_{\ff}})/\Gl_{\dd}$ for all $\dd,\ff$.   Now let $V=\CC^{\dd}$ (resp. $V'=\CC^{\dd}$) be the fundamental representation of $\Gl_{\dd}$ (resp. of $\Gl_{\dd}'$) and consider the following maps
\[
\begin{tikzcd}
    ((\mu_{\ff,\dd})^{-1}(t)\times\Hom(V',V))/\Gl_{\dd}\arrow[d, "p_1"] \arrow[rd, dashed, hookrightarrow, "k"] \arrow[r, hookrightarrow] & (\BoA_{(\dd,1)}(\overline{Q_{\ff}})\times T^*\Hom(V',V))/\Gl_{\dd}\\
    \Nak_{Q,t}(\ff,\dd)=(\mu_{\ff,\dd})^{-1}(t)/\Gl_{\dd}\arrow[r, hookrightarrow] & \Nak_{Q,t}(\ff+\dd,\dd).\arrow[u, hookrightarrow]
\end{tikzcd}
\]
It is easy to check that there is a unique choice of $k$ making the upper triangle commute.  For all $t\in U$, the image of $p_1\times k$ is a closed subscheme contained in the attracting set $\Att{+}^t$. Since they have the same dimension and $\Att{+}^t$ is irreducible, it follows that $\Att{+}^t$ coincides with the image of $p_1\times k$. Hence, $\Att{+}^t/\Gl'_{\dd}$ is the image of $p_1\times (q\circ k)$, where 
\[
q: \Nak_{Q,t}(\ff+\dd,\dd)\to \Nak_{Q,t}(\ff+\dd,\dd)/\Gl'_{\dd}
\]
is the quotient map. 
To conclude the proof, consider the commutative diagram
\[
\begin{tikzcd}
((\mu_{\ff,\dd})^{-1}(t)\times\Hom(V',V))/\Gl_{\dd}  \arrow[r, hookrightarrow,  " k"]\arrow[d, "q"] & \Nak_{Q,t}(\ff+\dd,\dd)\arrow[d, "q"]\\
((\mu_{\ff,\dd})^{-1}(t)\times\Hom(V',V))/\Gl_{\dd}\times \Gl'_{\dd}\arrow[r, hookrightarrow, "\bar k"]\arrow[dd, bend right=82, shift right=5, swap, "\bar p_1"]& \Nak_{Q,t}(\ff+\dd,\dd)/\Gl'_{\dd}\\
((\mu_{\ff,\dd})^{-1}(t)\times\Hom(V',V)^{\iso})/\Gl_{\dd}\times \Gl'_{\dd}\arrow[u, hookrightarrow]\arrow[r, hookrightarrow] \arrow[d, shift right=2, swap, "\bar p_1^{\iso}"] & \Nak^{\iso}_{Q,t}(\ff+\dd,\dd)/\Gl'_{\dd} \arrow[u, hookrightarrow]\\
(\mu_{\ff,\dd})^{-1}(t)/\Gl_{\dd}=\Nak_{Q,t}(\ff,\dd)\arrow[r, hookrightarrow]\arrow[u, shift right=2, "\cong" labl2] & \BoA_{(\dd,1)}(\overline{Q_{\ff}})/\Gl_{\dd}. \arrow[u, "\cong" labl] \arrow[uu, bend right, swap, shift right=15, "\iota_t"]
\end{tikzcd}
\]
All the vertical  maps on the left side and $k$ are already defined, $\bar p_1$ is the projection to the first factor and $\bar p_1^{\iso}$ its restriction to the iso-locus. All the remaining maps are obtained by taking Cartesian squares. Since $p_1$ is $\Gl'_{\dd}$-invariant, we have $p_1=\bar p_1\circ q$. Consequently, $\Att{+}^t/\Gl'_{\dd}$ is the image of $\bar p_1\times \bar k$, and hence its inverse image by $\id\times \iota_t$ is exactly the image of $\Delta \circ \bar p_1^{\iso}$, which is nothing but the diagonal
\[
\Delta_t\subset \Nak_{Q,t}(\ff,\dd)\times  \Nak_{Q,t}(\ff,\dd)\subset \Nak_{Q,t}(\ff,\dd)\times \BoA_{(\dd,1)}(\overline{Q_{\ff}})/\Gl_{\dd}
\]
This concludes the proof.
\end{proof}

\begin{remark}
    \label{remark NS and stability conditions} Let $\zeta$ and $\zeta'$ be two generic stability conditions for a quiver variety and let $\Psi^{\zeta}_{\ff,\dd}$ and $\Psi^{\zeta'}_{\ff,\dd}$ be the corresponding nonabelian stable envelopes. Let also $\Xi_{\ff,\dd}$ be the isomorphism from Remark \ref{rem: Yangians for different stability conditions}. Then the diagram 
    \[
    \begin{tikzcd}
        \HO^{\BoMo}(\Nak^{\zeta, T_\ff}_{Q}(\ff,\dd), \BoQ^{\vir})\arrow[r, "\Psi^{\zeta}_{\ff,\dd}"]\arrow[d, swap, "\Xi_{\ff,\dd}"] & \HO^{\BoMo}(\FM^{T^0_\ff}_{(\dd,1)}(\Pi_{Q_{\ff}})\BoQ^{\vir})\\
        \HO^{\BoMo}(\Nak^{\zeta', T_\ff}_{Q}(\ff,\dd), \BoQ^{\vir})\arrow[ur, swap, "\Psi^{\zeta'}_{\ff,\dd}"]
    \end{tikzcd}
    \]
    is commutative. This once again follows from the description of $\Xi_{\ff,\dd}$, $\Psi^{\zeta}_{\ff,\dd}$, and $\Psi^{\zeta'}_{\ff,\dd}$ as specialization of fundamental classes in Borel-Moore homology and compatibility of the latter with the convolution product.
\end{remark}

We finish the section with an easy but important corollary.

\begin{corollary}
    \label{corollary normalization nonabelian stable envelope}
    Let $i: \Nak^{T_{\ff}}_Q(\ff,\dd)\hookrightarrow  \FM^{T_{\ff_0}}_{(\dd,1)}(\Pi_{Q_{\ff}})$ be the natural open inclusion. We have $i^*\circ \Psi_{\ff,\dd}=\id$. In particular, $\Psi_{\ff,\dd}$ is injective.
\end{corollary}

\begin{proof}
    As before, we first show that the statement is true on the affine fibers of the deformation space and then we deduce the result by specialising to the central fiber. As a correspondence, the smooth pullback $i^*$ is given by the diagonal
    \[
    \Delta \subset \Nak^{T_{\ff}}_Q(\ff,\dd)\times_{\B T_{\ff}}\Nak^{T_{\ff}}_Q(\ff,\dd) \subset \Nak^{T_{\ff}}_Q(\ff,\dd) \times_{\B T_{\ff}}  \FM^{T_{\ff_0}}_{(\dd,1)}(\Pi_{Q_{\ff}}).
    \]
    As in the previous proof, this class is a specialisation at $t=0$ of a family of correspondences on the affine fibers, i.e. for $t\in U$. Since on the affine fibers all the representations are stable, this family of correspondences is given by $\Delta^{(21)}_t$, where $\Delta_t$ is the diagonal introduced in \eqref{eq:diagonal specializing to NS} and the superscript stands for the fact that the correspondence is meant to be read backwards (although this is purely formal since $\Delta_t$ becomes symmetric after the identification $\Nak^{T_{\ff}}_{Q,t}(\ff,\dd) = \FM^{T_{\ff_0}}_{(\dd,1),t}(\Pi_{Q_{\ff}})$). But then the result is trivial on the affine fibers and hence it also holds after specialising.
\end{proof}

\begin{remark}
    Notice that Proposition \ref{proposition support property and normalization of NS} follows at once by Proposition \ref{prop:nonabelian stable envelope as a correspondence} and Corollary \ref{corollary normalization nonabelian stable envelope}.
\end{remark}

\subsection{From stable envelopes to the preprojective CoHA}
\label{SE_to_COHA}
Throughout this section, we assume that the torus $T$ is associated to the restriction to $\overline{Q_{\ff}}\subset \tilde Q_{\ff}$ of a weighting $\wt: (\tilde Q_{\ff})_1\to \ZZ^r$ satisfying Assumption \ref{weighting_assumption}.
Fix a quiver $Q$, a framing vector $\ff$, and a nontrivial decomposition $\ff=\ff'+\ff''$. 
As in \S \ref{sect: R-matrices} this data defines a two dimensional subtorus $A\subset A_{\ff}$. The action of $A$ on $\Nak_Q(\ff,\dd)$ admits two chambers $\mathfrak{C}_{\pm}$ and hence induces two stable envelopes $\Stab_{\pm}$. We first focus on
\[
\Stab_+: \NakMod^{T_{\ff'}}_{Q,\ff', \dd'}\otimes_{\HO_{T}}\NakMod^{T_{\ff''}}_{Q,\ff'', \dd''}\to \NakMod^{T_{\ff}}_{Q,\ff,\dd}.
\]
These morphisms equip the $\BoN^{Q_0}\times\BoN^{Q_0}$-graded, cohomologically graded vector space
\[
\NakMod^{T_{\bullet}}_Q=\bigoplus_{\dd,\ff} \NakMod^{T_{\ff}}_{Q,\ff,\dd}
\]
with a graded $\HO_T$-algebra structure. The associativity of this multiplication defined via stable envelopes follows from Lemma \ref{lma:triangle lemma}. In the next section, we will compare this algebra with the following special type of CoHA. Consider the $\BoN^{Q_0}\times\BoN^{Q_0}$-graded, cohomologically graded vector space $\HO\!\CoHA_{\Pi_{Q_{\fr}}}^{T_{\bullet}}$ defined as
\[
\HO\!\CoHA_{\Pi_{Q_{\fr}}, \ff,\dd}^{T_{\bullet}}= \HO^{\BoMo}(\FM^{T^0_\ff}_{(\dd,1)}(\Pi_{Q_\ff}),\BoQ^{\vir}) \qquad \HO\!\CoHA_{\Pi_{Q_{\fr}}}^{T_{\bullet}}= \bigoplus_{\dd,\ff\in\BoN^{Q_0}} \HO\!\CoHA_{\Pi_{Q_{\fr}}, \ff,\dd}^{T_{\bullet}}.
\]

We now define a graded algebra structure on $\HO\!\CoHA_{\Pi_{Q_{\fr}}}^{T_{\bullet}}$. Recall the definition of the quiver $Q_{\ff',\ff''}$ from \S \ref{quiver_sec} and set $T^0_{\ff',\ff''}:=T^0_{\ff'}\times_T T^0_{\ff''}$. This torus acts naturally on the moduli space $\FM_{(\dd,1,1)} (\Pi_{Q_{\ff',\ff''}})$ for any $\dd\in \NN^{Q_0}$.  The  inclusions $Q_{\ff'}\hookrightarrow Q_{\ff',\ff''}$ and $Q_{\ff''}\hookrightarrow Q_{\ff',\ff''}$ and the identification $\infty'=\infty''$ induce, respectively, canonical isomorphisms
\[
\FM^{_{T^0_{\ff'}}}_{(\dd',1)}(\Pi_{Q_{\ff'}})\times \B A^{0}_{\ff''}\cong \FM^{T^0_{\ff',\ff''}}_{(\dd',1,0)}(\Pi_{Q_{\ff',\ff''}}) \qquad \FM^{T^0_{\ff''}}_{(\dd'',1)}(\Pi_{Q_{\ff''}})\times \B A^{0}_{\ff'}\cong \FM^{T^0_{\ff',\ff''}}_{(\dd'',0,1)}(\Pi_{Q_{\ff',\ff''}}) 
\]
and 
\[
\FM^{T^0_{\ff'+\ff''}}_{(\dd,1)}(\Pi_{Q_{\ff'+\ff''}})\cong \FM^{T^0_{\ff',\ff''}}_{(\dd,1,1)}(\Pi_{Q_{\ff',\ff''}}).
\]
We are now ready to introduce an algebra structure on $\HO\!\CoHA_{\Pi_{Q_{\fr}}}^{T_{\bullet}}$. We define a multiplication map
\[
\vmult^{\ff',\ff''}_{\dd',\dd''}: \HO\!\CoHA_{\Pi_{Q_{\fr}}, \ff', \dd'}^{T_{\bullet}}\otimes_{\HO_{T}} \HO\!\CoHA_{\Pi_{Q_{\fr}}, \ff'', \dd''}^{T_{\bullet}}\to \HO\!\CoHA_{\Pi_{Q_{\fr}}, \dd'+\dd'',\ff'+\ff''}^{T_{\bullet}}
\]
via the preprojective CoHA multiplication
\[
\vmult_{(\dd',1,0),(\dd'',0,1)}:\HO\!\CoHA^{T^0_{\ff',\ff''}}_{\Pi_{Q_{\ff',\ff''}},(\dd',1,0)}\otimes_{\HO_{T^0_{\ff',\ff''}}}\HO\!\CoHA^{T^0_{\ff',\ff''}}_{\Pi_{Q_{\ff',\ff''}},(\dd'',0,1)}\rightarrow \HO\!\CoHA^{T^0_{\ff',\ff''}}_{\Pi_{Q_{\ff',\ff''}},(\dd'+\dd'',1,1)}
\]
defined in \S \ref{subsection:Preprojective CoHA} and the isomorphisms above. Associativity of this multiplication map is a special case of associativity of the preprojective CoHA, cf. \S\ref{subsection:Preprojective CoHA}.

We can now state the main result of this section. Consider the morphism $\Psi: \NakMod^{T_{\bullet}}_Q\to \HO\!\CoHA_{\Pi_{Q_{\fr}}}^{T_{\bullet}}$ obtained by taking the direct sum across all $\dd,\ff\in\BoN^{Q_0}$ of the nonabelian stable envelopes
\[
\NakMod^{T_{\ff}}_{Q,\ff,\dd}=\HO^{\BoMo}(\Nak^{T_{\ff}}_{Q}(\ff,\dd), \BoQ^{\vir})\xrightarrow{\Psi_{\ff,\dd}} \HO^{\BoMo}(\FM^{T^0_{\ff}}_{(\dd,1)}(\Pi_{Q_{\ff}}), \BoQ^{\vir})=\HO\!\CoHA_{\Pi_{Q_{\fr}, \ff,\dd}}^{T_{\bullet}}.
\]
The following result provides a crucial connection between stable envelopes and CoHAs.
\begin{proposition}
\label{mult_compat}
    The map $\Psi: \NakMod^{T_{\bullet}}_Q\to \HO\!\CoHA_{\Pi_{Q_{\fr}}}^{T_{\bullet}}$ is a morphism of graded algebras. In other words, the diagram
    \begin{equation}
    \label{eq:diagram compatibility product}
    \begin{tikzcd}
        \NakMod^{T_{\ff'}}_{Q,\ff', \dd'}\otimes_{\HO_{T}} \NakMod^{T_{\ff''}}_{Q,\ff'',\dd''}\arrow[d, swap, "\Psi_{\ff',\dd'}\otimes_{\HO_T}\Psi_{\ff'',\dd''} "]\arrow[r, "\Stab_+"] & \NakMod^{T_{\ff}}_{Q,\ff,\dd}\arrow[d, "\Psi_{\ff,\dd}"]\\
        \HO\!\CoHA_{\Pi_{Q_{\fr}, \ff', \dd'}}^{T_{\bullet}} \otimes_{\HO_{T}} \HO\!\CoHA_{\Pi_{Q_{\fr}, \ff'', \dd''}}^{T_{\bullet}} \arrow[r, "\vmult^{\ff',\ff''}_{\dd',\dd''}"]& \HO\!\CoHA_{\Pi_{Q_{\fr}, \ff,\dd}}^{T_{\bullet}} 
    \end{tikzcd}
    \end{equation}
    commutes for all decompositions $\dd=\dd'+\dd''$ and $\ff=\ff'+\ff''$.
\end{proposition}
\begin{proof}
    In analogy with $\HO\!\CoHA_{\Pi_{Q_{\fr}}}^{T_{\bullet}}$, we define an $\HO_{T}$-algebra structure on 
    \[
    \HO\!\CoHA_{\tilde{Q}_{\fr}}^{T_{\bullet}}= \bigoplus_{\dd,\ff} \HO\!\CoHA_{\tilde{Q}_{\fr}, \ff,\dd}^{T_\bullet}; 
    \qquad 
    \HO\!\CoHA_{\tilde{Q}_{\fr}, \ff,\dd}^{T_\bullet} = \HO^*(\FM^{T^0_\ff}_{\dd,1}(\tilde{Q}_{\ff}),\BoQ^{\vir}) 
    \]
    via the isomorphisms 
    \[
\FM^{T^0_{\ff'}}_{(\dd',1)}(\tilde{Q}_{\ff'})\times \B A^{0}_{\ff''}\cong \FM^{T^0_{\ff,\ff''}}_{(\dd',1,0)}(\tilde{Q}_{\ff',\ff''}); \qquad \FM^{T^0_{\ff''}}_{(\dd'',1)}(\tilde{Q}_{\ff''})\times \B  A^{0}_{\ff'}\cong \FM^{T^0_{\ff,\ff''}}_{(\dd'',0,1)}(\tilde{Q}_{\ff',\ff''});
\]
\[
\FM^{T^0_{\ff}}_{(\dd,1)}(\tilde{Q}_{\ff})\cong \FM^{T^0_{\ff',\ff''}}_{(\dd,1,1)}(\tilde{Q}_{\ff',\ff''})
\]
and the multiplication map for the CoHA for the quiver $\tilde{Q}_{\ff',\ff''}$ and the trivial potential $W=0$ (see \S \ref{shuffle_sec}):
\[
\vmult_{(\dd',1,0),(\dd'',0,1)}:\HO\!\CoHA^{T^0_{\ff',\ff''}}_{\tilde{Q}_{\ff',\ff''},(\dd',1,0)}\otimes_{\HO_{T^0_{\ff',\ff''}}}\HO\!\CoHA^{T^0_{\ff',\ff''}}_{\tilde{Q}_{\ff',\ff''},(\dd'',0,1)}\rightarrow \HO\!\CoHA^{T^0_{\ff',\ff''}}_{\tilde{Q}_{\ff',\ff''},(\dd'+\dd'',1,1)}.
\]
We may alternatively realise this morphism as $\vmult_{(\dd',1,0),(\dd'',0,1)}$ for the CoHA for the quiver $\tilde{Q}_{\ff}^{\Abel}$, following Remark \ref{fr_from_ab}.

Now consider the composition
\begin{equation}
    \label{eq: nonabelian stab landing in shuffle algebra}
    \begin{tikzcd}
    \NakMod^{T_{\bullet}}_Q\arrow[r, "\Psi"] & \HO\!\CoHA_{\Pi_{Q_{\fr}}}^{T_{\bullet}}\arrow[r, hookrightarrow, "\bar \xi"] &  \HO\!\CoHA_{\tilde{Q}_{\fr}}^{T_{\bullet}}
    \end{tikzcd}
\end{equation}
where $\bar\xi$ is the framed version of the map from Proposition \ref{prop: embedding preprojective CoHA in shuffle algebra}. Since the action of the torus $T$ satisfies the hypotheses of Assumption \ref{weighting_assumption}, the map $\bar\xi$ is an embedding of algebras, by Proposition \ref{shuff_embed}. Hence, it suffices to show that the composition $\bar\xi\circ \Psi$ is a morphism of graded algebras. Focusing on a single graded component $\bar\xi_{\ff,\dd}\circ \Psi_{\ff,\dd}$, unraveling the definition of $\bar\xi$ from \S \ref{sec: From critical CoHA to shuffle algebra}, and applying Proposition \ref{prop:nonabelian stable envelope as a correspondence} we get the commutative diagram:
\[
\begin{tikzcd}
    \HO^{\BoMo}(\Nak^{T_\ff}_Q(\ff,\dd), \BoQ^{\vir})\arrow[r, "\Psi_{\ff,\dd}"]\arrow[dr, swap, "\overline \Psi_{\ff,\dd}"] &  \HO^{\BoMo}(\FM^{T^0_\ff}_{(\dd,1)}(\Pi_{Q_\ff}), \BoQ^{\vir}) \arrow[d, "j_*"]\arrow[r, "\bar\xi_{\ff,\dd}"] &\HO^*(\FM^{T^0_\ff}_{(\dd,1)}(\tilde{Q}_{\ff}), \BoQ^{\vir}) \\
    & \HO^*(\FM^{T^0_\ff}_{(\dd,1)}(\overline{Q_{\ff}}) ,\BoQ^{\vir}[\rk(\pi_{\ff,\dd})]) \arrow[ur, swap, "(\pi_{\ff,\dd})^*"]&
\end{tikzcd}
\]
where $j_*$ is the proper pushforward associated to the closed embedding $j: \FM_{(\dd,1)}(\Pi_{Q_\ff})\to\FM_{(\dd,1)}(\overline{Q_{\ff}})$ and $(\pi_{\ff,\dd})^*$ is the smooth pullback associated to the affine fibration $\pi_{\ff,\dd}:\FM_{(\dd,1)}(\tilde{Q}_{\ff})\to \FM_{(\dd,1)}(\overline{Q_{\ff}})$.
Then the result follows by \cite[Thm.5.6]{Botta_2023}, where the composition $(\pi_{\ff,\dd})^*\circ \overline\Psi_{\ff,\dd}$ is shown to be a morphism of algebras.  The proof follows.

We now provide an alternative, more direct, proof by exploiting the fact that both $\overline \Psi_{\ff, \dd}$ and the stable envelope $\Stab_+$ are defined via specialisation of easier classes defined on the deformation spaces introduced in \S \ref{section: Deformation of quiver varieties}.  Firstly, notice that $\pi_{\ff,\dd}$ is induced by the correspondence
\[
\Gamma_{\ff,\dd}:=\text{Graph}(\pi_{\ff,\dd})\subset \FM^{T^0_\ff}_{(\dd,1)}(\overline{Q_{\ff}})\times_{\B T^0_{\ff}}\FM^{T^0_\ff}_{(\dd,1)}(\tilde{Q}_{\ff}).
\]
Hence, the composition $(\pi_{\ff,\dd})^*\circ \overline \Psi_{\ff,\dd}\circ \Stab_+$ is induced by the convolution of cycles
\begin{equation}
	\label{eq:first triple convolution}
	[\Gamma_{\ff,\dd}]\ast [\NS] \ast [\Stab_+].
\end{equation}
On the other hand, the composition $\vmult^{\ff',\ff''}_{\dd',\dd''}\circ ((\pi_{\ff',\dd'})^*\otimes_{\HO_T} (\pi_{\ff'',\dd''})^*) \circ (\overline \Psi_{\ff',\dd'}\otimes_{\HO_T}  \overline \Psi_{\ff'',\dd''})$ is also induced by the correspondence
\begin{equation}
	\label{eq:second triple convolution}
	[\FA]\ast ([\Gamma_{\ff',\dd'}]\otimes_{\HO_T} [\Gamma_{\ff'',\dd''}]) \ast ([\NS]\otimes [\NS]),
\end{equation}
where $[\FA]$ is the image of the fundamental class of $\FM_{(\dd',1,0),(\dd'',0,1)}(\tilde{Q}_{\ff',\ff''})$ under the proper morphism
\[
\FM_{(\dd',1,0),(\dd'',0,1)}(\tilde{Q}_{\ff',\ff''})\xrightarrow{(q_1\times q_3)\times q_2}  \left(\FM_{(\dd',1,0)}(\tilde{Q}_{\ff',\ff''})\times \FM_{(\dd',1,0)}(\tilde{Q}_{\ff',\ff''})\right)\times \FM_{(\dd,1,1)}(\tilde{Q}_{\ff',\ff''})
\] 
introduced in \S \ref{gen_CoHA_sec} giving rise to the CoHA multiplication $\vmult^{\ff',\ff''}_{\dd',\dd''}$ for the tripled doubly framed quiver without potential.
As a result, to prove that $\pi_{\ff,\dd}^*\circ \overline\Psi_{\ff,\dd}$ is compatible with the algebra structures it suffices to show that the cycles \eqref{eq:first triple convolution} and \eqref{eq:second triple convolution} coincide in
\begin{equation*}
	\HO^{\BoMo}\left((\Nak^{T_{\ff'}}_Q(\ff',\dd')\times_{\B T}\Nak^{T_{\ff''}}_Q(\ff'',\dd'') )\times_{\B T_{\ff}}\FM^{T^0_{\ff}}_{(\dd,1)}(\tilde{Q}_{\ff}), \BoQ\right).
\end{equation*}

As in the proof of Proposition \ref{prop:nonabelian stable envelope as a correspondence}, we first work on the affine fibers of the deformation diagram of a quiver variety and then we deduce the claim by specialising to the central fiber. All the relevant classes admit deformations for $t\in \AA^{Q_0}$. The deformation of $[\Stab_{+}]$, introduced in \S \ref{sec: Stable envelopes via specialisation}, is $(-1)^{\dd'\cdot \dd''}$ times the fundamental class of the attracting set
\[
\Att{+}^t\subset \left(\Nak_{Q,t}(\ff',\dd')\times\Nak_{Q,t}(\ff'',\dd'')\right)\times \Nak_{Q,t}(\ff,\dd).
\]
The sign twist is due to our choice of polarisation in \S \ref{section: Stable envelopes for quiver varieties}. Over the affine fibers $t\in U$, the nonabelian stable envelope is $[\Delta_{t}]$, the fundamental class of the diagonal, i.e. the identity element in the convolution ring. Finally, we set $\FA_t$ and $\Gamma^t_{\ff,\dd}$ to be the trivial deformations. By compatibility of specialisation with convolution product and our analysis, we deduce that it suffices to prove that 
\begin{equation}
\label{eq: final statement alternative proof comparison}
    (-1)^{\dd'\cdot \dd''} [\Gamma^t_{\ff,\dd}]\ast [\Att{+}^t]=[\FA_{t}]\ast ([\Gamma_{\ff',\dd'}^t]\otimes_{\HO_T} [\Gamma_{\ff'',\dd''}^t])
\end{equation}
over the affine fibers $t\in U$. All the classes in the last equation can be made completely explicit. To describe them, consider the moment map
\[
\mu: \AA_{(\dd',1,0), (\dd'',0,1)}(\overline{Q_{\ff',\ff''}})\to \fg_{(\dd',1,0), (\dd'',0,1)}\to \fg_{\dd',\dd''}.
\]
 Arguing as at the end of the proof of Proposition \ref{prop:nonabelian stable envelope as a correspondence}, one identifies the attracting set $\Att{+}^t$ with $\mu^{-1}(t)/\Gl_{\dd',\dd''}$. By standard transversality arguments, it follows that the left hand side of \eqref{eq: final statement alternative proof comparison} is the fundamental class of 
\[
\widetilde V_1=\left(\mu^{-1}(t)\times \fg_{(\dd,1,1)}\right)/(\Gl_{(\dd',1,0),(\dd'',0,1)}\times T_{\ff}).
\]
Similarly, the right hand side is the fundamental class of
\[
\widetilde V_2=\left(\mu^{-1}(t+\fn)\times \fg_{(\dd',1,0),(\dd'',0,1)}\right)/(\Gl_{(\dd',1,0),(\dd'',0,1)}\times T_{\ff}),
\]
where $\fn$ is such that $\fg_{\dd',\dd''}=\fn\oplus \fg_{\dd'}\oplus \fg_{\dd''}$. To complete the proof, it suffices to show that $[\widetilde V_1]=(-1)^{\dd'\cdot \dd''}[\widetilde V_2]$ as equivariant cycles in $\HO^{\BoMo}((\mu^{-1}(t+\fn)\times \fg_{(\dd,1,1)})/(\Gl_{(\dd',1,0),(\dd'',0,1)}\times T_{\ff}))$.

Set $V=\mu^{-1}(t+\fn)\times \fg_{\dd}$ and consider its $T_{\ff}\times \Gl_{\dd',\dd''}$-equivariant subvarieties 
\[
V_1=\mu^{-1}(t)\times \fg_{\dd} \qquad V_2=\mu^{-1}(t+\fn)\times \fg_{\dd',\dd''}.
\]
Since we are working on the affine fibers of the moment map, $V$, $V_1$ and $V_2$ are smooth. Moreover, from the identities $\fg_{(\dd,1,1)}=\fg_{\dd}\times \fg_{(0,1,1)}$, $\Gl_{(\dd',1,0),(\dd'',0,1)}=\Gl_{\dd',\dd''}\times \Gl_{(0,1,1)}$ and $\fg_{(\dd',1,0),(\dd'',0,1)}= \fg_{\dd',\dd''}\times \fg_{(0,1,1)}$
it follows that 
\begin{align*}
        \widetilde V_1 & =(V_1/(\Gl_{\dd',\dd''}\times T_{\ff}))\times_{\B T_{\ff}} (\fg_{(0,1,1)}/ (\Gl_{(0,1,1)}\times T_{\ff} ))\\
         \widetilde V_2 & =(V_2/(\Gl_{\dd',\dd''}\times T_{\ff}))\times_{\B T_{\ff}} (\fg_{(0,1,1)}/ (\Gl_{(0,1,1)}\times T_{\ff} )).
\end{align*}
Hence, it suffices to show that $[V_1]=(-1)^{\dd'\cdot \dd''}[V_2]$ as equivariant cycles in $\HO^{\BoMo}_{T_{\ff}\times \Gl_{\dd',\dd''}}(V, \BoQ)$. Since the group homomorphism $\Gl_{\dd',\dd''}\to\Gl_{\dd'}\times \Gl_{\dd''}$ is an affine fibration, it is enough to prove the equality in $\HO^{\BoMo}_{T_{\ff}\times \Gl_{\dd'}\times \Gl_{\dd''}}(V, \BoQ)$. 
The $T_{\ff}\times \Gl_{\dd'}\times \Gl_{\dd''}$-equivariant projections 
\[
\mu^{-1}(t+\fn)\rightarrow (\mu_{\ff',\dd'}\times\mu_{\ff'',\dd''})^{-1}(t)\qquad \fg_{\dd}\rightarrow \fg_{\dd',\dd''}
\]
make the variety $V$ the total space of a $T_{\ff}\times \Gl_{\dd'}\times \Gl_{\dd''}$-equivariant vector bundle $p: V\to Z$.  As a consequence, we have that 
\[
[V_i]=p^*\left(\Eu(V/{V_i})\right)\cdot [V]\in \HO^{\BoMo}_{T_{\ff}\times \Gl_{\dd'}\times \Gl_{\dd''}}(V, \BoQ)
\]
for $i=1,2$. Here, $\Eu(V/{V_i})$ denotes the Euler class of the vector bundle $V/V_i$ and the dot product indicates the natural action of the cohomology of $V$ on its Borel--Moore homology \cite[\S 2.6.40]{chriss2009representation}.

To complete the proof, it suffices to show that $\Eu(V/V_1)=(-1)^{\dd'\cdot \dd''} \Eu(V/V_2)$. Indeed, the K-theory class of $V/V_1$ coincides with the class of the trivial bundle with fiber $\fn$. On the other hand, $V/V_2$ is the trivial bundle with fiber $\fg_{\dd}/\fg_{\dd',\dd''}=\fn^\vee$. As the notation suggests, $\fn$ and $\fn^\vee$ are dual $T_{\ff}\times \Gl_{\dd'}\times \Gl_{\dd''}$-representations of dimension $\dd'\cdot \dd''$. Therefore, their Euler classes exactly differ by the overall sign $(-1)^{\dd'\cdot \dd''}$.
\end{proof}

Let $\sw_{\tau}: V\otimes W\to W\otimes V$ be the morphism sending $v\otimes w\mapsto (-1)^{\chi}w\otimes v$, where we assume that $V,W$ are of $\BoN^{(Q_{\ff,\ff'})_0}$-degree $\ee$ and $\ee'$ respectively, and abbreviate $\chi=\chi_{\tilde{Q}_{\ff',\ff''}}(\ee,\ee')$.

The following corollary gives the proposition above for the other possible choice of chamber.
\begin{corollary}
\label{prod_compat_cor}
    The diagram 
    \begin{equation*}
    \begin{tikzcd}[ column sep= 2 cm]
        \NakMod^{T_{\ff'}}_{Q,\ff', \dd'}\otimes_{\HO_{T}} \NakMod^{T_{\ff''}}_{Q,\ff,\dd}\arrow[d, swap, "\Psi_{\ff',\dd'}\otimes_{\HO_T}\Psi_{\ff'',\dd''} "]\arrow[r, "\Stab_-"] & \NakMod^{T_{\ff}}_{Q,\ff,\dd}\arrow[d, "\Psi_{\ff,\dd}"]\\
        \HO\!\CoHA_{\Pi_{Q_{\fr}, \ff', \dd'}}^{T_{\bullet}} \otimes_{\HO_{T}} \HO\!\CoHA_{\Pi_{Q_{\fr}, \ff'', \dd''}}^{T_{\bullet}} \arrow[r, "\vmult^{\ff'',\ff'}_{\dd'',\dd'}\circ \sw_{\tau}"]& \HO\!\CoHA_{\Pi_{Q_{\fr}, \ff,\dd}}^{T_{\bullet}} 
    \end{tikzcd}
    \end{equation*}
    commutes for all decompositions $\dd=\dd'+\dd''$ and $\ff=\ff'+\ff''$.
\end{corollary}

\begin{proof}
    The proof is exactly the same as that of Proposition \ref{mult_compat}. The additional sign $(-1)^{\chi}$ encoded in the map $\sw_{\tau}$ compensates for the sign of $\Stab_-$, due to our choice of polarisation. Indeed, it is easy to check that $\chi=\frac{1}{2} \codim_X(F) \mod 2$, where $X=\Nak_Q(\ff,\dd)$ and $F=\Nak_Q(\ff',\dd')\times \Nak_Q(\ff'',\dd'')$. Hence, the proof follows from Remark \ref{rem: polarisation opposite chamber}.
\end{proof}

\section{Proof of the main results}
\subsection{Nonabelian stable envelopes and perverse filtrations}
\label{NASE_perv_sec}
We fix throughout this section a quiver $Q$.  We fix a weighting $\wt\colon \tilde{Q}_1\rightarrow \BoZ^r$ satisfying Assumption \ref{weighting_assumption}.  Until Corollary \ref{ThmC_cor2} we will consider degenerate stability conditions; on $Q$ we consider the stability condition $\zeta_{\mathrm{deg}}=\underline{0}$, so that all modules are semistable, and we extend this by considering $\zeta^+_{\mathrm{deg}}\in\BoN^{(Q_{\ff})_0}$ as in the definition of Nakajima quiver varieties given in \S\ref{sec: Nakajima quiver varieties}.

Given a framing dimension vector $\ff\in\BoN^{Q_0}$ we define the extended torus $T_{\ff}=T\times A_{\ff}$ as at the beginning of \S \ref{section: Stable envelopes for quiver varieties}, to which we refer also for the definitions of $A_{\triv}\cong \BoC^*$, $A^0_{\ff}=A_{\ff}/\BoC^*$ and $T_{\ff}^0=T\times A_{\ff}^0$.  The morphism $\pi_{\ff}\colon \Nak^{T_{\ff}}_Q(\ff,\dd)\rightarrow \CM^{T_{\ff}}_{(\dd,1)}(\Pi_{Q_{\ff}})$ is projective, with smooth domain, so that we may apply the Beilinson--Bernstein--Deligne--Gabber decomposition theorem \cite{BBD}.  Furthermore, by \cite[Cor.6.11]{Nak94} the morphism $\pi_{\ff}$ is semismall, meaning that ${}^{\Fp'}\!\!\Ho^{i}\!\left((\pi_{\ff})_{*}\BoQ^{\vir}_{\Nak^{T_{\ff}}_Q(\ff,\dd)}\right)=0$ for $i\neq 0$.  Putting this together, there is a decomposition into simple perverse sheaves:
\begin{equation}
\label{Nak_decomp}
(\pi_{\ff})_*\BoQ^{\vir}_{\Nak^{T_{\ff}}_Q(\ff,\dd)}\cong \bigoplus_{s\in J}\ICS_{\overline{Z_s}/T_{\ff}}(\CL_s).
\end{equation}
In the above decomposition, $Z_s\subset \CM_{(\dd,1)}(\Pi_{Q_{\ff}})$ are $T_{\ff}$-invariant locally closed smooth subvarieties, and each $\CL_s$ is a (shifted) simple local system on $Z_s$, while $\ICS_{\overline{Z_s}/T_{\ff}}(\CL_s)$ is the intermediate extension to $\overline{Z_s}/T_{\ff}$, considered as a perverse sheaf on $\CM^{T_{\ff}}_{(\dd,1)}(\Pi_{Q_{\ff}})$ via direct image along the closed embedding $\ol{Z_s}/T_{\ff} \hookrightarrow \CM^{T_{\ff}}_{(\dd,1)}(\Pi_{Q_{\ff}})$.  If $(\dd,1)\in\Phi^+_{Q_{\ff}}$ then the summand $\ICS(\CM^{T_{\ff}}_{(\dd,1)}(\Pi_Q))$ occurs with multiplicity one in the decomposition \eqref{Nak_decomp}, and is the only summand with full support.  

We define $T^+\coloneqq T\times A_{\triv}$.  Note that $A_{\triv}$ acts trivially on $\Nak_Q(\ff,\dd)$ and $\CM_{(\dd,1)}(\Pi_Q)$, and so there are isomorphisms 
\begin{equation*}
\Nak^{T_{\ff}}_Q(\ff,\dd)\cong \Nak^{T^0_{\ff}}_Q(\ff,\dd)\times \B A_{\triv};\quad\quad \varpi \colon \CM^{T_{\ff}}_{(\dd,1)}(\Pi_{Q_{\ff}})\cong \CM^{T^0_{\ff}}_{(\dd,1)}(\Pi_{Q_{\ff}})\times \B A_{\triv}
\end{equation*}
depending on a splitting $A_{\ff}\cong A_{\ff}^0\times A_{\triv}$.  We may alternatively write $\pi_{\ff}=\pi_{\ff}^0\times \id_{\B A_{\triv}}$, with $\pi_{\ff}^0\colon \Nak^{T^0_{\ff}}_Q(\ff,\dd)\rightarrow \CM^{T^0_{\ff}}_{(\dd,1)}(\Pi_{Q_{\ff}})$ the natural projective morphism.  If we denote by $\pi^+\colon \Nak^{T^+}_Q(\ff,\dd)\rightarrow \CM^{T^+}_{(\dd,1)}(\Pi_{Q_{\ff}})$ the morphism obtained after passing to $T^+$-quotients, we obtain the decomposition $\pi^+_*\BoQ^{\vir}_{\Nak^{T^+}_Q(\ff,\dd)}\cong \bigoplus_{s\in J}\ICS_{\overline{Z_s}/T^+}(\CL_s)$ analogous to \eqref{Nak_decomp}.

We denote $\kappa_{\ff}\colon \CM^{T_{\ff}}_{(\dd,1)}(\Pi_{Q_{\ff}})\rightarrow \CM^{T^0_{\ff}}_{(\dd,1)}(\Pi_{Q_{\ff}})$ the projection induced by the projection of groups $T_{\ff}\twoheadrightarrow T_{\ff}^0$, and $\kappa\colon \CM^{T^+}_{(\dd,1)}(\Pi_{Q_{\ff}})\rightarrow \CM^{T}_{(\dd,1)}(\Pi_{Q_{\ff}})$ the projection induced by the projection $T^+\twoheadrightarrow T$.  Expressed in terms of the isomorphism $\varpi$, the morphism $\kappa_{\ff}$ is simply the projection onto the first factor on the right hand side.  In particular, there is an isomorphism of split complexes of perverse sheaves
\begin{equation}
\label{can_split}
(\kappa_{\ff})_*(\pi_{\ff})_*\BoQ_{\Nak^{T_{\ff}}_Q(\ff,\dd)}^{\vir}\cong (\pi^0_{\ff})_*\BoQ_{\Nak^{T_{\ff}^0}_Q(\ff,\dd)}^{\vir}\otimes\HO_{A_{\triv}}
\end{equation}
and an isomorphism
\begin{equation}
\label{can_GS_split}
\HO^*\!\left(\Nak^{T_{\ff}}_Q(\ff,\dd),\BoQ^{\vir}\right)\cong \HO^*\!\left(\Nak^{T^0_{\ff}}_Q(\ff,\dd),\BoQ^{\vir}\right)\otimes\HO_{A_{\triv}}.
\end{equation}
\begin{remark}
\label{can_emb_rem}
We frequently consider $\HO^*\!\left(\Nak^{T^0_{\ff}}_Q(\ff,\dd),\BoQ^{\vir}\right)$ as a sub-object of $\HO^*\!\left(\Nak^{T_{\ff}}_Q(\ff,\dd),\BoQ^{\vir}\right)$ via the embedding sending $\alpha\mapsto \alpha\otimes 1$, where the target is considered as an element of the target of \eqref{can_GS_split}.  Geometrically, we have seen above that this is the inclusion of the zeroth piece of the perverse filtration defined with respect to the morphism $\kappa_{\ff}\circ\pi_{\ff}$, or pull back along $\kappa_{\ff}$.  In particular, it does not depend on the splitting $A_{\ff}\cong A_{\ff}^0\times A_{\triv}$.
\end{remark}
As in \S \ref{more_SE_sec} we may realise
\[
\Psi_{\ff,\dd}\colon \NakMod^{T_{\ff}}_{Q,\ff,\dd}\rightarrow \FM_{(\dd,1)}^{T_{\ff}^0}(\Pi_{Q_{\ff}})
\]
as the morphism defined by the correspondence $[\NS]$.  Considering $[\NS]$ instead as a $T^+$-equivariant cycle, we obtain a cycle in $\HO^{\BoMo}(\Nak^{T^+}_{Q}(\ff,\dd)\times_{\B T^+} \FM^{T}_{(\dd,1)}(\Pi_{Q_{\ff}}), \BoQ^{\vir})$.  By abuse of notation we denote by the same symbol the resulting morphism 
\[
\Psi_{\ff,\dd}\colon \NakMod^{T^+}_{Q,\ff,\dd}\rightarrow \FM_{(\dd,1)}^{T}(\Pi_{Q_{\ff}}).
\]

We denote by 
\[
z\colon \Nak_{Q}^{T_{\ff}}(\ff,\dd)\times_{\CM^{T_{\ff}}_{(\dd,1)}(\Pi_{Q_{\ff}})} \FM^{T_{\ff}^0}_{\dd,1}(\Pi_{Q_{\ff}})\rightarrow \Nak_{Q}^{T_{\ff}}(\ff,\dd)\times_{\B T_{\ff}} \FM^{T_{\ff}^0}_{(\dd,1)}(\Pi_{Q_{\ff}})
\]
the closed embedding.  We denote by the same symbol the embedding at the level of $T^+$-quotients (instead of $T_{\ff}$-quotients).  The following lemma follows from the construction of $[\NS]$ in \S \ref{more_SE_sec}.
\begin{lemma}
\label{NS_factoring_lemma}
The element $[\NS]$ lies in the image of 
\[
z_*\colon\HO^{\BM}( \Nak_{Q}^{T_{\ff}}(\ff,\dd)\times_{\CM^{T_{\ff}}_{(\dd,1)}(\Pi_{Q_{\ff}})} \FM^{T^0_{\ff}}_{(\dd,1)}(\Pi_{Q_{\ff}}),z^*\BoQ^{\vir})\rightarrow \HO^{\BM}( \Nak_{Q}^{T_{\ff}}(\ff,\dd)\times_{\B T_{\ff}} \FM^{T^0_{\ff}}_{(\dd,1)}(\Pi_{Q_{\ff}}),\BoQ^{\vir}).
\]
\end{lemma}

We define $\JH^{\ff}\colon \FM^{T_{\ff}^0}_{(\dd,1)}(\Pi_{Q_{\ff}})\rightarrow \CM^{T_{\ff}^0}_{(\dd,1)}(\Pi_{Q_{\ff}})$ as in \S \ref{moduli_stacks_sec}, and we denote by 
\[
\JH^{\ff,+}\colon \FM^{T_{\ff}^0}_{(\dd,1)}(\Pi_{Q_{\ff}})\rightarrow \CM^{T_{\ff}}_{(\dd,1)}(\Pi_{Q_{\ff}})
\]
the extended map.  Explicitly, writing the target as a Cartesian product via $\varpi$, we define $\JH^{\ff,+}\coloneqq \JH^{\ff}\times \det_{\infty}$, where $\det_{\infty}$ is the morphism to $\B A_{\triv}$ representing the tautological line bundle sending a $(\dd,1)$-dimensional $\Pi_{Q_{\ff}}$-module $M$ to the line $\lazy_{\infty}\cdotsh M$. In particular, $\kappa_{\ff}\circ \JH^{\ff,+}=\JH^{\ff}$. Similarly, we define $\JH^{+}=\JH\times \det_{\infty}\colon \FM^{T}_{(\dd,1)}(\Pi_{Q_{\ff}})\rightarrow \CM^{T^+}_{(\dd,1)}(\Pi_{Q_{\ff}})$ as the morphism obtained by passing to $T^+$-quotients. It satisfies $\kappa\circ \JH^+=\JH$.

The following is a modified version of \cite[Lem.8.6.1]{chriss2009representation}:
\begin{proposition}
\label{CG_prop}
There are natural isomorphisms
\begin{align*}
\HO^{\BM}( \Nak_{Q}^{T_{\ff}}(\ff,\dd)\times_{\CM^{T_{\ff}}_{(\dd,1)}(\Pi_{Q_{\ff}})} \FM^{T^0_{\ff}}_{(\dd,1)}(\Pi_{Q_{\ff}}),z^*\BoQ^{\vir})\cong &\Ext^{\bullet}\left((\pi_{\ff})_*\BoQ^{\vir}_{\Nak_{Q}^{T_{\ff}}(\ff,\dd)},(\JH^{\ff,+})_{*}\BoD\BoQ^{\vir}_{\FM^{T^0_{\ff}}_{(\dd,1)}(\Pi_{Q_{\ff}})}\right)
\\
\HO^{\BM}( \Nak_{Q}^{T^+}(\ff,\dd)\times_{\CM^{T^+}_{(\dd,1)}(\Pi_{Q_{\ff}})} \FM^{T}_{(\dd,1)}(\Pi_{Q_{\ff}}),z^*\BoQ^{\vir})\cong &\Ext^{\bullet}\left(\pi^+_*\BoQ^{\vir}_{\Nak_{Q}^{T^+}(\ff,\dd)},(\JH^+)_{*}\BoD\BoQ^{\vir}_{\FM^{T}_{(\dd,1)}(\Pi_{Q_{\ff}})}\right).
\end{align*}
\end{proposition}
\begin{proof}
We consider the first isomorphism, the second is constructed the same way.  Recall that $\Nak_{Q}^{T_{\ff}}(\ff,\dd)$ is smooth, and $\pi_{\ff}$ is projective.  These are the only two features of the assumptions of \cite[Lem.8.6.1]{chriss2009representation} used in the proof; the proposition admits exactly the same proof. 
\end{proof}
Applying $(\kappa_{\ff})_*$ (respectively $\kappa_*$) we deduce the following corollary:
\begin{corollary}
\label{CG_cor}
Let \begin{align*}
f\colon &\HO^{\BoMo}(\Nak_Q^{T_{\ff}}(\ff,\dd),\BoQ^{\vir})\rightarrow \HO^{\BoMo}(\FM^{T^0_{\ff}}_{(\dd,1)}(\Pi_{Q_{\ff}}),\BoQ^{\vir})\\
g\colon &\HO^{\BoMo}(\Nak_Q^{T^+}(\ff,\dd),\BoQ^{\vir})\rightarrow \HO^{\BoMo}(\FM^{T}_{(\dd,1)}(\Pi_{Q_{\ff}}),\BoQ^{\vir})
\end{align*}
be the morphisms induced by a correspondence $C\in \HO^{\BM}( \Nak_{Q}^{T_{\ff}}(\ff,\dd)\times_{\CM^{T_{\ff}}_{(\dd,1)}(\Pi_{Q_{\ff}})} \FM^{T^0_{\ff}}_{(\dd,1)}(\Pi_{Q_{\ff}}),z^*\BoQ^{\vir})$, then $f$ (respectively $g$) is obtained by applying the derived global sections functor to morphisms of complexes
\begin{align}
\label{sh_l_NASE}
&(\kappa_{\ff})_*(\pi_{\ff})_*\BoQ^{\vir}_{\Nak_{Q}^{T_{\ff}}(\ff,\dd)}\rightarrow (\JH^{\ff})_{*}\BoD\BoQ^{\vir}_{\FM^{T^0_{\ff}}_{(\dd,1)}(\Pi_{Q_{\ff}})}\\
\label{sh_l_NASE_red}
&\kappa_*\pi^+_*\BoQ^{\vir}_{\Nak_{Q}^{T^+}(\ff,\dd)}\rightarrow \JH_{*}\BoD\BoQ^{\vir}_{\FM^{T}_{(\dd,1)}(\Pi_{Q_{\ff}})}.
\end{align}
\end{corollary}
Combining this with Lemma \ref{NS_factoring_lemma} and Proposition \ref{prop:nonabelian stable envelope as a correspondence}, we deduce the following corollary:
\begin{corollary}
\label{NAHTPF}
The nonabelian stable envelope $\Psi_{\ff,\dd}$ preserves perverse filtrations, in the sense that $\Psi_{\ff,\dd}(\FP^n(\NakMod^{T_{\ff}}_{Q,\ff,\dd}))\subset \FP^n(\HO\!\CoHA^{T^0_{\ff}}_{\Pi_{Q_{\ff}},(\dd,1)})$ and $\Psi_{\ff,\dd}(\FP^n(\NakMod^{T^+}_{Q,\ff,\dd}))\subset \FP^n(\HO\!\CoHA^{T}_{\Pi_{Q_{\ff}},(\dd,1)})$, where the perverse filtrations are induced via the perverse truncation functors applied to the complexes of \eqref{sh_l_NASE} and \eqref{sh_l_NASE_red}, respectively.
\end{corollary}
The perverse filtrations on $\NakMod^{T_{\ff}}_{Q,\ff,\dd}$ and $\NakMod^{T^+}_{Q,\ff,\dd}$ are simply the order filtration in $\HO_{A_{\triv}}$ defined via \eqref{can_GS_split} and the canonical splitting $\NakMod^{T^+}_{Q,\ff,\dd}\cong \NakMod^{T}_{Q,\ff,\dd}\otimes \HO_{A_{\triv}}$, respectively.

\subsection{Proof of Theorem \ref{NaSE_vs_BPS}}

Our goal in this section is to prove the following theorem, which implies Theorem \ref{NaSE_vs_BPS}, since injectivity of $\Psi_{\ff,\dd}$ is given by Corollary \ref{corollary normalization nonabelian stable envelope}.  The proof is by induction on $\dd$, follows the strategy outlined in \S\ref{subsec: strategy for thm C}, and will involve establishing a number of preparatory lemmas; the proof is completed after Lemma \ref{half_comp_thm}.
\begin{theorem}
\label{NakBPS_prop}
Let $\dd,\ff\in\BoN^{Q_0}$ be arbitrary dimension vectors with $\ff$ nonzero.  Then $\Psi_{\ff,\dd}(\NakMod^{T^0_{\ff}}_{Q,\ff,\dd})=\Fg^{T^0_{\ff}}_{\Pi_{Q_{\ff}},(\dd,1)}\subset \HO\!\CoHA_{\Pi_{Q_{\ff}},(\dd,1)}^{T^0_{\ff}}$.
\end{theorem}

For $\dd\in \Phi^+_Q\subset \BoN^{Q_0}$ a simple positive root we use the notation $\CG(Q)_{\dd}\subset \CoHA_{\Pi_Q}^{T,0}$ introduced in \S \ref{GKM_thm_sec} to denote the generating perverse subsheaf of the relative Hall algebra $\CoHA_{\Pi_Q}^{T,0}={}^{\Fp'}\!\vtau^{\leq 0}\!\CoHA_{\Pi_Q}^{T}$ supported on $\CM^T_{\dd}(\Pi_Q)$.  In particular $\CG(Q)_{\dd}$ is a perverse sheaf on $\CM^T_{\dd}(\Pi_Q)$.

By Theorem \ref{KMA_thm} the zeroth perverse subalgebra $\FP^0\HO\!\CoHA^{T'}_{\Pi_Q}\coloneqq\HO^*(\CM^{T'}(\Pi_Q),\CoHA_{\Pi_Q}^{T',0})$ is the positive half of the generalised Kac--Moody algebra with Chevalley generators in degree $\dd\in \Phi^+_Q$ given by $\HO^*(\CM^{T'}_{\dd}(\Pi_Q),\CG(Q)_{\dd})$, and if $\dd\in \Sigma_Q$ then $\CG(Q)_{\dd}=\ICS(\CM^{T'}_{\dd}(\Pi_Q))$.  Here $T'$ is arbitrary.

Let $\dd,\ff\in\BoN^{Q_0}$ be dimension vectors such that there exists a simple $(\dd,1)$-dimensional $\Pi_{Q_{\ff}}$-module.  Equivalently, $(\dd,1)\in\Phi^+_{Q_{\ff}}$, which is in turn equivalent to $(\dd,1)\in\Sigma_{Q_{\ff}}$ since this dimension vector is indivisible.  We set $T'=T_{\ff}^0$.  Combining \eqref{Nak_decomp} and \eqref{can_split} yields a canonical morphism $\ICS(\CM^{T^0_{\ff}}_{(\dd,1)}(\Pi_{Q_{\ff}}))\rightarrow (\kappa_{\ff})_*(\pi_{\ff})_*\BoQ^{\vir}_{\Nak_{Q}^{T_{\ff}}(\ff,\dd)}$.  Passing to derived global sections, we obtain a canonical embedding
\begin{equation}
\label{can_fzero}
\iota\colon \IH^*(\CM^{T^0_{\ff}}_{(\dd,1)}(\Pi_{Q_{\ff}}),\BoQ^{\vir})\hookrightarrow \HO^*(\Nak^{T_{\ff}}_Q(\ff,\dd),\BoQ^{\vir})
\end{equation}
factoring through the inclusion $\HO^*(\Nak^{T^0_{\ff}}_Q(\ff,\dd),\BoQ^{\vir})\subset \HO^*(\Nak^{T_{\ff}}_Q(\ff,\dd),\BoQ^{\vir})$.

\begin{proposition}
\label{ChevG_prop}
Let $Q$ be a quiver, and let $\dd,\ff\in\BoN^{Q_0}$ satisfy $(\dd,1)\in\Sigma_{Q_{\ff}}$.  Then $\Psi_{\ff,\dd}\circ \iota(\IH^*(\CM^{T^0_{\ff}}_{(\dd,1)}(\Pi_{Q_{\ff}}),\BoQ^{\vir}))\subset \Fg^{T^0_{\ff}}_{\Pi_{Q_{\ff}},(\dd,1)}$, and moreover $\Psi_{\ff,\dd}\circ \iota$ identifies $\IH^*(\CM^{T^0_{\ff}}_{(\dd,1)}(\Pi_{Q_{\ff}}),\BoQ^{\vir})$ with the space $\HO^*\!\left(\CM_{(\dd,1)}^{T^0_{\ff}}(\Pi_{Q_{\ff}}),\CG(Q_{\ff})_{(\dd,1)}\right)$ of Chevalley generators of $\Fg^{T^0_{\ff}}_{\Pi_{Q_{\ff}}}$ for the dimension vector $(\dd,1)\in\BoN^{Q_{\ff}}$.
\end{proposition}
\begin{proof}
By the decomposition theorem \eqref{dec_thm_1}, and our assumptions on $\dd,\ff$, there is an isomorphism
\[
(\JH^{\ff})_*\BoD\BoQ_{\FM^{T^0_{\ff}}_{(\dd,1)}(\Pi_{Q_{\ff}})}^{\vir}\cong\bigoplus_{i\in 2\cdot\BoZ_{\geq 0}}{}^{\Fp'}\!\CH^i\!\left((\JH^{\ff})_*\BoD\BoQ_{\FM^{T^0_{\ff}}_{(\dd,1)}(\Pi_{Q_{\ff}})}^{\vir}\right)[-i]
\]
and there is a unique summand $\ICS(\CM^{T^0_{\ff}}_{(\dd,1)}(\Pi_{Q_{\ff}}))\subset {}^{\Fp'}\!\CH^0\!\left((\JH^{\ff})_*\BoD\BoQ_{\FM^{T^0_{\ff}}_{(\dd,1)}(\Pi_{Q_{\ff}})}^{\vir}\right)$ as well as a unique (up to scalar) nonzero morphism $\ICS(\CM^{T^0_{\ff}}_{(\dd,1)}(\Pi_{Q_{\ff}}))\rightarrow (\kappa_{\ff})_*(\pi_{\ff})_*\BoQ^{\vir}_{\Nak^{T_{\ff}}_Q(\ff,\dd)}$.  It follows from Proposition \ref{CG_prop}, and simplicity of the perverse sheaf $\ICS(\CM^{T^0_{\ff}}_{(\dd,1)}(\Pi_{Q_{\ff}}))$ that there is an equality of morphisms $\Psi_{\ff,\dd}\circ \iota=\iota'\circ g$, where 
\[
\iota'\colon \HO^*(\CM^{T^0_{\ff}}_{(\dd,1)}(\Pi_{Q_{\ff}}),\CG(Q_{\ff})_{(\dd,1)})\hookrightarrow \fg^{T^0_{\ff}}_{\Pi_{Q_{\ff}},(\dd,1)}
\]
is the inclusion provided by Theorem \ref{KMA_thm}, and $g\colon \IH^*(\CM^{T^0_{\ff}}_{(\dd,1)}(\Pi_{Q_{\ff}}),\BoQ^{\vir})\rightarrow  \HO^*(\CM^{T^0_{\ff}}_{(\dd,1)}(\Pi_{Q_{\ff}}),\CG(Q_{\ff})_{(\dd,1)})$ is a morphism that is either the zero map or an isomorphism.  By injectivity of $\Psi_{\ff,\dd}$ (Corollary \ref{corollary normalization nonabelian stable envelope}), we deduce that $g\neq 0$ and the result follows.
\end{proof}
We introduce some notation, to be used in the following.  Firstly, fix a pair of framing dimension vectors $\ff',\ff''\in\BoN^{Q_0}$, and subtori $A'\subset A^0_{\ff'}$ and $A''\subset A^0_{\ff''}$.  Set $A=A'\times A''$.  We define
\[
l^{(1)}_*\colon \HO\!\CoHA^{T\times A'}_{\Pi_{Q_{\ff'}}}\hookrightarrow \HO\!\CoHA^{T\times A}_{\Pi_{Q_{\ff',\ff''}}};\quad\quad
l^{(2)}_*\colon \HO\!\CoHA^{T\times A''}_{\Pi_{Q_{\ff''}}}\hookrightarrow \HO\!\CoHA^{T\times A}_{\Pi_{Q_{\ff',\ff''}}}
\]
as follows. We define $(\mu_{\ff',\dd})^{-1}(0)$ as in \S \ref{sec: Nakajima quiver varieties}.  We extend the $T\times A'$-action on $(\mu_{\ff',\dd})^{-1}(0)$ to a $T\times A$-action by letting the $A''$ factor act trivially.  We embed $\HO\!\CoHA^{T\times A'}_{\Pi_{Q_{\ff'}}}\hookrightarrow \HO\!\CoHA^{T\times A}_{\Pi_{Q_{\ff'}}}\cong \HO\!\CoHA^{T\times A'}_{\Pi_{Q_{\ff'}}}\otimes\HO_{A''}$ via the morphism $\alpha\mapsto \alpha\otimes 1$. Then we compose with the morphism 
\[
\bigoplus_{\substack{\dd\in\BoN^{Q_0}\\n\in\BoN}}\left(\HO\!\CoHA^{T\times A}_{\Pi_{Q_{\ff'}}(\dd,n)}\xrightarrow{\cong} \HO\!\CoHA^{T\times A}_{\Pi_{Q_{\ff',\ff''}},(\dd,n,0)}\right)
\]
induced by extension of $\Pi_{Q_{\ff'}}$-modules by zero.  The morphism $l^{(2)}_*$ is defined the same way.  Similarly, we define
\[
l'_*\colon \HO\!\CoHA^T_{\Pi_{Q}}\hookrightarrow \HO\!\CoHA^{T\times A'}_{\Pi_{Q_{\ff'}}};\quad\quad
l''_*\colon \HO\!\CoHA^T_{\Pi_{Q}}\hookrightarrow \HO\!\CoHA^{T\times A''}_{\Pi_{Q_{\ff''}}}
\]
to be the morphisms induced by extension of $\Pi_Q$-modules by zero and extension of scalars.  We define
\[
\Psi_{\ff'}^{(1)}\coloneqq l^{(1)}_*\circ \Psi_{\ff'};\quad\quad \Psi_{\ff''}^{(2)}\coloneqq l^{(2)}_*\circ \Psi_{\ff''}.
\]

Pick $\ff',\ff''\in\BoN^{Q_0}$ and $\dd'+\dd''=\dd$.  For now we make the choice $A'=\{1\}$ and $A''=A_{\ff''}^0$.  Both the stacks $\Nak^{T^+}_Q(\ff',\dd')\times_{\B T}\Nak^{T_{\ff''}}_Q(\ff'',\dd'')$ and $\FM_{(\dd',1)}^{T}(\Pi_{Q_{\ff'}})\times_{\B T}\FM_{(\dd'',1)}^{T^0_{\ff''}}(\Pi_{Q_{\ff''}})$ admit natural morphisms to $\CM_{(\dd,1,1)}^{T^0_{\ff''}}(\Pi_{Q_{\ff',\ff''}})$, taking pairs of modules to the direct sum of their semisimplifications.  These induce perverse filtrations on the Borel--Moore homology of these stacks in the (by now) standard way, via perverse truncation functors for the derived category of constructible complexes on $\CM_{(\dd,1,1)}^{T^0_{\ff''}}(\Pi_{Q_{\ff',\ff''}})$.  

Given a $\HO_{A_{\triv}}\otimes \HO_{A_{\ff''}}\otimes \HO_T$-module $M$ we define $M_{\loc}$ to be the localisation at the prime ideal $\mathfrak{p}=\HO_{A_{\triv}}\otimes\mathfrak{m}$, with $\mathfrak{m}\subset \HO_{A_{\ff''}}\otimes \HO_T$ the maximal graded ideal as in Remark \ref{loc_rem}.  As in \S\ref{PAGC} we extend the perverse filtration to a perverse filtration on e.g. $\left(\NakMod^{T^+}_{Q,\ff', \dd'}\otimes_{\HO_{T}} \NakMod^{T_{\ff''}}_{Q,\ff'',\dd''}\right)_{\loc}$ by declaring that $\FP^n\left(\left(\NakMod^{T^+}_{Q,\ff', \dd'}\otimes_{\HO_{T}} \NakMod^{T_{\ff''}}_{Q,\ff'',\dd''}\right)_{\loc}\right)$ is spanned by elements $\alpha/p$ where $p\in \HO_{T\times A}\setminus \mathfrak{p}$, $\alpha\in \FP^m\left(\left(\NakMod^{T^+}_{Q,\ff', \dd'}\otimes_{\HO_{T}} \NakMod^{T_{\ff''}}_{Q,\ff'',\dd''}\right)\right)$ and $m-2\deg(p)\leq n$.  Using Lemmas \ref{coarse_perv} and \ref{AB_lemma}, it is shown as in \S \ref{PAGC} that the natural morphisms 
\begin{align*}
&\NakMod^{T^+}_{Q,\ff', \dd'}\otimes_{\HO_{T}} \NakMod^{T_{\ff''}}_{Q,\ff'',\dd''}\rightarrow \left(\NakMod^{T^+}_{Q,\ff', \dd'}\otimes_{\HO_{T}} \NakMod^{T_{\ff''}}_{Q,\ff'',\dd''}\right)_{\loc}\\
&\HO\!\CoHA_{\Pi_{Q_{\ff',\ff''}, ( \dd',1,0)}}^{T_{\ff''}^0} \otimes_{\HO_{T_{\ff''}^0}} \HO\!\CoHA_{\Pi_{Q_{\ff',\ff''}, (\dd'',0,1)}}^{T^0_{\ff''}} \rightarrow \left(\HO\!\CoHA_{\Pi_{Q_{\ff',\ff''}, ( \dd',1,0)}}^{T_{\ff''}^0} \otimes_{\HO_{T_{\ff''}^0}} \HO\!\CoHA_{\Pi_{Q_{\ff',\ff''}, (\dd'',0,1)}}^{T^0_{\ff''}} \right)_I
\end{align*}
respect the perverse filtrations, and are moreover injective even after passing to the associated graded objects with respect to these filtrations.  The second localisation morphism is with respect to the ideal $I$ defined as in Proposition \ref{mult_prop}.

For the next lemma we use the operators $\psi(-)\in \Fg_Q^{\MO,T}$ defined in \eqref{little_psi_def}.
\begin{lemma}
\label{cor_pack}
Fix nonzero dimension vectors $\dd',\dd'',\ff',\ff''\in\BoN^{Q_0}$ and set $\dd=\dd'+\dd''$.  Choose 
\begin{align*}
\overline{\alpha}\in \HO^*\!\left(\Nak^{T}_Q(\ff',\dd'),\BoQ^{\vir}\right)\subset \HO^*\!\left(\Nak^{T^+}_Q(\ff',\dd'),\BoQ^{\vir}\right),
\\
\overline{\beta}\in \HO^*\!\left(\Nak^{T_{\ff''}^0}_Q(\ff'',\dd''),\BoQ^{\vir}\right)\subset \HO^*\!\left(\Nak^{T_{\ff''}}_Q(\ff'',\dd''),\BoQ^{\vir}\right).
\end{align*}
Assume that Theorem \ref{NakBPS_prop} holds for all $\ee<\dd$.  Then there is an equality of expansions in $x_{(1),\infty',1}^{-1}$
\begin{equation}
\label{claimed_expansion}
\hbar^{-1}\vDelta_{(\ul{0},1,0),(\dd,0,1)}([\Psi^{(1)}_{\ff',\dd'}(\ol{\alpha}),\Psi^{(2)}_{\ff'',\dd''}(\ol{\beta})])=x_{(1),\infty',1}^{-1}\Psi^{(2)}_{\ff'',\dd}(\psi(\overline{\alpha})(\ol{\beta}))+O(x_{(1),\infty',1}^{-2}).
\end{equation}
\end{lemma}

\begin{proof}
By Proposition \ref{mult_compat} and Corollary \ref{prod_compat_cor}, the left hand side of the equality claimed in the lemma is equal to
\begin{align*}
&\hbar^{-1}\vDelta_{(\ul{0},1,0),(\dd,0,1)}\Psi_{\ff'+\ff'',\dd}\left(\Stab_+(\ol\alpha\otimes\ol\beta)- \Stab_-(\ol\alpha\otimes\ol\beta)\right)\\
=&\hbar^{-1}\vDelta_{(\ul{0},1,0),(\dd,0,1)}\Psi_{\ff'+\ff'',\dd}\Stab_+\left( 1 -R_{\ff',\ff''}\right)(\ol\alpha\otimes\ol\beta).\\
\end{align*}
By Proposition \ref{mult_compat} this is the same as
\[
    \vDelta_{(\ul{0},1,0),(\dd,0,1)}\sum_{\ee'+\ee''=\dd}\vmult_{(\ee', 1,0), (\ee'', 0,1)}(\Psi_{\ff',\ee'}\otimes \Psi_{\ff'', \ee''})\left( 1 -R_{\ff',\ff''}\right)(\ol\alpha\otimes\ol\beta).
\]
Next, by Proposition \ref{mult_cor}, it follows that $\vDelta_{(\ul{0},1,0),(\dd,0,1)}\vmult_{(0, 1,0), (\dd, 0,1)}=\id$, so by the R-matrix expansion \eqref{eq:R-matrix expansion} the left hand side of \eqref{claimed_expansion} is equal to
\begin{multline}
\label{eq: expansion}
    x_{(1), \infty', 1}^{-1}(\Psi_{\ff',0}\otimes \Psi_{\ff'', \dd})\left(\rmat_{\ff',\ff''}\right)(\ol\alpha\otimes\ol\beta)+ O(x_{(1),\infty',1}^{-2})\\
    +\hbar^{-1 }\sum_{\substack{\ee'\neq 0\\\ee'+\ee''=\dd}}\vDelta_{(\ul{0},1,0),(\dd,0,1)}\vmult_{(\ee', 1,0), (\ee'', 0,1)}(\Psi_{\ff',\ee'}\otimes \Psi_{\ff'', \ee''})\left( 1-R_{\ff',\ff''}\right)(\ol\alpha\otimes\ol\beta)
\end{multline}
Notice that $\Psi_{\ff', 0}=\id$, so the first term is exactly $x_{(1),\infty',1}^{-1}\Psi^{(2)}_{\ff'',\dd}(\psi(\overline{\alpha})(\ol{\beta}))$ by definition of the raising operator $\psi(\ol\alpha)$, cf. \eqref{eq:MO raising operator}. Therefore, to complete the proof, it suffices to show that each term in the last summation contributes to an additional $O(x_{(1), \infty', 1}^{-2})$ term. This uses our inductive assumption regarding Theorem \ref{NakBPS_prop}. Denote by $\gamma_{\ee',\ee''}$ the projection of $(1-R_{\ff',\ff''})(\ol{\alpha}\otimes\ol{\beta})$ onto the summand $\left(\NakMod^{T^+}_{Q,\ff', \ee'}\otimes_{\HO_{T}} \NakMod^{T_{\ff''}}_{Q,\ff'',\ee''}\right)_{\loc} $.  From the expansion \eqref{eq:R-matrix expansion} we deduce
\begin{equation}
\label{first_jump}
\gamma_{\ee',\ee''}\in\FP^{-2}\!\left( \left(\NakMod^{T^+}_{Q,\ff', \ee'}\otimes_{\HO_{T}} \NakMod^{T_{\ff''}}_{Q,\ff'',\ee''}\right)_{\loc} \right).
\end{equation}
Next, we claim that for $\ol{\gamma}\in\NakMod^{T}_{Q,\ff',\ee'} \otimes_{\HO_{T}} \NakMod^{T_{\ff''}^0}_{Q,\ff'',\ee''}$ with $\ee''\neq \dd$, there is an inclusion
\begin{equation}
\label{extra_jump}
\left(\vDelta_{(\ul{0},1,0),(\dd,0,1)} \circ \vmult\circ  (\Psi_{\ff',\ee'}\otimes_{\HO_T}\Psi_{\ff'',\ee''}) \right)(\ol{\gamma})\in \FP^{-2}\!\left(\left(\HO\!\CoHA_{\Pi_{Q_{\ff',\ff''}, (\ul 0,1,0)}}^{T^0_{\ff''}} \otimes_{\HO_{T^0_{\ff''}}} \HO\!\CoHA_{\Pi_{Q_{\ff',\ff''}, (\dd,0,1)}}^{T^0_{\ff''}} \right)_I\right).
\end{equation}
Under the inductive assumption regarding Theorem \ref{NakBPS_prop}, the left hand side of \eqref{extra_jump} is equal to $\vDelta_{(\ul{0},1,0),(\dd,0,1)}(\gamma'\ast \gamma'')$ for $\gamma'\in \Fg^{T_{\ff''}^0}_{\Pi_{Q_{\ff',\ff''}},(\ee',1,0)}$ and $\gamma''\in \Fg^{T^0_{\ff''}}_{\Pi_{Q_{\ff',\ff''}},(\ee'',0,1)}$.  We calculate
\[
\vDelta_{(\ul{0},1,0),(\dd,0,1)}(\gamma'\ast \gamma'')=\vDelta_{(\ul{0},1,0),(\ee',0,0)}(\gamma')\ast(1\otimes \gamma'').
\]
Now $\vDelta_{(\ul{0},1,0),(\ee',0,0)}(\gamma')\in \FP^{-2}\left(\HO\!\CoHA_{\Pi_{Q_{\ff',\ff''}, (\underline{0},1,0)}}^{T^0_{\ff''}}\tilde{\otimes}\HO\!\CoHA_{\Pi_{Q_{\ff',\ff''}, (\ee',0,0)}}^{T^0_{\ff''}}\right)$ by Corollary \ref{perv_prim}, and the claim follows.
Putting \eqref{first_jump} and \eqref{extra_jump} together, we deduce that each term in the second line of \eqref{eq: expansion} belongs to
\[
\FP^{-4}\left(\HO\!\CoHA_{\Pi_{Q_{\ff',\ff''}, (\underline{0},1,0)}}^{T^0_{\ff''}}\otimes_{\HO_{{T^0_{\ff''}}}}\HO\!\CoHA_{\Pi_{Q_{\ff',\ff''}, (\dd,0,1)}}^{T^0_{\ff''}}\right)_I.
\]
Thus, by Proposition \ref{ptrack}, each summand contributes to the $O(x_{(1), \infty', 1}^{-2})$ term. The result follows.
\end{proof}

\begin{lemma}
\label{Nakinsc}
Let $\dd\in\Phi_Q^+$ be a simple positive root, and let $\ff\in\BoN^{Q_0}$ satisfy $\ff\cdotsh\dd\neq 0$.  We extend the $T$-action on $\CM_{\dd}(\Pi_Q)$ to a $T_{\ff}^0=T\times A_{\ff}^0$-action by letting the $A_{\ff}^0$ component act trivially.  Let $l^{\circ}\colon \CM^{T^0_{\ff}}_{\dd}(\Pi_Q)\hookrightarrow \CM^{T^0_{\ff}}_{(\dd,1)}(\Pi_{Q_{\ff}})$ be the embedding induced by extending modules by zero.  Then the semisimple perverse sheaf $(\pi^0_{\ff})_*\BoQ^{\vir}_{\Nak^{T^0_{\ff}}_Q(\ff,\dd)}$ contains a unique copy of the simple perverse sheaf $l^{\circ}_*\CG(Q)_{\dd}$.
\end{lemma}
\begin{proof}
By Corollary \ref{KMLA_Nak_comp} there is an isomorphism 
\begin{equation}
\label{stage_pst}
(\pi^0_{\ff})_*\BoQ^{\vir}_{\Nak^{T^0_{\ff}}_Q(\ff,\dd)}\cong (\Fn^{{T^0_{\ff}},+}_{\CG(Q_{\ff})})_{(\dd,1)}.
\end{equation}
Firstly, we assume that $\dd\in\Sigma_Q$.  The only summand on the right hand side of \eqref{stage_pst} having support $l^{\circ}(\CM^{T^0_{\ff}}_{\dd}(\Pi_Q))$ is given by $[-,-]\left(\ICS(\CM^{T^0_{\ff}}_{(\dd,0)}(\Pi_{Q_{\ff}}))\boxdot \ICS(\CM^{T^0_{\ff}}_{(\ul{0},1)}(\Pi_{Q_{\ff}}))\right)\cong l^{\circ}_*\CG(Q)_{\dd}$.
Similarly, if $\dd=n\dd'\in\Phi^+_Q\setminus \Sigma_Q$ with $\dd'\in\Sigma_Q$ and $n\geq 2$, the only summand on the right hand side of \eqref{stage_pst} with support $(l^{\circ}\circ\Diag_n)(\CM^{T^0_{\ff}}_{\dd'}(\Pi_Q))$ is $[-,-]\left(\Diag_{n,*}\ICS(\CM^{T^0_{\ff}}_{(\dd',0)}(\Pi_{Q_{\ff}}))\boxdot \ICS(\CM^{T^0_{\ff}}_{(\ul{0},1)}(\Pi_{Q_{\ff}}))\right)\cong l^{\circ}_*\CG(Q)_{\dd}$.  The result follows.
\end{proof}
Let $\dd\in\Phi^+_Q$, and fix $\ff\in\BoN^{Q_0}$ such that $\ff\cdotsh\dd\neq 0$.  It follows from $\dd\in\Phi^+_Q$ that $(\dd,0)\in\Phi^+_{Q_{\ff}}$.  We denote by 
\[
g'_1\colon \HO^*(\CM^{T^0_{\ff}}_{\dd}(\Pi_Q),\CG(Q)_{\dd})\hookrightarrow \HO^*(\Nak^{T_{\ff}}_Q(\ff,\dd),\BoQ^{\vir})
\]
the inclusion defined (up to multiplication by a nonzero scalar) by Lemma \ref{Nakinsc} and \eqref{can_GS_split}, and then define
\[
g_1\coloneqq \Psi_{\ff,\dd}\circ g'_1\colon\HO^*(\CM^{T^0_{\ff}}_{\dd}(\Pi_Q),\CG(Q)_{\dd})\hookrightarrow\HO\!\CoHA^{T^0_{\ff}}_{\Pi_{Q_{\ff}},(\dd,1)}.
\]
Recall that we define $l\colon \CM^{T^0_{\ff}}_{\dd}(\Pi_{Q})\hookrightarrow \CM^{T^0_{\ff}}_{(\dd,0)}(\Pi_{Q_{\ff}})$ to be the extension by zero map (see \eqref{ldef}).  There is a canonical isomorphism $l_*\CG(Q)_{\dd}\cong \CG(Q_{\ff})_{(\dd,0)}$.  Since, by definition, $\Fn^{{T^0_{\ff}},+}_{\Pi_{Q_{\ff}}}$ is generated by the spaces $\HO^*(\CM^{T^0_{\ff}}_{(\dd,n)}(\Pi_{Q_{\ff}}),\CG(Q_{\ff})_{(\dd,n)})$ for $(\dd,n)\in\Phi^+_{Q_{\ff}}$ we obtain an inclusion
\[
g'_2\colon\HO^*(\CM^{T^0_{\ff}}_{\dd}(\Pi_Q),\CG(Q)_{\dd})\hookrightarrow\Fn^{T^0_{\ff},+}_{\Pi_{Q_{\ff}},(\dd,0)}. 
\]
Let $e_{\infty}\in \Fn^{{T^0_{\ff}},+}_{\Pi_Q,(\ul{0},1)}$ be a Chevalley generator, of cohomological degree zero: this specifies it uniquely, up to multiplication by an element of $\BoQ\setminus\{0\}$.  We define
\[
g_2=[e_{\infty},-]\circ g'_2\colon \HO^*(\CM^{T^0_{\ff}}_{\dd}(\Pi_Q),\CG(Q)_{\dd})\hookrightarrow\Fn^{T^0_{\ff},+}_{\Pi_{Q_{\ff}},(\dd,1)}.
\]
The morphism $g_2$ is an injection by Theorem \ref{KMA_thm}.

\begin{lemma}
\label{weak_F1_lem}
Assume, as above, that $\dd\in\Phi^+_Q$, $\ff\in\BoN^{Q_0}$, and $\ff\cdotsh\dd\neq 0$.  Assume that Theorem \ref{NakBPS_prop} holds for dimension vectors $\ee< \dd$.  The image of $g_1$ lies in $\Fg^{T^0_{\ff}}_{\Pi_{Q_{\ff}},(\dd,1)}$, and considered as morphisms to $\Fg^{T^0_{\ff}}_{\Pi_{Q_{\ff}},(\dd,1)}$, there is an equality $g_1=\lambda g_2$ for some $\lambda\in\BoQ\setminus\{0\}$.
\end{lemma}
\begin{proof}
We define $\CoHA_{\Pi_{Q_{\ff}}}^{{T^0_{\ff}},0}$ as in \eqref{ZPSA} to be the zeroth perverse subalgebra of $\CoHA_{\Pi_{Q_{\ff}}}^{{T^0_{\ff}}}$.  By Theorem \ref{KMA_thm}, there is an isomorphism
\begin{equation}
\label{rem_UEA}
\CoHA_{\Pi_{Q_{\ff}}}^{{T^0_{\ff}},0}\cong \UEA_{\boxdot}\!\left(\fn^{{T^0_{\ff}},+}_{\CG(Q_{\ff})}\right).
\end{equation}
Firstly, we consider the case $\dd\in\Sigma_Q$.  The right hand side of \eqref{rem_UEA} is defined as a quotient of the free unital associative algebra $\Tens_{\boxdot}\!\left(\bigoplus_{(\dd',n)\in\Phi_{Q_{\ff}}^+ }\SG(Q_{\ff})_{(\dd',n)}\right)$ by the Serre relations.  It follows that there are precisely two summands in the decomposition of $\CoHA_{\Pi_{Q_{\ff}}}^{T_{\ff}^0,0}$ into simple summands with support equal to $l^{\circ}(\CM^{T_{\ff}^0}_{\dd}(\Pi_Q))$; they are given by the (isomorphic) summands $\CG(Q_{\ff})_{(\dd,0)}\boxdot \CG(Q_{\ff})_{(\ul{0},1)}$ and $\CG(Q_{\ff})_{(\ul{0},1)}\boxdot\CG(Q_{\ff})_{(\dd,0)}$.  Similarly, if $\dd\in\Phi_Q^+\setminus\Sigma_Q$, with $\dd=n\dd'$, $\dd'\in\Sigma_Q$ and $n\geq 2$ there are precisely two summands in $\Tens_{\boxdot}\!\left(\bigoplus_{(\dd',n)\in\Phi_{Q_{\ff}}^+ }\SG(Q_{\ff})_{(\dd',n)}\right)$ with support equal to $l^{\circ}(\Diag_n(\CM_{\dd'}^{T_{\ff}^0}(\Pi_Q)))$; they are $\CG(Q_{\ff})_{(\dd,0)}\boxdot \CG(Q_{\ff})_{(\ul{0},1)}$ and $\CG(Q_{\ff})_{(\ul{0},1)}\boxdot\CG(Q_{\ff})_{(\dd,0)}$.  After passing to derived global sections, the composition
\[
\CG(Q_{\ff})_{(\dd,0)}\boxdot \CG(Q_{\ff})_{(\ul{0},1)}\hookrightarrow \Tens_{\boxdot}\!\left(\bigoplus_{(\dd',n)\in\Phi_{Q_{\ff}}^+ }\SG(Q_{\ff})_{(\dd',n)}\right)\twoheadrightarrow \UEA_{\boxdot}\!\left(\Fn^{T_{\ff}^0,+}_{\CG(Q_{\ff})}\right)\cong \CoHA_{\Pi_{Q_{\ff}}}^{T_{\ff}^0,0}\rightarrow \CoHA_{\Pi_{Q_{\ff}}}^{T_{\ff}^0}
\]
becomes the morphism 
\[
\HO^*(\CM^{T_{\ff}^0}_{(\dd,0)}(\Pi_{Q_{\ff}}),\CG(Q_{\ff})_{(\dd,0)}))\otimes_{\HO_{T_{\ff}^0}}\HO^*(\CM^{T^0_{\ff}}_{(\ul{0},1)}(\Pi_{Q_{\ff}}),\CG(Q_{\ff})_{(\ul{0},1)}))\xrightarrow{\vmult}\HO\!\CoHA^{T_{\ff}^0}_{\Pi_{Q_{\ff}}},
\]
and the analogous statement holds also for the summand $\CG(Q_{\ff})_{(\ul{0},1)}\boxdot\CG(Q_{\ff})_{(\dd,0)}$.

Note that $\CG(Q_{\ff})_{(\ul{0},1)}=\ICS(\CM^{T_{\ff}^0}_{(\ul{0},1)}(\Pi_{Q_{\ff}}))=\BoQ_{\B T_{\ff}^0}$ after identifying $\CM^{T_{\ff}^0}_{(\ul{0},1)}(\Pi_{Q_{\ff}})=\B T_{\ff}^0$ and we have isomorphisms
\begin{align*}
\HO^*\!\left(\CM^{T_{\ff}^0}_{(\dd,1)}(\Pi_{Q_{\ff}}),\CG(Q_{\ff})_{(\dd,0)}\boxdot \CG(Q_{\ff})_{(\underline{0},1)}\right)\cong &\HO^*\left(\CM^{T_{\ff}^0}_{(\dd,1)}(\Pi_{Q_{\ff}}),l^{\circ}_*\CG(Q)_{\dd}\right)\\
\xrightarrow[\cong]{\;\; r\;\;} &\HO^*(\CM^{T_{\ff}^0}_{\dd}(\Pi_Q),\CG(Q)_{\dd})\\
\HO^*\!\left(\CM^{T_{\ff}^0}_{(\dd,1)}(\Pi_{Q_{\ff}}), \CG(Q_{\ff})_{(\underline{0},1)}\boxdot \CG(Q_{\ff})_{(\dd,0)}\right)\xrightarrow[\cong]{\;\; s\;\;}&\HO^*(\CM^{T_{\ff}^0}_{\dd}(\Pi_Q),\CG(Q)_{\dd}).
\end{align*}
From Lemma \ref{NS_factoring_lemma} and Corollary \ref{CG_cor}, the morphism $g_1$ is obtained as a linear combination of $r^{-1}$ and $s^{-1}$, followed by the inclusions of the left hand sides of the above isomorphisms into $\CoHA^{T_{\ff}^0,0}_{\Pi_Q}$ induced by Theorem \ref{KMA_thm}, and finally the adjunction morphism $\CoHA^{T_{\ff}^0,0}_{\Pi_Q}\hookrightarrow \CoHA^{T_{\ff}^0}_{\Pi_Q}$.  Putting this all together, there exist $\lambda_1,\lambda_2\in\BoQ$ such that $g_1=\lambda_1 m_1+\lambda_2 m_2$, where 
\begin{align*}
m_1\colon \HO^*(\CM^{T_{\ff}^0}_{\dd}(\Pi_Q),\CG(Q)_{\dd})\xrightarrow{\alpha\mapsto \vmult(e_{\infty},\alpha)}\HO\!\CoHA^{T_{\ff}^0}_{\Pi_{Q_{\ff}},(\dd,1)}\\
m_2\colon \HO^*(\CM^{T_{\ff}^0}_{\dd}(\Pi_Q),\CG(Q)_{\dd})\xrightarrow{\alpha\mapsto \vmult(\alpha,e_{\infty})}\HO\!\CoHA^{T_{\ff}^0}_{\Pi_{Q_{\ff}},(\dd,1)}\\
\end{align*}
are defined by left, respectively right, multiplication by our chosen cohomological degree zero Chevalley generator for the dimension vector $(\underline{0},1)$.

By Corollary \ref{corollary normalization nonabelian stable envelope} $\Psi_{\ff,\dd}$ is injective, from which it follows that we cannot have $\lambda_1=0=\lambda_2$.  Note that (by definition) $g_2=m_1-(-1)^{\chi_{\tilde{Q}_{\ff}}((\underline{0},1),(\dd,0))}m_2$, and so the lemma will follow if we can show $\lambda_1=-(-1)^{\chi_{\tilde{Q}_{\ff}}((\underline{0},1),(\dd,0))}\lambda_2$.  It will reduce clutter below to observe that $\chi_{\tilde{Q}_{\ff}}((\underline{0},1),(\dd,0))=-\ff\cdot\dd$.

Consider the diagram
\[
\begin{tikzcd}
\IC^*(\CM^{T_{\ff}^0}_{\dd}(\Pi_Q),\BoQ^{\vir})\arrow[dd, bend right=70, shift right=1, swap, "g_1\otimes e_{\infty''}"]
\arrow[d, "g'_1\otimes \ol{e}_{\infty''}"]
\\
\NakMod^{T_{\ff}}_{Q,\ff,\dd}\otimes_{\HO_{T}} \NakMod^{T_{\ff}}_{Q,\ff,\underline{0}}\arrow[r, "F"]\arrow[d, "\Psi^{(1)}_{\ff,\dd}\otimes_{\HO_T}\Psi^{(2)}_{\ff,\underline{0}} "]& (\NakMod^{T_{\ff}}_{Q,\ff,0}\otimes_{\HO_T}\NakMod^{T_{\ff}}_{Q,\ff,\dd})_{\loc}\arrow[d, "(\Psi^{(1)}_{\ff,\underline{0}}\otimes_{\HO_T}\Psi^{(2)}_{\ff,\dd})_{\loc}"]
\\
\HO\!\CoHA_{\Pi_{Q_{\ff,\ff}, (\dd,1,0)}}^{T_{\ff}^0} \otimes_{\HO_{T}} \HO\!\CoHA_{\Pi_{Q_{\ff,\ff}, (\underline{0},0,1)}}^{T_{\ff}^0}\arrow[r, "G"]&(\HO\!\CoHA_{\Pi_{Q_{\ff,\ff}, (\ul{0},1,0)}}^{T_{\ff}^0}\otimes_{\HO_T}\HO\!\CoHA_{\Pi_{Q_{\ff,\ff}, (\dd,0,1)}}^{T_{\ff}^0})_I
\end{tikzcd}
\]
with $\ol{e}_{\infty''}$ denoting the unique (up to scalar) cohomological degree zero element in $\NakMod^{T_{\ff}}_{Q,\ff,\underline{0}}$ and
\[
F= ((\Stab_+)_{\loc}^{-1})_{\underline{0},\dd}\circ(\Stab_+-\Stab_-);\quad\quad
G=\vDelta_{(\underline{0},1,0),(\dd,0,1)}\circ [-,-].
\]
By Lemma \ref{cor_pack}, the diagram commutes up to order one in $x_{(1),\infty',1}^{-1}$.  In particular, by \eqref{eq:R-matrix expansion} we find that $\Image((\Psi^{(1)}_{\ff,\underline{0}}\otimes_{\HO_T}\Psi^{(2)}_{\ff,\dd})_{\loc}\circ F\circ (g'_1\otimes\ol{e}_{\infty''}))$ is divisible by $x_{(1),\infty',1}^{-1}$, since it is given by an off-diagonal term in the R-matrix.  Hence, the same must be true for $G\circ (g_1\otimes e_{\infty''})$.  Now let $\alpha\in  \Fg^{T_{\ff}^0}_{\Pi_Q,\dd}$.  We calculate
\begin{align*}
\vDelta_{(\underline{0},1,0),(\dd,0,1)}(\alpha\ast e_{\infty'}\ast e_{\infty''})&=(1\otimes \alpha)\ast(e_{\infty'}\otimes 1)\ast(1\otimes e_{\infty''})\\
&=\left(\EU^{T_{\ff,\ff}^0}_{(\dd,0,0),\delta_{\infty'}}((\tilde{Q}_{\ff,\ff})^{\opp}_1)^{-1}\EU^{T_{\ff,\ff}^0}_{(\dd,0,0),\delta_{\infty'}}((\tilde{Q}_{\ff,\ff})_1)\right)_{x_{(2),\ldots}\leftrightarrow x_{(1),\ldots}}\cdot\\
&(-1)^{\ff\cdot\dd}(e_{\infty'}\otimes (\alpha\ast e_{\infty''}))\\
&=(-1)^{\ff\cdot\dd}\cdot e_{\infty'}\otimes (\alpha\ast e_{\infty''})+O(x^{-1}_{(1),\infty',1})\\
\vDelta_{(\underline{0},1,0),(\dd,0,1)}(e_{\infty'}\ast\alpha \ast e_{\infty''})&=(e_{\infty'}\otimes 1)\ast(1\otimes \alpha)\ast(1\otimes e_{\infty''})\\
&=e_{\infty'}\otimes (\alpha\ast e_{\infty''})
\end{align*}
(using \eqref{constexp}) and similarly
\begin{align*}
\vDelta_{(\underline{0},1,0),(\dd,0,1)}(e_{\infty''}\ast\alpha\ast e_{\infty'})&=\left(\EU^{T_{\ff,\ff}^0}_{(\dd,0,0),\delta_{\infty'}}((\tilde{Q}_{\ff,\ff})^{\opp}_1)^{-1}\EU^{T_{\ff,\ff}^0}_{(\dd,0,0),\delta_{\infty'}}((\tilde{Q}_{\ff,\ff})_1)\right)_{x_{(2),\ldots}\leftrightarrow x_{(1),\ldots}}\cdot\\
&(-1)^{\ff\cdot\dd}(e_{\infty'}\otimes  (e_{\infty''}\ast \alpha))\\
\vDelta_{(\underline{0},1,0),(\dd,0,1)}(e_{\infty''}\ast e_{\infty'}\ast\alpha  )&=e_{\infty'}\otimes (e_{\infty''}\ast\alpha).
\end{align*}

We deduce that, since
\[
G\circ (g_1(\alpha)\otimes e_{\infty''})=\lambda_1\vDelta_{(\underline{0},1,0),(\dd,0,1)}([\alpha\ast e_{\infty'} , e_{\infty''}])+\lambda_2\vDelta_{(\underline{0},1,0),(\dd,0,1)}([e_{\infty'}\ast \alpha,e_{\infty''}])
\]
has zero constant term when expanded in $x^{-1}_{(1),\infty',1}$, we must have $\lambda_1+(-1)^{\ff\cdot\dd}\lambda_2=0$, and the lemma follows.
\end{proof}

So far we have shown:
\begin{enumerate}
\item
If $(\dd,1)\in \Phi_{Q_{\ff}}^+$ then $\IH^*(\CM_{(\dd,1)}^{T_{\ff}^0}(\Pi_{Q_{\ff}}),\BoQ^{\vir})$, identified with the space of Chevalley generators in $\Fg^{T_{\ff}^0}_{\Pi_{Q_{\ff}},(\dd,1)}$, lies in the image of $\Psi_{\ff,\dd}$ (Proposition \ref{ChevG_prop}).
\item
If $\dd\in \Phi_{Q}^+$ and Theorem \ref{NakBPS_prop} holds for all dimension vectors $\ee<\dd$, then $\HO^*(\CM^{T_{\ff}^0}_{\dd}(\Pi_Q),\CG(Q)_{\dd})$ is identified with the space of Chevalley generators in $\Fg^{T_{\ff}^0}_{\Pi_{Q_{\ff}},(\dd,0)}$ and for $\ff\in\BoN^{Q_0}$ satisfying $\ff\cdotsh\dd\neq 0$ we have an embedding $\HO^*(\CM^{T_{\ff}^0}_{\dd}(\Pi_Q),\CG(Q)_{\dd})\hookrightarrow \Fg^{T_{\ff}^0}_{\Pi_{Q_{\ff}},(\dd,1)}$ sending each Chevalley generator $\alpha$ to $[e_{\infty},\alpha]$.  The image of this embedding lies in the image of $\Psi_{\ff,\dd}$ (Lemma \ref{weak_F1_lem}).
\end{enumerate}
\begin{lemma}
\label{half_comp_thm}
Let $\dd\in\BoN^{Q_0}$, and assume that Theorem \ref{NakBPS_prop} holds for all dimension vectors strictly less than $\dd$.  For arbitrary $\ff\in\BoN^{Q_0}$, the subspace $\Fg^{T_{\ff}^0}_{\Pi_{Q_{\ff}},(\dd,1)}\subset\FP^0\HO\!\CoHA^{T_{\ff}^0}_{\Pi_{Q_{\ff}}}$ is contained in $\Image(\Psi_{\ff,\dd})$.
\end{lemma}
\begin{proof}
We prove the lemma by induction on $\dd$, with the base case $\dd=0$ satisfied by definition of $\Psi_{\ff,\ul{0}}$.  

An arbitrary $\gamma\in \Fg^{T_{\ff}^0}_{\Pi_{Q_{\ff}},(\dd,1)}$ can be written as a linear combination
\[
\gamma=\gamma'+[e_{\infty},\gamma'']+L
\]
where $\gamma'$ and $\gamma''$ are Chevalley generators, and $L$ is a linear combination of elements $[\ldotsh[\beta,\alpha_i],\alpha_{i-1}],\ldotsh ],\alpha_1]$ where each of the $\alpha_i$ are Chevalley generators for $\Fg^{T_{\ff}^0}_{\Pi_Q}$ (embedded in $\Fg^{T_{\ff}^0}_{\Pi_{Q_{\ff}}}$ via $l_*$), and $\beta$ is a Chevalley generator in $\Fg^{T_{\ff}^0}_{\Pi_{Q_{\ff}},(\dd',1)}$ for some $0<\dd'<\dd$.  By Proposition \ref{ChevG_prop}, $\gamma'\in\Image(\Psi_{\ff,\dd})$, and by Lemma \ref{weak_F1_lem} $[e_{\infty},\gamma'']\in\Image(\Psi_{\ff,\dd})$, so it remains to deal with the term $L$.

Set $\ff'=\ul{1}$ and denote $\ff''=\ff$.  Set $\delta=[\ldotsh[\beta,\alpha_i],\alpha_{i-1}],\ldotsh ],\alpha_2]$ as in the expression for $L$.  By the inductive hypothesis $\delta=\Psi_{\ff'',\bullet}(\overline{\delta})$ for some $\ol{\delta}$.  By Lemma \ref{weak_F1_lem} there is a $\overline{\alpha}_1\in \HO^*\!\left(\Nak^{T^0_{\ff'}}_Q(\ff',\dd'),\BoQ^{\vir}\right)$, for some nonzero $\dd'<\dd$, satisfying 
\[
\Psi^{(1)}_{\ff',\dd'}(\overline{\alpha}_1)=[e_{\infty'},\alpha_1]\in\Fg_{\Pi_{Q_{\ff',\ff''}},(\dd',1,0)}^{T^0_{\ff'}}.
\]
Furthermore, there are equalities
\begin{align*}
\Psi_{\ff,\dd}(\psi(\overline{\alpha}_1)(\ol{\delta}))=&\hbar^{-1}\Res_{x_{(1),\infty',1}}\vDelta_{(\ul{0},1,0),(\dd,0,1)}([[e_{\infty'},\alpha_1],\delta])\\
=&[\hbar^{-1}\Res_{x_{(1),\infty',1}}\vDelta_{(\ul{0},1,0),(\dd,0,1)}([e_{\infty'},\alpha_1]),\delta]\\
=&(\ff'\cdot\dd')[\delta,\alpha_1]
\end{align*}
where the first equality follows by Lemma \ref{cor_pack}, the second from Corollary  \ref{Res_cor_2}, and the third from Corollary \ref{old_lowering_argument}. The inductive step follows.

\end{proof}

We finally have all of the preparatory results we require in order to complete the proof of Theorem \ref{NakBPS_prop}
\begin{proof}[Proof of Theorem \ref{NakBPS_prop}]
We prove the theorem by induction on $\dd$.  The base case $\dd=\ul{0}$ follows by the definition of $\Psi_{\ff,\ul{0}}$.  So now we assume that the theorem holds for all $\ee<\dd$, so we can apply Lemma \ref{half_comp_thm}.  By Corollary \ref{Nak_BPS_comp} there is an isomorphism of cohomologically graded vector spaces $\NakMod^{T_{\ff}^0}_{Q,\ff,\dd}\cong \Fg^{T_{\ff}^0}_{\Pi_{Q_{\ff}},(\dd,1)}$, so in particular their cohomologically graded pieces have the same (finite) dimensions.  By Lemma \ref{half_comp_thm}, the image of the morphism $\Psi_{\ff,\dd}\colon \NakMod^{T^0_{\ff}}_{Q,\ff,\dd}\rightarrow\HO\!\CoHA_{\Pi_{Q_{\ff}}}^{T^0_{\ff}} $ contains $\Fg^{T^0_{\ff}}_{\Pi_{Q_{\ff}},(\dd,1)}\subset \HO\!\CoHA_{\Pi_{Q_{\ff}}}^{T^0_{\ff}}$.  The theorem follows.
\end{proof}

We can strengthen Theorem \ref{NaSE_vs_BPS}, enabling the direct comparison between nonabelian stable envelopes and inclusion of BPS cohomology, via the following corollary.
\begin{corollary}
\label{ThmC_cor}
The following diagram commutes.
\[
	\begin{tikzcd}
	\Fg^{T_{\ff}^0}_{\Pi_{Q_{\ff}},(\dd,1)}\arrow[d,"d"]\arrow[rd,"j'"]& \\
	\NakMod^{T^0_{\ff}}_{Q,\ff,\dd}\arrow[r,"\Psi_{\ff,\dd}"]&\HO\!\CoHA_{\Pi_{Q_{\ff}}}^{T^0_{\ff}}
	\end{tikzcd}
\]
where the vertical morphism is given in Corollary \ref{Nak_BPS_comp}, $j'$ is the inclusion of BPS cohomology, and we have denoted by $\Psi_{\ff,\dd}$ the restriction of $\Psi_{\ff,\dd}$ to $\NakMod^{T^0_{\ff}}_{Q,\ff,\dd}$.
\end{corollary}
\begin{proof}
The statement is equivalent to the claim that the diagram
\[
	\begin{tikzcd}
	\HO_{\BoC^*}\otimes \Fg^{T_{\ff}^0}_{\Pi_{Q_{\ff}},(\dd,1)}\arrow[d,"\overline{d}"]\arrow[rd,"j"]& \\
	\NakMod^{T_{\ff}}_{Q,\ff,\dd}\arrow[r,"\Psi_{\ff,\dd}"]&\HO\!\CoHA_{\Pi_{Q_{\ff}}}^{T^0_{\ff}}
	\end{tikzcd}
\]
commutes, with $j$ as in \eqref{jdef}.  By Theorem \ref{NakBPS_prop} both $\Psi_{\ff,\dd}$ and $j\circ\overline{d}^{-1}$ have the same image, and combining Remark \ref{toda_rem} and Corollary \ref{corollary normalization nonabelian stable envelope} we see that they are sections to the restriction morphism $\HO\!\CoHA_{\Pi_{Q_{\ff}}}^{T^0_{\ff}}\rightarrow \NakMod^{T_{\ff}}_{Q,\ff,\dd}$.
\end{proof}
The following generalisation of Theorem \ref{NaSE_vs_BPS} follows directly from it and the construction of the nonabelian stable envelope (see Remark \ref{remark NS and stability conditions}).
\begin{corollary}
\label{ThmC_cor2}
Let $\dd,\ff\in\BoN^{Q_0}$ be arbitrary dimension vectors with $\ff$ nonzero.  Let $\zeta\in\BoQ^{(Q_{\ff})_0}$ be generic for the dimension vector $(\dd,1)$.  Then $\Psi^{\zeta}_{\ff,\dd}\colon \HO^{\BoMo}(\Nak^{\zeta, T^0_\ff}_{Q}(\ff,\dd), \BoQ^{\vir})\rightarrow \HO\!\CoHA_{\Pi_{Q_{\ff}}}^{T^0_{\ff}}$ is injective, with image $\Fg^{T^0_{\ff}}_{\Pi_{Q_{\ff}},(\dd,1)}$.
\end{corollary}
\subsection{Proof of Theorem \ref{main_thm}}
\label{main_thm_sec}
Given a dimension vector $\ff\in\BoN^{Q_0}$ we define $\mathbb{E}_{\ff}=\End_{\HO_{T^0_{\ff}}}\left(\NakMod^{T^0_{\ff}}_{Q,\ff}\right)$, which we consider as a $\HO_T$-linear Lie algebra, under the commutator Lie bracket.  For each $\ff$ we have, by construction, a morphism $\imath_{\ff}^{\MO}\colon\Fg^{\MO,T}_{Q}\rightarrow \mathbb{E}_{\ff}$.  We define $\imath^{\MO}=\prod_{i\in Q_0}\imath^{\MO}_{\delta_i}\colon \Fg^{\MO,T}_Q\rightarrow \mathbb{E}\coloneqq \prod_{i\in Q_0}\mathbb{E}_{\delta_{i}}$.  We denote by $\imath^{\MO,+}$ the restriction of $\imath^{\MO}$ to $\Fn^{\MO,T,+}_{Q}$.
\begin{proposition}\cite[Prop.5.3.4]{MO19}
The morphism $\imath^{\MO}$ is an embedding of Lie algebras.
\end{proposition}
Pick $\ff\in\BoN^{Q_0}$.  By Theorem \ref{NaSE_vs_BPS}, the morphisms $\Psi_{\ff,\dd}$ for $\dd\in\BoN^{Q_0}$ provide an isomorphism
\[
\Psi_{\ff,\bullet}\colon \NakMod^{T_{\ff}^0}_{Q,\ff}\xrightarrow{\cong}\Fg^{T_{\ff}^0}_{\Pi_{Q_{\ff}},(\bullet,1)},
\]
and thus an isomorphism 
\begin{align*}
J\colon &\mathbb{E}_{\ff} \xrightarrow{\cong} \End_{\HO_{T_{\ff}^0}}(\Fg^{T_{\ff}^0}_{\Pi_{Q_{\ff}},(\bullet,1)})\\
&g\mapsto \Psi_{\ff,\bullet}\circ g\circ \Psi_{\ff,\bullet}^{-1}.
\end{align*}
We embed $\Fg^T_{\Pi_Q}\subset \Fg^{T_{\ff}^0}_{\Pi_{Q_{\ff}}}$ via extension of scalars and the isomorphism $l_*\colon \Fg^{T_\ff^0}_{\Pi_Q,\bullet}\cong \Fg^{T_\ff^0}_{\Pi_{Q_{\ff}},(\bullet,0)}$.  Then $\Fg^T_{\Pi_Q}$ acts on $\Fg^{T_\ff^0}_{\Pi_{Q_{\ff}},(\bullet,1)}$ via the Lie bracket.  Composing with $J^{-1}$ yields a morphism of Lie algebras \[
\imath^{\BPS}_{\ff}\colon \Fg_{\Pi_Q}^{T}\rightarrow \mathbb{E}_{\ff}.  
\]
We set $\imath^{\BPS}\coloneqq \prod_{i\in Q_0} \imath^{\BPS}_{\delta_i}\colon\Fg_{\Pi_Q}^{T}\rightarrow \mathbb{E}$. We denote by $\imath^{\BPS,+}$ the restriction to $\Fn_{\Pi_Q}^{T,+}$.
\begin{proposition}
The morphism $\imath^{\BPS}$ is injective.
\end{proposition}
\begin{proof}
Since the morphism is a morphism of free $\HO_T$-modules, it is sufficient to prove the proposition for the special case $T=\{1\}$, equivalently $\HO_T=\BoQ$.  

Fix a dimension vector $\dd\in\BoN^{Q_0}\setminus \{0\}$, and pick $i\in \supp(\dd)$. Set $\ff=\delta_i$. We consider the action of $\mathfrak{sl}_2$ given by the $\mathfrak{sl}_2$-triple $\{e_{\infty},f_{\infty},h_{\infty}\}$ on the left module $\bigoplus_{n\in\BoN}\Fg_{\Pi_{Q_{\ff}}(\dd,n)}$.  The lowest weight space of this $\mathfrak{sl}_2$-module is given by $\Fg_{\Pi_{Q_{\ff}},(\dd,0)}$, and it has weight $(\delta_{\infty},(\dd,0))_{Q_{\ff}}=-\ff\cdot\dd\neq 0$.  In particular, for every nonzero $\alpha\in \Fg_{\Pi_{Q_{\ff}},(\dd,0)}$ we find $[e_{\infty},\alpha]\neq 0$.  Since $e_{\infty}\in\fg_{\Pi_{Q_{\ff}},(\underline{0},1)}$ we deduce that $\imath^{\BPS}(\alpha)\neq 0$.  

If $\beta\in \Fg_{\Pi_Q,-\dd}\cong \Fg_{\Pi_Q,\dd}^{\vee}$ satisfies $\beta(\alpha)\neq 0$, then we calculate $[[e_{\infty},\alpha],\beta]=\beta(\alpha)(\ff\cdot \dd) e_{\infty}$, and so $\beta$ also acts faithfully.
\end{proof}
\begin{proposition}
\label{good_residue_prop}
Let $\dd\in\BoN^{Q_0}$, and let $\ff'\in\BoN^{Q_0}$ satisfy $\ff'\cdot \dd\neq 0$.  Let $\alpha\in\Fg^{T_{\ff'}^0}_{\Pi_{Q_{\ff'}},(\dd,1)}$.  Then 
\[
\alpha_{(1)}\coloneqq\Res_{x_{(1),\infty,1}}\vDelta_{(\underline{0},1),(\dd,0)}(\alpha)\in\hbar \cdot\Fg^{T_{\ff'}^0}_{\Pi_{Q},\dd}\subset \HO\!\CoHA^{T^0_{\ff'}}_{\Pi_Q,\dd}.
\]
\end{proposition}
\begin{proof}
Set $\ff''=\delta_j$ with $j\in \supp(\dd)$.  By Proposition \ref{alpha_one_perv} $\alpha_{(1)}\in \FP^0\HO\!\CoHA^{T_{\ff'}^0}_{\Pi_Q}$.  We embed $\FP^0\HO\!\CoHA^{T_{\ff'}^0}_{\Pi_Q}\subset \FP^0\HO\!\CoHA_{\Pi_{Q_{\ff',\ff''}}}^{T_{\ff'}^0}$ via the natural isomorphisms $\FP^0\HO\!\CoHA^{T_{\ff'}^0}_{\Pi_Q,\dd}\cong \FP^0\HO\!\CoHA^{T_{\ff'}^0}_{\Pi_{Q_{\ff',\ff''}},(\dd,0,0)}$.  We denote by $e_{\infty''}$ the unique (up to multiplication by a nonzero scalar) Chevalley generator in $\Fg^{T_{\ff'}^0}_{\Pi_{Q_{\ff',\ff''}},(\ul{0},0,1)}$.  Then we claim that it is sufficient to show that 
\[
\gamma=[e_{\infty''},\alpha_{(1)}]\in \hbar\cdot \Fg^{T_{\ff'}^0}_{\Pi_{Q_{\ff',\ff''}},(\dd,0,1)}.
\]
For assume that this is the case, then $[f_{\infty''},\gamma]\in \hbar\cdot\Fg^{T_{\ff'}^0}_{\Pi_{Q_{\ff',\ff''}},(\dd,0,0)}$, but we calculate $[f_{\infty''},\gamma]=(\ff''\cdot \dd)\alpha_{(1)}$.

By Theorem \ref{NaSE_vs_BPS} we can find $\ol{\alpha}$ satisfying $\Psi_{\ff',\dd'}(\ol{\alpha})=\alpha$.  By construction, there exists $\ol{e}_{\infty''}\in \NakMod^{T_{\ff'}^0}_{Q,\ff'',\ul{0}}$ such that $\Psi_{\ff'',\ul{0}}^{(2)}(\ol{e}_{\infty''})=e_{\infty''}$.  We have the equalities
\[
[e_{\infty''},\alpha_{(1)}]=\Res_{x_{(1),\infty',1}}\vDelta_{(\ul{0},1,0),(\dd,0,1)}([e_{\infty''},\alpha])
=-\hbar\cdotsh\Psi_{\ff'',\dd}(\psi(\ol{\alpha})(\ol{e}_{\infty''})),
\]
where the first follows from Corollary \ref{Res_cor_2} and the second from Lemma \ref{cor_pack}.  Then the inclusion $\gamma\in  \hbar\cdotsh \Fg^{T_{\ff'}^0}_{\Pi_{Q_{\ff',\ff''}},(\dd,0,1)}$ follows by Theorem \ref{NaSE_vs_BPS}.
\end{proof}

We have provided embeddings $\imath^{\MO}\colon \fg^{\MO,T}_{Q}\hookrightarrow \mathbb{E}$ and $\imath^{\BPS}\colon \Fg^T_{\Pi_Q}\hookrightarrow \mathbb{E}$.  The restriction of Theorem \ref{main_thm} to the positive halves of the two Lie algebras we are comparing is a consequence of the following theorem.

\begin{theorem}
\label{main_thm_redo}
There is an equality $\imath^{\MO,+}(\Fn^{\MO,T,+}_{Q})=\imath^{\BPS,+}(\Fn^{T,+}_{\Pi_Q})$ of Lie subalgebras of $\mathbb{E}$. Moreover, the induced isomorphism $\Lambda\colon \Fn^{\MO,T,+}_{Q}\rightarrow \Fn^{T,+}_{\Pi_Q}$ intertwines the actions of these two Lie algebras on $\NakMod^{T_{\ff''}^0}_{Q,\ff''}$ for each $\ff''\in\BoN^{Q_0}$.
\end{theorem}
\begin{proof}
Fix $\ff'=\delta_i$ for some $i\in Q_0$.  Let $\ol{\alpha}\in \NakMod^T_{Q,\ff',\dd'}$, and let $\ol{\beta}\in \NakMod^{T^0_{\ff''}}_{Q,\ff'',\dd''}$ with $\ff''\in\BoN^{Q_0}$.  Set $\dd=\dd'+\dd''$.  By Lemma \ref{cor_pack} and Corollary \ref{Res_cor_2} we have
\begin{align*}
\Psi_{\ff'',\dd}(\psi(\ol{\alpha})(\ol{\beta}))=&\hbar^{-1}\Res_{x_{(1),\infty',1}}\vDelta_{(\ul{0},1,0),(\dd,0,1)}[\Psi^{(1)}_{\ff',\dd'}(\ol{\alpha}),\Psi^{(2)}_{\ff'',\dd''}(\ol{\beta})]\\
=&\hbar^{-1}[l''_*\Res_{x_{(1),\infty,1}}\vDelta_{(\ul{0},1),(\dd',0)}\Psi_{\ff',\dd'}(\ol{\alpha}),\Psi_{\ff'',\dd''}(\ol{\beta})].
\end{align*}
By Theorem \ref{NakBPS_prop} and Proposition \ref{good_residue_prop} we have $l''_*\Res_{x_{(1),\infty,1}}\vDelta_{(\ul{0},1),(\dd',0)}\Psi_{\ff',\dd'}(\ol{\alpha})\in\hbar\cdot \Fg^{T_{\ff''}^0}_{\Pi_{Q_{\ff''}},(\dd',0)}$ and so $\imath^{\MO,+}$ factors through $\imath^{\BPS,+}$.  Precisely, we define $\Lambda(\psi(\overline{\alpha}))\coloneqq \hbar^{-1}\cdot \Res_{x_{(1),\infty,1}}\vDelta_{(\ul{0},1),(\dd',0)}\Psi_{\ff',\dd'}(\ol{\alpha})$.  Since we chose $\ff''$ arbitrary, the statement regarding intertwining of representations follows from the same calculation and Proposition \ref{lowest_weight_prop}.

Let $\gamma\in \Fg^T_{\Pi_Q,\dd}$ be a Chevalley generator, and fix $\ff=\delta_i$ with $i\in \supp(\dd)$.  By Lemma \ref{weak_F1_lem} there exists a $\ol{\beta}\in \NakMod^T_{Q,\ff,\dd}$ for which $\Psi_{\ff,\dd}(\overline{\beta})=[e_{\infty},\gamma]$.   From the above calculation and Corollary \ref{old_lowering_argument} we deduce that $\imath^{\MO}(\psi(\ol{\beta}))=\imath^{\BPS}(\gamma)$.  This shows that $\imath^{\MO}$ surjects onto the image of $\imath^{\BPS}$, as well as being injective, and the theorem is proved.
\end{proof}

\begin{remark}
\label{redundancy_remark}
In \S \ref{MO_LA_sec} we defined the operators $\psi(\ol{\alpha})$ for $\alpha\in\NakMod^T_{Q,\ff,\dd}$ with $\ff\in\BoN^{Q_0}$ not necessarily equal to $\delta_i$ for some $i\in Q_0$.  The proof of Theorem \ref{main_thm_redo}, shows that these operators are identified with taking the Lie bracket $[\hbar^{-1}\cdotsh l''\Res_{x_{(1),\infty,1}}\vDelta_{(\ul{0},1),(\dd,0)}\Psi_{\ff,\dd}(\ol{\alpha}),-]$.  Then Proposition \ref{good_residue_prop}, along with Theorem \ref{main_thm}, demonstrates that these operators are in $\Fg^{\MO,T}_Q$.  In particular, adding these operators in to the definition of $\Yang^{\MO}_Q$ would be redundant.
\end{remark}

\begin{theorem}[Theorem \ref{main_thm}]
\label{main_thm_rdx}
There is an isomorphism $\fg_Q^{\MO,T}\cong \Fg_{\Pi_Q}^T$ intertwining the actions of these two Lie algebras on $\NakMod^{T^0_{\ff}}_{Q,\ff}$ for arbitrary $\ff\in\BoN^{Q_0}\setminus\{0\}$. Moreover, this isomorphism preserves the cohomological degree.
\end{theorem}
\begin{proof}
Fix a dimension vector $\dd\in\BoN^{Q_0}$.  We write 
\[
\fg^T_{\Pi_Q,\dd}=\sum_{\dd'+\dd''=\dd}\Image(\fg^T_{\Pi_Q,\dd'}\otimes_{\HO_T}\fg^T_{\Pi_Q,\dd''}\xrightarrow {[-,-]}\fg^T_{\Pi_Q,\dd})\oplus E_{\dd}
\]
where $E_{\dd}=\HO^*(\CM^T_{\dd}(\Pi_Q),\CG(Q)_{\dd})$.  Set $E'_{\dd}=\Lambda^{-1}(E_{\dd})$.  The perfect pairing $\fg_{Q,\dd}^{\MO,T}\otimes_{\HO_T}\Fg_{Q,-\dd}^{\MO,T}\rightarrow \HO_T$ is constructed in \cite[\S 5.3]{MO19} via the commutator pairing.  Set $F'_{-\dd}=(E'_{\dd})^{\vee}\subset \fg^{\MO,T}_{Q,-\dd}$.  Then for $f'\in F'_{-\dd}$ and $e'\in E'_{\dd}$ we get, by construction $[e',f']=f'(e')h_{\dd}\in \bar\fh^{\MO,T}_{Q}\setminus\{0\}$.  By Proposition \ref{lowest_weight_prop}, $\NakMod_{Q,\ff}^{T_{\ff}^0}$ is a direct sum of irreducible lowest weight modules for the action of $\fg_{\Pi_Q}^T$.  So as a $\fg_{Q}^{\MO,T}$-module, we can write arbitrary elements of $\NakMod_{Q,\ff}^{T_{\ff}^0}$ as linear combinations of elements $[\ldots [\alpha,\beta'_l],\ldots],\beta'_1]$ where $\alpha$ is a lowest weight vector for the action of $\Fg_{\Pi_Q}$, and the $\beta'$s belong to spaces $E'_{\dd'}$ for $\dd'\in\BoN^{Q_0}$.  It follows that the action of $F'_{\dd}$ on $\NakMod_{Q,\ff}^{T_{\ff}^0}$ matches that of the Chevalley lowering operators in $\fg_{\Pi_Q}^T$ if and only if each such $\alpha$ is a lowest weight vector for the action of $\fg^{\MO,T}_{Q}$.  This is the statement of Proposition \ref{MO_lowest_weights}.
\end{proof}

\subsection{Proof of Theorem \ref{OC_thm} and Corollaries \ref{OQ_cor} -- \ref{faithful_cor}}
\label{corrs_sec}
From Proposition \ref{Triv_BPS_ext} and Theorem \ref{main_thm} we deduce that there is an isomorphism $\Fg_Q^{\MO,T}\cong \Fg_{\Pi_Q}\otimes \HO_T$, so the Maulik--Okounkov Lie algebra is obtained from a Lie algebra $\Fg_Q^{\MO}$ defined over $\BoQ$, by extending scalars. Then Theorem \ref{OC_thm} is a consequence of Theorem \ref{main_thm} and Proposition \ref{BPS_char_function}. 
 Corollary \ref{OQ_cor} is a consequence of Theorem \ref{KMA_thm}.

In a little more detail: for $\dd\in \Sigma_Q$ set $V_{\dd}=\IH^*(\CM_{\dd}(\Pi_Q),\BoQ^{\vir})$, and for $\dd=n\dd'$ with $\dd\in \Phi^+_Q\setminus \Sigma_Q$ an imaginary isotropic root with $\dd'\in\Sigma_Q$, set $V_{\dd}=\IH^*(\CM_{\dd'}(\Pi_Q),\BoQ)$.  Then we have $\Fg^{\MO}_Q\coloneqq \Fg_V$, the generalised Kac--Moody Lie algebra determined by the $\Phi^+_Q$-graded cohomologically graded vector space $V$ as defined in \S \ref{GKM_sec} and $\Fg_Q^{\MO,T}\cong \Fg_Q^{\MO}\otimes \HO_T$.

By Theorem \ref{PBW_thm}, $\HO\!\CoHA_{\Pi_Q}^T$ is generated by $\fg_{\Pi_Q}^T$ and the action of tautological classes.  Corollary \ref{Y_cor} then follows from Proposition \ref{gen_by_taut}.  Then Corollary \ref{faithful_cor} is itself a corollary of Corollary \ref{Y_cor}: by definition, the action of $\Yang^{\MO,+}_Q$ on $\bigoplus_{\ff\in\BoN^{Q_0}}\NakMod_{Q,\ff}^{T_{\ff}}$ is faithful.

\bibliographystyle{alpha}
\bibliography{Literatur}

\vfill

\end{document}